\documentclass[twoside,11pt]{article}

\usepackage{jmlr2e}
\usepackage{amsmath}
\usepackage{algorithm, algpseudocode}
\usepackage{mathtools}
\usepackage{xcolor}
\usepackage{caption}
\usepackage{upgreek}
\usepackage{dsfont}
\usepackage{pifont}
\usepackage{array}
\usepackage{ragged2e}
\usepackage{subcaption}
\usepackage{titletoc}
\usepackage{bm}
\usepackage[capitalise]{cleveref}

\newtheorem{assm}{Assumption}
\newcommand{\R}{\mathbb{R}}
\newcommand{\E}{\mathbb{E}}
\newcommand{\bigO}{\mathcal{O}}
\newcommand{\IdxB}{\mathcal{B}}
\newcommand{\Otil}{\widetilde{\mathcal{O}}}
\newcommand{\bigOmega}{\Upomega}
\newcommand{\Ps}{P^{(s)}_k}
\newcommand{\Psin}{(P^{(s)}_k)^{-1}}
\newcommand{\Lmc}{\mathcal{L}}
\newcommand{\Model}{\mathcal{M}}
\newcommand{\gammaMax}{\gamma_{u}^{\textup{max}}}
\newcommand{\gammaMin}{\gamma_{\ell}^{\textup{min}}}
\newcommand{\gammatlm}{{\tilde{\gamma}_{\ell}}^{\textrm{min}}}
\newcommand{\deff}{d^{\nu}_{\textup{eff}}}
\newcommand{\nys}{Nystr\"{o}m}
\newcommand{\diag}{\textrm{diag}}
\newcommand{\ssn}{\textsc{SSN}}
\newcommand{\nyssn}{\textsc{NySSN}}
\newcommand{\sassn}{\textsc{SASSN}}
\newcommand{\sassnr}{\textsc{SASSN-R}}
\newcommand{\sassnc}{\textsc{SASSN-C}}
\newcommand{\diagssn}{\textsc{DiagSSN}}
\newcommand{\pssn}{\mathcal{P}_{\textrm{ssn}}}
\newcommand{\pnyssn}{\mathcal{P}_\textrm{nyssn}}
\newcommand{\psassnr}{\mathcal{P}_\textrm{sassn-r}}
\newcommand{\psassnc}{\mathcal{P}_\textrm{sassn-c}}
\newcommand{\pdiagssn}{\mathcal{P}_\textrm{diagssn}}

\newcommand{\cmark}{\ding{51}}%
\newcommand{\xmark}{\ding{55}}%
\newcommand{\q}{\mathfrak{q}}
\newcommand{\qbar}{\bar{\mathfrak{q}}}
\newcommand{\qmax}{\mathfrak{q}_{\textup{max}}}
\newcommand{\kappaMax}{\kappa_{\textup{max}}}
\newcommand{\LMax}{L_{\textup{max}}}
\newcommand{\Nstar}{\mathcal{N}_{\varepsilon_0}(w_\star)}

\newcolumntype{L}[1]{>{\raggedright\let\newline\\\arraybackslash\hspace{0pt}}m{#1}}
\newcolumntype{C}[1]{>{\centering\let\newline\\\arraybackslash\hspace{0pt}}m{#1}}
\newcolumntype{R}[1]{>{\raggedleft\let\newline\\\arraybackslash\hspace{0pt}}m{#1}}

\algdef{SE}[SUBALG]{Indent}{EndIndent}{}{\algorithmicend\ }%
\algtext*{Indent}
\algtext*{EndIndent}

\algdef{SE}[SUBALG]{Function}{EndFunction}[2]{\textbf{function} #1(#2)}{}%
\algtext*{EndFunction}

\algnewcommand\algorithmicinit{\textbf{Initialize:}}
\algnewcommand\Init{\item[\algorithmicinit]}

\crefname{assm}{Assumption}{Assumptions}

\newtheorem{theorem}{Theorem}
\newtheorem{lemma}[theorem]{Lemma} 
\newtheorem{proposition}[theorem]{Proposition} 

\newtheorem{corollary}[theorem]{Corollary}
\newtheorem{definition}[theorem]{Definition}

 \newcommand{\znote}[1]{}
\newcommand{\pnote}[1]{}
 \newcommand{\snote}[1]{}
 \newcommand{\mnote}[1]{}
\renewcommand{\znote}[1]{\textcolor{blue}{\textbf{[ZF: #1]}}}
\renewcommand{\pnote}[1]{\textcolor{red}{\textbf{[PR: #1]}}}
\renewcommand{\snote}[1]{\textcolor{green}{\textbf{[SZ: #1]}}}
\renewcommand{\mnote}[1]{\textcolor{purple}{\textbf{[MU: #1]}}}

\newif\ifpreprint
\preprintfalse

\footnotetext{Equal contribution. A longer, more detailed version of this manuscript is available at \texttt{https://arxiv.org/abs/2309.02014v2}.}

\ShortHeadings{PROMISE: Preconditioned Stochastic Optimization Methods}{Frangella, Rathore, Zhao, and Udell}
\firstpageno{1}

\begin{document}

\title{PROMISE: Preconditioned Stochastic Optimization Methods by Incorporating Scalable Curvature Estimates}

\author{\name Zachary Frangella$^{*}$ \email zfran@stanford.edu \\
       \addr Stanford University \\
       \AND
        \name Pratik Rathore$^{*}$ \email pratikr@stanford.edu \\
       \addr Stanford University \\
       \AND 
       \name Shipu Zhao  \email sz533@cornell.edu \\
       \addr Cornell University
       \AND
       \name Madeleine Udell \email udell@stanford.edu \\
       \addr Stanford University}

\editor{}

\maketitle

\begin{abstract}
Ill-conditioned problems are ubiquitous in large-scale machine learning:
as a dataset grows to include more and more features correlated with the labels,
the condition number increases.
Yet traditional stochastic gradient methods
converge slowly on these ill-conditioned problems,
even with careful hyperparameter tuning.
This paper introduces PROMISE (\textbf{Pr}econditioned Stochastic \textbf{O}ptimization \textbf{M}ethods by \textbf{I}ncorporating \textbf{S}calable Curvature \textbf{E}stimates), a suite of sketching-based preconditioned stochastic gradient algorithms 
that deliver fast convergence on ill-conditioned large-scale convex optimization problems arising in machine learning.
PROMISE includes preconditioned versions of SVRG, SAGA, and Katyusha;
each algorithm comes with a strong theoretical analysis and
effective default hyperparameter values. 
Empirically, we verify the superiority of the proposed algorithms 
by showing that, using default hyperparameter values, 
they outperform or match popular \emph{tuned} stochastic gradient optimizers 
on a test bed of $51$ ridge and logistic regression problems 
assembled from benchmark machine learning repositories.
On the theoretical side, this paper introduces the notion of \emph{quadratic regularity}
in order to establish linear convergence of all proposed methods
even when the preconditioner is updated infrequently. 
The speed of linear convergence is determined by the \emph{quadratic regularity ratio}, 
which often provides a tighter bound on the convergence rate compared to the condition number, 
both in theory and in practice, 
and explains the fast global linear convergence of the proposed methods. 
\end{abstract}


\newpage

\ifpreprint
\startcontents[sections] 
\section*{Table of Contents}
\printcontents[sections]{}{1}{} 

\clearpage
\else
\fi

\section{Introduction}
\label{section:Introduction}
Modern machine learning (ML) poses significant challenges for optimization, owing to the sheer scale of the problems. 
Modern datasets are both enormous and high-dimensional, 
often with millions of samples and features.
As a consequence, classic methods such as gradient descent and L-BFGS, 
which make a full pass through the data at each iteration, 
are prohibitively expensive.
In this context, stochastic gradient descent (SGD) and its variants,
which operate on only a small mini-batch of data at each iteration,
have become the dominant optimization methods for modern ML.

When the problem is well-conditioned, 
SGD quickly finds models that are nearly optimal.
Further, although classic SGD converges to a ball around the optimum (with fixed learning rate) or sublinearly (with decaying learning rate) \citep{moulines2011non,gower2019sgd}, 
\emph{variance reduction} techniques like SVRG, SAGA, Katyusha, and L-Katyusha significantly improve performance on convex problems,
and converge linearly to the optimum for strongly convex problems \citep{johnson2013accelerating,defazio2014saga,allenzhu2018katyusha,kovalev2020lkatyusha}.

\begin{figure}[h]
    \centering
    \includegraphics[scale=0.7]{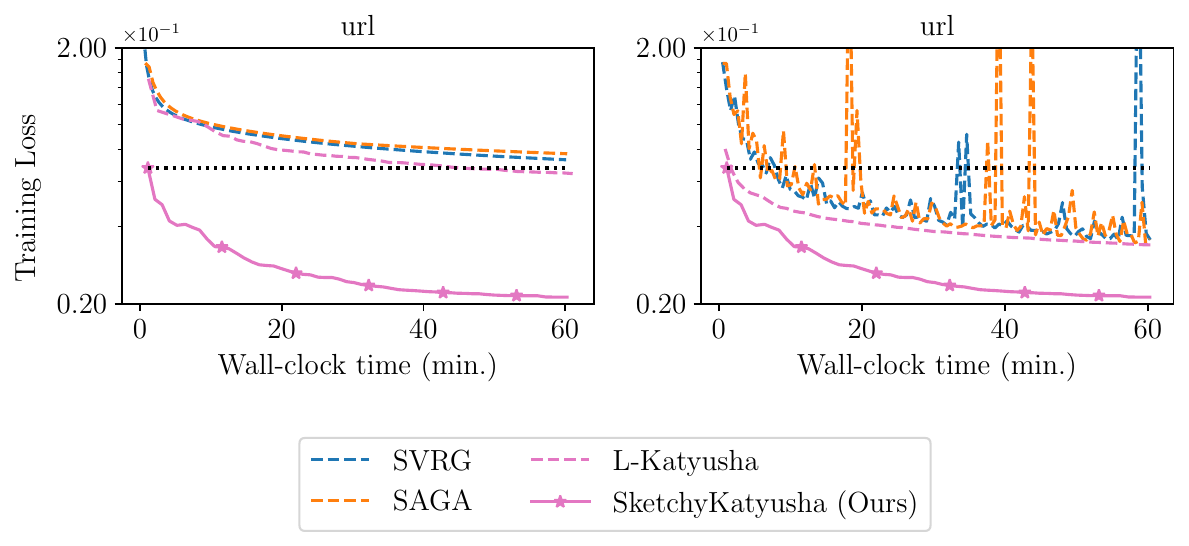}
    \caption{SketchyKatyusha (an algorithm in the PROMISE suite, see \cref{alg:skkat}) with its default hyperparameters 
    outperforms standard stochastic gradient optimizers
    with both default (left) and tuned (right) hyperparameters. 
    The loss curves start after a single epoch of training has been completed; the black dotted line indicates the training loss attained by SketchyKatyusha after a single epoch. 
    Each optimizer is allotted $1$ hour of runtime.}
    \label{fig:url_intro}
\end{figure}

Unfortunately SGD and related algorithms are difficult to tune and converge slowly  
when the data is poorly conditioned.
Parameters like the learning rate are difficult to choose and important to get right: 
slow convergence or divergence loom on either side of the best parameter choice \citep{nemirovski2009robust}.
Even with the best learning rate, in the worst case,
variance-reduced stochastic gradient methods require  at least $\bigO(\left(n+\sqrt{n\kappa}\right) \log (1/\epsilon))$ stochastic gradient evaluations to reach $\epsilon$ accuracy \citep{woodworth2016tight}, 
 which is problematic as the condition number for many ML problems is typically on the order of $10^{4}$ to $10^{8}$ (see Figure 13 in \cite{frangella2023sketchysgd}).
Hence the convergence of stochastic gradient methods can be excruciatingly slow (see \cref{fig:url_intro})
and popular stochastic optimizers often provide low-quality solutions even with generous computational budgets.

How should the challenges of ill-conditioning and sensitivity to the learning rate be addressed? 
Classical optimization wisdom suggests using second-order information based on the Hessian.
Second-order methods converge locally at superlinear rates under mild assumptions and 
beat first-order methods in practice \citep{nocedal1999numerical, boyd2004convex}.

Many authors have attempted to develop second-order methods 
that use stochastic gradients and Hessians, 
but major difficulties remain; 
see \cref{section:related_work} and especially \cref{table:2nd-ord-comp} for details.
First, no previous stochastic second-order method delivers fast local-linear convergence 
without vanishing noise in the gradient estimates,
as noted in \cite{kovalev2019stochastic,frangella2023sketchysgd}.
The strategems used to reduce this noise,
including exponentially increasing gradient batchsizes and Hessian batchsizes 
that may depend on the condition number,
lead to exceedingly slow iterations as the algorithm converges.
Second, even the methods 
that allow a stochastic Hessian require an expensive (and slow) 
new estimate of the Hessian at every iteration.
Finally, most of these methods are difficult to deploy in real-world ML pipelines,
as they introduce new hyperparameters without practical guidelines for choosing them.


In this paper, we introduce PROMISE, 
a suite of preconditioned stochastic gradient methods that use 
scalable curvature estimates to directly address each of these problems.
PROMISE methods estimate second-order information
from minibatch data (i.e., stochastic Hessians)
to avoid difficulties with hyperparameter selection and ill-conditioning, 
and they use infrequent (``lazy'') Hessian updates. 
The resulting algorithms, such as SketchyKatyusha (\cref{alg:skkat}), 
are fast enough per-iteration to compete with first-order methods,
yet converge much faster on ill-conditioned problems
with minimal or no hyperparameter tuning.

\Cref{fig:url_intro} illustrates the benefits of PROMISE by applying  SketchyKatyusha to a malicious link detection task using the url dataset ($n$ = 2,396,130, $p$ = 3,231,961), which yields a large $l^2$-regularized logistic regression problem. 
Popular stochastic optimizers, using the default learning rates, perform poorly.
In contrast, with default hyperparameters, 
the proposed method SketchyKatyusha achieves a loss three times smaller 
than the best competing method after an hour of training!
In fact, even after an hour of training, 
the other optimizers barely match the training loss that 
SketchyKatyusha achieves after a single epoch (pass through the data).
Even with extensive hyperparameter tuning 
(which, in practice, increases the cost of optimization by orders of magnitude), 
the other first-order methods cannot 
match the performance of SketchyKatyusha with its default hyperparameters.

PROMISE also improves on preexisting theory for stochastic second-order methods. 
In contrast to prior approaches, 
PROMISE methods achieve linear convergence with lazy updates to the preconditioner 
and without large batchsizes for the gradient and Hessian.
Significantly, PROMISE methods come with default hyperparameters 
(including the learning rate) that enable them to work out-of-the-box 
and outperform or match popular stochastic optimizers \emph{tuned} 
to achieve their best performance. 
Numerical experiments on a test bed of 51 ridge and logistic-regression problems
verify this claim.
Hence our methods avoid the usual theory-practice gap: 
our theoretical advances yield practical algorithms.

In order to show linear convergence under lazy updates, 
our analysis introduces a new analytic quantity  
we call the \emph{quadratic regularity ratio} that  
controls the convergence rate of all PROMISE methods.
The quadratic regularity ratio generalizes the condition number to the Hessian norm.
Unlike the condition number, the quadratic regularity ratio equals one for quadratic objectives and approaches one as the iterate approaches the optimum for any objective with a Lipschitz Hessian. 
Hence the quadratic regularity ratio often gives tighter convergence rates than the condition number and explains why the proposed methods empirically exhibit fast global linear convergence and outperform the competition. 

\subsection{PROMISE}

PROMISE methods solve convex finite-sum minimization (FSM) problems of the form
 \begin{equation}
 \tag{FSM}
 \label{eq:EmpRiskProb}
     \underset{w\in \mathbb{R}^p}{\textrm{minimize}}~F(w) := \frac{1}{n}\sum_{i=1}^{n}f_i(w)+\frac{\nu}{2}\|w\|^2,
 \end{equation}
where each $f_i$ is real-valued, smooth, and convex, and $\nu >0$.

We provide a high-level overview of the PROMISE methods in the \hyperref[alg:PSGM]{Meta-algorithm} and \cref{table:inputs_description}. 
Each iteration consists of two phases: 
a (lazy) preconditioner update and a parameter update.
By default, 
PROMISE methods update the preconditioner at a fixed frequency, 
such as once per epoch,
using a stochastic Hessian estimate at the current iterate $w_k$. 
The learning rate is then recomputed to adapt to the new preconditioner.
For the parameter update, 
our methods compute a stochastic gradient $g_k$ and 
preconditioned direction $v_k = P^{-1} g_k$.
Our methods then use a parameter update subroutine $\mathcal S$ 
to compute the next iterate $w_{k+1}$. 
The $*$ in the call to $\mathcal S$ denotes additional arguments 
to perform variance reduction and acceleration.

\begin{algorithm}[h]
\scriptsize
	\centering
	\caption*{\textbf{Meta-algorithm:} PROMISE}
	\label{alg:PSGM}
	\begin{algorithmic}
        \Require{
		initial iterate $w_0$, stochastic gradient oracle $\mathcal O_{g}$, stochastic Hessian oracle $\mathcal O_{H}$, gradient and Hessian batchsizes $b_g$ and $b_H$, preconditioner object $\mathcal P$, preconditioner update times $\mathcal U \subseteq \mathbb{N}$, parameter update subroutine $\mathcal S$}

        \vspace{0.5pc}
  
        \For{$k = 0, 1, \ldots$}
        
        
        \State {\textbf{\# Preconditioner update}}
        \If{$k \in \mathcal U$}
        \State {$\mathcal P.\texttt{update}(\mathcal O_{H}(w_k, b_H))$} \Comment{Update preconditioner $P$ via stochastic Hessian}
        \State {$\eta = \mathcal P.\texttt{get\_learning\_rate}()$} \Comment{Compute learning rate based on $P$}
        \EndIf

        \vspace{0.5pc}
        
        \State {\textbf{\# Parameter update}}
        \State $g_k = \mathcal{O}_g(w_k, b_g)$ \Comment{Compute stochastic gradient}
        \State $v_k = \mathcal P.\texttt{direction}(g_k)$ \Comment{Compute $v_k = P^{-1} g_k$}
        \State $w_{k+1} = \mathcal S(w_k, g_k, v_k, *)$ \Comment{Compute next iterate}
        \EndFor
	\end{algorithmic}
\end{algorithm}

\begin{table}[h]
\centering
    \begin{tabular}{C{2cm}C{10cm}}
        Input & Description \\ \hline
        $w_0$ & Initial iterate, typically set to $0$. \\ \hline
        $\mathcal{O}_g$ & Computes a stochastic gradient. \\ \hline
        $\mathcal{O}_H$ & Used for computing a stochastic/subsampled Hessian. Does not compute the entire subsampled Hessian in practice. \\ \hline
        $b_g, b_H$ & Batchsizes for computing stochastic gradients and Hessians. Used as inputs to $\mathcal{O}_g$ and $\mathcal{O}_H$. \\ \hline
        $\mathcal{P}$ & Preconditioner object. Examples provided in \cref{subsection:precond}. \\ \hline
        $\mathcal U$ & Times at which to update the preconditioner. \\ \hline
        $\mathcal S$ & Subroutine that updates the iterate. May include calculations related to variance reduction and acceleration.
    \end{tabular}
    \caption{Inputs to the \hyperref[alg:PSGM]{Meta-algorithm}.}
    \label{table:inputs_description}
\end{table}

The finite-sum structure of the objective \eqref{eq:EmpRiskProb} 
makes it easy to construct unbiased estimators of the gradient $\nabla F(w)$ and the Hessian $\nabla^2 F(w)$,
given batchsizes $b_g$ and $b_H$, as
  \begin{equation*}
    \widehat \nabla F(w) = \frac{1}{b_g}\sum_{i\in \mathcal B_g} \nabla f_i(w) +\nu w, \quad \widehat \nabla^2 F(w) = \frac{1}{b_H}\sum_{i\in \mathcal B_H} \nabla^2 f_i(w) +\nu I,
 \end{equation*}
 where $\mathcal B_g$ and $\mathcal B_H$, with size $b_g$ and $b_H$ respectively, are
 sampled independently and uniformly from $\{1,\dots,n\}$.
 As a concrete example, consider a generalized linear model (GLM) with $f_i(w) = \phi_i(a^T_iw)$.
 Then
  \begin{equation*}
    \nabla f_i(w) = \phi'_i(a^T_iw)a_i, \quad \nabla^2 f_i(w) = \phi''_i(a^T_iw)a_i a^T_i.
 \end{equation*}

 The stochastic Hessian as written is a $p \times p$ matrix: rather large! 
 But none of the methods we discuss instantiate such a matrix. 
 Instead, they take advantage of the low-rank structure of the preconditioner to 
 compute the approximate Newton direction $P^{-1} g_k$ efficiently using the Woodbury formula (see \cref{table:precond_comp}).

\paragraph{Contributions.} We summarize the contributions of this work as follows:
\begin{enumerate}
    \item We propose preconditioned versions of SVRG, SAGA, and Katyusha, which we call SketchySVRG, SketchySAGA, and SketchyKatyusha. 
    These methods use stochastic approximations to the Hessian to perform preconditioning.
    \item We formally describe a wide array of preconditioners that are compatible with our methods.
    We show that any preconditioner that approximates the Hessian sufficiently well is compatible with our theoretical convergence results.
    \item We define the quadratic regularity ratio, which generalizes the condition number globally to the Hessian norm, 
    and use this ratio to prove our methods converge linearly to the optimum despite lazy updates to the preconditioner.
    \item We show global linear convergence, independent of the condition number, 
    for SketchySVRG, SketchySAGA, and SketchyKatyusha applied to ridge regression. 
    We also show local linear convergence, independent of the condition number, for SketchySVRG on any strongly convex finite-sum problem with Lipschitz Hessians.
    \item We provide default hyperparameters and a heuristic to automatically compute a good learning rate for our proposed methods.
    \item We present extensive experiments demonstrating that SketchySVRG, SketchySAGA, and SketchyKatyusha, equipped with their default hyperparameters and learning rate heuristic, outperform popular stochastic optimizers for GLMs.
\end{enumerate}

\subsection{Roadmap}
\cref{section:precond_techniques} introduces several scalable preconditioning techniques that are compatible with the PROMISE framework; we provide both implementation details and theoretical results for these preconditioners.
\cref{section:algs} presents the algorithms that comprise the PROMISE framework, along with default hyperparameters and algorithmic recommendations for various GLMs.
\cref{section:related_work} reviews the literature on preconditioning and stochastic second-order methods and places PROMISE in the context of these existing works.
\cref{section:theory} establishes linear convergence of all of the proposed methods for strongly convex machine learning problems. 
\cref{section:experiments} demonstrates the superior performance of the algorithms in PROMISE over popular \emph{tuned} stochastic optimizers through extensive numerical experiments.

\subsection{Notation}
Define $[n] \coloneqq \{1,\ldots,n\}$.
Throughout the paper, $\mathcal B_{k}$ and $\mathcal S_{k}$ denote subsets of $[n]$ 
that are sampled independently and uniformly without replacement.
The corresponding (unregularized) minibatch gradient and Hessian are given by
\begin{align*}\widehat \nabla f(w) = \frac{1}{b_{g_k}}\sum_{i\in \mathcal B_{k}}\nabla f_{i}(w),
\\
\widehat \nabla^2 f(w) = \frac{1}{b_{h_k}}\sum_{i\in \mathcal S_{k}}\nabla^{2} f_{i}(w),
\end{align*}
where $b_{g_k} \coloneqq |\mathcal B_k|, b_{h_k} \coloneqq |\mathcal S_k|.$
We abbreviate positive-semidefinite as psd and use $\mathbb S_p^+ (\R)$ to denote the convex cone of psd matrices in $\R^{p \times p}$. 
The symbol $\preceq$ denotes the Loewner order on the convex cone of psd matrices: 
$A\preceq B$ means $B-A$ is psd.
Given a psd matrix $A\in \mathbb{R}^{p\times p}$, 
its eigenvalues in decreasing order are $\lambda_{1}(A)\geq \lambda_2(A)\geq \cdots \geq \lambda_p(A)$. 
For a psd matrix $A \in \R^{p \times p}$, define the condition number $\kappa_2(A) \coloneqq \lambda_1(A) / \lambda_p(A)$.
Define $B(w,r)$ to be the closed Euclidean norm ball of radius $r$, centered at $w$.

Throughout the remainder of this paper, we assume we have access to some GLM $\Model$. 
We note our theoretical convergence results do not require the objective to be a GLM, but the implementation of the SASSN preconditioner (\cref{subsection:precond}) requires this structure.
Moreover, most convex machine learning problems arising in practice are GLMs, so we specialize our implementation to GLMs.
Given that $F$ is a GLM, we assume access to oracles for obtaining the regularization parameter $\nu$, row subsamples of the data matrix $A \in \R^{n \times p}$, the diagonal of the stochastic Hessian of $F$ (excluding the regularization $\nu$), stochastic gradients of $F$, and full gradients of $F$. We present the names, inputs, and outputs of these oracles in \cref{table:m_funcs}. 

\begin{table}[H]
\scriptsize
\centering
    \begin{tabular}{c c}
        Oracle & Output\\ \hline
        $\Model.\texttt{get\_reg}()$ & $\nu$ \\ \hline
        $\Model.\texttt{get\_data}(\mathcal{B})$ & $A_\mathcal{B}$ \\ \hline
        $\Model.\texttt{get\_hessian\_diagonal}(\mathcal{B},w)$ & $\Phi''(A_{\mathcal B}w)$ \\ \hline
        $\Model.\texttt{get\_stoch\_grad}(\mathcal{B}, w)$ & $\frac{1}{|\mathcal{B}|} \sum_{i \in \mathcal{B}} \nabla f_i(w) + \nu w$ \\ \hline
        $\Model.\texttt{get\_full\_grad}(w)$ & $\nabla F(w)$ \\ \hline
    \end{tabular}
    \caption{Oracles associated with the GLM $\Model$. $\mathcal{B} \subseteq [n]$ is a batch of indices and $w \in \R^p$.}
    \label[table]{table:m_funcs}
\end{table}

The oracles \texttt{get\_data} and  \texttt{get\_hessian\_diagonal} output $A_{\mathcal B}w$ and $\Phi''(A_{\mathcal B}w)$, respectively, for minibatch $\mathcal B = \{i_1, i_2, \ldots, i_{|\mathcal B|}\} \subseteq [n]$,
where
\begin{align*}
    A_{\mathcal{B}} \coloneqq 
    \begin{pmatrix}
        a_{i_1}^T \\
        a_{i_2}^T \\
        \vdots \\
        a_{i_{|\mathcal{B}|}}^T \\
    \end{pmatrix},
\qquad
    \Phi''(A_{\mathcal{B}} w) \coloneqq 
    \diag
    \left(
    \begin{pmatrix}
        \phi_{i_1}''\left(a_{i_1}^T w\right) \\
        \phi_{i_2}''\left(a_{i_2}^T w\right) \\
        \vdots \\
        \phi_{i_{|\mathcal{B}|}}''\left(a_{i_{|\mathcal{B}|}}^T w\right)
    \end{pmatrix}
    \right)
    .
\end{align*}
The oracles \texttt{get\_reg}, \texttt{get\_data} and \texttt{get\_hessian\_diagonal} are used in the preconditioners described in \cref{subsection:precond}, while the oracles \texttt{get\_stoch\_grad} and \texttt{get\_full\_grad} are used in the optimization algorithms in \cref{section:algs}.
In practice, \texttt{get\_hessian\_diagonal} returns the diagonal as a vector, not a matrix.
\section{Scalable preconditioning techniques}
\label{section:precond_techniques}
\ifpreprint
We present four scalable preconditioning techniques: Subsampled Newton (\ssn{}), \nys{} Subsampled Newton (\nyssn{}), Sketch-and-Solve Subsampled Newton (\sassn{}), and Diagonal Subsampled Newton (\diagssn{}), all of which are based on stochastic approximations of the Hessian. 
\else
We present three scalable preconditioning techniques: Subsampled Newton (\ssn{}), \nys{} Subsampled Newton (\nyssn{}), and Sketch-and-Solve Subsampled Newton (\sassn{}), all of which are based on stochastic approximations of the Hessian. 
\fi
The key driver behind the scalability of these methods is that subsampling and randomized low-rank approximation provide cheap, reliable estimates of the curvature.

\ifpreprint
For general convex, finite-sum objectives, the \ssn{}, \nyssn{}, and \diagssn{} preconditioners require access to a stochastic Hessian, while the \sassn{} preconditioner requires access to the square root of the stochastic Hessian.
\else
For general convex, finite-sum objectives, the \ssn{} and \nyssn{} preconditioners require access to a stochastic Hessian, while the \sassn{} preconditioner requires access to the square root of the stochastic Hessian.
\fi
Fortunately, the stochastic Hessian in GLMs, $\frac{1}{b_H}A^{T}_\IdxB \Phi''(A_\IdxB w) A_\IdxB$, has a structure that allows us to compute all of these preconditioners using its square root, $\frac{1}{\sqrt{b_H}} [\Phi''(A_\IdxB w)]^{1/2} A_\IdxB$. 
\ifpreprint
An overview of the proposed preconditioners is provided in \cref{table:precond_overview}.
\else
\fi

\ifpreprint
\begin{table}[H]
\centering
    \begin{tabular}{C{3cm}C{7cm}C{4cm}}
        Preconditioner & Key ideas & Oracle requirements \\ \hline
        \ssn{} \newline (\cref{subsubsection:ssn}) & Subsampling & Stochastic Hessian \\ \hline
        \nyssn{} \newline (\cref{subsubsection:nyssn}) & Subsampling, randomized low-rank approximation, randomized \nys{} approximation & Stochastic Hessian \\ \hline
        \sassn{} \newline (\cref{subsubsection:sassn}) & Subsampling, randomized low-rank approximation, sparse embedding, Newton sketch & Square root of stochastic Hessian \\ \hline
        \diagssn{} \newline (\cref{subsubsection:diagssn}) & Subsampling, diagonal approximation & Stochastic Hessian
    \end{tabular}
    \caption{Overview of proposed preconditioners.}
    \label{table:precond_overview}
\end{table}

We present mathematical and algorithmic formulations of the four preconditioning techniques in \cref{table:precond_overview} and \cref{subsection:precond}. 
\else

We present mathematical and algorithmic formulations of the four preconditioning techniques in \cref{subsection:precond}. 
\fi
We then compare the computational costs associated with each preconditioner and provide preconditioner recommendations for various problem regimes in \cref{subsection:precond_comp}. 
\ifpreprint
Finally, we analyze the approximation quality of the proposed preconditioners (with the exception of \diagssn{}) in \cref{subsec:precond_qual}.
\else
Finally, we analyze the approximation quality of the proposed preconditioners in \cref{subsec:precond_qual}.
Any proofs not provided in this section can found in the \href{https://arxiv.org/abs/2309.02014v2}{arxiv report}.
\fi

\subsection{Mathematical and algorithmic formulation of preconditioners}
\label{subsection:precond}
\ifpreprint
This section provides mathematical formulations for the preconditioners in \cref{table:precond_overview} in the GLM setting,
and object-oriented pseudocode.
The preconditioners have two key methods: $\texttt{update}$ and $\texttt{direction}$. 
The $\texttt{update}$ method constructs the preconditioner and 
updates an estimate of the preconditioned smoothness constant 
(effectively combining the \texttt{update} and \texttt{get\_learning\_rate} methods in the \hyperref[alg:PSGM]{Meta-algorithm}), 
while the $\texttt{direction}$ method computes a preconditioned stochastic gradient (similar to \texttt{direction} in the \hyperref[alg:PSGM]{Meta-algorithm}). 
\else
This section provides mathematical formulations for each proposed preconditioner in the GLM setting,
and object-oriented pseudocode for \ssn{}.
Each preconditioner has \texttt{update} and \texttt{direction} methods.
The \texttt{update} method constructs the preconditioner and estimates the preconditioned smoothness constant 
(i.e., it combines \texttt{update} and \texttt{get\_learning\_rate} methods in \hyperref[alg:PSGM]{Meta-algorithm}), while the $\texttt{direction}$ method applies the preconditioner to a vector (similar to \texttt{direction} in \hyperref[alg:PSGM]{Meta-algorithm}).
\fi
These preconditioners play a critical role in the optimization algorithms presented in \cref{section:algs}. 

\subsubsection{Subsampled Newton (\ssn{})}
\label{subsubsection:ssn}
The first preconditioning method we present is \ssn{} \citep{erdogdu2015convergence,roosta2019sub}. 
\ssn{} forms a preconditioner using the Hessian of a random subsample of the terms in the finite-sum objective \eqref{eq:EmpRiskProb}. 
Given a point $w\in \R^p$, \ssn{} constructs the preconditioner
\begin{equation}
\label{eq:ssn_precond}
    P = \frac{1}{b_H}\sum_{i\in \mathcal B} \nabla^2 f_i(w)+\rho I,
\end{equation}
where $\rho \geq \nu$ and $\mathcal B$ consists of $b_H$ elements sampled uniformly at random from $[n]$. 
For GLMs, $f_i = \phi_i(a_i^Tw)$, so \cref{eq:ssn_precond} simplifies to
\begin{equation}
\label{eq:ssn_precond_glm}
    P = \frac{1}{b_H}\sum_{i\in \mathcal B} \phi_i''(a_i^{T}w)a_i a^{T}_i+\rho I = \frac{1}{b_H}A_{\mathcal B}^{T}\Phi''(A_{\mathcal B}w) A_{\mathcal B}+\rho I.    
\end{equation}
If $b_H \geq p$, we may form and factor $P$ (via Cholesky) in $\bigO(b_H p^2 + p^3)$ time and compute $P^{-1} v$ for $v \in \R^p$ in $\bigO(p^2)$ time via triangular solves.
When $b_H \leq p$ (as is typical), we compute the Cholesky factorization $LL^T = \Phi''(A_{\mathcal B} w)^{1/2} A_{\mathcal B} A_{\mathcal B}^T \Phi''(A_{\mathcal B} w)^{1/2} + \rho I$ and 
\[
P^{-1}v = (v - A_{\mathcal B}^T \Phi''(A_{\mathcal B} w)^{1/2} L^{-T} L^{-1} \Phi''(A_{\mathcal B} w)^{1/2} A_{\mathcal B} v)/\rho
\]
via the Woodbury formula \citep{higham2002accuracy}.
\cref{lemma:ssn_cost} summarizes the operational costs.
\begin{lemma}
Let $v \in \R^p$ and let $P$ be as in \eqref{eq:ssn_precond_glm}. If $b_H\leq p$, then the Cholesky factorization can be constructed in $\bigO(b_H^2 p + b_H^3)$ time and $P^{-1}v$ can be computed in $\bigO(b_Hp)$ time. Furthermore, if the data matrix $A$ is row-sparse with sparsity parameter $s$, the computational cost of $P^{-1}v$ can be reduced to $\bigO(b_Hs)$ time.
\label[lemma]{lemma:ssn_cost}
\end{lemma}

The $\pssn$ class (\cref{tab:ssn_attr,alg:ssn_upd,alg:ssn_dir})
provides an implementation of the \ssn{} preconditioner. 
The attributes of the $\pssn$ class are given in \cref{tab:ssn_attr} and pseudocode for the \texttt{update} and \texttt{direction} methods is provided in \cref{alg:ssn_upd,alg:ssn_dir}, respectively.

\ifpreprint
    \begin{table}[H]
\else
    \begin{table}[H]
    \scriptsize
    \centering
    \begin{tabular}{c c}
        Attribute & Description \\ \hline
        $\rho$ & Regularization for preconditioner \\ \hline
        $b$ & Size of Hessian batch used for preconditioner construction \\ \hline
        $X$ & Square root of subsampled Hessian (excluding $l^2$-regularization) \\ \hline
        $L$ & Lower-triangular Cholesky factor for storing preconditioner \\ \hline
        $\lambda_\mathcal{P}$ & Estimate of preconditioned smoothness constant
    \end{tabular}
    \caption{Attributes of the $\pssn$ class.}
    \label{tab:ssn_attr}
\end{table}
\fi



        

\ifpreprint
    \begin{algorithm}[H]
\else
    \begin{algorithm}[t]
    \scriptsize
    \centering
    \caption{Update $\pssn$ preconditioner and preconditioned smoothness constant}
    \label{alg:ssn_upd}
    \begin{algorithmic}
        \Require{$\pssn$ object with attributes $\rho, b, X, L, \lambda_\mathcal{P}$}
        \Function{$\pssn$.\texttt{update}}{$\Model, \mathcal{B}_1, \mathcal{B}_2, w$}
        \State $\rho \gets \pssn.\rho$ \Comment{Get attributes}
        \State
        \State \# Phase 1: Update preconditioner
        \State $A_{\textrm{sub}} \gets \Model.\texttt{get\_data}(\mathcal{B}_1)$
        \State $d_{\textrm{sub}} \gets \Model.\texttt{get\_hessian\_diagonal}(\mathcal{B}_1, w)$
        \State $X \gets \diag(\sqrt{d_{\textrm{sub}}}) A_{\textrm{sub}}$ \Comment{Square root of subsampled Hessian}
        \If {$|\mathcal{B}_1| \geq p$}
        \State $L \gets \texttt{cholesky}(X^T X + \rho I)$
        \Else
        \State $L \gets \texttt{cholesky}(X X^T + \rho I)$
        \EndIf

        \State 
        \State \# Phase 2: Update estimated preconditioned smoothness constant
        \State $A_{\textrm{sub}} \gets \Model.\texttt{get\_data}(\mathcal{B}_2)$
        \State $d_{\textrm{sub}} \gets \Model.\texttt{get\_hessian\_diagonal}(\mathcal{B}_2, w)$
        \State $Z \gets A_{\textrm{sub}}^T \diag(d_{\textrm{sub}}) A_{\textrm{sub}} + \Model.\texttt{get\_reg}() I$ \Comment{Subsampled Hessian}
        \State $\lambda_\mathcal{P} \gets \texttt{eig}(Z (X^T X + \rho I)^{-1}, k = 1)$ \Comment{Compute largest eigenvalue}
        \State
        \State $\pssn.b \gets |\mathcal B_1|, \pssn.X \gets X, \pssn.L \gets L, \pssn.\lambda_\mathcal{P} \gets \lambda_\mathcal{P}$ \Comment{Set attributes}
        \EndFunction
    \end{algorithmic}
\end{algorithm}
\fi
The \texttt{update} method takes a GLM $\Model$, Hessian batches $\mathcal{B}_1, \mathcal{B}_2$, and vector $w \in \R^p$ as input. 
In the first phase, this method constructs the \ssn{} preconditioner $P$ at $w$ by computing the square root of the subsampled Hessian, followed by an appropriate Cholesky factorization. 
The matrix used in the Cholesky factorization changes depending on the Hessian batchsize in order to obtain the computational costs in \cref{lemma:ssn_cost}. 
In the second phase, this method estimates the preconditioned smoothness constant by computing $\lambda_1(P^{-1/2} \widehat{\nabla}^2 F(w) P^{-1/2}) = \lambda_1(\widehat{\nabla}^2 F(w) P^{-1})$.
We never instantiate the subsampled Hessian to perform this calculation. 
Instead, we define matrix-vector products with the subsampled Hessian and inverse preconditioner and compute the largest eigenvalue via powering (our implementation uses \texttt{scipy.sparse.linalg.eigs}).

The \texttt{direction} method takes a vector $g \in \R^p$ (typically a stochastic gradient) as input. 
This method then computes $P^{-1}g$ using the Cholesky factor $L$ and the square root of the subsampled Hessian $X$ (as necessary). 
The reason for having two cases is to achieve the computational complexity in \cref{lemma:ssn_cost} by taking advantage of the Woodbury formula.

\ifpreprint
    \begin{algorithm}[H]
\else
    \begin{algorithm}[t]
    \scriptsize
    \centering
    \caption{Compute $\pssn$ direction}
    \label{alg:ssn_dir}
    \begin{algorithmic}
        \Require{$\pssn$ object with attributes $\rho, b, X, L, \lambda_\mathcal{P}$}
        \Function{$\pssn$.\texttt{direction}}{$g$}
        \State $b \gets \pssn.b, L \gets \pssn.L, X \gets \pssn.X$ \Comment{Get attributes}
        \If {$b \geq p$}
        \State $v \gets L^{-1} g$ \Comment{Triangular solve}
        \State $v \gets L^{-T} v$ \Comment{Triangular solve}
        \State \Return $v$
        \Else
        \State $v \gets X g$
        \State $v \gets L^{-1} v$ \Comment{Triangular solve}
        \State $v \gets L^{-T} v$ \Comment{Triangular solve}
        \State $v \gets X^T v$
        \State \Return $(g - v) / \rho$
        \EndIf
        \EndFunction
    \end{algorithmic}
\end{algorithm}
\fi


\subsubsection{\nys{} Subsampled Newton (\nyssn{})}
\label{subsubsection:nyssn}
\nyssn{} combines the \ssn{} preconditioner with randomized low-rank approximation, specifically the randomized Nystr{\"o}m approximation \citep{williams2000using,gittens2016revisiting,tropp2017fixed}.
This approach was previously developed by in \cite{frangella2023sketchysgd} to precondition stochastic gradient descent. 
Given a symmetric psd matrix $H \in \R^{p\times p}$, the randomized \nys{} approximation with respect to a random test matrix $\Omega \in \R^{p\times r}$ is given by
\begin{equation}
\label{eq:RandNysAppx}
    \hat H = (H\Omega)\left(\Omega^{T}H\Omega\right)^{\dagger}(H\Omega)^{T}.
\end{equation}
Common choices for $\Omega$ include standard normal random matrices, subsampled randomized Hadamard transforms, and sparse sign embeddings \citep{tropp2017fixed}. 
The benefit of the latter two test matrices is that computation of the sketch $H\Omega$ becomes cheaper. 

For a minibatch $\IdxB$ ($|\IdxB| = b_H$) and query point $w\in \R^p$, \nyssn{} takes $H = \frac{1}{b_H}A^{T}_\IdxB \Phi''(A_\IdxB w) A_\IdxB$ in \eqref{eq:RandNysAppx}
and produces a randomized low-rank approximation $\hat{H}$ of (the un-regularized portion of) the subsampled Hessian.

\ifpreprint
A practical algorithm (\cref{alg:nyssn_upd}) for constructing the randomized low-rank approximation outputs $\hat{H}$ in the factored form $U \hat{\Lambda} U^T$, where $U \in \R^{p \times r}$ is an orthogonal matrix containing approximate eigenvectors and $\hat{\Lambda} \in \R^{r \times r}$ is a diagonal matrix containing approximate eigenvalues.
\else
A practical algorithm for constructing the randomized low-rank approximation outputs $\hat{H}$ in the factored form $U \hat{\Lambda} U^T$, where $U \in \R^{p \times r}$ is an orthogonal matrix containing approximate eigenvectors and $\hat{\Lambda} \in \R^{r \times r}$ is a diagonal matrix containing approximate eigenvalues.
\fi
We emphasize that this algorithm never forms $\frac{1}{b_H}A^{T}_\IdxB \Phi''(A_\IdxB w) A_\IdxB$ explicitly.
The resulting preconditioner is
\begin{equation}
\label{eq:nyssn_precond}
    P = \hat{H} + \rho I = U \hat{\Lambda} U^T + \rho I.
\end{equation}

The dominant costs in the (practical) construction of $P$ are computing the sketch $H \Omega$ and a SVD of a $p \times r$ matrix.
Furthermore, we can compute $P^{-1}v$ via the Woodbury formula, which yields
\[
P^{-1}v = U\left(\hat{\Lambda}+\rho I\right)^{-1}U^T v + \frac{1}{\rho}(v - UU^T v).
\]
We summarize the costs of these operations in \cref{lemma:nyssn_cost}.
\begin{lemma}
    Let $v \in \R^p$ and let $P$ be as in \eqref{eq:nyssn_precond}. Then $P$ can be constructed in $\bigO(b_H r p + r^2 p)$ time and $P^{-1} v$ can be computed in $\bigO(rp)$ time. For GLMs, $b_H \geq r$, so the construction cost is reduced to $\bigO(b_H r p)$.
\label[lemma]{lemma:nyssn_cost}
\end{lemma}

The main advantage of \nyssn{} over \ssn{} is in the setting where the data matrix $A$ is dense and has rapid spectral decay. 
When $A$ has rapid spectral decay, we can use a relatively small value of $r$ ($r \leq 10$) to construct the \nyssn{} preconditioner. 
With this small value of $r$, the $\bigO(rp)$ cost of applying the \nyssn{} preconditioner to a vector is usually cheaper than the $\bigO(b_Hp)$ cost of applying the \ssn{} preconditioner.
On the other hand, when $A$ is row-sparse with sparsity parameter $s$, the cost of applying the \ssn{} preconditioner to a vector is reduced to $\bigO(b_Hs)$, which negates the speedups provided by the \nyssn{} preconditioner.

\ifpreprint
An implementation of the \nyssn{} preconditioner is provided in the $\pnyssn$ class (\cref{tab:nyssn_attr,alg:nyssn_upd,alg:nyssn_dir}). 
The attributes of the $\pnyssn$ class are given in \cref{tab:nyssn_attr}, and pseudocode for the \texttt{update} and \texttt{direction} methods is provided in \cref{alg:nyssn_upd,alg:nyssn_dir}, respectively.

\begin{table}[H]
    \centering
    \begin{tabular}{c c}
        Attribute & Description \\ \hline
        $r$ & Rank for constructing preconditioner \\ \hline
        $\rho$ & Regularization for preconditioner \\ \hline
        $d$ & Approximate eigenvalues in \nys{} approximation \\ \hline
        $U$ & Approximate eigenvectors in \nys{} approximation \\ \hline
        $\lambda_\mathcal{P}$ & Estimate of preconditioned smoothness constant
    \end{tabular}
    \caption{Attributes of the $\pnyssn$ class.}
    \label{tab:nyssn_attr}
\end{table}

The \texttt{update} method takes a GLM $\Model$, Hessian batches $\mathcal{B}_1, \mathcal{B}_2$, and vector $w \in \R^p$ as input. 

\begin{algorithm}[H]
    \centering
    \caption{Update $\pnyssn$ preconditioner and preconditioned smoothness constant}
    \label{alg:nyssn_upd}
    \begin{algorithmic}
        \Require{$\pnyssn$ object with attributes $r, \rho, d, U, \lambda_{\mathcal{P}}$}
        \Function{$\pnyssn$.\texttt{update}}{$\Model, \mathcal{B}_1, \mathcal{B}_2, w$}
        \State $r \gets \pnyssn.r, \rho \gets \pnyssn.\rho$ \Comment{Get attributes}
        \State
        \State \# Phase 1: Update preconditioner
        \State $A_{\textrm{sub}} \gets \Model.\texttt{get\_data}(\mathcal{B}_1)$
        \State $d_{\textrm{sub}} \gets \Model.\texttt{get\_hessian\_diagonal}(\mathcal{B}_1, w)$
        \State $X \gets \diag(\sqrt{d_{\textrm{sub}}}) A_{\textrm{sub}}$ \Comment{Square root of subsampled Hessian}
        \State $\Omega \gets \texttt{randn}(p, r)$ \Comment{Generate random Gaussian embedding}
        \State $\Omega \gets \texttt{qr\_econ}(\Omega)$ \Comment{Improves numerical stability}
        \State $Y \gets X^T (X \Omega)$ \Comment{Sketch of subsampled Hessian}
        \State $\delta \gets \sqrt{p} \texttt{eps}(\texttt{norm}(Y))$ \Comment{Compute shift for stability}
        \State $Y_{\delta} \gets Y + \delta \Omega$ \Comment{Shifted sketch}
        \State $C \gets \texttt{cholesky}(\Omega^T Y_{\delta})$ \Comment{Lower triangular Cholesky factor}
        \State $S \gets Y_{\delta} C^{-T}$ \Comment{Triangular solve}
        \State $(U, \Sigma, \sim) \gets \texttt{svd}(S)$ \Comment{Thin SVD}
        \State $d \gets \max(\Sigma^2 - \delta, 0)$ \Comment{Remove shift, compute eigenvalues}

        \State
        \State \# Phase 2: Update estimated preconditioned smoothness constant
        \State $A_{\textrm{sub}} \gets \Model.\texttt{get\_data}(\mathcal{B}_2)$
        \State $d_{\textrm{sub}} \gets \Model.\texttt{get\_hessian\_diagonal}(\mathcal{B}_2, w)$
        \State $Z \gets A_{\textrm{sub}}^T \diag(d_{\textrm{sub}}) A_{\textrm{sub}} + \Model.\texttt{get\_reg}() I$ \Comment{Subsampled Hessian}
        \State $\lambda_\mathcal{P} \gets \texttt{eig}(Z (U (\diag(d) + \rho I)^{-1} U^T + \frac{1}{\rho} (I - U U^T)), k = 1)$ \Comment{Compute largest eigenvalue}
        \State
        \State $\pnyssn.d \gets d, \pnyssn.U \gets U, \pnyssn.\lambda_\mathcal{P} \gets \lambda_\mathcal{P}$ \Comment{Set attributes}
        \EndFunction
    \end{algorithmic}
\end{algorithm}

In the first phase, this method constructs a randomized low-rank approximation $\hat{H}$ to the subsampled Hessian at $w$; updating this low-rank approximation is equivalent to updating the preconditioner since $P = \hat{H} + \rho I$.
The first phase starts by computing the square root of the subsampled Hessian, $X$. 
From this point forward, we follow the stabilized procedure in \cite{tropp2017fixed} to compute $U$ and $d$, which leads to a factored form of the randomized low-rank approximation, $\hat{H} = U \diag(d) U^T$.
In the second phase, this method estimates the preconditioned smoothness constant by taking a similar approach to the $\pssn$ class.

The \texttt{direction} method takes a vector $g \in \R^p$ (typically a stochastic gradient) as input. 

\begin{algorithm}[H]
    \centering
    \caption{Compute preconditioned $\pnyssn$ direction}
    \label{alg:nyssn_dir}
    \begin{algorithmic}
        \Require{$\pnyssn$ object with attributes $r, \rho, d, U, \lambda_{\mathcal{P}}$}
        \Function{$\pnyssn$.\texttt{direction}}{$g$}
        \State $U \gets \pnyssn.U, d \gets \pnyssn.d, \rho \gets \pnyssn.\rho$ \Comment{Get attributes}
        \State $v \gets U^T g$
        \State $v \gets (1 / (d + \rho) - 1 / \rho) v$ \Comment{Elementwise division, subtraction, and multiplication}
        \State \Return $g / \rho + U v$
        \EndFunction
    \end{algorithmic}
\end{algorithm}

The method then computes $P^{-1}g$ using the Woodbury formula with the preconditioner factors $d$ and $U$; this computation has complexity $\mathcal{O}(rp)$.
\else
Our code repository, linked in \cref{section:experiments}, implements the \nyssn{} preconditioner.
\fi

\subsubsection{Sketch-and-Solve Subsampled Newton (\sassn{})}
\label{subsubsection:sassn}
The next preconditioning technique we discuss is \sassn{}. 
Similar to \nyssn{}, the fundamental goal of \sassn{} is to reduce the cost of the \ssn{} preconditioner by replacing it with a randomized low-rank approximation.
However, instead of using the randomized Nystr{\"o}m approximation, \sassn{} computes an approximation in the style of the Newton Sketch \citep{pilanci2017newton,lacotte2021adaptive}.
To start, observe that the subsampled Hessian $\widehat \nabla^2 F(w)$ has the form
\[
\widehat \nabla^2 f(w)+\nu I = R^T R+\nu I,
\]
where $R = \Phi''(A_{\IdxB}w)^{1/2}A_{\IdxB}$.
Hence, given a test matrix $\Omega \in \R^{r\times b_H}$, we construct the preconditioner
\begin{equation}
\label{eq:sassn_preconditioner}
    P = R^{T}\Omega^{T}\Omega R+\rho I.
\end{equation}

The dominant costs in the construction of $P$ are computing the sketch $\Omega R$ and a Cholesky factorization of $\Omega R (\Omega R)^{T}+\rho I$. 
Taking $\Omega$ to be a column-sparse or row-sparse (LESS-uniform) embedding \citep{derezinski2021newtonless}, $\Omega R$ can be computed in $\bigO(b_Hp)$ time.
A preconditioner generated by a column-sparse embedding is referred to as SASSN-C, while a preconditioner generated by a row-sparse embedding is referred to as SASSN-R. 
SASSN-R tends to be better than SASSN-C because it is cheaper to apply to vectors when the data matrix $A$ is row-sparse.
Similar to the \nyssn{} preconditioner, we can compute $P^{-1}v$ via the Woodbury Formula, which yields
\[
P^{-1}v= \frac{1}{\rho}\left(v-(\Omega R)^{T}\left(\Omega R (\Omega R)^{T}+\rho I\right)^{-1}(\Omega R)v\right).
\] 
We summarize the costs of these operations in \cref{lemma:sassn_cost}.
\begin{lemma}
Let $v \in \R^p$ and let $P$ be as in \eqref{eq:sassn_preconditioner}. Then $P$ can be constructed in $\bigO(b_Hp + r^2p + r^3)$ time and $P^{-1}$ may be applied to vectors in $\bigO(rp)$ time. Furthermore, if the data matrix $A$ is row-sparse with sparsity parameter $s$, the computational cost of $P^{-1}v$ for SASSN-R can be reduced to $\bigO(rs)$ time.
\label[lemma]{lemma:sassn_cost}
\end{lemma} 

Similar to \nyssn{}, the costs of constructing and applying the \sassn{} preconditioner are lower than the costs incurred by \ssn{}.
A potential advantage of \sassn{} over \nyssn{} is that \sassn{} requires $\bigO(b_Hp + r^2p + r^3)$ time to construct the preconditioner, whereas \nyssn{} requires $\bigO(b_Hrp)$. 
Furthermore, the \sassnr{} preconditioner takes $\bigO(rs)$ time to apply when the data matrix $A$ is row-sparse, whereas \nyssn{} takes $\bigO(rp)$.

However, our experiments (\cref{section:experiments}) suggest that the \sassn{} preconditioner tends to be of lower quality than the \nyssn{} preconditioner (i.e., it does not reduce the condition number as much), and the theoretical complexity advantage of \sassn{} is not always realized as the computations in \nyssn{} benefit from (embarassing) parallelism.
Concrete comparisons and recommendations between \ssn{}, \nyssn{}, \sassnc{}, and \sassnr{} are given in \cref{table:precond_comp,table:precond_when_to_use} below. 
\ifpreprint
An implementation of the \sassnc{}/\sassnr{} preconditioners is provided in \cref{subsection:methodology_sassn_appdx}.
\else
An implementation of the \sassnc{}/\sassnr{} preconditioners can be found in our code repository, which is linked in \cref{section:experiments}.
\fi

\ifpreprint
\subsubsection{Diagonal Subsampled Newton (\diagssn{})}
\label{subsubsection:diagssn}
The final preconditioning technique we discuss is \diagssn{}. 
\diagssn{} does exactly what its name suggests: it computes the diagonal of the subsampled Hessian (excluding regularization) $\widehat \nabla^2 f(w) = \frac{1}{b_H}A^{T}_\IdxB \Phi''(A_\IdxB w) A_\IdxB$, and constructs the preconditioner 
\[P = \diag(\widehat \nabla^2 f(w) + \rho I).\]
The resulting preconditioner can be applied to vectors in $\mathcal{O}(p)$ time.
\diagssn{} is a stochastic generalization of Jacobi preconditioning, a classical preconditioning strategy in optimization \citep{nocedal1999numerical,jambulapati2020fast,qu2022optimal}. 
It serves as a baseline in this work 
to illustrate that low-rank approximation is much more powerful than diagonal approximation
for the subsampled Hessian:
\diagssn{} can be computed quickly but its performance is poor in practice.
The experiments in \cref{section:experiments} show that the performance of our PROMISE methods  using \diagssn{} is worse than with any of the other preconditioners.
An implementation of the \diagssn{} preconditioner is provided in \cref{subsection:methodology_diagssn_appdx}.
\else
\fi

\subsection{Preconditioner defaults and comparisons}
\label{subsection:precond_comp}
\ifpreprint
We provide default values for rank $r$ and regularization $\rho$ for each preconditioner in \cref{table:precond_hyperparams}; we use these defaults for $r$ and $\rho$ in the experiments in \cref{section:experiments}.
We also summarize the costs to construct and apply each preconditioner in \cref{table:precond_comp}.
These costs assume that the Hessian batchsize $b_H = \lfloor \sqrt{n} \rfloor$ in the PROMISE algorithms; we provide motivation for this selection in \cref{subsec:precond_qual} and use this selection in our experiments in \cref{section:experiments}.
We also provide guidelines for which preconditioner to use as the problem size varies in \cref{table:precond_when_to_use}. 
We do not recommend using \diagssn{} or \sassnc{} in practice and so we omit them from \cref{table:precond_when_to_use}.

\begin{table}[H]
\centering
    \begin{tabular}{C{4.5cm}C{1.5cm}C{1.5cm}C{1.5cm}}
        Preconditioner & $r$ & $\rho$ \\ \hline
        \ssn{} \newline (\cref{alg:ssn_upd,alg:ssn_dir}) & N/A & $10^{-3}$ \\ \hline
        \nyssn{} \newline (\cref{alg:nyssn_upd,alg:nyssn_dir}) & $10$ & $10^{-3}$ \\ \hline
        \sassnc{} \newline (\cref{alg:sassnc_gen,alg:sassn_upd,alg:sassn_dir}) & $10$ & $10^{-3}$ \\ \hline
        \sassnr{} \newline (\cref{alg:sassnr_gen,alg:sassn_upd,alg:sassn_dir}) & $10$ & $10^{-3}$ \\ \hline
        \diagssn{} \newline (\cref{alg:diagssn_upd,alg:diagssn_dir}) & N/A & $10^{-3}$ \\
    \end{tabular}
    \caption{Default values of $r$ and $\rho$ for each preconditioner. The default for $\rho$ assumes that $\nu \leq 10^{-3}$, which is common in practice and is always the case in our experiments; if $\nu > 10^{-3}$, this default for $\rho$ should be increased.
    }
    \label{table:precond_hyperparams}
\end{table}
\else
All of the proposed preconditioners require a regularization $\rho$, and the \nyssn{}, \sassnr{}, and \sassnc{} preconditioners require a rank $r$ for forming the low-rank approximation. 
We recommend setting $\rho = 10^{-3}$ and $r = 10$.
We also summarize the costs to construct and apply each preconditioner in \cref{table:precond_comp}.
These costs assume the Hessian batchsize $b_H = \lfloor \sqrt{n} \rfloor$ in the PROMISE algorithms; we provide motivation for this selection in \cref{subsec:precond_qual}. 
We also provide guidelines for which preconditioner to use as the problem size varies in \cref{table:precond_when_to_use}; we do not recommend \sassnc{} in practice.
\fi

\ifpreprint
\begin{table}[H]
\scriptsize
\centering
    \begin{tabular}{C{4.5cm}C{4.0cm}C{1.6cm}C{2.2cm}}
        Preconditioner & Construction cost  & Cost to apply & Cost to apply (sparse) \\ \hline
        \ssn{} \newline (\cref{alg:ssn_upd,alg:ssn_dir}) & $\mathcal{O}(np+n^{3/2})$ & $\mathcal{O}(\sqrt{n}p)$ & $\mathcal{O}(\sqrt{n}s)$ \\ \hline
        \nyssn{} \newline (\cref{alg:nyssn_upd,alg:nyssn_dir}) & $\mathcal{O}(\sqrt{n}rp)$ & $\mathcal{O}(rp)$ & $\mathcal{O}(rp)$ \\ \hline
        \sassnc{} \newline (\cref{alg:sassnc_gen,alg:sassn_upd,alg:sassn_dir}) & $\mathcal{O}(\sqrt{n}p)$ & $\mathcal O(rp)$ & $\mathcal O(rp)$ \\ \hline
        \sassnr{} \newline (\cref{alg:sassnr_gen,alg:sassn_upd,alg:sassn_dir}) & $\mathcal{O}(\sqrt{n}p)$ & $\mathcal O(rp)$ & $\mathcal O(rs)$ \\ \hline
        \diagssn{} \newline (\cref{alg:diagssn_upd,alg:diagssn_dir}) & $\mathcal{O}(\sqrt{n}p)$ & $\mathcal{O}(p)$ & $\mathcal{O}(p)$
    \end{tabular}
    \caption{Summary of costs of proposed preconditioners. $s$ denotes the row sparsity of the data matrix $A$.}
    \label{table:precond_comp}
\end{table}
\else
\begin{table}[H]
\scriptsize
\centering
    \begin{tabular}{C{4.5cm}C{4.0cm}C{1.6cm}C{2.2cm}}
        Preconditioner & Construction cost  & Cost to apply & Cost to apply (sparse) \\ \hline
        \ssn{} & $\mathcal{O}(np+n^{3/2})$ & $\mathcal{O}(\sqrt{n}p)$ & $\mathcal{O}(\sqrt{n}s)$ \\ \hline
        \nyssn{} & $\mathcal{O}(\sqrt{n}rp)$ & $\mathcal{O}(rp)$ & $\mathcal{O}(rp)$ \\ \hline
        \sassnc{} & $\mathcal{O}(\sqrt{n}p)$ & $\mathcal O(rp)$ & $\mathcal O(rp)$ \\ \hline
        \sassnr{} & $\mathcal{O}(\sqrt{n}p)$ & $\mathcal O(rp)$ & $\mathcal O(rs)$ \\ \hline
        \ifpreprint
        \diagssn{} & $\mathcal{O}(\sqrt{n}p)$ & $\mathcal{O}(p)$ & $\mathcal{O}(p)$
        \else
        \fi
    \end{tabular}
    \caption{Summary of costs of proposed preconditioners. $s$ denotes the row sparsity of the data matrix $A$.}
    \label{table:precond_comp}
\end{table}
\fi

\begin{table}[H]
\scriptsize
\centering
    \begin{tabular}{C{3cm}C{1.5cm}C{1.5cm}C{2cm}}
        Regime & \ssn{} & \nyssn{} & SASSN-R \\ \hline
        $n\gg p$ (dense) & 2 & $\mathbf{1}$ & 3\\ \hline
        $n\gg p$ (sparse) & $\mathbf{1}$ & 2 & 3 \\ \hline
        $n\sim p$ (dense) & 3  & $\mathbf{1}$ & 2 \\ \hline
        $n\sim p$ (sparse) & $\mathbf{1}$ & 3 & 2 \\ \hline
        $n\ll p$ (sparse) & $\mathbf{1}$  & $3$ & 2 \\ 
    \end{tabular}
    \caption{Guidelines for selecting a preconditioner. 
    The best preconditioner for each regime is assigned a rank of 1. \nyssn{} is effective for dense problems, but \ssn{} generally works better for sparse problems because it preserves the sparsity of the data.
    }
    \label{table:precond_when_to_use}
\end{table}

\subsection{Quality of the preconditioners}
\label{subsec:precond_qual}
We now analyze the quality of the \ssn{}, \nyssn{}, and \sassn{} preconditioners that were introduced in \cref{subsection:precond}.
Our goal is to show that these preconditioners satisfy the following \textit{$\zeta$-spectral approximation property} with high probability.
\begin{definition}[$\zeta$-spectral approximation]
\label[definition]{def:good_precond}
    Let $w\in \R^p$, $\zeta \in (0,1)$. Then we say $P$ is a \emph{$\zeta$-spectral approximation} of $\nabla^2 F(w)$ if the following relation holds:
    \begin{equation}
       (1-\zeta)P\preceq  \nabla^2F(w) \preceq (1+\zeta)P. 
    \end{equation}
\end{definition}
If $P$ satisfies \cref{def:good_precond}, then 
\[
\kappa_2(P^{-1/2}\nabla^2 F(w)P^{-1/2})\leq \frac{1+\zeta}{1-\zeta}.
\]
Hence preconditioning $\nabla^2F(w)$ by $P$ results in a good (small) condition number for moderate $\zeta$ (e.g., $\zeta \leq .9$).
Moreover, as $P^{-1/2}\nabla^2F(w)P^{-1/2}$ is nearly the identity, $P^{-1}$ is close to $\nabla^2F(w)^{-1}$, which ensures the approximate Newton direction computed with $P^{-1}$ is close to the true Newton direction.
As a consequence of this last observation, essentially all works on approximate Newton methods require the Hessian approximation to satisfy the conditions of \cref{def:good_precond} \citep{pilanci2017newton,roosta2019sub,marteau2019globally,ye2021Appx}.  

 \subsubsection{Preliminaries on sampling}
To establish the $\zeta$-approximation property for the preconditioners, we require some fundamental concepts from matrix approximation via random sampling, which we now review. 
We start with the definition of ridge leverage scores \citep{cohen2017input,li2020subsampled}.
\begin{definition}[Ridge leverage scores]
\label[definition]{def:RLS}
Let $\nu\geq 0$ and $i\in [n]$. Then the $ith$ \emph{ridge leverage score} of a matrix $A\in 
\R^{n\times p}$ is given by
\[l_i^{\nu}(A) \coloneqq a_i^{T}\left(A^TA+n\nu I\right)^{\dagger}a_i.\]
where $a^{T}_i$ is the $i{th}$ row of $A$ and the \emph{maximum} ridge leverage score is
$l^{\nu}_\infty(A) = \max_{1\leq i\leq n}l_{i}^{\nu}(A).$

\end{definition}
The $ith$ ridge leverage score measures the importance of row $i$ in the matrix $A$. 
These scores play a crucial role in determining how well the matrix $\frac{1}{n}A^{T}A+\nu I$ may be approximated via uniform sampling.
To understand this relation, we recall the notions of effective dimension and ridge leverage incoherence. 

\begin{definition}[Effective dimension and ridge leverage coherence]
\label[definition]{def:EffDimRLSCoh}
    Given $A\in \mathbb{R}^{n\times p}$ and $\nu\geq 0$, the \emph{effective dimension} of $A$ is given by
    \begin{equation}
        d^\nu_{\textup{eff}}(A) \coloneqq \sum_{i=1}^{n}l_i^{\nu}(A) = \sum^{p}_{j=1}\frac{\sigma_j^2(A)}{\sigma^2_j(A)+n\nu} = \sum^{p}_{j=1}\frac{\lambda_j(A^TA)}{\lambda_j(A^TA)+n\nu}.
    \end{equation}
    If $H\in \mathbb{S}^{+}_{p}(\R)$ with $H = A^{T}A$, then we overload notation and define $\deff(H) \coloneqq \deff(A)$. \\
    The \emph{ridge leverage coherence} is given by
    \begin{equation}
        \chi^\nu(A) \coloneqq \frac{n}{d^\nu_{\textup{eff}}(A)}l_\infty^{\nu}(A).
    \end{equation}
    Similarly, if $H\in \mathbb{S}^{+}_{p}(\R)$ with $H = A^{T}A$, we overload notation and define $\chi^\nu(H) \coloneqq \chi^\nu(A)$.
\end{definition}

\paragraph{Effective dimension: discussion.} The effective dimension $d^\nu_{\textup{eff}}(A)$ has an intuitive interpretation: it provides a smoothed count of the eigenvalues greater than or equal to the regularization $\nu$.
In the regularized setting, only directions associated with eigenvalues larger than $\nu$ matter, so $d^\nu_{\textup{eff}}(A)$ rather than $p$ is the relevant measure of degrees of freedom for the problem. 
Consequently, the effective dimension often appears in fields that deal with $l^2$-regularized problems, including non-parametric learning, RandNLA, and statistical learning \citep{caponnetto2007optimal,hsu2014random,marteau2019beyond}.
As many data matrices have fast spectral decay \citep{derezinski2020precise}, or obtain it through some algorithmic transformation, such
as the celebrated random features method of \cite{rahimi2007random},
$d^{\nu}_{\textrm{eff}}(A)$ is often much smaller than $\min\{n,p\}$.

When the loss function $f$ belongs to the GLM family and $A$ has polynomial spectral decay, the following lemma demonstrates that the effective dimension of the Hessian, $A^T\Phi''(A w)A$, is much smaller than the ambient dimension $p$.
Various results of this form are well-known in the literature, see for instance \cite{caponnetto2007optimal,bach2013sharp,marteau2019globally}.
 
\begin{lemma}[Effective dimension under polynomial decay] 
\label[lemma]{lemma:EffDimBnd}
Let $f$ be a GLM loss satisfying $\sup_{w\in \R^p}\phi''(w)\leq B$, with data matrix $A \in \mathbb{R}^{n\times p}$ and regularization $\nu$.
Suppose the matrix $A$ has polynomial (or faster) spectral decay:
\[\sigma_j(A) = \bigO(j^{-\beta}) \quad (1 \leq j \leq p),\]
for some $\beta \in \mathbb Z_+$ satisfying $\beta \geq 1$.
    Then for any $w\in \R^{p}$,
    \[
    d_\textup{eff}^\nu(A^T\Phi''(A w)A)\leq \frac{\pi/(2\beta)}{\sin(\pi/(2\beta))}\left(\frac{C}{n\nu}\right)^{1/{2\beta}}.\]
    Hence, if $\nu = \bigO (\frac{1}{n})$ we have
    \[
     d_\textup{eff}^\nu(A^T\Phi''(A w)A) = \bigO \left(\sqrt{n}\right),
    \]
\end{lemma}
Given mild hypotheses, \cref{lemma:EffDimBnd} shows that the effective dimension of the Hessian for GLMs is no larger than $\bigO(\sqrt{n})$.
Thus, we generally expect the effective dimension of the Hessian, $d_\textup{eff}^\nu(A^T\Phi''(A w)A)$, to be significantly smaller than the ambient dimension of the problem, $p$. 
The ``smallness'' of the effective dimension has been exploited in numerous works to develop fast algorithms for solving a variety of machine learning problems \citep{bach2013sharp,alaoui2015fast,rudi2017falkon,marteau2019globally,lacotte2021adaptive,zhao2022nysadmm,frangella2023randomized}.
Similar to these prior works, we will also exploit the small effective dimension of the Hessian to develop effective preconditioners that can be constructed at negligible cost.  

\ifpreprint
\paragraph{Ridge leverage coherence: discussion.} 
When $\nu = 0$, the ridge leverage coherence $\chi^\nu(A)$ is equivalent to the coherence parameter from compressed sensing and matrix completion \citep{candes2007sparsity,candes2012exact}, and its formulation in \cref{def:EffDimRLSCoh} is the proper generalization to the regularized setting.  
The ridge leverage coherence measures the uniformity of ridge leverage scores, which determines how challenging it is to approximate $(1/n)A^{T}A+\nu I$ through row sampling from $A$.
When $\chi^{\nu}(A) = \bigO(1)$, the ridge leverage scores are uniform, with each row contributing relatively equally to $d^\nu_{\textrm{eff}}(A)$. 
As a result, $A$ has no distinguished rows carrying more weight than the rest, so uniform sampling will perform well. 

A classic family of matrices that have small ridge-leverage coherence are matrices whose left singular vectors are uniformly distributed on the Stiefel manifold $\mathcal S_p(\R^n)$ (the set of all $p$-dimensional orthonormal frames in $\R^n$ \citep{milman2009asymptotic}).
Concretely, let $A\in \R^{n\times p}$ with SVD $A = U\Sigma V^{T}$, $U\in \R^{n\times p}$, $\Sigma, V\in \R^{p\times p}$, and suppose $U$ is uniformly distributed on $\mathcal S_p(\R^n)$. Then with high probability, $\chi^{\nu}(A) = \Otil(1)$.
When $\nu = 0$, in which case the ridge-leverage coherence is simply the coherence, this result was established \cite{candes2012exact}. 
We now give a sketch of the argument for $\nu>0$.
To this end, first observe that a routine calculation reveals
\[l_i^{\nu}(A) = u_i^{T}(\Sigma^2(\Sigma^2+n\nu I)^{-1})u_i,\]
where $u_i$ is the ith row of $U$.
From this relation we find
\begin{align*}
    \E[l^{\nu}_i(A)] &= \E[u_i^{T}(\Sigma^2(\Sigma^2+n\nu I)^{-1})u_i] = \frac{1}{n}\textrm{trace}(\Sigma^2(\Sigma^2+n\nu I)^{-1}) = \frac{1}{n}\deff(A).
\end{align*}
To control the deviation of $l^{\nu}_i(A)$ from its expectation, we may apply the Hanson-Wright inequality \citep{rudelson2013hanson} to find $|l^{\nu}_i(A)-\deff(A)/n| = \Otil(\deff(A)/n)$ with high probability.
Applying a union bound, we conclude $l^{\nu}_\infty(A) = \Otil(\deff(A)/n)$, which immediately implies $\chi^{\nu}(A) = \Otil(1)$.
Hence matrices whose left-singular vectors are uniformly distributed on the Stiefel manifold are ridge-leverage incoherent with high probability.


In contrast, when $\chi^{\nu}(A) = \bigOmega(n)$, a small number of rows significantly contribute to $d^\nu_{\textrm{eff}}(A)$, leading to highly non-uniform ridge leverage scores.
In this setting, uniform sampling is unlikely to yield satisfactory results, as it may fail to sample the important rows.
As an example, consider the family of matrices $A\in \R^{n\times p}$ whose rows satisfy 
\[a_1 = nu, \quad a_j = u, \quad \text{for all}~2\leq j\leq n, \quad \text{and}~\|u\|_2 = 1.\]
Note for any $n$, $A$ is rank-1 matrix with $\lambda_{1}(A^TA) = n^2+n-1$, and that the first row is a heavy-hitter.
Suppose $\nu = \frac{1}{n}$, then direct calculation shows
\[l^{\nu}_{\infty}(A) = \frac{n^2}{n^2+n},\quad \deff(A) = \frac{n^2+n-1}{n^2+n}.\]
Hence, by the definition of ridge leverage coherence,
\[
\chi^{\nu}(A) = \frac{n}{1+n^{-1}+n^{-2}}.
\]
Thus, even though the matrix is rank-$1$, uniform sampling must effectively sample every row of $A$ to obtain a good approximation. 
This example shows only one outlier row is needed to make uniform sampling ineffective. 

The preceding observations are made rigorous in the following lemma, which  is well-known in the literature \citep{alaoui2015fast,cohen2017input,li2020subsampled}.
\begin{lemma}[Uniform sampling: spectral approximation]
\label[lemma]{lemma:ridgesamplingappx}
    Let $A\in \mathbb{R}^{n\times p}$. Suppose we sample the rows of $A$ according to the set $\mathcal{J}$, where $\mathcal{J} \subseteq [n]$ is drawn uniformly at random and $|\mathcal J| =  \Upomega\left(\frac{\chi^\nu(A) d^\nu_{\textup{eff}}(A)\log\left(\frac{d^\nu_{\textup{eff}}(A)}{\delta}\right)}{\zeta^2}\right)$.
    Then with probability at least $1-\delta$
    \begin{equation}
        (1-\zeta)\left(\frac{1}{|\mathcal J|}A^{T}_\mathcal{J}A_\mathcal{J}+\nu I\right){\preceq} \left(\frac{1}{n}A^TA+\nu I\right) \preceq (1+\zeta)\left(\frac{1}{|\mathcal J|}A_\mathcal{J}^TA_\mathcal{J}+\nu I\right).
    \end{equation}
\end{lemma}
\Cref{lemma:ridgesamplingappx} shows that for an incoherent matrix with $\chi^\nu(A) = \bigO(1)$, it only takes $\Otil\left(d^\nu_{\textrm{eff}}(A)\right)$ rows to obtain a good spectral approximation to $(1/n)A^{T}A+\nu I$. 
In contrast, for a coherent matrix with $\chi^\nu(A) = \bigO(n)$, as many as $\bigO(n)$ rows may be needed to obtain a good spectral approximation. 
Thus, when the coherence is large, uniform sampling cannot exploit low effective dimensionality.
The sample complexity for coherent matrices can be reduced using approximate leverage score sampling (ALS) procedures such as BLESS \citep{rudi2018fast}, which sample rows according to their importance.
However, since uniform sampling performs very well in our experiments, we omit ALS procedures, which can increase algorithmic complexity and computational overhead.
\else
\fi

We now establish that the various preconditioning methods introduced in \cref{subsection:precond} 
provide a $\zeta$-spectral approximation with high probability. 

\subsubsection{Subsampled Newton}
\ssn{} yields a $\zeta$-spectral approximation for GLMs with high probability,
formalized below.
\begin{proposition}
\label[proposition]{proposition:SSNPrecond}
     Let $w\in \mathbb{R}^{p}$, $\zeta_0 \in (0,1)$, and suppose $f$ is a GLM. 
     Construct the subsampled Hessian with  batchsize $b_H =  \bigOmega\left(\frac{\chi^\rho(\nabla^2 f(w)) d^\rho_{\textup{eff}}(\nabla^2 f(w))\log\left(\frac{d^\rho_{\textup{eff}}(\nabla^2 f(w))}{\delta}\right)}{\zeta_0^2}\right)$.
     Then for $\zeta = 1-(1-\zeta_0)\nu/\rho$, with probability at least $1-\delta$,
     \begin{equation}
        (1-\zeta)(\widehat{\nabla}^2f(w)+\rho I) \preceq \nabla^2 f(w)+\nu I \preceq (1+\zeta)(\widehat{\nabla}^2f(w)+\rho I).
    \end{equation}
\end{proposition}
\cref{proposition:SSNPrecond} is well-known in the literature \citep{li2020subsampled}.
It shows that when $\chi^\rho(\nabla^2 f(w)) = \bigO(1)$, a batchsize of $b_H = \Otil(d^{\rho}_{\textrm{eff}}(\nabla^2 f(w)))$ is sufficient to ensure that the subsampled Hessian is a $\zeta$-spectral approximation.
Furthermore, when the data matrix exhibits polynomial spectral decay, applying \cref{lemma:EffDimBnd} reduces this requirement to $b_H = \Otil(\sqrt{n})$.
This latter reduction motivates our default hyperparameter setting $b_H = \lfloor \sqrt{n} \rfloor$ for the PROMISE algorithms in \cref{section:algs}.

\cref{proposition:SSNPrecond} should be contrasted with the Hessian batchsize requirements of works where the $f_i$'s are taken to be arbitrary convex functions.
To facilitate this comparison, we first state the following simple lemma. 
\begin{lemma}[$\deff$ vs. $\kappaMax$ for GLMs]
\label[lemma]{lem:deff_vs_kappa}
    Let $f$ be a GLM, $\nu>0$, and $\kappaMax = \LMax/\nu$. Then
    \[
    \chi^\nu(\nabla^2 f(w))\deff(\nabla^2 f(w))\leq \kappaMax.
    \]
\end{lemma}
As stated above, some works \citep{roosta2019sub,ye2021Appx,derezinski2022stochastic} assume the $f_i$'s possess no structure aside from convexity, which leads to the batchsize requirement $b_H = \bigO(\kappaMax\log(p/\delta)/\zeta^2_0)$.
\cref{{lem:deff_vs_kappa}} shows that $\kappaMax\geq  \chi^\nu(\nabla^2 f(w))\deff(\nabla^2 f(w))$, 
so the needed batchsize is always larger than the one prescribed by \cref{proposition:SSNPrecond}.
Moreover, when the data matrix $A$ is ill-conditioned, the required batchsize is significantly larger.
As a concrete example, consider the setting of \cref{lemma:EffDimBnd} with ridge leverage incoherent $A$: $\chi^{\nu}(\nabla^2f(w))\deff(\nabla^2f(w)) = \bigO(\sqrt{n})$, while $\kappaMax = \bigO(n)$.
Hence, \cref{proposition:SSNPrecond} predicts a small Hessian batchsize, while the requirement based on $\kappaMax$ states the full Hessian must be used.
Thus, the Hessian batchsize required for convex GLMs is considerably smaller than the one for a sum of arbitrary convex functions.

\subsubsection{Nystr{\"o}m Subsampled Newton}
\nyssn{} yields a $\zeta$-spectral approximation for GLMs with high probability,
formalized below.
\ifpreprint
\begin{proposition}
\label[proposition]{proposition:NysSSNPrecond}
        Let $w\in \mathbb{R}^{p}$, $\zeta_0\in (0,1)$, and suppose $f$ is a GLM.
        Construct the subsampled Hessian with  batchsize $b_H =  \bigOmega\left(\frac{\chi^\nu(\nabla^2 f(w)) d^\nu_{\textup{eff}}(\nabla^2 f(w))\log\left(\frac{d^\nu_{\textup{eff}}(\nabla^2 f(w))}{\delta}\right)}{\zeta_0^2}\right)$. 
     Further, assume either:
     \begin{enumerate}
         \item $\Omega$ is a Gaussian random matrix with $r = \bigOmega\left(\frac{d^{\rho}_{\textup{eff}}(\widehat \nabla^2 f(w))+\log(\frac{1}{\delta})}{{\zeta_0^2}}\right)$ columns.
         \item $\Omega$ is an SRHT matrix with $ r =\bigOmega\left(\frac{d^{\rho}_{\textup{eff}}(\widehat \nabla^2 f(w))+\log(\frac{1}{\zeta_0 \delta})\log\left(\frac{d^{\rho}_{\textup{eff}}(\widehat \nabla^2 f(w))}{\delta}\right)}{{\zeta_0^2}}\right)$ columns.
         \item $\Omega$ is an sparse sign embedding with sparsity $s = \bigOmega\left(\frac{\log\left(\frac{d^{\rho}_{\textup{eff}}(\widehat \nabla^2 f(w))}{\delta}\right)}{{\zeta_0}}\right)$ and $r =\bigOmega\left(\frac{d^{\rho}_{\textup{eff}}(\widehat \nabla^2 f(w))\log\left(\frac{d^{\rho}_{\textup{eff}}(\widehat \nabla^2 f(w))}{\delta}\right)}{{\zeta_0^2}}\right)$ columns.
     \end{enumerate}
     Then with probability at least $1-\delta$
     \begin{equation}
        (1-\zeta)(\hat H+\rho I)\preceq \nabla^2f(w)+\nu I\preceq (1+\zeta)(\hat H+\rho I),
    \end{equation}
    where $\zeta = 1-(1-\zeta_0)\nu/\rho$.
\end{proposition} 
\else
\begin{proposition}
\label[proposition]{proposition:NysSSNPrecond}
        Let $w\in \mathbb{R}^{p}$, $\zeta_0\in (0,1)$, and suppose $f$ is a GLM.
        Construct the subsampled Hessian with  batchsize $b_H =  \bigOmega\left(\frac{\chi^\nu(\nabla^2 f(w)) d^\nu_{\textup{eff}}(\nabla^2 f(w))\log\left(\frac{d^\nu_{\textup{eff}}(\nabla^2 f(w))}{\delta}\right)}{\zeta_0^2}\right)$. Further, assume $\Omega$ is a Gaussian random matrix with $r = \bigOmega\left(\frac{d^{\rho}_{\textup{eff}}(\widehat \nabla^2 f(w))+\log(\frac{1}{\delta})}{{\zeta_0^2}}\right)$ columns. 
        Then for $\zeta = 1-(1-\zeta_0)\nu/\rho$, with probability at least $1-\delta,$
     \begin{equation}
        (1-\zeta)(\hat H+\rho I)\preceq \nabla^2f(w)+\nu I\preceq (1+\zeta)(\hat H+\rho I).
    \end{equation}
\end{proposition} 
\fi
The regularization parameter $\rho$ controls how much we may truncate the rank parameter $r$. As $\rho$ increases, $d^{\rho}_{\textup{eff}}(\widehat \nabla^2 f(w))$ decreases, so we can use a smaller value of $r$ to construct the preconditioner; conversely, as $\rho$ approaches $\nu$, we must use a larger value of $r$.
 
We observe a trade-off: a smaller rank parameter leads to faster computation and less storage, 
but potentially a less effective preconditioner and slower convergence.
In practice, this tradeoff is not as dramatic as the theory might suggest, 
and we find a rank of $r = 10$ provides excellent performance in a wide range of applications (\cref{section:experiments}). 

\subsubsection{Sketch-and-solve Subsampled Newton} 
\ifpreprint
An analogous result for the \sassn{} preconditioners can be found in \cref{subsec:samp_lr_pfs}.
\else
An analogous result for the \sassn{} preconditioners can be found in the \href{https://arxiv.org/abs/2309.02014v2}{arxiv report}.
\fi

\section{Algorithms}
\label{section:algs}
\ifpreprint
In this section, we introduce the PROMISE algorithms SketchySGD (\cref{subsection:sksgd}), SketchySVRG (\cref{subsection:sksvrg}), SketchySAGA (\cref{subsection:sksaga}), and SketchyKatyusha (\cref{subsection:skkat}). 
SketchySGD was first proposed in \cite{frangella2023sketchysgd}; the other algorithms are new.
\else
In this section, we introduce the PROMISE algorithms SketchySVRG (\cref{subsection:sksvrg}), SketchySAGA (\cref{subsection:sksaga}), and SketchyKatyusha (\cref{subsection:skkat}). 
\fi
A summary of these algorithms in provided in \cref{table:algo_summary}.
\ifpreprint
Each algorithm is compatible with all five preconditioning methods (\ssn{}, \nyssn{}, \sassnc{}, \sassnr{}, \diagssn{}) described in \cref{subsection:precond}. 
\else
Each algorithm is compatible with all four preconditioning methods (\ssn{}, \nyssn{}, \sassnc{}, \sassnr{}) described in \cref{subsection:precond}. 
\fi
Each algorithm comes with default hyperparameters, which we use in the empirical evaluation in \cref{section:experiments}. 
In particular, we describe how to automatically compute the learning rate for each algorithm using the estimated preconditioned smoothness constant.
The learning rate is hard to tune in stochastic optimization, and it is remarkable that this automated selection works across a wide range of problems (\cref{section:experiments}).
Finally, we recommend the best algorithm to use for two important applications, ridge and $l^2$-regularized logistic regression (\cref{subsection:algo_recs}).

\begin{table}[H]
\scriptsize
\centering
    \begin{tabular}{C{2.9cm}C{4.6cm}C{1.6cm}C{2.2cm}C{1.75cm}}
        Algorithm & Base Algorithm & Variance reduction & Acceleration & Stochastic gradients only? \\ \hline
        \ifpreprint
        SketchySGD \newline (\cite{frangella2023sketchysgd}) & SGD & \textcolor{purple}{\xmark} & \textcolor{purple}{\xmark} & \textcolor{blue}{\cmark} \\ \hline
        \else
        \fi
        SketchySVRG \newline (\cref{alg:sksvrg}) & SVRG \newline \citep{johnson2013accelerating} & \textcolor{blue}{\cmark} & \textcolor{purple}{\xmark} & \textcolor{purple}{\xmark} \\ \hline
        SketchySAGA \newline (\cref{alg:sksaga}) & b-nice SAGA \newline \citep{gazagnadou2019mbsaga} & \textcolor{blue}{\cmark} & \textcolor{purple}{\xmark} & \textcolor{blue}{\cmark} \\ \hline
        SketchyKatyusha \newline (\cref{alg:skkat}) & Loopless Katyusha \newline \citep{kovalev2020lkatyusha} & \textcolor{blue}{\cmark} & \textcolor{blue}{\cmark} & \textcolor{purple}{\xmark} \\
    \end{tabular}
    \caption{Summary of algorithms in PROMISE. \textcolor{blue}{\textbf{Ticks}} are pros while \textcolor{purple}{\textbf{crosses}} are cons. SVRG and Katyusha require some full gradients rather than stochastic gradients only.}
    \label{table:algo_summary}
\end{table}

\subsection{Notation in algorithms}
\label{subsection:algo_notation}
Throughout this section, $\Model$ denotes a GLM with the oracles defined in \cref{table:m_funcs}. 
\ifpreprint
We use $P$ to denote a preconditioner object, which is a member of one of the five preconditioner classes ($\pssn, \pnyssn, \psassnc, \psassnr, \pdiagssn$).
\else
We use $P$ to denote a preconditioner object, which is a member of one of the four preconditioner classes ($\pssn, \pnyssn, \psassnc, \psassnr$).
\fi
$\mathcal U \subseteq \mathbb N$ denotes a (possibly infinite) set of times that indicate when to update the preconditioner.
We also use the index $j$ to track the time when the preconditioner is constructed: every time the preconditioner is updated, the index $j$ is updated to  the most recently used element of $\mathcal U$.
This index does not have to be tracked in the implementations of these algorithms, but it plays a key role in the theoretical analysis of the proposed algorithms (\cref{section:theory}).

\ifpreprint
\subsection{SketchySGD}
\label{subsection:sksgd}
We present SketchySGD in \cref{alg:sksgd}.

\begin{algorithm}
    \caption{SketchySGD}
    \label{alg:sksgd}
    \begin{algorithmic}
        \Require{initialization $w_0$, gradient and Hessian batchsizes $b_g$ and $b_H$, preconditioner object $P$, model $\Model$, preconditioner update times $\mathcal U$, learning rate multiplier $\alpha$}
        
        \vspace{0.5pc}
        
        \For{$k = 0, 1, \ldots$}
        \If{$k \in \mathcal U$} \Comment{Update preconditioner \& learning rate}
            \State Sample independent batches $\mathcal{S}_k^1, \mathcal{S}_k^2$
            \Comment{$|\mathcal{S}_k^1| = |\mathcal{S}_k^2| = b_H$}
            \State $P.\texttt{update}(\Model, \mathcal{S}_k^1, \mathcal{S}_k^2, w_k)$ \Comment{Compute preconditioner $P_j$ at $w_k$ \& update $P.\lambda_{\mathcal{P}}$}
            \State $\eta \gets \alpha / P.\lambda_{\mathcal{P}}$ \Comment{Update learning rate} 
        \EndIf
        \State Sample batch $\mathcal{B}_k$ \Comment{$|\mathcal{B}_k| = b_g$}
        \State $g_k \gets \Model.\texttt{get\_stoch\_grad}(\mathcal{B}_k, w_k)$
        \State $v_k \gets P.\texttt{direction}(g_k)$ \Comment{Get approx. Newton step $P_j^{-1} g_k$}
        \State $w_{k+1} \gets w_k - \eta v_k$ \Comment{Update parameters}
        \EndFor
    \end{algorithmic}
\end{algorithm}

\paragraph{Explanation of algorithm} SketchySGD is a preconditioned version of SGD that was originally proposed in \cite{frangella2023sketchysgd}. 
Whenever the number of iterations $k \in \mathcal U$, 
both the preconditioner and learning rate are updated.
SketchySGD uses a preconditioned stochastic gradient $v_k$ to update the parameters
rather than performing an SGD step at every iteration using $g_k$.

\paragraph{Default hyperparameters} The main hyperparameters in SketchySGD are the gradient and Hessian batchsizes $b_g$ and $b_H$, preconditioner update times $\mathcal U$, and learning rate multiplier $\alpha$. 
The gradient batchsize $b_g$ can be set to a wide range of values (our experiments on medium-scale datasets use $b_g = 256$, while our experiments on large-scale datasets use $b_g = 4096$); we recommend setting the Hessian batchsize $b_H = \lfloor \sqrt{n} \rfloor$.
The selection $b_H = \lfloor \sqrt{n} \rfloor$ is motivated by the discussion in \cref{subsec:precond_qual}.
We recommend setting $\mathcal U$ according to a constant update frequency $u$, i.e., $\mathcal U = \{0, u, 2u, 3u, \ldots \}$. 
For least squares/ridge regression, which has a constant Hessian, we recommend setting $u = \infty$, which results in the preconditioner being computed exactly once and held constant for the remainder of the optimization. 
For problems with a non-constant Hessian, such as logistic regression, we recommend setting $u = \lceil n/b_g \rceil$, which is equivalent to updating the preconditioner after each pass through the training set.
Based on numerical experiments, we recommend setting the multiplier $\alpha = 1/2$.
\else
\fi

\subsection{SketchySVRG} 
\label{subsection:sksvrg}
We formally introduce SketchySVRG in \cref{alg:sksvrg}.

\begin{algorithm}
\scriptsize
    \caption{SketchySVRG}
    \label{alg:sksvrg}
    \begin{algorithmic}
        \Require{
            initialization $\hat{w}_0$, gradient and Hessian batchsizes $b_g$ and $b_H$, preconditioner object $P$, model $\Model$, preconditioner update times $\mathcal U$, learning rate multiplier $\alpha$, snapshot update frequency $m$
        }
        \Init snapshot $\hat{w} \gets \hat{w}_0$
        
        \vspace{0.5pc}
        
        \For{$s = 0, 1, \ldots$} \Comment{Outer loop}
        \State $\bar{g} \gets \Model.\texttt{get\_full\_grad}(\hat{w})$ \Comment{Full gradient at snapshot}
        \State $w_0 \gets \hat{w}$
        \For{$k = 0, 1, \ldots, m - 1$} \Comment{Inner loop}
        \If{$ms+k \in \mathcal U$} \Comment{Update preconditioner \& learning rate}
        \State Sample independent batches $\mathcal{S}_k^1, \mathcal{S}_k^2$ \Comment{$|\mathcal{S}_k^1| = |\mathcal{S}_k^2| = b_H$}
        \State $P.\texttt{update}(\Model, \mathcal{S}_k^1, \mathcal{S}_k^2, w_k)$ \Comment{Compute preconditioner $P_j$ at $w_k$ \& update $P.\lambda_{\mathcal{P}}$}
        \State $\eta \gets \alpha / P.\lambda_{\mathcal{P}}$ \Comment{Update learning rate} 
        \EndIf
        \State Sample batch $\mathcal{B}_k$ \Comment{$|\mathcal{B}_k| = b_g$}
        \State $\widehat{\nabla}F(w_k) \gets \Model.\texttt{get\_stoch\_grad}(\mathcal{B}_k, w_k)$
        \State $\widehat{\nabla}F(\hat{w}) \gets \Model.\texttt{get\_stoch\_grad}(\mathcal{B}_k, \hat{w})$
        \State $g_k \gets \widehat{\nabla}F(w_k) - \widehat{\nabla}F(\hat{w}) + \bar{g}$ \Comment{Unbiased estimate of $\nabla F(w_k)$}
        \State $v_k \gets P.\texttt{direction}(g_k)$ \Comment{Get approx. Newton step $P_j^{-1} g_k$}
        \State $w_{k+1} \gets w_k - \eta v_k$ \Comment{Update parameters}
        \EndFor
        \State \textbf{Option I:} $\hat{w} \gets w_m$ \Comment{Update snapshot to final inner iterate}
        \State \textbf{Option II:}
        $\hat{w} \gets w_t$ for $t \sim \textrm{Unif}(\{0, 1, \ldots, m - 1\})$ \Comment{Update snapshot randomly}
        \EndFor
    \end{algorithmic}
\end{algorithm}

\ifpreprint
\paragraph{Explanation of algorithm} SketchySVRG is a preconditioned version of SVRG \citep{johnson2013accelerating}.
Similar to SVRG, SketchySVRG consists of an ``outer'' and ``inner'' loop indexed by $s$ and $k$, respectively. 

The algorithm starts in the outer loop by computing a full gradient $\bar{g}$ at the snapshot $\hat{w}$, which is critical for performing variance reduction. The algorithm then sets the first iterate in the inner loop, $w_0$, equal to $\hat{w}$.

The inner loop of the algorithm updates the parameters with a preconditioned, variance-reduced stochastic gradient, $v_k$. 
Similar to SketchySGD, SketchySVRG uses the preconditioner update times $\mathcal U$ to determine when the preconditioner and learning rate should be updated. 

After $m$ iterations of the inner loop, the algorithm returns to the outer loop and updates the snapshot $\hat{w}$ by either using the final inner iterate $w_m$ (Option I) or sampling the previous $m$ iterates uniformly randomly (Option II). 
In practice, we use Option I, but the theoretical analysis is conducted using Option II (\cref{subsection:sksvrg_theory}). 
This discrepancy also appears in the original SVRG analysis \citep{johnson2013accelerating} and is therefore not a drawback of the analysis in this paper.

\paragraph{Default hyperparameters}  The main hyperparameters in SketchySVRG are the gradient and Hessian batchsizes $b_g$ and $b_H$, preconditioner update times $\mathcal U$, learning rate multiplier $\alpha$, and snapshot update frequency $m$. 
We recommend setting $b_g$, $b_H$, and $\mathcal U$ similar to SketchySGD.
We recommend setting $\alpha \in [1/3, 1/2]$; our practical implementation uses the SAGA-inspired update rule $\eta \gets \max \{ 1/(2(\nu n + P.\lambda_{\mathcal{P}})), 1/(3 P.\lambda_{\mathcal{P}}) \}$.
We recommend computing a full gradient every one or two passes through the dataset, i.e., setting $m \in [n/b_g , 2n/b_g]$ or so. Our experiments use $m = \lceil n/b_g \rceil$, which corresponds to computing a full gradient every pass through the training set.
\else
\paragraph{Explanation of algorithm} SketchySVRG is a preconditioned version of SVRG \citep{johnson2013accelerating}.
Similar to SVRG, SketchySVRG consists of an ``outer'' and ``inner'' loop indexed by $s$ and $k$, respectively. 

The algorithm starts in the outer loop by computing a full gradient $\bar{g}$ at the snapshot $\hat{w}$, which is critical for performing variance reduction. The algorithm then sets the first iterate in the inner loop, $w_0$, equal to $\hat{w}$.

The inner loop of the algorithm updates the parameters with a preconditioned, variance-reduced stochastic gradient, $v_k$. 
\ifpreprint
Similar to SketchySGD, SketchySVRG uses the preconditioner update times $\mathcal U$ to determine when the preconditioner and learning rate should be updated. 
\else
SketchySVRG uses the preconditioner update times $\mathcal U$ to determine when the preconditioner and learning rate should be updated. 
\fi

After $m$ iterations of the inner loop, the algorithm returns to the outer loop and updates the snapshot $\hat{w}$ by either using the final inner iterate $w_m$ (Option I) or sampling the previous $m$ iterates uniformly randomly (Option II). 
In practice, we use Option I, but the theoretical analysis is conducted using Option II (\cref{subsection:sksvrg_theory}). 
This discrepancy also appears in the original SVRG analysis \citep{johnson2013accelerating} and is therefore not a drawback of the analysis in this paper.

\paragraph{Default hyperparameters}
SketchySVRG's key hyperparameters include gradient and Hessian batch sizes $b_g$ and $b_H$, preconditioner update times $\mathcal U$, learning rate multiplier $\alpha$, and snapshot update frequency $m$.
For gradient batch size $b_g$, we suggest 256 for medium and 4096 for large datasets, while $b_H$ should be $\lfloor \sqrt{n} \rfloor$, as motivated in \cref{subsec:precond_qual}. 
We recommend setting $\mathcal{U} = \{0, u, 2u, \ldots\}$, where $u = \infty$ for problems with a constant Hessian (i.e., we only update the preconditioner once in total), like least squares/ridge regression, and $u = \lceil n/b_g \rceil$ (i.e., update the preconditioner after reach pass through the dataset) for problems with a non-constant Hessian, such as logistic regression.
We recommend $\alpha \in [1/3, 1/2]$; our practical implementation uses the SAGA-inspired update rule $\eta \gets \max \{ 1/(2(\nu n + P.\lambda_{\mathcal{P}})), 1/(3 P.\lambda_{\mathcal{P}}) \}$.
We recommend snapshot update frequency $m \in [n/b_g , 2n/b_g]$ to compute a full gradient every one or two passes through the dataset; our experiments set $m = \lceil n/b_g \rceil$.
\fi

\subsection{SketchySAGA}
We formally introduce SketchySAGA in \cref{alg:sksaga}.

\label{subsection:sksaga}
\begin{algorithm}
\scriptsize
    \caption{SketchySAGA}
    \label{alg:sksaga}
    \begin{algorithmic}
        \Require{
            initialization $w_0$, gradient and Hessian batchsizes $b_g$ and $b_H$, preconditioner object $P$, model $\Model$, preconditioner update times $\mathcal U$, learning rate multiplier $\alpha$
        }
        \Init gradient table $\psi_0 \gets 0 \in \R^{p \times n}$, table avg. $x_0 \gets \frac{1}{n} \psi_0 \mathbf{1}_n \in \R^p$
        
        \vspace{0.5pc}

        \For{$k = 0, 1, \ldots$}
        \If{$k \in \mathcal U$} \Comment{Update preconditioner \& learning rate}
        \State Sample independent batches $\mathcal{S}_k^1, \mathcal{S}_k^2$ \Comment{$|\mathcal{S}_k^1| = |\mathcal{S}_k^2| = b_H$}
        \State $P.\texttt{update}(\Model, \mathcal{S}_k^1, \mathcal{S}_k^2, w_k)$ \Comment{Compute preconditioner $P_j$ at $w_k$ \& update $P.\lambda_{\mathcal{P}}$}
        \State $\eta \gets \alpha / P.\lambda_{\mathcal{P}}$ \Comment{Update learning rate} 
        \EndIf
        \State Sample batch $\mathcal{B}_k$ \Comment{$|\mathcal{B}_k| = b_g$}
        \State $\mathrm{aux} \gets \sum_{i \in \mathcal{B}_k} (\Model.\texttt{get\_stoch\_grad}(i, w_k) - \psi_k^i$)
        \State $g_k \gets x_k + \frac{1}{|\mathcal{B}_k|} \mathrm{aux}$ \Comment{Unbiased estimate of $\nabla F(w_k)$}
        \State $x_{k+1} \gets x_k + \frac{1}{n} \mathrm{aux}$ \Comment{Update table average}
        \State $\psi_{k+1}^i \gets 
        \begin{cases}
        \psi_k^i, & i \notin \mathcal{B}_k \\
        \Model.\texttt{get\_stoch\_grad}(i, w_k), & i \in \mathcal{B}_k
        \end{cases}$ \Comment{Update table columns for all $i \in [n]$}
        \State $v_k \gets P.\texttt{direction}(g_k)$ \Comment{Get approx. Newton step $P_j^{-1} g_k$}
        \State $w_{k+1} \gets w_k - \eta v_k$ \Comment{Update parameters}
        \EndFor
    \end{algorithmic}
\end{algorithm}

\ifpreprint
\paragraph{Explanation of algorithm} SketchySAGA is a preconditioned version of a minibatch variant of SAGA \citep{gazagnadou2019mbsaga}. 

Similar to SketchySGD, SketchySAGA updates the preconditioner and learning rate at each preconditioner update time in $\mathcal U$.
At each iteration, the algorithm computes the stochastic gradient at every index in the batch $\mathcal{B}_k$. 
These stochastic gradients are used to update $\mathrm{aux}$, an auxiliary vector used to reduce computation. 

The auxiliary vector is then used to update the variance-reduced stochastic gradient $g_k$ and table average $x_{k+1}$.
The gradient table $\psi$ is then updated: if $i$ is in the batch, then the corresponding row of $\psi$ is updated with the appropriate stochastic gradient; otherwise it is left unchanged. This update to $\psi$ is a critical step for variance reduction.

After updating the table, SketchySAGA computes the preconditioned version of the variance-reduced stochastic gradient, $v_k$, and uses this quantity to update the parameters.
Note that SketchySAGA does not require full gradient computations, making this algorithm well-suited for large-scale GLMs.

The memory requirement of SketchySAGA is dominated by the gradient table $\psi$, which requires $\mathcal{O}(np)$ storage in a naive implementation. 
Similar to \cite{defazio2014saga}, the storage of $\psi$ can be reduced to $\mathcal{O}(n)$ for GLMs, since $\nabla f_i(w)$ is a multiple of $a_i$. 
For example, in least squares regression, $\nabla f_i(w) = \nabla (a_i^T w_i - b_i)^2/2 = (a_i^T w_i - b_i) a_i$, so we can simply store $a_i^T w_i - b_i$ in place of $(a_i^T w_i - b_i) a_i$. 
Fully modifying the algorithm to reduce the storage of $\psi$ requires fairly simple changes to the updates of $\mathrm{aux}$ and $\mathrm{\psi}$, while also taking care to decouple the regularization-induced term $\nu w$ from the stochastic gradient calculation. 
We use this improved version of the algorithm in our experiments. 

\paragraph{Default hyperparameters} The main hyperparameters in SketchySAGA are the gradient and Hessian batchsizes $b_g$ and $b_H$, preconditioner update times $\mathcal U$, and learning rate multiplier $\alpha$. We recommend setting $b_g$, $b_H$, and $\mathcal U$ similar to SketchySGD. Similar to SketchySVRG, we recommend setting $\alpha \in [1/3, 1/2]$, although our practical implementation again uses $\eta \gets \max \{ 1/(2(\nu n + P.\lambda_{\mathcal{P}})), 1/(3 P.\lambda_{\mathcal{P}}) \}$.
\else
\paragraph{Explanation of algorithm} 
SketchySAGA, a minibatch variant of SAGA with preconditioning, updates the preconditioner and learning rate at specified times in $\mathcal U$. Each iteration involves computing stochastic gradients for each index in batch $\mathcal{B}_k$, which then update an auxiliary vector $\mathrm{aux}$. This auxiliary vector aids in updating the variance-reduced stochastic gradient $g_k$ and the table average $x_{k+1}$. The gradient table $\psi$ is updated accordingly; if an index $i$ is in the batch, its row in $\psi$ gets updated with the stochastic gradient (essential for variance reduction), otherwise it remains unchanged.

SketchySAGA then calculates the preconditioned variance-reduced stochastic gradient $v_k$ for parameter updates, eliminating the need for full gradient computations and making it efficient for large-scale GLMs. The memory usage is dominated by the gradient table $\psi$, which typically requires $\mathcal{O}(np)$ storage. However, this can be reduced to $\mathcal{O}(n)$ for GLMs \citep{defazio2014saga}. Implementing this storage optimization involves straightforward modifications to the updates of $\mathrm{aux}$ and $\mathrm{\psi}$, and separating the regularization term $\nu w$ from the stochastic gradient calculation. This improved algorithm is used in our experiments.

\paragraph{Default hyperparameters} The main hyperparameters in SketchySAGA are the gradient and Hessian batchsizes $b_g$ and $b_H$, preconditioner update times $\mathcal U$, and learning rate multiplier $\alpha$. We recommend setting $b_g$, $b_H$, and $\mathcal U$ similar to SketchySVRG. Furthermore, we recommend setting $\alpha \in [1/3, 1/2]$, although our practical implementation again uses $\eta \gets \max \{ 1/(2(\nu n + P.\lambda_{\mathcal{P}})), 1/(3 P.\lambda_{\mathcal{P}}) \}$.
\fi


\subsection{SketchyKatyusha}
\label{subsection:skkat}
We formally introduce SketchyKatyusha in \cref{alg:skkat}.

\begin{algorithm}
\scriptsize
    \centering
    \caption{SketchyKatyusha}
    \label{alg:skkat}
    \begin{algorithmic}
        \Require{
        initialization $w_0$, gradient and Hessian batchsizes $b_g$ and $b_H$, preconditioner object $P$, model $\Model$, preconditioner update times $\mathcal U$, momentum multiplier $\alpha$, momentum parameter $\theta_2$, snapshot update probability $\pi$, strong convexity parameter $\mu$ 
        } 
        \Init snapshot $y \gets w_0$, $z_0 \gets w_0$, full gradient $\bar{g} \gets \Model.\texttt{get\_full\_grad}(w_0)$
        
        \vspace{0.5pc}
        
        \For{$k = 0, 1, \ldots$}
        \If{$k \in \mathcal U$} \Comment{Update preconditioner \& learning rate}
        \State Sample independent batches $\mathcal{S}_k^1, \mathcal{S}_k^2$ \Comment{$|\mathcal{S}_k^1| = |\mathcal{S}_k^2| = b_H$}
        \State $P.\texttt{update}(\Model, \mathcal{S}_k^1, \mathcal{S}_k^2, w_k)$ \Comment{Compute preconditioner $P_j$ at $w_k$ \& update $P.\lambda_{\mathcal{P}}$}
        \State $L \gets P.{\lambda_{\mathcal{P}}}$
        \State $\sigma \gets \mu/L$ \Comment{Estimate of inverse condition number}
        \State $\theta_1 \gets \min(\sqrt{\alpha n \sigma}, 1/2)$ \Comment{Update momentum parameter}
        \State $\eta \gets \frac{\theta_2}{(1 + \theta_2) \theta_1}$ \Comment{Update learning rate}
        \EndIf
        \State $x_k \gets \theta_1 z_k + \theta_2 y + (1 - \theta_1 - \theta_2) w_k$ \Comment{``Negative momentum'' step}
        \State Sample batch $\mathcal{B}_k$ \Comment{$|\mathcal{B}_k| = b_g$}
        \State $\widehat{\nabla} F(x_k) \gets \Model.\texttt{get\_stoch\_grad}(\mathcal{B}_k, x_k)$
        \State $\widehat{\nabla} F(y) \gets \Model.\texttt{get\_stoch\_grad}(\mathcal{B}_k, y)$
        \State $g_k \gets \widehat{\nabla} f(x_k) - \widehat{\nabla} f(y) + \bar{g}$ \Comment{Unbiased estimate of $\nabla F(x_k)$}
        \State $v_k \gets P.\texttt{direction}(g_k)$ \Comment{Get approx. Newton step $P_j^{-1} g_k$}
        \State $z_{k+1} \gets \frac{1}{1 + \eta \sigma} (\eta \sigma x_k + z_k - \frac{\eta}{L} v_k)$
        \State $w_{k+1} \gets x_k + \theta_1 (z_{k+1} - z_k)$ \Comment{Update parameters}
        \State Sample $\mathcal{U} \sim \textrm{Unif}([0, 1])$
        \If{$\mathcal{U} \leq \pi$} \Comment{Update snapshot \& full gradient with probability $\pi$}
        \State $y \gets w_k$
        \State $\bar{g} \gets \Model.\texttt{get\_full\_grad}(y)$
        \EndIf
        \EndFor
    \end{algorithmic}
\end{algorithm}

\ifpreprint
\paragraph{Explanation of algorithm} SketchyKatyusha is a preconditioned version of Loopless Katyusha \citep{kovalev2020lkatyusha}. The following explanation of the algorithm is adapted from both \cite{allenzhu2018katyusha} and \cite{kovalev2020lkatyusha}.

Similar to SketchySGD, SketchyKatyusha uses the preconditioner update times $\mathcal U$ to determine when the preconditioner and learning rate should be updated. 
However, the update to the learning rate no longer depends on just the learning rate multiplier and estimated preconditioned smoothness $P.\lambda_{\mathcal{P}}$. 
The learning rate for SketchyKatyusha is calculated by estimating the inverse condition number $\sigma$ to determine momentum parameter $\theta_1$, which is followed by using the momentum parameters $\theta_1, \theta_2$ to determine $\eta$. 

The vectors $x_k$ and $z_k$ are key to performing acceleration. 
At each iteration, SketchyKatyusha performs a ``negative momentum'' step, which computes $x_k$ using a convex combination of $z_k$, the snapshot $y$, and current iterate $w_k$.
This step counteracts the momentum provided by $z_k$ by using the snapshot $y$ as a ``magnet'' that prevents $x_k$ from moving too far away from $y$. 
By using negative momentum, we are able to obtain the benefits of both variance reduction and acceleration at the same time.

After the negative momentum step, SketchyKatyusha computes the preconditioned version of the variance-reduced stochastic gradient, $v_k$. 
SketchyKatyusha then computes $z_{k+1}$, which is followed by a Nesterov momentum-like step to update the parameters $w_{k+1}$.

Finally, SketchyKatyusha randomly updates the snapshot $y$ and full gradient $\bar{g}$ with probability $\pi$. 
This random updating allows us to use just one loop in the implementation, unlike the original formulation of Katyusha \citep{allenzhu2018katyusha}, which uses two loops.

\paragraph{Default hyperparameters}
The main hyperparameters in SketchyKatyusha are the gradient and Hessian batchsizes $b_g$ and $b_H$, preconditioner update times $\mathcal U$, momentum multiplier $\alpha$, momentum parameter $\theta_2$, snapshot update probability $\pi$, and strong convexity parameter $\mu$. We recommend setting $b_g$, $b_H$, and $\mathcal U$ similar to SketchySGD. We recommend setting $\alpha = 2/3$, $\theta_2 = 1/2$, $\pi = b_g/n$, and $\mu = \nu$, where $\nu$ is the regularization parameter in the GLM.
\else
\paragraph{Explanation of algorithm} 
SketchyKatyusha, a preconditioned version of Loopless Katyusha \citep{kovalev2020lkatyusha}, updates the preconditioner and learning rate at specified times in $\mathcal{U}$.
The keys to its acceleration are the vectors $x_k$ and $z_k$; at each iteration, a ``negative momentum'' step computes $x_k$ as a convex combination of $z_k$, snapshot $y$, and current iterate $w_k$, which moderates $x_k$'s deviation from $y$. 
This approach merges the advantages of variance reduction and acceleration.

Following this, SketchyKatyusha calculates the preconditioned variance-reduced stochastic gradient $v_k$ and then $z_{k+1}$, and proceeds with a Nesterov momentum-like step to update the parameters $w_{k+1}$. It sporadically updates the snapshot $y$ and full gradient $\bar{g}$ based on a probability $\pi$, enabling a simpler, single-loop implementation instead of Katyusha's original double-loop design \cite{allenzhu2018katyusha}.

\paragraph{Default hyperparameters}
The main hyperparameters in SketchyKatyusha are the gradient and Hessian batchsizes $b_g$ and $b_H$, preconditioner update times $\mathcal U$, momentum multiplier $\alpha$, momentum parameter $\theta_2$, snapshot update probability $\pi$, and strong convexity parameter $\mu$. We recommend setting $b_g$, $b_H$, and $\mathcal U$ similar to SketchySVRG. We recommend setting $\alpha = 2/3$, $\theta_2 = 1/2$, $\pi = b_g/n$, and $\mu = \nu$, where $\nu$ is the regularization parameter in the GLM.
\fi

\subsection{Algorithm recommendations}
\label{subsection:algo_recs}
We present recommended algorithms for ridge regression and $l^2$-regularized logistic regression in \cref{table:alg_to_use_lsq_reg,table:alg_to_use_log_reg}, respectively. 


\begin{table}[H]
 \footnotesize
 \centering
     \begin{tabular}{C{2.5cm}C{2.7cm}C{2.7cm}C{2.7cm}C{2.25cm}}
         Data Regime & Recommendation \newline (full gradients) & Recommendation \newline (streaming $\leq$ 10 epochs) & Recommendation \newline (streaming $>$ 10 epochs) & Preconditioner \\ \hline
         Dense & SketchyKatyusha & SketchySGD & SketchySAGA & \nyssn{} \\ \hline
         Sparse & SketchyKatyusha & SketchySGD & SketchySAGA & \ssn{} \\ \hline
     \end{tabular}
     \caption{Recommended algorithms for ridge regression. 
     We recommend SketchySGD for streaming settings with limited computation ($\leq$ 10 epochs) as SketchySAGA offers no significant advantage over short durations.}
     \label{table:alg_to_use_lsq_reg}
\end{table}


\begin{table}[H]
 \scriptsize
 \centering
     \begin{tabular}{C{2.5cm}C{2.7cm}C{2.7cm}C{2.7cm}C{2.25cm}}
         Data Regime &  Recommendation \newline (full gradients) & Recommendation \newline (streaming $\leq$ 10 epochs) & Recommendation \newline (streaming $>$ 10 epochs) & Preconditioner \\ \hline
         Dense & SketchyKatyusha \newline SketchySAGA & SketchySGD & SketchySAGA & \nyssn{} \\ \hline
         Sparse & SketchyKatyusha \newline SketchySAGA & SketchySGD & SketchySAGA & \ssn{} \\ \hline
     \end{tabular}
     \caption{Recommended algorithms for $l^2$-regularized logistic regression.
     We recommend SketchySGD for streaming settings with limited computation ($\leq$ 10 epochs) as SketchySAGA offers no significant advantage over short durations.}
     \label{table:alg_to_use_log_reg}
\end{table}

\section{Related work}
\label{section:related_work}
We review the literature on stochastic second-order and preconditioned stochastic gradient methods for solving \eqref{eq:EmpRiskProb}, with emphasis on work that assumes strong convexity. 

\subsection{Stochastic second-order and preconditioned stochastic gradient methods}
The deficiencies of the stochastic first-order methods presented in \cref{section:Introduction} are well-known within the optimization and machine learning communities.
Indeed, in the past decade or so, research on stochastic second-order methods and stochastic preconditioning techniques for finite-sum optimization has exploded.
Roughly, these methods can be divided into three categories: 1) stochastic second-order methods that use full gradients, 2) stochastic second-order methods that use stochastic gradients, and 3) preconditioned stochastic gradient methods.
The dividing line between stochastic second-order methods and preconditioned methods is not always clear, as many preconditioners use second-order information, including the PROMISE framework. 
We review the literature on these three approaches in detail below.

\ifpreprint
\paragraph{Stochastic second-order methods with full gradients.}
We begin with stochastic second-order methods that use full gradients and a stochastic approximation to the Hessian.
To the authors' knowledge, the earliest method of this form targeting \eqref{eq:EmpRiskProb} is \cite{byrd2011use}.
\cite{byrd2011use} subsample the Hessian and use this stochastic approximation in conjunction with an L-BFGS style update. \cite{erdogdu2015convergence,roosta2019sub} independently investigated the application of Newton's method to solve \eqref{eq:EmpRiskProb}, where the Hessian is replaced with an approximation constructed through subsampling.
The subsampled Newton method, as pioneered in these works, serves as the foundation for many subsequent developments in stochastic second-order methods.

In addition to introducing new algorithms, the works discussed above also provide analysis that lead to various convergence guarantees, which we now review. 
The earliest analysis of \cite{byrd2011use} is quite coarse, only showing their method converges to the global optimum, provided the objective is strongly convex and the subsampled Hessian is always positive definite.
\cite{byrd2011use} provides neither a convergence rate nor a theoretical advantage over first-order methods.
The analyses of \cite{erdogdu2015convergence} and \cite{roosta2019sub} yield considerably stronger results.
Both works establish linear convergence in the strongly convex setting. 
Furthermore, \cite{roosta2019sub} prove local superlinear convergence of Subsampled Newton, albeit under certain unattractive assumptions such as exponentially growing the Hessian batchsize $b_H$.  
Despite these assumptions, the results of \cite{roosta2019sub} point to potential benefits of stochastic second-order methods over first-order methods.
We also note the analyses of these papers have been refined by \cite{ye2021Appx,na2022hessian}.
In particular, \cite{na2022hessian} propose a novel averaging scheme for the subsampled Hessian, which achieves local superlinear convergence without requiring a growing Hessian batchsize.
Unfortunately, this approach is limited to settings where the dimension $p$ of the feature vectors is modest,
as it requires forming the subsampled Hessian for averaging, at a computational cost of $O(b_Hp^2)$ and a storage cost of $O(p^2).$
  
As an alternative to subsampling, some methods use sketching to construct a stochastic approximation to the full Hessian \citep{pilanci2017newton,gower2019rsn,lacotte2021adaptive}. 
Sketching the Hessian has two main benefits over subsampling: (i) it generally produces higher-accuracy approximations to the Hessian \citep{martinsson2020randomized}, and (ii) it is robust to the origins of the data. 
By robust, we mean that with an appropriate sketching matrix, the sketch size required to ensure the $\zeta$-spectral approximation property is independent of the ridge leverage coherence.
In detail, if the Hessian approximation is constructed from a sketching matrix belonging to an appropriate random ensemble, the sketch size required to ensure the $\zeta$-spectral approximation property holds with high probability, is only $\Otil(\deff(A))$ \citep{lacotte2021adaptive}.
In contrast, the Hessian batchsize required by subsampling to ensure the $\zeta$-spectral approximation property depends upon the ridge leverage coherence, which can be quite large when the data contains outliers.

However, sketching has several disadvantages relative to subsampling. 
A notable disadvantage of existing sketching-based methods is that they require a full pass through the data to approximate the Hessian, while subsampling methods do not. 
It is desirable to minimize full passes through the data when solving large-scale problems, which limits the usefulness of existing sketching-based methods in this setting.
Another limitation of sketching-based methods, is that they may only be applicable to problems with certain structure.
As a concrete example, the Newton Sketch \citep{pilanci2017newton,lacotte2021adaptive} requires access to a matrix $R$ such that $\nabla^2 f(w) = R^{T}R$. 
While such a matrix is always available when $f$ is a GLM, this is not the case for more general losses. 
In contrast, subsampling can always be used to approximate the Hessian of a finite-sum objective, regardless of the form of the loss function. 
Last, we note that if $f$ is a GLM, then the Newton Sketch with a row-sampling sketching matrix is equivalent to Subsampled Newton.

The analysis guarantees of sketching-based (approximate) second-order methods are similar to their subsampled counterparts. 
\cite{pilanci2017newton,lacotte2021adaptive} focus on self-concordant functions for their global convergence analysis and show fast linear convergence independent of the condition number. 
Additionally, \cite{pilanci2017newton} show local superlinear convergence for smooth and strongly convex objectives with Lipschitz Hessians,  but require a sketch size that depends on the condition number of the problem, which can be larger than $n$ when the problem is ill-conditioned. 
However, in the setting where $p$ is moderate in size, this issue may be resolved by using the Hessian averaging scheme of \cite{na2022hessian}.
\cite{gower2019rsn} prove convergence for functions that are relatively smooth and relatively convex, a generalization of smoothness and strong convexity to the local Hessian norm. 
They establish linear convergence with a rate that depends upon the \emph{relative condition number} and the smallest non-zero eigenvalue of an expected projection matrix.
When the objective is quadratic, the relative condition number equals $1$ and so the convergence rate is independent of the condition number, which shows an improvement of over first-order methods.
\else
\fi

\ifpreprint
\paragraph{Stochastic second-order methods with full gradients for kernel methods.} 
In machine learning, the primary metric of interest is generalization error, i.e., the error made by the model on unseen data. 
Motivated by this, a line of work beginning with \cite{rudi2017falkon} and culminating with \cite{marteau2019globally}, has developed fast algorithms for $l^2$-regularized kernel methods, for which low generalization error is the main goal.  
These approaches may be viewed as sketching-based Newton methods to solve a reduced form of the objective.
Both \cite{rudi2017falkon}, and \cite{marteau2019beyond} leverage the small effective dimension of the kernel operator to reduce the original problem involving an $n\times n$ kernel matrix to a smaller problem with an $n\times \bigO(\deff)$ kernel matrix.
However, the resulting optimization problem is highly ill-conditioned, and $n$ is very large, so that exact Newton's method is prohibitively expensive.
To address this issue, \cite{rudi2017falkon} and \cite{marteau2019globally} use preconditioned Newton Conjugate Gradient (Newton-CG). 
In particular, they exploit that $n\gg \deff$, so that sketch-and-precondition style preconditioners \citep{martinsson2020randomized} may be used with PCG, which leads to a fast solution of the Newton system at each iteration.
This approach is closely related to the Newton Sketch. 
Indeed, the Newton Sketch directly solves an approximate Newton system with an approximate Hessian, while \cite{rudi2017falkon,marteau2019globally} solves the exact Newton system with PCG, using the approximate Hessian as the preconditioner.
Significantly, \cite{rudi2017falkon,marteau2019globally} both show that their proposed algorithms achieve optimal statistical rates, despite reducing the original problem to a smaller problem more amenable to computation.
When statistical error is the primary metric of interest, the algorithms of \cite{rudi2017falkon,marteau2019globally} can deliver significant computational gains.
The main deficit of these approaches is that it can be difficult to select the number of centers a priori, as the effective dimension is unknown.
This can be problematic in practice, as if too few centers are chosen, the statistical performance will be much worse than that obtained by solving the exact problem \citep{diaz2023robust}.
Last, when $n$ is large enough, even computing full gradients involving the reduced kernel matrix is too expensive, which necessitates the use of stochastic gradients.
\else
\fi

\paragraph{Stochastic second-order methods with stochastic gradients for solving \eqref{eq:EmpRiskProb}.}

In scenarios where both $n$ and $p$ are large, full gradients become prohibitively expensive. 
Consequently, a scalable second-order method must use both stochastic gradients and stochastic Hessian approximations. 
To address this challenge, many methods that employ fully stochastic first- and second-order information have been proposed for solving \eqref{eq:EmpRiskProb}.
All such proposals employ subsampling-based approximations to the Hessian \citep{byrd2016stochastic,moritz2016linearly,gower2016stochastic,bollapragada2018progressive, roosta2019sub, bollapragada2019exact,derezinski2022stochastic}.
These methods can be further categorized based on whether they directly compute the search direction by applying the inverse subsampled Hessian \citep{roosta2019sub,bollapragada2019exact,derezinski2022stochastic}, or by using the subsampled Hessian to stabilize an L-BFGS style update \citep{byrd2016stochastic,moritz2016linearly,gower2016stochastic,bollapragada2018progressive}.
Hence most stochastic second-order methods that use stochastic gradients have their roots in the full-gradient methods of \cite{byrd2011use,erdogdu2015convergence,roosta2019sub}. 

The convergence guarantees of existing proposals vary greatly, often leaving much to be desired. 
\cite{byrd2016stochastic} established an $\bigO(1/k)$-rate for their stochastic L-BFGS method assuming strong convexity. 
However, their analysis relies on two restrictive assumptions: (i) bounded variance of stochastic gradients and (ii) strict positive definiteness of the subsampled Hessian. 
The first is known to be false for strongly convex functions, unless the iterates lie in a compact set, and the second fails in common applications such as GLMs, where the subsampled Hessian is singular unless $b_H \geq p$.
\cite{roosta2019sub} present a range of convergence results for a variety of settings, including when $F$ is strongly convex, for which they establish fast local linear convergence. 
Hence the Subsampled Newton method enjoys the fast local convergence of Newton's method, albeit at a linear rather than quadratic rate.
Nevertheless, the fast local convergence rate shows an advantage over stochastic first-order methods, whose convergence depends upon the condition number, regardless of how close the iterates are to the optimum. 
Unfortunately, in order to achieve their fast local convergence result, \cite{roosta2019sub} require an exponentially increasing gradient batchsize, and that the subsampled Hessian batchsize satisfy $b_H = \Otil(\kappa/\epsilon^2)$, where $\epsilon \in (0,1)$. 
Hence the theoretical analysis requires rapidly growing gradient batchsizes and large Hessian batchsizes, which is antithetical to the purpose of stochastic methods.

The first work to obtain a linear rate of convergence for solving \eqref{eq:EmpRiskProb} without requiring large/growing gradient and Hessian batchsizes was \cite{moritz2016linearly}.
\cite{moritz2016linearly} combine the stochastic L-BFGS method of \cite{byrd2016stochastic} with SVRG to reduce the variance of the stochastic gradients without growing the gradient batchsize. 
Although \cite{moritz2016linearly} proves global linear convergence, fast local linear convergence is not established, and no theoretical benefit over SVRG is demonstrated.
Similar remarks hold for the stochastic L-BFGS methods of \cite{gower2016stochastic,bollapragada2018progressive}.
More recently, \cite{derezinski2022stochastic} proposed Stochastic Variance Reduced Newton (SVRN), which combines Subsampled Newton with SVRG. 
\cite{derezinski2022stochastic} shows SVRN exhibits fast local linear convergence.
However, the analysis requires the gradient batchsize to satisfy $b_g = \Otil(\kappa)$, which is still very large, 
and can easily exceed $n$ for ill-conditioned problems.

\paragraph{Preconditioned stochastic gradient methods.}

Taking a general viewpoint, the stochastic second-order methods discussed above are all special cases of preconditioned stochastic gradient methods, where the current preconditioner $P_k$ is based on an approximation to the Hessian matrix.
An early notable proposal is the preconditioned SVRG algorithm by \cite{gonen2016solving}, which employs a preconditioner obtained through low-rank approximation to the Hessian using the randomized block Krylov method \citep{musco2015randomized}. 
Despite exhibiting improvements over SVRG, this approach requires multiple passes through the data matrix, hindering its suitability for larger problems. Another method, SVRG2 by \cite{gower2018tracking}, combines SVRG with a preconditioner based on a randomized Nyström approximation to the Hessian, but also requires a costly full pass through the data at every outer iteration. 
\cite{liu2019acceleration} propose preconditioned variants of SVRG and Katyusha using either the covariance matrix or a diagonal approximation. 
However, the former is impractical for large-scale settings, and the latter, although scalable, may perform poorly. 
Additionally, their fixed preconditioner approach may lead to suboptimal performance for non-quadratic problems. 

\paragraph{Relation to PROMISE.}
The methods most closely related to PROMISE are Subsampled Newton \citep{roosta2019sub} and SVRN \citep{derezinski2022stochastic}.
When a PROMISE method uses the \ssn{} preconditioner, it may be viewed as combining Subsampled Newton with the corresponding stochastic gradient algorithm.
Variance reduction stabilizes PROMISE iterations to allow linear convergence without exponentially growing gradient batchsizes,
unlike the batchsizes required by Subsampled Newton \citep{roosta2019sub}. 
As SVRN is simply Subsampled Newton combined with SVRG, it follows that SketchySVRG equipped with the \ssn{} preconditioner is equivalent to SVRN. 
However, despite this equivalence, the aims of this work and those of \cite{derezinski2022stochastic} are quite different. 
\cite{derezinski2022stochastic} focuses on proving fast local linear convergence under the hypotheses of large gradient minibatches: their results require $b_g = \Otil(\kappa)$.
Moreover, \cite{derezinski2022stochastic} only suggests using SVRN to finish off the optimization, and that in the beginning, Subsampled Newton with full gradients should be used to get the iterates sufficiently close to the optimum.
This recommendation contrasts with this work, which shows global linear convergence, allows for lazy updates, and admits variance reduction schemes beyond SVRG.
We also prove fast local linear convergence of SketchySVRG with moderate gradient batchsizes, 
a significant theoretical and practical improvement over the requirements of \cite{derezinski2022stochastic}. 

To facilitate a straightforward comparison between PROMISE and prior work, we present \cref{table:2nd-ord-comp}. 
\cref{table:2nd-ord-comp} compares the properties of various stochastic second-order methods for solving \eqref{eq:EmpRiskProb} when $F$ is a GLM.
We select SketchySVRG as a representative for PROMISE. 
Inspection of \cref{table:2nd-ord-comp} shows SketchySVRG is the method that enjoys the best batchsize requirements, while still attaining fast local linear convergence.
Indeed, $\chi^\nu(\nabla^2f(w))\deff(\nabla^2f(w))$ and $\tau_\star^\nu$ are always smaller than $\kappaMax$ and $n$ (\cref{lem:deff_vs_kappa}, \cref{lem:hess_sim}).
Moreover, it is the only method in \cref{table:2nd-ord-comp} whose theory accounts for lazy updates to the preconditioner, which is essential for good practical performance.  

\ifpreprint
\subsection{Lazy updating and Newton's method.}
One of the major contributions of this work is to show linear convergence is still attainable (without harming the rate) even when the stochastic Hessian approximation is updated \emph{infrequently}. 
In the literature, this is known as \emph{lazy updating} or \emph{lazy Hessians} \citep{doikov2023second}. 
Despite being of immense practical interest, convergence of Newton's method under lazy updating in the \emph{full gradient} setting was only recently analyzed in \cite{doikov2023second}, which proved an optimal update frequency of $p$ iterations.
Our setting is quite different from \cite{doikov2023second}, as our focus is on solving the smooth convex finite-sum problem \eqref{eq:EmpRiskProb} with preconditioned stochastic gradient methods based on stochastic Hessian approximations. 
Hence the analysis in \cite{doikov2023second} is incompatible with PROMISE.
\else
\fi

\ifpreprint
\subsection{Fast global convergence in theory?}
Of all the stochastic second-order methods with stochastic gradients that we have discussed, none show a \emph{global improvement} over stochastic first-order methods---at best they show improved local convergence. 
Moreover, the worst-case global convergence rate of each method scales as $\Otil(\kappa^2)$, which is a factor of $\kappa$ worse than that of stochastic variance-reduced first-order methods.
This gap has been noticed in the literature: \cite{gower2020variance} discusses the disappointing convergence guarantees of existing second-order stochastic methods, noting that none of them have a better global convergence rate than that of SVRG.
Given that the optimal number of stochastic gradient queries for any method using a stochastic first-order oracle is achieved by Katyusha (with $b_g = 1$) \citep{woodworth2016tight,allenzhu2018katyusha}, it provides the more appropriate baseline. 
Making the previous substitution leads to the following problem, which to our knowledge, is open.

\paragraph{Open problem.} \textit{Suppose} $F$ \textit{is given by} \eqref{eq:EmpRiskProb}, \textit{is smooth, strongly convex, and each} $F_i$ \textit{has a Lipschitz Hessian.} \textit{Does there exist a a stochastic second-order method that matches or improves upon the} $\bigO\left((n+\sqrt{n\kappaMax})\log(1/\epsilon)\right)$ \textit{stochastic gradient queries of Katyusha?}

In this work we provide an affirmative answer to this question, albeit it is mainly of theoretical interest.
The solution is based on the fact that SketchySVRG exhibits fast local linear convergence, independent of the condition number.
To obtain the improved query complexity, we propose a hybrid algorithm that consists of two phases.
Phase 1 runs Katyusha with batchsize $1$ until the iterates reach the region of local convergence of SketchySVRG, at which point Phase 2 commences. 
Phase 2 simply runs SketchySVRG as in \cref{thm:sksvrg_loc_con}, until the desired precision is reached.
Now, with high probability, the number of queries required by Katyusha is $\bigO(n+\sqrt{n\kappaMax})$, while for SketchySVRG the number of queries is bounded as $\bigO(n\log(1/\epsilon))$, where we have used \cref{thm:sksvrg_loc_con} and the fact that $\tau_\star^\nu\leq n$. 
Hence, the total number of queries required hybrid algorithm is $\bigO(\sqrt{n\kappaMax}+n\log(1/\epsilon))$.
Thus, the proposed hybrid algorithm improves upon Katyusha by a logarithmic factor, providing an affirmative answer to the open question.
However, from a practical viewpoint, this algorithm is very unsatisfactory, as it does not match how stochastic second-order methods are actually used. 
Moreover, empirically SketchySVRG outperforms vanilla Katyusha, see \cref{section:experiments}.

Although the hybrid algorithm discussed above is purely theoretical, we believe it is unlikely that any other stochastic second-order method can significantly improve upon it in the worst-case. 
Our reasoning for this belief is rooted in work on lower bounds for second-order methods in the full gradient setting, where it has been known for some time that second-order methods do not converge faster than first-order methods in the worst-case.  
Indeed, in their seminal work, \cite{nemirovskij1983problem} note (without proof) that for an arbitrary smooth and strongly convex function $F$, second-order methods do not improve upon first-order methods.
More relevant to our setting is \cite{arjevani2017oracle}, who consider \eqref{eq:EmpRiskProb}, and in addition to assuming smoothness and strong convexity, allow $F$ to have a Lipschitz Hessian.
They show even under these hypotheses, that in the worst-case, second-order methods do not provide an improvement (up to logarithmic factors) over the convergence rate of accelerated gradient descent.
\cite{arjevani2017oracle}'s disappointing conclusion also applies to approximate schemes like Subsampled Newton methods \citep{erdogdu2015convergence,roosta2019sub,bollapragada2019exact}.
Hence, in the worst-case, when full gradients are used, no global improvement is possible from using second-order information.
These negative conclusions are quite surprising, as in practice Newton's method often converges in tens of iterations on problems for which first-order methods converge slowly.
Nevertheless, given these negative results for the full gradient setting, we hypothesize that (without further assumptions) no stochastic second-order method can significantly improve upon the query complexity of the proposed hybrid algorithm. 
It remains an interesting open question whether there is a stochastic second-order algorithm that achieves the complexity of the hybrid algorithm, while using second-order information throughout the optimization, which would match how second-order methods are used in practice. 

The previous paragraph is pessimistic in its conclusions, but it does not imply that stochastic second-order methods do not yield faster global convergence in practice---they clearly do. 
The purpose of the preceding discussion is to delineate what kind of convergence guarantees are reasonable to expect from stochastic second-order methods. 
If the class of functions considered is not further restricted by adding more assumptions, the previous discussion shows it is unreasonable to expect an improvement over the global convergence rate of stochastic first-order methods. 
Thus, in terms of theoretical guarantees, any reasonable stochastic second-order method should possess either (or both) of the following properties: (i) fast global linear convergence on quadratic objectives, (ii) fast local convergence, independent of the condition number, once the iterates are sufficiently close to the optimum.
Improved convergence results beyond these two items seem very unlikely, unless further assumptions are imposed.
We note the proposed PROMISE family satisfy these desiderata.
\else
\fi

\begin{table}[t] 
\scriptsize
\centering 
    \begin{tabular}{C{2cm}C{2cm}C{5cm}C{2cm}C{2cm}}
        Algorithm & $b_g$ & $b_H$/\newline Sketch size & Lazy preconditioner updates & Fast local-linear convergence \\ \hline
        SketchySVRG \newline (\cref{alg:sksvrg}) & $\Otil(\tau_\star^{\nu})$ & \newline $\Otil\left(\frac{\chi^\nu (A^T\Phi''(Aw)A)\deff(A^T\Phi''(Aw)A)}{\zeta^2_0}\right)$ \newline & \textcolor{blue}{\cmark} & \textcolor{blue}{\cmark} \\ \hline
        Subsampled Newton \newline \citep{roosta2019sub} & Exponentially \newline increasing & $\Otil(\frac{\kappaMax}{\zeta^2_0})$& \textcolor{purple}{\xmark} & \textcolor{blue}{\cmark}\\ \hline
        Newton Sketch \newline \citep{lacotte2021adaptive} & Full & $\Otil\left(\frac{d^{\nu}_{\textrm{eff}}(A^T\Phi''(Aw)A)}{\zeta^2_0}\right)$& \textcolor{purple}{\xmark} & \textcolor{blue}{\cmark} \\ \hline
        SVRN \newline \citep{derezinski2022stochastic} & $\Otil(\kappaMax)$ & $\Otil\left(\frac{\kappaMax}{\zeta_0^2}\right)$ & \textcolor{purple}{\xmark} & \textcolor{blue}{\cmark} \\ \hline
        SLBFGS \\ \citep{moritz2016linearly} & Constant & Constant & \textcolor{purple}{\xmark} & \textcolor{purple}{\xmark} \\ \hline
        Progressive \newline Batching L-BFGS \newline \citep{bollapragada2018progressive} & Increasing & Increasing & \textcolor{purple}{\xmark} & \textcolor{purple}{\xmark} \\
    \end{tabular}
    \caption{Comparison of preconditioned stochastic gradient methods for solving \eqref{eq:EmpRiskProb} when $F$ is a GLM. Here $b_g$ and $b_H$ are gradient and Hessian batchsizes, $\kappaMax = L_{\textrm{max}}/\nu$ is the condition number, $\chi^\nu(A^{T}\Phi(Aw)A)$ and $\deff(A^{T}\Phi(Aw)A)$ are the ridge-leverage coherence and effective dimension of the Hessian, while $\tau_\star^\nu$ denotes the Hessian dissimilarity.
    Note $\chi^\nu(A^{T}\Phi(Aw)A)\deff(A^{T}\Phi(Aw)A)$ and $\tau_\star^\nu$ are always smaller than $\kappaMax$. 
    Hence, of all the methods, SketchySVRG has the best required gradient and Hessian batchsizes, and is the only method whose theory accounts for lazy updates.} 
    \label{table:2nd-ord-comp}
\end{table}

\section{Theory}
\label{section:theory}
In this section we establish (global and local) linear convergence results for the PROMISE methods on smooth, strongly convex, finite-sum objectives 
(including but not restricted to GLMs).
This section begins with our assumptions and then introduces the key concepts of quadratic regularity and the quadratic regularity ratio, which generalize the notions of strong convexity, smoothness, and condition number to the Hessian norm.
\ifpreprint
We then provide a simple theoretical example illustrating why $\zeta$-spectral approximations and quadratic regularity should improve convergence over vanilla gradient descent.
We follow this by introducing Hessian dissimilarity, which plays a key role in analyzing PROMISE methods with stochastic gradients, and preconditioner stability, which is important for proving convergence under lazy updating.
\else
We follow this by introducing Hessian dissimilarity, which plays a key role in analyzing PROMISE methods with stochastic gradients.
\fi
Finally, we state the main convergence theorems and provide a convergence proof for SketchySVRG which illustrates the techniques in our analysis. Our first theorem establishes global linear convergence, while our second theorem shows global SketchySVRG achieves fast, local convergence independent of the condition number.
Any proofs not provided in this section can found in the \href{https://arxiv.org/abs/2309.02014v2}{arxiv report}.
Linear convergence results for SketchySAGA and SketchyKatyusha are also available in the \href{https://arxiv.org/abs/2309.02014v2}{arxiv report}.
\ifpreprint
Proofs not presented in the main text are in \cref{sec:theory_pfs}.

As there are numerous ideas and results in this section, we provide \cref{table:theory_summary}, which lists the key ideas and results along with cross references to where they appear.

\begin{table}[t]
 \centering
     \begin{tabular}{C{4cm}C{6.5cm}C{2.5cm}}
         \textbf{Idea/result} & \textbf{Expression(s)} &\textbf{Reference} \\ \hline
         Quadratic regularity & $\gamma_u(\mathcal C), \gamma_\ell(\mathcal C), \q(\mathcal C)$   & \cref{def:QuadReg} \\ \hline
         Hessian dissimilarity & $\tau_\star^{\nu}(\mathcal C)$ & \cref{def:hess_sim} \\ \hline
         Preconditioner stability & $\|w_j-w_\star\|^2_{P_{j+1}}/\|w_j-w_\star\|^2_{P_j},$\newline $\|w_{j}-w_\star\|^2_{P_{j-1}}/\|w_{j}-w_\star\|^2_{P_{j}}$ & \cref{prop:stable_precond} \\ \hline
         Expected preconditioned smoothness & $\E\|\widehat \nabla F(w)-\widehat \nabla F(w')\|^2_{P^{-1}}$ & \cref{prop:precond_grad_var} \\ \hline
         SketchySVRG convergence & $F(w)-F(w_\star)$ & \cref{thm:conv_SketchySVRG} \\ \hline
         SketchySVRG fast local convergence & $F(w)-F(w_\star)$& \cref{thm:sksvrg_loc_con} \\ \hline
         SketchySAGA convergence & $\|w-w_\star\|^2_{P_k}$ & \cref{thm:conv_sksaga} \\ \hline
         SketchyKatyusha convergence &  $F(w)-F(w_\star), F(y)-F(w_\star), \|z-w_\star\|_{P_k}^2$ & \cref{thm:conv_skkat}  \\
         \hline
     \end{tabular}
     \caption{A summary of the main ideas and theoretical results of this work.}
     \label{table:theory_summary}
\end{table}
\else
\fi

\subsection{A subtlety in notation}
We index the preconditioner in two different ways in this section;
sometimes we denote the preconditioner by $P_j$, and other times we denote the preconditioner by $P_k$ (or in the case of SketchySVRG, $P_k^{(s)}$).
In this setting, $j$ indexes the iterate where the preconditioner is constructed, while $k$ (or $_k^{(s)}$) indexes the current iterate in the algorithm.

There is a simple way to map a $P_k$ to the corresponding $P_j$. 
If the preconditioner update indices $\mathcal U = \{u_1, u_2, \ldots, u_m \}$, then a given $P_k$ for $k \in \{u_i, u_i+1, \ldots, u_{i+1} - 1\}$ is the same as $P_j$ for $j = u_i$.
As a concrete example, suppose $\mathcal U = \{0, 4, 10\}$.
Then, for $k \in \{0, 1, 2, 3\}$, $P_k$ is the same as $P_j$ where $j = 0$; for $k \in \{4, 5, 6, 7, 8, 9\}$, $P_k$ is the same as $P_j$ where $j = 4$; for $k \in \{10, 11, \ldots\}$, $P_k$ is the same as $P_j$ where $j = 10$.

\subsection{Assumptions}
Here we provide assumptions that will be needed in the convergence analyses of the PROMISE methods.
\begin{assm}[Smoothness and convexity]
\label{assm:SmoothSC}
For each $i\in [n]$, $f_i(w)$ is $L_i$-smooth and convex.
\end{assm}
The above assumption is standard in the analysis of stochastic gradient methods for solving \eqref{eq:EmpRiskProb}.

\begin{assm}[$\zeta$-spectral approximation]
\label{assm:GoodPrecond}
    If the preconditioner $P_{j}$ was constructed at $w_{j}$, where $j\in \mathcal U$, then
    \[
    (1-\zeta) P_{j} \preceq \nabla^2 f(w_j)+\nu I \preceq (1+\zeta) P_{j},
    \]
    where $ \zeta \in (0,1)$.
\end{assm}
\cref{assm:GoodPrecond} states each preconditioner constructed by the algorithm satisfies the $\zeta$-spectral approximation property, which is reasonable as PROMISE preconditioners satisfy this property with high probability.
\cref{assm:GoodPrecond} can be viewed as conditioning on the event that the preconditioners constructed by the algorithm satisfy the $\zeta$-spectral approximation property, which holds with high probability via union bound.
\cref{assm:GoodPrecond} makes it easy to establish convergence of any PROMISE method equipped with this preconditioner. 

\ifpreprint
\begin{assm}[Finite number of preconditioner updates]
\label{assm:FiniteUpates}
Let $P_k$ denote the preconditioner at iteration $k$. There exists $J\geq 0$, such that for all $k\geq J$ we have
\[
P_k = P_{J}.
\]
\end{assm}
This last assumption is somewhat technical in nature, and is needed only for SketchySAGA and SketchyKatyusha.
The need for \cref{assm:FiniteUpates} stems from the convergence analysis of SketchySAGA and SketchyKatyusha involving the metric $\|\cdot\|_{P_k}$, and having to compare it to the next (previous) metric $\|\cdot\|_{P_{k+1}} (\|\cdot\|_{P_{k-1}})$.
Consequently, if the number of preconditioner updates is not finite, we cannot guarantee that the quantities in $\mathcal E_P$ and $\mathcal B_P$ in \cref{prop:stable_precond} are finite, which prevents establishment of convergence.
By instating \cref{assm:FiniteUpates}, we remove this obstacle.
From the viewpoint of theory, this assumption isn't very limiting, as once the iterates are close enough to the optimum, the benefit from updating the preconditioner is negligible.
Nevertheless, we believe \cref{assm:FiniteUpates} is unnecessary in practice, and is an artifact of the analysis.
\else

\fi

\subsection{Technical preliminaries}
\subsubsection{Quadratic regularity}
We start by introducing the upper and lower quadratic regularity constants for a smooth, convex function $F:\mathcal C \mapsto \R$, where $\mathcal C$ is a closed convex subset of $\R^p$.
These ideas are crucial for establishing linear convergence under infrequent updating of the preconditioner.  
\begin{definition}[Quadratic regularity]
\label[definition]{def:QuadReg}
    Let $F$ be a twice differentiable function, and $\mathcal C$ a closed convex set.
    Then we say $F$ is $\mathcal C$-upper quadratically regular, if there exists $0\leq \gamma_u(\mathcal C)<\infty$, such that for all $w_0,w_1,w_2\in \mathcal C$
    \[
    F(w_2)\leq F(w_1)+\langle \nabla F(w_1),w_2-w_1\rangle +\frac{\gamma_{u}(\mathcal C)}{2}\|w_2-w_1\|^2_{\nabla^2 F(w_0)}.
    \]
    Similarly, we say $F$ is $\mathcal C$-lower quadratically regular, if there exists $0<\gamma_{\ell}(\mathcal C)$,  such that for all $w_0,w_1,w_2\in \mathcal C$
    \[
        F(w_2)\geq F(w_1)+\langle \nabla F(w_1), w_2 -w_1\rangle+\frac{\gamma_\ell(\mathcal C)}{2}\|w_2-w_1\|_{\nabla^2 F(w_0)}^2.
    \]
    We say $F$ is $\mathcal{C}$-quadratically regular if $0<\gamma_\ell$ and $\gamma_u<\infty$.
    Further, if $F$ is $\mathcal{C}$-quadratically regular, we define the \emph{quadratic regularity ratio} to be 
    \[
    \q(\mathcal C) \coloneqq \frac{\gamma_u(\mathcal C)}{\gamma_\ell(\mathcal C)}.
    \]
    Moreover, if $F(w) = \frac{1}{n}\sum_{i=1}^{n}F_i(w)$, and each $F_i$ is $\mathcal C$-quadratically regular, we denote the corresponding quadratic regularity constants by $\gamma_{u_i}(\mathcal C)$ and $\gamma_{\ell_i}(\mathcal C)$, and we define
    \[
    \gammaMax(\mathcal C) \coloneqq \max_{i\in[n]} \gamma_{u_i}, \quad \gammaMin(\mathcal C) \coloneqq \min_{i\in [n]}\gamma_{\ell_i}.
    \]
\end{definition}
Quadratic regularity holds whenever $F$ can be upper- and lower-bounded in terms of the Hessian at any $w_0 \in \mathcal C$. 
Hence the upper and lower quadratic regularity constants may be viewed as global generalizations of the smoothness and strong convexity constants to the Hessian norm $\|\cdot\|_{\nabla^2F(w)}$. 
Moreover, the quadratic regularity ratio $\q$ generalizes the condition number $\kappa$ to the Hessian norm. 

Upper and lower quadratic regularity expand upon stable Hessians from \cite{karimireddy2018global} and its refinements relative smoothness and relative convexity from
\cite{gower2019rsn}. 
The relative smoothness and relative convexity parameters from \cite{gower2019rsn} are defined similarly to the quadratic regularity constants, except they lack the outer supremum and infimum present in the definitions of $\gamma_u$ and $\gamma_\ell$, respectively. 
Unfortunately, relative smoothness and relative convexity are insufficient for our analysis, which incorporates infrequent updating. 
Under infrequent updating, the preconditioner is constructed at a point $w_0\neq w_1,w_2$.
Relative smoothness only provides bounds in terms of $\nabla^2 F(w_1)$, whereas our analysis requires bounds in terms of $\nabla^2 F(w_0)$, i.e., the Hessian at the iterate where the preconditioner is constructed, which is exactly what quadratic regularity provides.

The additional power given by quadratic regularity could imply it only holds for a restrictive class of functions.
However, the following proposition shows this is not the case, as quadratic regularity holds under many standard hypotheses, including smoothness and strong convexity.
\ifpreprint
\begin{proposition}[Sufficient conditions for quadratic regularity]
\label[proposition]{proposition:QuadRegConds}
    The following conditions all imply $F$ is $\mathcal C$-quadratically regular: 
    \begin{enumerate}
        \item The function $F$ is $L$-smooth and $\mu$-strongly convex over $\mathcal C$. 
        Then $F$ is $\mathcal C$-quadratically regular with 
        \[
        \frac{\mu}{L}\leq \gamma_\ell(\mathcal C) \leq \gamma_u(\mathcal C) \leq \frac{L}{\mu}.
        \]
        \item The function $F$ is $\mu$-strongly convex and has a $M$-Lipschitz Hessian over $\mathcal C$, and $\mathcal C$ is compact with diameter $D$. 
        Then $F$ is $\mathcal C$-quadratically regular with 
        \[
        \left(1+\frac{MD}{\mu}\right)^{-1}\leq \gamma_\ell(\mathcal C) \leq \gamma_u(\mathcal C) \leq 1+\frac{M D}{\mu}.
        \]
        \item The function $F$ is $k$-generalized self-concordant \citep{bach2010self} over $\mathcal{C}$, and $\mathcal C$ is compact with diameter $D$. 
        Then $F$ is $\mathcal C$-quadratically regular with 
        \[
        \exp{(-kD)}\leq \gamma_{\ell}(\mathcal C)\leq \gamma_u(\mathcal C)\leq \exp{(kD)}.
        \]
        \item The function $F(w) = \frac{1}{n}\sum_{i=1}^n \phi_i(a_i^Tw)+\frac{\nu}{2}\|w\|^2$, that is $F$ is a \emph{GLM}. 
        Then $F$ is $\mathcal{C}$-quadratically regular with
        \[
        \frac{\ell \lambda_{1}(\hat \Sigma)+\nu}{u \lambda_{1}(\hat \Sigma)+\nu } \leq \gamma_\ell(\mathcal C) \leq \gamma_u(\mathcal C) \leq \frac{u \lambda_{1}(\hat \Sigma)+\nu}{\ell \lambda_{1}(\hat \Sigma)+\nu},
        \]
        where $\hat \Sigma = \frac{1}{n}A^TA$, $u = \sup_{1\leq i\leq n}\left(\sup_{w \in \mathcal C} \phi''_i(a^T_iw)\right)$, and $\ell =  \inf_{1\leq i\leq n}\left(\inf_{w \in \mathcal C} \phi''_i(a^T_iw)\right)$.
    \end{enumerate}
\end{proposition}
The proof is given in \cref{subsec:quad_reg_pfs}.
\else
\begin{proposition}[Sufficient conditions for quadratic regularity]
\label[proposition]{proposition:QuadRegConds}
    The following conditions all imply $F$ is $\mathcal C$-quadratically regular: 
    \begin{enumerate}
        \item The function $F$ is $L$-smooth and $\mu$-strongly convex over $\mathcal C$. 
        Then $F$ is $\mathcal C$-quadratically regular with 
        \[
        \frac{\mu}{L}\leq \gamma_\ell(\mathcal C) \leq \gamma_u(\mathcal C) \leq \frac{L}{\mu}.
        \]
        \item The function $F$ is $\mu$-strongly convex and has a $M$-Lipschitz Hessian over $\mathcal C$, and $\mathcal C$ is compact with diameter $D$. 
        Then $F$ is $\mathcal C$-quadratically regular with 
        \[
        \left(1+\frac{MD}{\mu}\right)^{-1}\leq \gamma_\ell(\mathcal C) \leq \gamma_u(\mathcal C) \leq 1+\frac{M D}{\mu}.
        \]
\end{enumerate}
\end{proposition}
\fi
\Cref{proposition:QuadRegConds} is similar to Theorem 1 in \cite{karimireddy2018global}, which establishes analogous sufficient conditions for ensuring a stable Hessian. 
Moreover, the bounds attained on the quadratic regularity constants are identical to those attained in \cite{karimireddy2018global}.
Hence, the bounds on the quadratic regularity constants are no worse than those of \cite{karimireddy2018global}, even though they account for lazy updating, while Hessian stability does not. 

To better understand quadratic regularity and its properties, it is instructive to compare the bounds in \cref{proposition:QuadRegConds} and their implications.
Notably, the bounds guaranteed by smoothness and strong convexity are looser than the bounds guaranteed by the other conditions. 
To see why, suppose $F$ is an ill-conditioned quadratic.
Clearly, $\gamma_\ell = \gamma_u = 1$, but the bound given by $1.$ in \cref{proposition:QuadRegConds} is $\kappa^{-1}\leq \gamma_\ell\leq \gamma_u \leq \kappa$.
\ifpreprint
In contrast, items $2$--$4.$ in \cref{proposition:QuadRegConds} all yield $\gamma_\ell = \gamma_u = 1$. 
\else
In contrast, item $2.$ in \cref{proposition:QuadRegConds} yields $\gamma_\ell = \gamma_u = 1$. 
\fi
Thus, the bound given by $1.$ is pessimistic, which is unsurprising since the proof of $1.$ is based on simple worst-case norm conversion bounds.

\ifpreprint
\subsubsection{$\zeta$-spectral approximation and quadratic regularity in action}

This section provides intuition for why $\zeta$-spectral approximation and quadratic regularity lead to improved local convergence.
To do so, we will compare the one-step improvement of an (approximate) lazy Newton method to that of gradient descent. The convergence proofs for PROMISE use stochastic gradients, which adds a layer of complexity to their analysis, while the analysis in this section uses full gradients.
However, the ideas underlying the convergence proofs of PROMISE methods are similar to those in the following analysis.

We will assume the objective $F$ is strongly convex and has $M$-Lipschitz Hessians. 
Furthermore, we will assume that the iterates are localized to a ball that contains the optimum.
Under these assumptions, we show that an approximate, lazy version of Newton's method has a better one-step improvement than gradient descent. This is accomplished with the following sequence of results:
\begin{enumerate}
    \item Quadratic regularity leads to a ``good'' local model of $F$ in a neighborhood of the iterate where the Hessian is computed (\cref{lem:lazy_hess_good_model})
    \item Newton's method with lazy updates makes better progress compared to gradient descent (\cref{cor:one_step_lazy_newton})
    \item A $\zeta$-spectral approximation in place of the Hessian still provides a ``good'' local model of $F$ (\cref{lem:approx_hess_good_model})
    \item Approximate Newton's method with lazy updates still makes fast progress. (\cref{cor:one_step_approx_lazy_newton})
\end{enumerate}

\begin{lemma}[Lazy Hessians provide a good local model]
\label[lemma]{lem:lazy_hess_good_model}
        Let $w\in \R^p$, $\varepsilon\in (0,1)$. Suppose $F$ is strongly convex and has $M$-Lipschitz Hessians. Let 
 $\mathcal C = B(w,\frac{\varepsilon\mu}{2M})$.
    Then 
    \[ 
    \frac{1}{1+\varepsilon}\leq \gamma_\ell(\mathcal C)\leq \gamma_u(\mathcal C) \leq 1+\varepsilon,
    \]
    so the quadratic regularity ratio is moderate:
    \[
    \mathfrak{q}(\mathcal C) = \gamma_u / \gamma_l \leq (1+\varepsilon)^2.
    \]
    Furthermore, for any $w',w''\in \mathcal C$,
    \[
    F(w'')\leq F(w')+\langle \nabla F(w'),w''-w'\rangle +\frac{1+\varepsilon}{2}\|w''-w'\|^2_{\nabla^2 F(w)},
    \]
    \[
     F(w'')\geq F(w')+\langle \nabla F(w'),w''-w'\rangle +\frac{1}{2(1+\varepsilon)}\|w''-w'\|^2_{\nabla^2 F(w)}.
    \]
\end{lemma}

\begin{proof}
The result follows immediately by combining the definition of $\mathcal C$ with item 2 in \cref{proposition:QuadRegConds}.
\end{proof}

In \cref{lem:lazy_hess_good_model}, $w$ is the iterate where the Hessian is computed, while $w'$ and $w''$ are contained in a ball about $w$.
\cref{lem:lazy_hess_good_model} shows that within $\mathcal C$, $F$ is well-conditioned (i.e., $\nabla^2 F(w)$ provides a good local model) with respect to the $\nabla^2 F(w)$ norm, with condition number given by $\q(\mathcal C) \leq (1 + \varepsilon)^2$.

\begin{corollary}[One-step improvement for lazy Newton's method]
\label[corollary]{cor:one_step_lazy_newton}
     Instate the hypotheses of \cref{lem:lazy_hess_good_model}. Set $w'' = w'-\frac{1}{1+\varepsilon}\nabla^2 F(w)^{-1}\nabla F(w')$ and suppose $w', w'', w_\star \in B(w,\frac{\varepsilon\mu}{2M})$. Then
     \[
     F(w'')-F(w_\star)\leq \left[1-\frac{1}{(1+\varepsilon)^2}\right](F(w')-F(w_\star)).
     \]
\end{corollary}

\ifpreprint
\begin{proof}
    The proof is similar to that of linear convergence of gradient descent for smooth, strongly convex functions.

    Since $w', w'' \in B(w,\frac{\varepsilon\mu}{2M})$,
    \begin{align*}
        F(w'') &\leq F(w') + \langle \nabla F(w'), w''-w' \rangle + \frac{1+\varepsilon}{2} \|w''-w'\|^2_{\nabla^2 F(w)} \\
        &= F(w') + \left \langle \nabla F(w'), - \frac{1}{1+\varepsilon}\nabla^2 F(w)^{-1} \nabla F(w') \right \rangle + \frac{1+\varepsilon}{2} \left\| -\frac{1}{1+\varepsilon}\nabla^2 F(w)^{-1} \nabla F(w') \right\|^2_{\nabla^2 F(w)} \\
        &= F(w') - \frac{1}{2 (1+\varepsilon)} \| 
 \nabla F(w') \|^2_{\nabla^2 F(w)^{-1}},
    \end{align*}
    where the inequality follows from upper quadratic regularity and the first equality follows from the definition of $w''$.

    To complete the proof, we use quadratic regularity to derive a lower bound on $\| 
 \nabla F(w') \|^2_{\nabla^2 F(w)^{-1}}$ in terms of objective suboptimality. Since $w', w_\star \in B(w,\frac{\varepsilon\mu}{2M})$,
 \begin{align*}
     F(w_\star) &\geq F(w') + \langle \nabla F(w'), w_\star - w' \rangle + \frac{1}{2(1 + \varepsilon)} \| w_\star - w' \|^2_{\nabla^2 F(w)} \\
     &\geq F(w') + \min_y ~ \langle \nabla F(w'), y \rangle + \frac{1}{2(1 + \varepsilon)} \| y \|^2_{\nabla^2 F(w)} \\
     &= F(w') - \frac{1 + \varepsilon}{2} \| F(w') \|^2_{\nabla^2 F(w)^{-1}},
 \end{align*}
 where the first inequality is due to lower quadratic regularity and the equality is due to minimizing the quadratic function with respect to $y$.

 Rearranging the previous display yields 
 \begin{align*}
      \frac{2}{1 + \varepsilon} (F(w') - F(w_\star)) \leq \| \nabla F(w') \|^2_{\nabla^2 F(w)^{-1}},
 \end{align*}
 which is precisely the lower bound we are seeking.

 Therefore,
 \begin{align*}
     F(w'') &\leq F(w') - \frac{1}{2 (1+\varepsilon)} \| \nabla F(w') \|^2_{\nabla^2 F(w)^{-1}} \\
     & \leq F(w') - \frac{1}{(1 + \varepsilon)^2} (F(w') - F(w_\star)).
 \end{align*}

 Adding $-F(w_\star)$ to both sides of this display yields the claimed result.

\end{proof}
\else 
\begin{proof}
    The proof is similar to the argument used for gradeint descent, and is omitted.
\end{proof}
\fi

We can compare the one-step improvement in \cref{cor:one_step_lazy_newton} with that of a gradient descent update $(w'' = w'-\frac{1}{L}\nabla F(w'))$, which has one-step improvement 
\[
F(w'')-F(w_\star)\leq \left[1-\frac{1}{\kappa}\right](F(w')-F(w_\star)).
\]

This comparison suggests that Newton's method with lazy Hessian updates (locally) converges linearly at a rate independent of the condition number, unlike gradient descent, whose convergence rate depends on the condition number.
\begin{lemma}[$\zeta$-spectral approximations provide a good local model]
\label[lemma]{lem:approx_hess_good_model}
    Suppose $F$ is quadratically regular and $P$ is a $\zeta$-spectral approximation of $\nabla^2 F(w)$ and let $w', w'' \in \R^p$. Then
    \[
    F(w'')\leq F(w')+\langle \nabla F(w'),w''-w'\rangle+\frac{(1+\zeta)\gamma_u}{2}\|w''-w'\|_P^2, 
    \]
    \[
    F(w'')\geq F(w')+\langle \nabla F(w'),w''-w'\rangle+\frac{(1-\zeta)\gamma_\ell}{2}\|w''-w'\|_P^2.
    \]
\end{lemma}

\ifpreprint
\begin{proof}
    Since $F$ is quadratically regular,
    \begin{align}
        F(w'') &\leq F(w') + \langle \nabla F(w'), w'' - w' \rangle + \frac{\gamma_u}{2} \|w'' - w'\|^2_{\nabla^2 F(w)}, \label{eq:qr_u} \\
        F(w'') &\geq F(w') + \langle \nabla F(w'), w'' - w' \rangle + \frac{\gamma_\ell}{2} \|w'' - w'\|^2_{\nabla^2 F(w)} \label{eq:qr_l}.
    \end{align}

    Since $P$ is a $\zeta$-spectral approximation of $\nabla^2 F(w)$, $(1 - \zeta) P \preceq \nabla^2 F(w) \preceq (1 + \zeta) P$, which implies
    \begin{align*}
        (1 - \zeta) \|w'' - w'\|^2_{P} \leq \|w'' - w'\|^2_{\nabla^2 F(w)} \leq (1 + \zeta) \|w'' - w'\|^2_{P}.
    \end{align*}

    Substituting the previous display into \cref{eq:qr_u,eq:qr_l} yields the claimed result.
\end{proof}
\else
\fi

Quadratic regularity tells us that $F$ has a condition number of $\gamma_u/\gamma_\ell = \q$ in the $\nabla^2 F(w)$ norm. \cref{lem:approx_hess_good_model} shows that we pay a multiplicative factor of $\frac{1 + \zeta}{1 - \zeta}$ on the quadaratic regularity ratio $\q$ to extend quadratic regularity to the $P$ norm, where $P$ is a $\zeta$-spectral approximation of $\nabla^2 F(w)$.
In other words, using a $\zeta$-spectral approximation in place of the Hessian still provides a reliable local model for $F$.

A concrete example is given by ridge regression, for which $\q = 1$.
If $P$ is a $0.9$-spectral approximation of $\nabla^2 F(w)$, then $\frac{1 + \zeta}{1 - \zeta} \q = 19$, which is better than the condition number $\kappa$ for most ridge regression problems.

\begin{corollary}[One-step improvement for approximate, lazy Newton's method]
\label[corollary]{cor:one_step_approx_lazy_newton}
    Instate the hypotheses of the \cref{lem:lazy_hess_good_model,lem:approx_hess_good_model}. Set $w'' = w'-\frac{1}{(1+\zeta)(1+\varepsilon)}P^{-1}\nabla F(w')$ and suppose $w', w'', w_\star \in B(w,\frac{\varepsilon\mu}{2M})$. 
    Then
    \[
        F(w'')-F(w_\star) \leq \left[1-\frac{1-\zeta}{1+\zeta}\frac{1}{(1+\varepsilon)^2}\right](F(w')-F(w_\star)).
    \]
\end{corollary}

\ifpreprint
\begin{proof}
    Combining \cref{lem:lazy_hess_good_model,lem:approx_hess_good_model}, we have
    \begin{align}
    F(w'') &\leq F(w')+\langle \nabla F(w'),w''-w'\rangle+\frac{(1+\zeta) (1 + \varepsilon)}{2}\|w''-w'\|_P^2, \label{eq:approx_qr_u} \\
    F(w_\star) &\geq F(w') + \langle \nabla F(w'), w_\star - w' \rangle + \frac{1 - \zeta}{2 (1 + \varepsilon)} \|w''-w'\|_P^2 \label{eq:approx_qr_l}.
    \end{align}

    The remainder of the proof is analogous to that of \cref{cor:one_step_lazy_newton}---we use \cref{eq:approx_qr_u} to derive $F(w'') \leq F(w') - \frac{1}{2 (1 + \zeta)(1 + \varepsilon)} \| \nabla F(w') \|^2_{P^{-1}}$ and \cref{eq:approx_qr_l} to derive $\frac{2 (1 - \zeta)}{1 + \varepsilon}(F(w') - F(w_\star)) \leq \| \nabla F(w') \|^2_{P^{-1}}$.
    Combining these two bounds yields the claim.
\end{proof}
\else
\fi

\cref{cor:one_step_approx_lazy_newton}, which analyzes the one-step improvement of approximate, lazy Newton's method, is best understood by comparing with \cref{cor:one_step_lazy_newton}, which analyzes the one-step improvement of (exact) lazy Newton's method.
First, the stepsize in \cref{cor:one_step_approx_lazy_newton} is $1/[(1 + \zeta)(1 + \varepsilon)]$, while the stepsize in \cref{cor:one_step_lazy_newton} is $1/(1+ \varepsilon)$.
The $\zeta$-spectral approximation multiplies the upper quadratic regularity constant by $1 + \zeta$ (\cref{lem:approx_hess_good_model}), which explains the reduced stepsize.
Furthermore, the multiplicative factor in \cref{cor:one_step_approx_lazy_newton} is $1-\left[(1-\zeta)/(1+\zeta)\right]/(1+\varepsilon)^2$, while the multiplicative factor in \cref{cor:one_step_lazy_newton} is $1 - 1/(1+\varepsilon)^2$.
This is also consistent---\cref{lem:approx_hess_good_model} multiplies $\q$ by $\frac{1 + \zeta}{1 - \zeta}$, so we would expect the multiplicative factor to be reduced from $1 - 1/\q$ to $1 - 1/\left[(1+\zeta)/(1-\zeta)\q\right] = 1 - \left[(1-\zeta)/(1+\zeta)\right]/\q$.

Despite this reduced improvement, we would expect approximate, lazy Newton's method to perform better than gradient descent. 
For example, if $\zeta = \varepsilon = 0.5$, then $\left[(1-\zeta)/(1+\zeta)\right]/(1+\varepsilon)^2 = 4/27$. 
Hence, if $\kappa \geq 7$ (which is common in ML problems), approximate, lazy Newton's method would be expected to (locally) converge faster than gradient descent.
\else
\fi

\subsubsection{Hessian dissimilarity}
The next important idea is the Hessian dissimilarity, which quantifies how large the gradient batchsize must be to realize the benefits of preconditioning. 
\begin{definition}
\label[definition]{def:hess_sim}
Let $\mathcal C$ be a closed convex subset of $\R^p$. The \emph{Hessian dissimilarity} is
    \[
    \tau_\star^\nu(\mathcal C) \coloneqq \sup_{w\in \mathcal C}\max_{1\leq i\leq n}\lambda_1\left((\nabla^2f(w)+\nu I)^{-1/2}(\nabla^2 f_i(w)+\nu I)(\nabla^2f(w)+\nu I)^{-1/2}\right).
    \]
\end{definition}
The Hessian dissimilarity bounds the worst-case value over $\mathcal C$, of the norm ratio $\|v\|^2_{\nabla^2 F_i(w)}/\|v\|^2_{\nabla^2 F(w)}$ between the full Hessian $\nabla^2 F(w)$ and the Hessian of any term $\nabla^2 F_i(w)$ in the sum. 
Hessian dissimilarity is analogous to the ridge leverage coherence in \cref{subsec:precond_qual}: while ridge leverage coherence measures the uniformity of the rows of a matrix, the Hessian dissimilarity measures the uniformity of constituent psd Hessians of a finite-sum convex function $F(w)$. 
Furthermore, for GLMs, we will see that the Hessian dissimilarity is controlled by a uniform version of the ridge leverage coherence.

In our analysis of PROMISE, the Hessian dissimilarity parameter $\tau^\nu_\star$ appears in the bound on the preconditioned smoothness constant (\cref{prop:precond_grad_var}),
where it controls the gradient batchsize needed to ensure a good preconditioned smoothness constant.
More precisely, it is the gradient batchsize required to ensure the preconditioned smoothness constant is $\bigO(1)$ on average, when the gradients are sampled uniformly at random.
We believe this issue could be alleviated by importance sampling, but leave it as a direction for future work, as we have found uniform sampling to be sufficient in our experiments. 

The Hessian dissimilarity never exceeds $n$, as shown by the following lemma. 
\ifpreprint
\begin{lemma}[Hessian dissimilarity never exceeds $n$]
\label[lemma]{lem:hess_sim}
    The dissimilarity parameter satisfies 
    \begin{equation*}
        1\leq \tau^\nu_\star(\mathcal C) \leq \min\left\{n, 1+\kappaMax\right\}.
    \end{equation*}
\end{lemma}
The proof is given in \cref{subsec:hess_sim_pfs}.
\else
\begin{lemma}[Hessian dissimilarity never exceeds $n$]
\label[lemma]{lem:hess_sim}
    Define $\kappaMax$ is as in \cref{lem:deff_vs_kappa}. 
    The dissimilarity parameter satisfies 
    \begin{equation*}
        1\leq \tau^\nu_\star(\mathcal C) \leq \min\left\{n, 1+\kappaMax\right\}.
    \end{equation*}
\end{lemma}
\fi
\cref{lem:hess_sim} shows the the Hessian dissimilarity never exceeds $n$, and may be much smaller if the $f_i$'s are well-conditioned.
Unfortunately, for ill-conditioned problems, we can easily have $\tau^\nu_\star = n$, in which case there is 
no improvement in the gradient batchsize needed (in theory) over a full-gradient algorithm.
In practice, large gradient batchsizes are unnecessary in any of our numerical experiments, 
which suggests the bound in \cref{lem:hess_sim} is pessimistic.
In particular, this bound does not account for the structure of the objective.
For GLMs, we can derive a more informative bound on the Hessian dissimilarity (see \cref{prop:hess_sim_glm} below), 
which reveals a deep connection to ridge leverage scores and ridge leverage coherence.
\ifpreprint
\begin{proposition}[Hessian dissimilarity for GLMs] Consider a regularized GLM
\label[proposition]{prop:hess_sim_glm}
\[
F(w) = \frac{1}{n}\sum_{i=1}^{n}\phi_{i}(a_i^Tw)+\frac{\nu}{2}\|w\|^2.
\]
Further suppose that $\sup_{x\in \R}\phi_i^{''}(x)\leq B$, for some $B>0$. Then
\[
\tau_\star^{\nu}\leq 1+\chi^{\nu}_{\star}d_{\textup{eff}}^{\nu/B}(A),
\]
where $\chi^{\nu}_\star = \sup_{w\in \R^p} \chi^{\nu}(\Phi''(Aw)^{1/2}A)$ and $\Phi''(Aw) = \textup{diag}\left([\phi''_1(a_1^{T}w)\dots \phi''_n(a_n^{T}w)]\right)$.
In particular, for least squares and logistic regression, 
\[
\tau_\star^{\nu}\leq 1+\chi_{\star}^{\nu}d_{\textup{eff}}^{\nu}(A).
\]
\end{proposition}
The proof is given in \cref{subsec:hess_sim_pfs}.
\else
\begin{proposition}[Hessian dissimilarity for GLMs] Consider minimizing a regularized GLM
\label[proposition]{prop:hess_sim_glm}
\[
F(w) = \frac{1}{n}\sum_{i=1}^{n}\phi_{i}(a_i^Tw)+\frac{\nu}{2}\|w\|^2.
\]
Further suppose that $\sup_{x\in \R}\phi_i^{''}(x)\leq B$, for some $B>0$. Then
\[
\tau_\star^{\nu}\leq 1+\chi^{\nu}_{\star}d_{\textup{eff}}^{\nu/B}(A),
\]
where $\chi^{\nu}_\star = \sup_{w\in \R^p} \chi^{\nu}(\Phi''(Aw)^{1/2}A)$ and $\Phi''(Aw) = \textup{diag}\left([\phi''_1(a_1^{T}w)\dots \phi''_n(a_n^{T}w)]\right)$.
In particular, for least squares and logistic regression, 
\[
\tau_\star^{\nu}\leq 1+\chi_{\star}^{\nu}d_{\textup{eff}}^{\nu}(A).
\]
\end{proposition}
\fi
\cref{prop:hess_sim_glm} shows that for GLMs, the Hessian dissimilarity is controlled by the global ridge leverage coherence of the Hessian, $\chi_\star^{\nu}$. 
When $\chi_\star^{\nu}$ is close to 1, the batchsize required to see the full effects of preconditioning is not much larger than the effective dimension of the data matrix.
As $\deff(A)$ is much smaller than $n$ under mild assumptions (recall \cref{lemma:EffDimBnd}), this implies 
that a much smaller gradient batchsize suffices to enjoy the effects of preconditioning.
This improved theory agrees with our empirical results. 
Conversely, when the data matrix has high coherence, \cref{prop:hess_sim_glm} suggests that large gradient batchsizes may be necessary to realize the benefits of preconditioning.

\ifpreprint
\subsubsection{Preconditioner stability}
We also establish the following result that shows the preconditioners evolve in a stable fashion, which is needed for showing convergence of SketchySAGA and SketchyKatyusha.
\begin{proposition}[Stable evolution of preconditioners]
\label[proposition]{prop:stable_precond}
    Instate the hypotheses of \cref{assm:SmoothSC,assm:GoodPrecond,assm:FiniteUpates}.
    Consider the sequences of random variables $\{X_{j}\}_{j\in [J]}, \{Y_{j}\}_{j\in [J]}$ defined by 
    \[
    X_{j} \coloneqq \max\left\{\frac{\|w_{j}-w_\star\|^2_{P_{j+1}}}{\|w_{j}-w_\star\|^2_{P_{j}}}, 1\right\}, \quad Y_{j} \coloneqq \max\left\{\frac{\|w_{j}-w_\star\|^2_{P_{j-1}}}{\|w_{j}-w_\star\|^2_{P_{j}}}, 1\right\}
    \]
    Then with probability $1$, there exist constants $\xi_j, \beta_j$ such that
    \[
    X_{j}\leq \xi_j, \quad Y_{j} \leq \beta_j.
    \]
    Hence with probability $1$,
    \[
    \prod^{J}_{j=1}X_{j} \leq \prod^{J}_{j=1}\xi_{j} \coloneqq \mathcal E_P, \quad \prod_{j=1}^{J}Y_j \leq \prod_{j=1}^{J}\beta_j \coloneqq B_P.
    \]
\end{proposition}
The proof is given in \cref{subsec:stable_precond}.
\else
\fi

\subsubsection{The smoothness of the preconditioned stochastic gradient}
Convergence analysis of stochastic gradient methods
requires control of the smoothness constant of the minibatch stochastic gradient.
To analyze preconditioned methods, 
we must control the smoothness in the preconditioned norm ($\|\cdot\|_{P^{-1}}$) instead of the Euclidean norm ($\|\cdot\|_2$). 
The following proposition provides such control in expectation, in terms of the quantity $\mathcal L_P$, which we call the preconditioned expected smoothness constant.
This proposition extends Proposition 3.8 of \cite{gower2019sgd}, which handles the case $P = I$.
\begin{proposition}[Preconditioned expected smoothness]
\label[proposition]{prop:precond_grad_var} 
    Let $F$ be $\gamma_u$ upper-quadratically regular, $P$ be a $\zeta$-spectral approximation, and recall $\gammaMax = \max_{i\in [n]}\gamma_{u_i}$
    Instate \cref{assm:SmoothSC} and \cref{assm:GoodPrecond}. Then for any $w',w \in \R^p$,
\[
\E\|\widehat \nabla F(w)-\widehat \nabla F(w')\|_{P^{-1}}^2\leq 2\mathcal L_P\left(F(w)-F(w')-\langle\nabla F(w'),w-w'\rangle \right),
\]
where
\[
\mathcal L_P = \left(\frac{n(b_g-1)}{b_g(n-1)}\gamma_u+\tau_\star^\nu\frac{n-b_g}{b_g(n-1)}\gammaMax\right)(1+\zeta).
\]
\end{proposition}
The proof is given in \cref{subsec:precond_grad_var}.
The preconditioned expected smoothness constant $\mathcal L_P$ in \cref{prop:precond_grad_var}
is the 
preconditioned analogue of the smoothness constant in the stochastic gradient setting.
\ifpreprint
Indeed, if $n = b_g$, then $\mathcal L_P = (1+\zeta)\gamma_u$, which is the smoothness constant of $F$ with respect to the preconditioned norm $\|\cdot\|_{P}$ (\cref{lem:approx_hess_good_model}).
\else
Indeed, if $n = b_g$, then $\mathcal L_P = (1+\zeta)\gamma_u$, which is the smoothness constant of $F$ with respect to the preconditioned norm $\|\cdot\|_{P}$.
\fi
We have seen that when $P$ provides a good quadratic model, $\gamma_u \leq 1+\varepsilon$, and so $\mathcal L_P \leq (1+\varepsilon)(1+\zeta)$.
However, PROMISE operates in the setting $b_g\ll n$, in which case \cref{prop:precond_grad_var} yields a new phenomena not present for full gradients $b_g = n$; namely, that even if $\gammaMax = 1$, so that $P$ provides a perfect quadratic model, it is not guaranteed that $\mathcal L_P = \bigO(1)$.
To ensure $\mathcal L_P = \bigO(1)$, the gradient batchsize must satisfy $b_g = \bigO(\tau_\star^{\nu})$.
Hence to realize the benefits of preconditioning, the gradient batchsize must be sufficiently large:
the required batchsize is determined by the Hessian dissimilarity.
The dependence on $\tau^{\nu}_\star$ reflects the fact that the gradient batchsize must be large enough 
to ensure that Hessian of the minibatch objective is a good approximation to the Hessian of the full objective.
If the gradient batchsize is too small, the corresponding minibatch Hessian may have curvature that is quite different from the full Hessian.
In this case, we should not expect a preconditioner built from a good approximation of the full Hessian to help, 
as it contains information unrelated to that of the minibatch stochastic gradient.
Moreover, we have the natural conclusion that the required gradient batchsize is smaller when the $\nabla^2F_i$ are more similar, and larger when they are more dissimilar.

\ifpreprint
To obtain a more concrete understanding of \cref{prop:precond_grad_var}, consider the special case when $F$ is a GLM. 
\cref{prop:hess_sim_glm} shows the Hessian dissimilarity is bounded as $\tau_\star^{\nu}\leq 1+\chi_\star^{\nu}d^{\nu/B}_{\textrm{eff}}(A)$, 
which is consistent with \cref{proposition:SSNPrecond}.
This bound depends on the global ridge leverage coherence $\chi_\star^\nu$ instead of the ridge leverage coherence of the Hessian at $w$, as smoothness is a global quantity.
When the global ridge leverage coherence is small, $\tau^\nu_\star$ is no larger than the effective dimension, so the required gradient batchsize is small.
Conversely, a large global ridge leverage coherence implies the required gradient batchsize may be large.
\else
\fi

Overall, \cref{prop:precond_grad_var} shows that preconditioning is not a panacea.
For problems with highly non-uniform data, convergence of PROMISE methods may be slow if the gradient batchsize 
is smaller than $\bigO(\tau^{\nu}_\star)$.
We emphasize this limitation is independent of any particular preconditioning technique, and would remain true even if PROMISE used the perfect preconditioner---the Hessian itself. 
The problem stems from the use of uniform sampling to construct the stochastic gradient.  
Evidently, given our extensive empirical results in \cref{section:experiments}, problem instances requiring large gradient batches seem to be uncommon.
This is unsurprising, as the data in many ML problems is (approximately) i.i.d.; hence we have strong reasons to believe that data is relatively uniform due to statistical similarity.
  
To our knowledge, the analysis above is the first to demonstrate the necessity of a minimum gradient batchsize to see the benefits of preconditioning.
Previously, \cite{derezinski2022stochastic} observed that least-squares with highly coherent data matrices requires large gradient batchsizes, and demonstrated the phenomena empirically. 
Moreover, \cite{derezinski2022stochastic} shows that leverage score sampling to select the gradients
can reduce the required gradient batchsize 
Our analysis generalizes this observation to arbitrary loss functions, and provides a simple explanation based on the expected preconditioned smoothness constant.  

\subsubsection{Weighted quadratic regularity ratio}
We have seen that $\mathcal L_P$ is the analogue of the smoothness constant for PROMISE methods. As a consequence, the convergence rates of PROMISE methods will not depend on $\q$, but on the $\emph{weighted}$ quadratic regularity ratio, $\qbar$, which we define below.  
\begin{definition}[Weighted quadratic regularity ratio]
    Let $F:\R^p\mapsto \R$ be quadratically regular and define $\qmax \coloneqq \frac{\gammaMax}{\gamma_\ell}$. 
    Define the \emph{weighted quadratic regularity ratio}
    \begin{align*}
    \qbar &\coloneqq \frac{\mathcal L_P}{\gamma_\ell} = \frac{\left(\frac{n(b_g-1)}{b_g(n-1)}\gamma_u+\tau_\star^{\nu}\frac{(n-b_g)}{b_g(n-1)}\gammaMax\right)(1+\zeta)}{\gamma_\ell} \\
    &= \left(\frac{n(b_g-1)}{b_g(n-1)}\q+\tau_{\star}^{\nu}\frac{(n-b_g)}{b_g(n-1)}\qmax\right)(1+\zeta).    
    \end{align*}
\end{definition}

\subsection{SketchySVRG}
\label{subsection:sksvrg_theory}
\cref{thm:conv_SketchySVRG} shows the global linear convergence of SketchySVRG.

\begin{theorem}[SketchySVRG convergence]
\label{thm:conv_SketchySVRG}
Instate the hypotheses of \cref{assm:SmoothSC}-\cref{assm:GoodPrecond}. 
Run SketchySVRG with fixed learning rate $\eta = \frac{1}{8\mathcal L_P}$ and $m = \frac{19}{(1-\zeta)}\qbar$ inner iterations.
Then 
    \begin{align*}
           \mathbb E[F(\hat w^{(s)})]-F(w_\star) \leq \epsilon
    \end{align*}
after $s = 10\log(1/\epsilon)$ outer iterations. 
Hence, the total number of stochastic gradient queries to reach an $\epsilon$-suboptimal point is bounded by 
\[
10\left( n + \frac{19b_g}{(1-\zeta)}\qbar\right)\log \left( \frac{1}{\epsilon} \right).
\] 
\end{theorem}
The proof is given in \cref{subsec:sksvrg_convergence}. 
\cref{thm:conv_SketchySVRG} shows that SketchySVRG converges linearly at rate controlled by the weighted quadratic regularity ratio $\qbar$.
This should be contrasted with vanilla SVRG, whose rate is controlled by the condition number $\kappa$.  
As $\qbar$ depends upon the quadratic regularity constants, it may be much smaller than $\kappa$ for structured functions.
Indeed, for quadratic objectives with batchsize $b_g =  \tau_\star^\nu$, $\qbar \leq 2(1+\zeta)$, which is a significant smaller than $\kappa$.   
\ifpreprint
However, despite the explicit absence of $\kappa$ in the convergence rate of \cref{thm:conv_SketchySVRG}, the overall guarantee provided by the theorem is pessimistic.  
The pessimism stems from the convergence rate's dependence upon the global value of $\qbar$, which \cref{proposition:QuadRegConds} shows might be as large as $\kappa^2$ in the worst-case, which is a consequence of \cref{thm:conv_SketchySVRG} failing to reflect that optimization trajectory is \emph{localized}. 
Indeed, in practice SketchySVRG can take a large stepsize based on the good local model provided by the current preconditioner, and not the global one.  
Hence, the rate of convergence SketchySVRG achieves in practice is much faster than what \cref{thm:conv_SketchySVRG} predicts, as it uses a much larger stepsize than the one required by \cref{thm:conv_SketchySVRG}, which is based on pessimistic global constants. 
Moreover, the quality of the quadratic model provided by preconditioning improves as the iterates approach the optimum, as $\qbar$ approaches a constant close to $1$ provided $b_g = \bigO(\tau_\star^\nu)$.  
The analysis underlying \cref{thm:conv_SketchySVRG} treats the entire optimization trajectory as if though it is always far from the optimum, and so is unable to exploit that $\qbar$ ameliorates as the iterates near the optimum.
\else
\fi

\ifpreprint
Unfortunately, while it is easy to identify the shortcomings of \cref{thm:conv_SketchySVRG}, it seems difficult to eliminate them and obtain a theorem showing fast global convergence.
Given the discussion in \cref{section:related_work}, this is unsurprising, as under our current hypotheses even Newton's method does not yield faster global convergence than first-order methods in the worst case; it can only guarantee fast local convergence.
Thus, further assumptions beyond smoothness and strong convexity are likely required to show an improved global convergence convergence rate.
Nevertheless, we show in \cref{thm:sksvrg_loc_con} that SketchySVRG achieves fast local convergence just like Newton's method, albeit 
the rate is only linear instead of quadratic, as it uses stochastic gradients.
So, SketchySVRG possesses the Newton-like property of fast local convergence, providing an advantage over stochastic first-order methods, which do not enjoy this property. 
\else
\fi

The exception to the preceding discussion on global convergence is the case where $F$ is quadratic, in which case the quadratic regularity constants equal $1$, so we always obtain global linear convergence at a rate independent of the condition number (assuming an appropriate gradient batchsize):

\begin{corollary}[SketchySVRG: Fast ridge regression]
    Instate the hypotheses of \cref{thm:conv_SketchySVRG}, and suppose $F$ is quadratic. 
    Run SketchySVRG with gradient batchsize $b_g = 1+\chi^{\nu}(A)\deff(A)$ and $m = 48\frac{1+\zeta}{1-\zeta}$ inner iterations. Then
\begin{align*}
           \mathbb E[F(\hat w^{(s)})]-F(w_\star) \leq \epsilon
    \end{align*}
where $\hat w^{(s)}$ is the output after running $s = 10\log\left(\frac{1}{\epsilon}\right)$ outer iterations.
Hence the total number of stochastic gradient evaluations to reach an $\epsilon$-suboptimal point is bounded by
\[
10\left( n + 48\frac{1+\zeta}{1-\zeta}[1+\chi^{\nu}(A)\deff(A)]\right)\log \left(\frac{1}{\epsilon}\right).
\]
\end{corollary}

\subsection{SketchySVRG: Fast local convergence}
We establish local linear convergence of SketchySVRG independent of the condition number in the neighborhood 
\[\Nstar = \left\{w\in \R^p: \|w-w_\star\|_{\nabla^2F(w_\star)}\leq\frac{\varepsilon_0 \nu^{3/2}}{2M}\right\},\] 
where $M$ is the uniform Lipschitz constant for each $\nabla^2 F_i$. 
This result is analagous to the fast local convergence of Newton's method in the full gradient setting. 
\begin{theorem}
\label{thm:sksvrg_loc_con}
    Let $\varepsilon_0 \in (0, 1/6]$.
    Suppose that each $F_i$ has an $M$-Lipschitz Hessian, and that $w_0\in \Nstar$.
    Instate \cref{assm:SmoothSC} and \cref{assm:GoodPrecond} with $\zeta = \varepsilon_0$. 
    Run \cref{alg:sksvrg} using Option I with $\mathcal U = \{0\}$, 
    $m = 6$ inner iterations, $s = 2\log(1/\epsilon)$ outer iterations, $\eta = 1$, and $b_g = \Otil\left(\tau(\Nstar)\log(\frac{1}{\delta})\right)$. Then with probability at least $1-\delta$,
    \[
    F(\hat w^{(s)})-F(w_\star)\leq \epsilon.
    \]
    Hence the total number of stochastic gradient queries required to reach an $\epsilon$-suboptimal point is bounded by
    \[
    3\left[n+6\Otil\left(\tau(\Nstar)\log\left(\frac{1}{\delta}\right)\right)\right]\log\left(\frac{1}{\epsilon}\right).
    \]
\end{theorem}
The proof is given in \cref{subsec:sksvrg_fast_local_convergence}.
\cref{thm:sksvrg_loc_con} shows that once the iterates are close enough to the optimum, SketchySVRG converges linearly at a rate independent of the condition number, provided the gradient batchsize satisfies $b_g = \Otil(\tau_\star^{\nu}(\Nstar))$. 
Recall $\tau_\star^{\nu}(\Nstar)$ is always smaller than $n$, and is significantly smaller when there are no outliers amongst the individual Hessians $\nabla^2 F_i$.

Previously, \cite{derezinski2022stochastic} established a result similar to \cref{thm:sksvrg_loc_con} when $P$ is the \ssn{} preconditioner. 
However, \cite{derezinski2022stochastic} requires $b_g = \Otil(\kappa)$, which can easily exceed $n$ for ill-conditioned problems.
\cref{thm:sksvrg_loc_con} significantly improves this result by reducing the required gradient batchsize to $\Otil(\tau_\star^\nu(\Nstar))$.
For GLMs (with some mild hypotheses), we show in \cref{cor:sksvrg_fast_glm} that the required gradient batchsize is as small as $\Otil(\sqrt{n})$, which is often orders of magnitude smaller than $n$ or $\kappa$. 
Hence \cref{thm:sksvrg_loc_con} holds with small gradient batchsizes, which agrees with practice, as PROMISE methods provide excellent empirical performance without large gradient batchsizes. 
The key idea for achieving the improvements in \cref{thm:sksvrg_loc_con} is quadratic regularity, which enables tighter control over the gradient in the inverse Hessian norm with high probability.

To better understand the implications of \cref{thm:sksvrg_loc_con}, we present the following corollary, which addresses the setting where $F$ is a GLM. 
\begin{corollary}
\label[corollary]{cor:sksvrg_fast_glm}
    Instate the hypotheses of \cref{thm:sksvrg_loc_con} and let $F$ be a bounded GLM. 
    Moreover suppose its data matrix $A$ has polynomially decaying singular values, the regularization satisfies $\nu =\bigO(1/n)$, and the ridge leverage incoherence satisfies $\chi^{\nu}_\star(\Nstar) = \bigO(1)$\footnote{Equivalently, $\Phi''(Aw_\star)^{1/2}A$ is ridge leverage incoherent.}.
    Run $\cref{alg:sksvrg}$ with $b_g = \Otil\left(\sqrt{n}\log(\frac{1}{\delta})\right)$. 
    Then with probability at least $1-\delta$, at most 
    \[
    3\left[n+6\Otil\left(\sqrt{n}\log\left(\frac{1}{\delta}\right)\right)\right]\log\left(\frac{1}{\epsilon}\right)
    \]
    stochastic gradient queries are required to find an $\epsilon$-suboptimal point. 
\end{corollary}
\cref{cor:sksvrg_fast_glm} shows that under mild hypotheses on $A$ and the Hessian, SketchySVRG achieves fast local convergence with a small gradient batchsize of $\Otil(\sqrt{n})$, for common values of the regularization parameter $\nu$. 
Hence large gradient batchsizes are unnecessary to see the benefits of preconditioning for GLMs.
\ifpreprint
\cref{cor:sksvrg_fast_glm}'s prescription that the gradient batchsize should satisfy $b_g = \Otil(\sqrt{n})$ translates well to practice, as \cref{table:grad_batchsize} shows the median value of $b_g$ in our experiments is around $\sqrt{n}$.
\else
\fi

\ifpreprint
\subsection{SketchySAGA}
We now turn to the convergence of SketchySAGA. Similar to the analysis in \cite{defazio2014saga}, we set $b_g = 1$ and use a Lyapunov function argument.
The Lyapunov function is
\begin{align}
\label{eq:lyap_fun_sksaga}
    T_k \coloneqq B_k \left( \frac{1}{n} \sum_{i = 1}^n F_i(\psi_k^i) - F(w_\star) - \frac{1}{n} \sum_{i = 1}^n \langle \nabla F_i(w_\star), \psi_k^i - w_\star \rangle + c \| w_k - w_\star \|_{P_{k-1}}^2 \right), 
\end{align}
where
\[
B_k \coloneqq 
\begin{cases}
\prod_{i = 0}^{k-1} \beta_i & k \geq 1\\
1 & k = 0
\end{cases}.
\] 
The Lyapunov function in \eqref{eq:lyap_fun_sksaga} is identical to that of \cite{defazio2014saga}, except $w_k-w_\star$ is measured in the $P_{k-1}$-norm, and there is a factor of $B_k$ that arises from the changing metric. 

To establish linear convergence, we must show that for appropriate learning rate, $\eta$, and constant, $c$, the Lyapunov function contracts in expectation. 
This is shown in \cref{lem:sksaga_contraction}: if we set $\eta = \frac{1}{2(n(1-\zeta)\gammaMin  + \Lmc_P)}$, $c = \frac{1}{2 \eta (1 - \eta (1-\zeta)\gammaMin) n}$,  and $\kappa_P = \frac{\mathcal L_P}{(1-\zeta)\gammaMin}$, then conditioned on the first $k$ iterations,
\begin{align*}
    \E_k [T_{k+1}] \leq \left( 1 - \frac{1}{\kappa_P} \right) \beta_k T_k.
\end{align*}
Using the preceding result, we can easily establish the following theorem.
\begin{theorem}
\label{thm:conv_sksaga}
Instate the hypotheses of \cref{assm:SmoothSC}-\cref{assm:FiniteUpates}. 
Set $\kappa_P = \frac{\mathcal L_P}{(1-\zeta)\gammaMin}$ and $\tilde{\mathfrak{q}}_{\textup{max}} = \gammaMax/\gammaMin$.
Run SketchySAGA with fixed learning rate $\eta = \frac{1}{2\left(\mathcal L_P+n(1-\zeta)\gammaMin\right)}$.
Then  
\begin{align*}
    \E \|w_k-w_\star\|_{P_k}^2  \leq \left( 1 - \frac{1}{2(n+\kappa_P)}\right)^{k}\frac{n B_P}{\mathcal L_P+(1-\zeta)n\gamma_{\ell_{\textup{min}}}}T_0.
\end{align*}
Hence, 
\[
\E \|w_k-w_\star\|_{P_k}^2\leq \epsilon
\] 
after $k = 2\left( n+\tau_\star^{\nu}\frac{1+\zeta}{1-\zeta} \tilde{\mathfrak{q}}_{\textup{max}} \right) \log\left(\frac{nB_PT_0}{(\mathcal L_P+n(1-\zeta)\gamma_{\ell_{\textup{min}}})\epsilon}\right)$ iterations.
\end{theorem}
The proof is given in \cref{subsec:sksaga_convergence}.
\cref{thm:conv_sksaga} proves the linear convergence of SketchySAGA, which allows it to attain high accuracy solutions of \eqref{eq:EmpRiskProb}. 
Since SketchySAGA does not require full gradient computations, it can achieve linear convergence in a streaming setting, where computing full gradients is prohibitively expensive. 
This is a unique advantage of SketchySAGA over SketchySVRG and SketchyKatyusha, which require periodic full gradient computations. 
We also note the dependence upon $\tilde{\mathfrak{q}}_{\textup{max}}$ is identical to that of SketchySVRG up to constant factors, with the only difference being $\tilde{\mathfrak{q}}_{\textup{max}} = \left(\gamma_\ell/\gammaMin\right)\q_{\textrm{max}}$, as we have assumed $b_g = 1$.

As an immediate corollary of \cref{thm:conv_sksaga}, SketchySAGA converges linearly to the optimum, at a rate independent of the condition number, for ridge regression problems.
\begin{corollary}[SketchySAGA: Fast ridge regression]
Instate the hypotheses of \cref{thm:conv_sksaga} and suppose $F$ is quadratic. Then
\begin{align*}
    \E \|w_k-w_\star\|_{P_k}^2  \leq \epsilon
\end{align*}
after $k = 2\left(n+\frac{1+\zeta}{1-\zeta}[1+\chi^{\nu}(A)\deff(A)] \right) \log\left(\frac{nT_0}{(\tau_\star^{\nu}+(1-\zeta)n)\epsilon}\right)$ iterations.
\end{corollary}

\subsection{SketchyKatyusha}
Last, we come to the convergence of the SketchyKatyusha algorithm.
Similar to \cite{allenzhu2018katyusha,kovalev2020lkatyusha}\footnote{Although \cite{allenzhu2018katyusha} does not explicitly introduce a Lyapunov function, it is implicit in their analysis.} and SketchySAGA, the convergence proof is based on showing the contraction of a Lyapunov function. 
The Lyapunov function is
\begin{equation}
\label{eq:lyap_fun_skkat}
\Psi_k \coloneqq \mathcal {W}_k+\mathcal {Y}_k+\mathcal {Z}_k,
\end{equation}
where
\begin{align*}
    \mathcal W_k = \frac{1}{\theta_1}\left(F(w_k)-F(w_\star)\right), \quad \mathcal {Y}_k = \frac{\theta_2 (1+\theta_1)}{\pi\theta_1}(F(y_k)-F(w_\star)), \quad \mathcal Z_k = \frac{\mathcal L_P(1+\eta \sigma)}{2\eta}\|z_k-w_\star\|_{P_k}^2. 
\end{align*}
The definition of the Lyapunov function in \eqref{eq:lyap_fun_skkat} differs from that of \cite{kovalev2020lkatyusha}.
The difference arises as we work in a preconditioned metric, so $\|\cdot\|$ in \cite{kovalev2020lkatyusha} is replaced by $\|\cdot\|_{P_k}$. 
For the same reason, $L$ is replaced by $\mathcal L_P$.
Similar to SketchySAGA, the proof establishes that \eqref{eq:lyap_fun_skkat} contracts in expectation (at each iteration) for appropriately chosen values of $\pi, \sigma, \theta_1, \theta_2$, and $\eta$. 
\begin{theorem}[SketchyKatyusha convergence]
\label{thm:conv_skkat}
Instate the hypotheses of \cref{assm:SmoothSC}-\cref{assm:FiniteUpates}. 
Run SketchyKatyusha with $\pi = b_g/n$, $\sigma = (1-\zeta)/\qbar, \theta_1 = \min\left\{\sqrt{\frac{(1-\zeta)n}{b_g\qbar}}\theta_2,1/2\right\}$, $\theta_2 = 1/2$, and fixed learning rate $\eta = \frac{\theta_2}{(1+\theta_2)\theta_1}$.
Then  
\[
\E[\Psi_{k}] \leq \epsilon\Psi_0
\]
after $k = \max\left\{3n/b_g,n/b_g+2\sqrt{n\qbar/[(1-\zeta)b_g]}\right\}\log(\frac{\mathcal E_P}{\epsilon})$ iterations. 
\end{theorem}

The proof is given in \cref{subsec:skkat_convergence}.
\cref{thm:conv_skkat} shows SketchyKatyusha converges linearly, just like the other variance-reduced PROMISE methods.   
The main difference is the convergence rate depends upon $\sqrt{\qbar}$ instead of $\qbar$, due to the use of acceleration in the algorithm.
When $\qbar$ is moderate, (i.e., $\qbar = \bigO(1)$ for ridge regression) it is unclear whether SketchyKatyusha offers any benefit over SketchySVRG and SketchySAGA. 
Indeed, comparing their respective convergence bounds is insufficient, as they all essentially require $\bigO\left(n\log(\frac{1}{\epsilon})\right)$ gradient computations in this setting.
However, our numerical simulations show that acceleration provides an advantage, particularly for ridge regression problems, which tend to be more ill-conditioned than logistic regression problems.
This observation motivates our recommendation that SketchyKatyusha be the default PROMISE method for ridge regression when full gradient computations are feasible.
Last, we note if $F$ is quadratic, then SketchyKatyusha also enjoys fast global linear convergence; we omit writing down the explicit corollary, as it yields nothing new relative to the ones for SketchySVRG and SketchySAGA.
\else
\fi

\subsection{Convergence proof of SketchySVRG}
\label{subsec:sksvrg_convergence}
In this section we prove \cref{thm:conv_SketchySVRG}, which establishes linear convergence of SketchySVRG. 
The proof is divided into a sequence of helper lemmas; taken together these lemmas allow us to easily establish the theorem.

\subsubsection{Notation}
For clarity in the proof, we explicitly keep track of the outer iteration that a quantity belongs to. 
Specifically, we write $w_k^{(s)}$ for the $k$th iterate in outer iteration $s$, and do the same for other quantities.
Under this convention, $v^{(s)}_k$ denotes the variance-reduced gradient at the $k$th iteration of outer iteration $s$, and $P^{(s)}_k$ is the current preconditioner at the $k$th iteration of outer iteration $s$.

\subsubsection{Preliminary lemmas}
Here, we establish helper lemmas needed to prove \cref{thm:conv_SketchySVRG}.
We start by bounding the second moment of the preconditioned variance-reduced stochastic gradients.
\begin{lemma}[Variance bound]
\label[lemma]{lemma:second_mom_bound}
Let $v^{(s)}_k = \widehat{\nabla}F(w^{(s)}_k) - \widehat{\nabla} F(\hat{w}^{(s)}) + \nabla F(\hat{w}^{(s)})$ be the variance-reduced stochastic gradient at inner iteration $k$ in outer iteration $s$. Then
\begin{align*}
\E \|v^{(s)}_k\|^2_{\Psin} &\leq 4\mathcal L_P[F(w^{(s)}_k) - F(w_\star) + F(\hat{w}^{(s)}) - F(w_\star)].
\end{align*}
\end{lemma}
\begin{proof}
We have
\begin{align*}
    \E \|v^{(s)}_k\|_{\Psin}^2 \overset{(1)}{\leq} {} & 2 \E \|\widehat{\nabla}F(w^{(s)}_k) - \widehat{\nabla}F(w_\star)\|^2_{\Psin} + 2 \E \|[\widehat{\nabla}F(\hat w^{(s)})- \widehat{\nabla}F(w_\star)] - \nabla F(\hat w^{(s)})\|^2_{\Psin} \\
    = {} & 2 \E \|\widehat{\nabla}F(w^{(s)}_k) - \widehat{\nabla}F(w_\star)\|^2_{\Psin} \\ 
    & + 2 \E \|[\widehat{\nabla}F(\hat w^{(s)})- \widehat{\nabla}F(w_\star)]-\E[\widehat \nabla F(\hat w^{(s)}) -\widehat \nabla F(w_\star)]\|^2_{\Psin}\\
    \overset{(2)}{\leq} {} & 2 \E \|\widehat{\nabla}F(w^{(s)}_k) - \widehat{\nabla}F(w_\star)\|^2_{\Psin}+
    2\E \|\widehat{\nabla}F(\hat w^{(s)})- \widehat{\nabla}F(w_\star)\|^2_{\Psin}
    \\
    \overset{(3)}{\leq} {} &  4\mathcal L_P[F(w^{(s)}_k) - F(w_\star) + F(\hat{w}^{(s)}) - F(w_\star)].
\end{align*}
\ifpreprint
Here (1) invokes \cref{lem:A-norm-par}, (2) uses \cref{lem:AVarBnd}, and (3) applies \cref{prop:precond_grad_var} with $w' = w_\star$ twice.
\else
Here, (1) uses $\|a+b\|_A^2 \leq 2\left(\|a\|_A^2+\|b\|_A^2\right)$ and (2) uses $\E\|X-\E X\|_A^2 \leq \E\|X\|_A^2$, which are valid for any random variable $X$ and symmetric positive definite matrix $A$.
Finally, (3) applies \cref{prop:precond_grad_var} with $w' = w_\star$ twice.
\fi
\end{proof}

Next, we have the following one-step relation.
\begin{lemma}[One-step bound]
\label[lemma]{lem:sksvrg_one_step}
Suppose we are in outer iteration $s$ at inner iteration $k$ and $w^{(s)}_{k+1} = w^{(s)}_k-\eta\Psin v^{(s)}_k$. Then
\begin{equation*}
    \E_k\|w^{(s)}_{k+1}-w_{\star}\|^2_{\Ps} \leq \|w^{(s)}_{k}-w_{\star}\|^2_{\Ps}+2\eta \left(2\eta \mathcal L_P-1\right)[F(w_k^{(s)})-F(w_\star)]+4\eta^2\mathcal L_{P}[F(\hat w^{(s-1)})-F(w_\star)].  
\end{equation*}
\end{lemma}
\begin{proof}
    Simply use the definition of the update, expand the square, and invoke \cref{lemma:second_mom_bound}.
\end{proof}

We now come to the key lemma for establishing convergence, which shows a contraction of suboptimalities between consecutive outer iterations. 
\begin{lemma}[outer iteration contraction]
\label[lemma]{lem:sksvrg_out_contract}
Suppose we are in outer iteration $s+1$. Then
\begin{equation}
    \E_{0:s}[F(\hat w^{(s+1)})]-F(w_\star)\leq \left[\frac{1}{(1-\zeta)\gamma_\ell\eta(1-2\eta \mathcal L_P)m} +\frac{2\eta\mathcal L_P}{1-2\eta \mathcal L_P}\right]\left(F(\hat w^{(s)})-F(w_\star)\right),
\end{equation}
where $\E_{0:s}$ denotes the expectation conditioned on outer iterations $0$ through $s$.
\end{lemma}
\begin{proof}
Summing the bound in \cref{lem:sksvrg_one_step} over $k = 0,\dots m-1$, we reach
\begin{align*}
    \sum_{k=0}^{m-1}\E_k\|w^{(s)}_{k+1}-w_{\star}\|^2_{\Ps} \leq {} & \sum_{k=0}^{m-1}\|w^{(s)}_{k}-w_{\star}\|^2_{\Ps}+2\eta m\left(2\eta \mathcal L_P-1\right)\frac{1}{m}\sum_{k=0}^{m-1}[F(w_k^{(s)})-F(w_\star)]\\
    &+ 4m\eta^2\mathcal L_{P}[F(\hat w^{(s)})-F(w_\star)]. 
\end{align*}
Now, taking the expectation over all the inner iterations conditioned on outer iterations $0$ through $s$, we find 
\begin{align*}
    \E_{0:s}\|w_m^{(s)}-w_\star\|_{\Ps}^2 \leq {} & \|\hat w^{(s)}-w_{\star}\|^2_{P^{(s)}_0}+2\eta m\left(2\eta \mathcal L_P-1\right)\left(\E_{0:s}\left[F(\hat w^{(s+1)})\right]-F(w_\star)\right)\\
    &+ 4m\eta^2\mathcal L_{P}[F(\hat w^{(s)})-F(w_\star)].
\end{align*} 
Rearranging and invoking quadratic regularity of $f$, we reach
\begin{align*}
    &\E_{0:s}\|w_m^{(s)}-w_\star\|_{\Ps}^2+2\eta m\left(1-2\eta \mathcal L_P\right)\left(\E_{0:s}\left[F(\hat w^{(s+1)})\right]-F(w_\star)\right) \\
    &\leq  2\left(\frac{1}{(1-\zeta)\gamma_\ell}+2m\eta^2 \mathcal L_P\right)[F(\hat w^{(s)})-F(w_\star)].
\end{align*}
Hence we conclude
\begin{align*}
 \E_{0:s}[F(\hat w^{(s+1)})]-F(w_\star)\leq \left[\frac{1}{(1-\zeta)\gamma_\ell\eta(1-2\eta \mathcal L_P)m} +\frac{2\eta\mathcal L_P}{1-2\eta \mathcal L_P}\right]\left(F(\hat w^{(s)})-F(w_\star)\right).
\end{align*}
\end{proof}
\subsubsection{SketchySVRG convergence: Proof of \cref{thm:conv_SketchySVRG}}
\begin{proof}
    From \cref{lem:sksvrg_out_contract} we have,
\begin{align*}
         \E_{0:s-1}[F(\hat w^{(s)})]-F(w_\star)\leq \left[\frac{1}{(1-\zeta)\gamma_\ell\eta(1-2\eta \mathcal L_P)m} +\frac{2\eta\mathcal L_P}{1-2\eta \mathcal L_P}\right]\left(F(\hat w^{(s-1)})-F(w_\star)\right).
\end{align*}
Setting $\eta = \frac{1}{8\mathcal L_P}$ and $m = \frac{19}{1-\zeta}\qbar$, we obtain
\[
\E_{0:s-1}[F(\hat w^{(s)})]-F(w_\star)\leq \frac{9}{10}\left(F(\hat w^{(s-1)})-F(w_\star)\right).
\]
Taking the total expectation over all outer iterations, and recursing, we reach
\[
\E[F(\hat w^{(s)})]-F(w_\star)\leq \left(\frac{9}{10}\right)^s\left(F(w_0)-F(w_\star)\right).
\]
Hence after $s = 10\log\left(\frac{F(w_0)-F(w_\star)}{\varepsilon}\right)$ outer iterations we have
\[
\E[F(\hat w^{(s)})]-F(w_\star)\leq \varepsilon.
\]
\end{proof}

\section{Numerical experiments}
\label{section:experiments}
In this section, we provide four sets of experiments to demonstrate 
the effectiveness of the PROMISE methods for $l^2$-regularized least squares and logistic regression problems. 
\ifpreprint
We also provide studies that examine the sensitivity of our methods to hyperparameters and investigate the quadratic regularity ratio.
\else
We also investigate the quadratic regularity ratio.
\fi
We present the following results:
\begin{itemize}
    \item Performance experiments (\cref{subsection:performance}): We compare PROMISE methods to SVRG, b-nice SAGA (henceforth referred to as SAGA), Loopless Katyusha (L-Katyusha), and stochastic L-BFGS (SLBFGS), whose learning rates are tuned. 
    We find that our methods outperform the competition on a testbed of $51$ medium-sized least squares and logistic regression problems. 
    \item Suboptimality experiments (\cref{subsection:subopt}): We show that PROMISE methods achieve global linear convergence on several least squares and logistic regression problems, which matches the global linear convergence guarantees in \cref{section:theory}. Furthermore, our methods converge faster than the competition.
    \item Showcase experiments (\cref{subsection:showcase}): We evaluate PROMISE methods against the competition on the url, yelp, and acsincome datasets, which originate in real-world applications and lead to large-scale problems.
    We again find that our methods outperform the competition. 
    \ifpreprint
    \item Streaming experiments (\cref{subsection:large_scale}): We test PROMISE methods on performing logistic regression with large-scale transformations of the HIGGS and SUSY datasets. 
    These transformed datasets are so large that they do not fit in the memory of most computers, putting these experiments in a streaming setting where the computation of full gradients is prohibitive. 
    Our methods continue to outperform the competition.
    \else
    \item Streaming experiments (\cref{subsection:large_scale}): We test PROMISE methods on performing logistic regression with a large-scale transformation of the HIGGS dataset. 
    This transformed dataset is so large that it does not fit in the memory of most computers, putting these experiments in a streaming setting where the computation of full gradients is prohibitive. 
    Our methods continue to outperform the competition.
    \fi
    \ifpreprint
    \item Sensitivity study (\cref{subsection:sensitivity}): We investigate how the choice of rank $r$ and update frequency $u$ impacts the performance of PROMISE methods.
    We find that the impact of these hyperparameters depends on the spectral properties of the data.
    \else
    \fi
    \item Regularity study (\cref{subsection:regularity}): We demonstrate that the quadratic regularity ratio, $\gamma_u/\gamma_\ell$, is well-behaved over the optimization trajectory, which provides empirical support for our claims in \cref{section:fast_global_convergence}. 
\end{itemize}

The experiments in \cref{subsection:performance,subsection:subopt,subsection:showcase,subsection:large_scale,subsection:regularity} run PROMISE methods with the default hyperparameters given in \cref{section:algs}. 
Throughout the experiments, we set the $l^2$-regularization parameter $\nu = 10^{-2}/n_{\mathrm{tr}}$, where $n_{\mathrm{tr}}$ is the number of samples in the training set,
which typically results in an ill-conditioned problem. 
\ifpreprint
The results are qualitatively similar for larger values of $\nu$; when $\nu = 10^{-1}/n_{\mathrm{tr}}$, PROMISE methods still outperform the competition on the performance experiments (\cref{subsubsection:performance_mu_1e-1_appdx}).
\else
\fi
\ifpreprint
    All preconditioners use the default values of $r$ and $\rho$ in \cref{table:precond_hyperparams}, unless stated otherwise.
    Additional details appear in \cref{appndx:experiment_details} and code for our experiments can be found at \href{https://github.com/udellgroup/PROMISE}{https://github.com/udellgroup/PROMISE}.
\else
    All preconditioners use the default values of $r$ and $\rho$ in \cref{subsection:precond_comp}.
    Additional details appear in Appendix D of \href{https://arxiv.org/abs/2309.02014v2}{https://arxiv.org/abs/2309.02014v2} and code for our experiments can be found at \href{https://github.com/udellgroup/PROMISE}{https://github.com/udellgroup/PROMISE}.
\fi



\subsection{Performance experiments}
\label{subsection:performance}

Our first set of experiments compares the performance of SketchySVRG, SketchySAGA, and SketchyKatyusha, with their \textit{default} hyperparameters, to SVRG, SAGA, L-Katyusha, and SLBFGS, with \textit{tuned} hyperparameters, on solving ridge and $l^2$-regularized logistic regression problems. These experiments therefore \emph{understate} the performance improvement that can be expected by using PROMISE methods. Moreover, we modify SLBFGS to compute the preconditioner once per epoch rather than at every iteration for a fair comparison. 

SAGA/SketchySAGA require one full pass through the data per epoch, while SVRG/L-Katyusha/SLBFGS/SketchySVRG/SketchyKatyusha use two full passes through the data per epoch since they compute full gradients\footnote{L-Katyusha and SketchyKatyusha compute full gradients with random probability, and our hyperparameter settings result in one full gradient computation per epoch, in expectation.}.
By using the number of full data passes we (roughly) equate the computation required for computing gradients, making for a fair comparison. 
We compute the minimum $F(w^{\star})$ for all ridge and logistic regression problems via scikit-learn \citep{pedregosa2011scikit}.
\ifpreprint
We run neither SketchySGD nor SGD because these algorithms do not converge linearly and hence underperform the others (\cref{subsubsection:sketchy_opt_comparison_appdx}).
\else
We run neither SketchySGD nor SGD because these algorithms do not converge linearly.
\fi

Our primary metrics for comparing the performance of these methods are the wall-clock time and number of full data passes to reach suboptimality within $10^{-4}$ of the minimum, $F(w^\star)$. 
Each optimizer is run either until this suboptimality condition is met (i.e., the problem is solved), 
or for $200$ full data passes ($100$ epochs for SVRG, L-Katyusha, SLBFGS, SketchySVRG, and SketchyKatyusha, $200$ epochs for SAGA and SketchySAGA).

\subsubsection{Ridge regression}
\label{subsubsection:ridge}
We solve ridge regression problems of the form
\begin{align*}
    \textrm{minimize}_{w\in \mathbb{R}^p} ~ \frac{1}{n_{\mathrm{tr}}} \sum_{i=1}^{n_{\mathrm{tr}}} \frac{1}{2} (a_i^T w - b_i)^2 + \frac{\nu}{2} \|w\|_2^2,
\end{align*}
where $a_i \in \R^p$ is a datapoint, $b_i \in \R$ is a label, and $\nu > 0$ is the regularization parameter.

Our experiments in this setting are performed on a testbed of $17$ datasets from OpenML \citep{vanschoren2013openml} and LIBSVM \citep{chang2011libsvm}.
\ifpreprint
We apply random features \citep{rahimi2007random,mei2022randfeatures} to most, but not all, datasets; further details regarding preprocessing are provided in \cref{subsubsection:performance_exp_data_appdx}.
\else
We apply random features \citep{rahimi2007random,mei2022randfeatures} to most, but not all, datasets; further details regarding preprocessing may be found in the arxiv report.
\fi

\ifpreprint
The results of these experiments appear in \cref{fig:prop_solved_least_squares,fig:ranking_least_squares}. 
\cref{fig:prop_solved_least_squares} shows the proportion of problems solved by both our methods and the competitor methods as a function of wall-clock time and full data passes. 
\else
Results appear in \cref{fig:prop_solved_least_squares}, 
which shows the proportion of problems solved by both our methods and the competitor methods as a function of wall-clock time and full data passes. 
\fi
When combined with any of the \ssn{}, \nyssn{}, \sassnc{}, and \sassnr{} preconditioners, SketchySVRG, SketchySAGA, and SketchyKatyusha uniformly outperform competitor methods.
\ifpreprint
Furthermore, the performance of SketchySVRG, SketchySAGA, and SketchyKatyusha degrades considerably when combined with the \diagssn{} preconditioner (\cref{fig:prop_solved_least_squares_appdx} in \cref{subsubsection:precond_comparison_appdx}), demonstrating the value of a low-rank approximation to the subsampled Hessian over a diagonal approximation.
\else
\fi
\ifpreprint

\cref{fig:ranking_least_squares} shows the average ranking of all the optimization methods over the course of optimization.
To compute these rankings, we rank each optimizer by the number of problems solved 
at every $10$ timesteps (either $10$ seconds or $10$ full data passes), 
and then compute the mean of these ranks over the entire optimization trajectory. 
We observe that the \ssn{} and \nyssn{} preconditioners tend to outperform the others.
\else
\fi
SketchyKatyusha and SketchySVRG perform slightly better than SketchySAGA,
supporting our recommendation 
to use SketchyKatyusha for ridge regression.


\begin{figure}[p]
    \centering
    \includegraphics[scale=0.4]{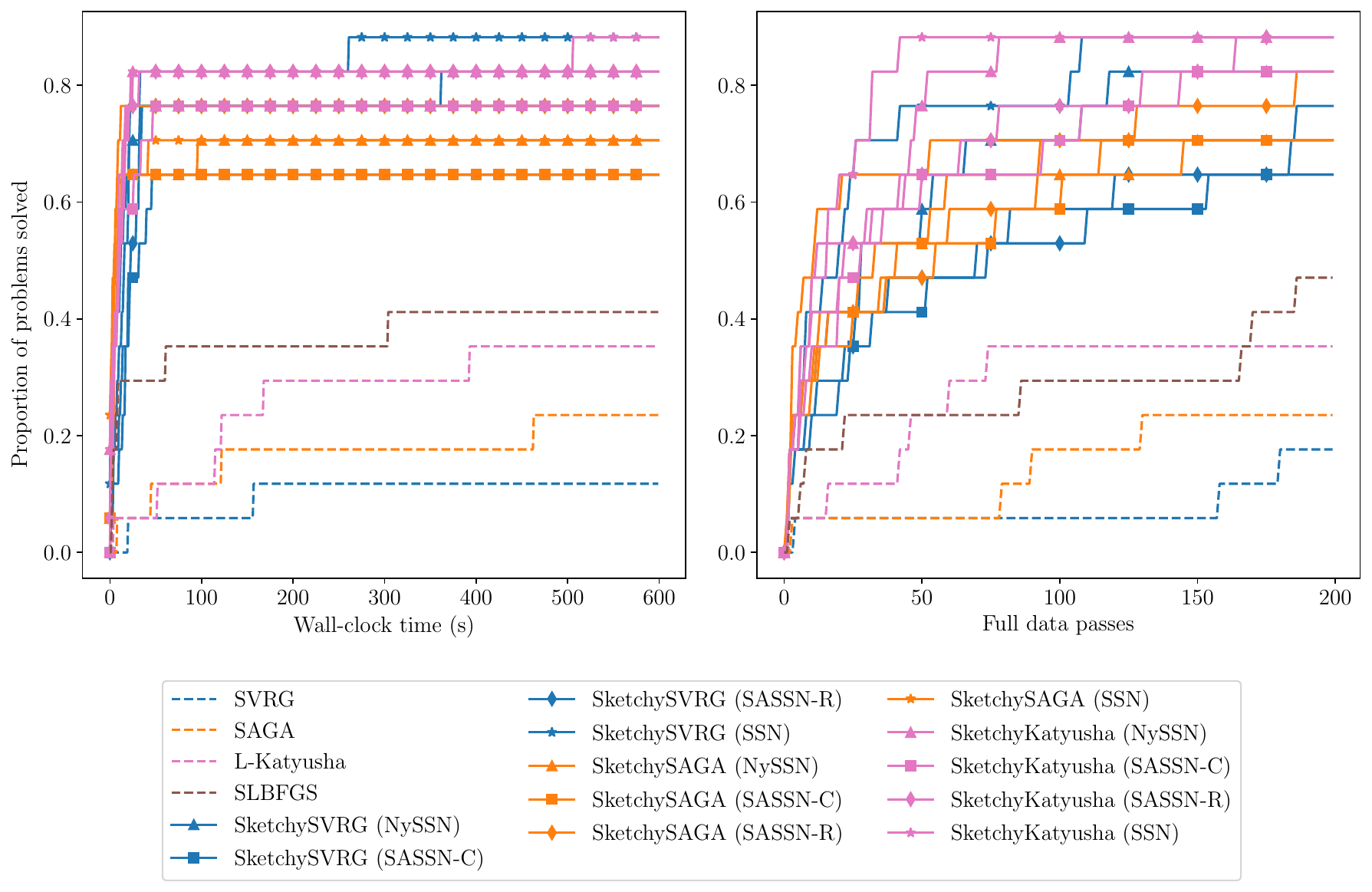}
    \caption{PROMISE methods solve ridge regression problems faster than competitors.}
    \label{fig:prop_solved_least_squares}
\end{figure}

\ifpreprint
\begin{figure}[p]
    \centering
    \includegraphics[scale=0.4]{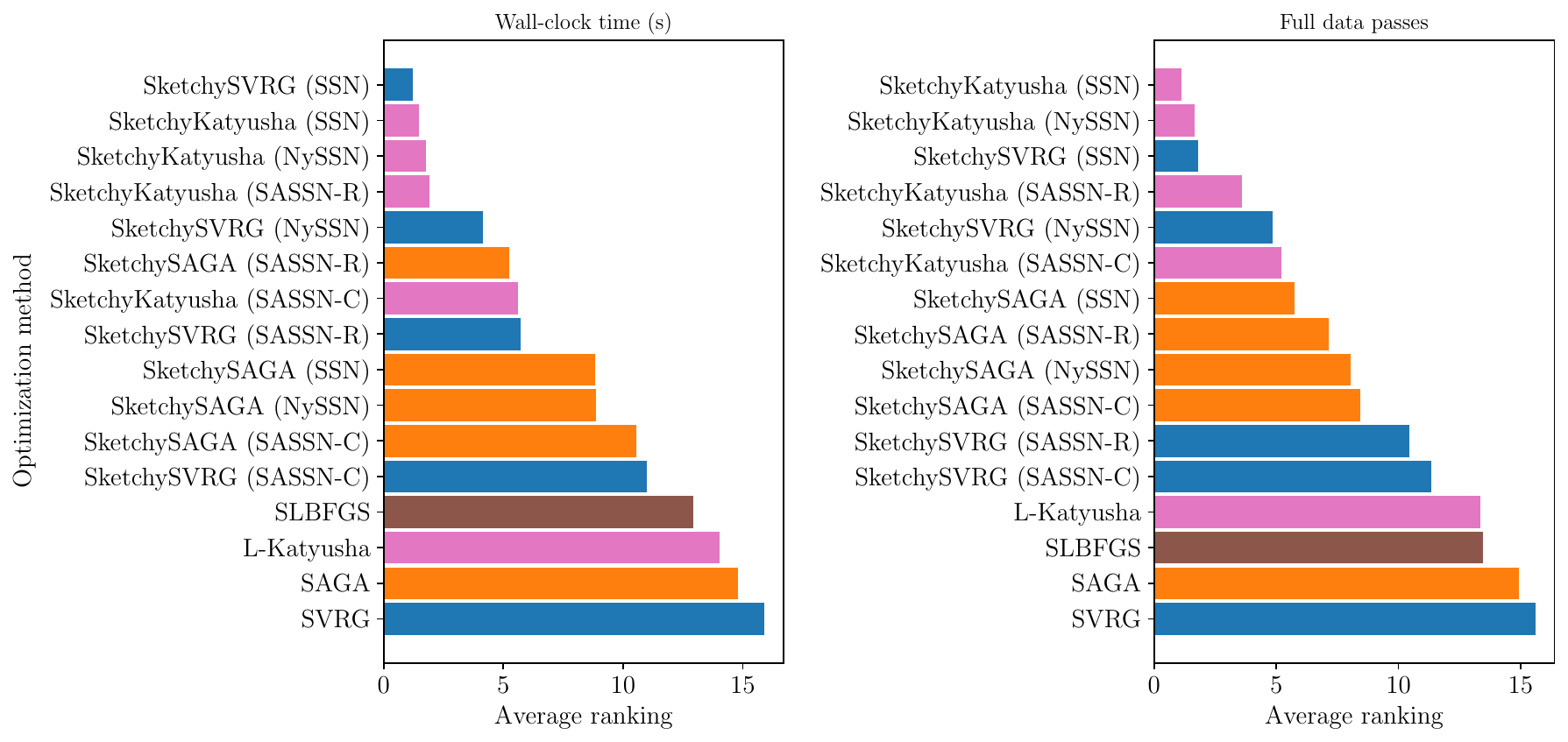}
    \caption{Average ranking of each method by number of ridge regression problems solved with respect to wall-clock time (left) and full gradient computations (right).}
    \label{fig:ranking_least_squares}
\end{figure}
\else
\fi


\subsubsection{$l^2$-regularized logistic regression}
\label{subsubsection:logistic}
\begin{figure}[p]
    \centering
    \includegraphics[scale=0.4]{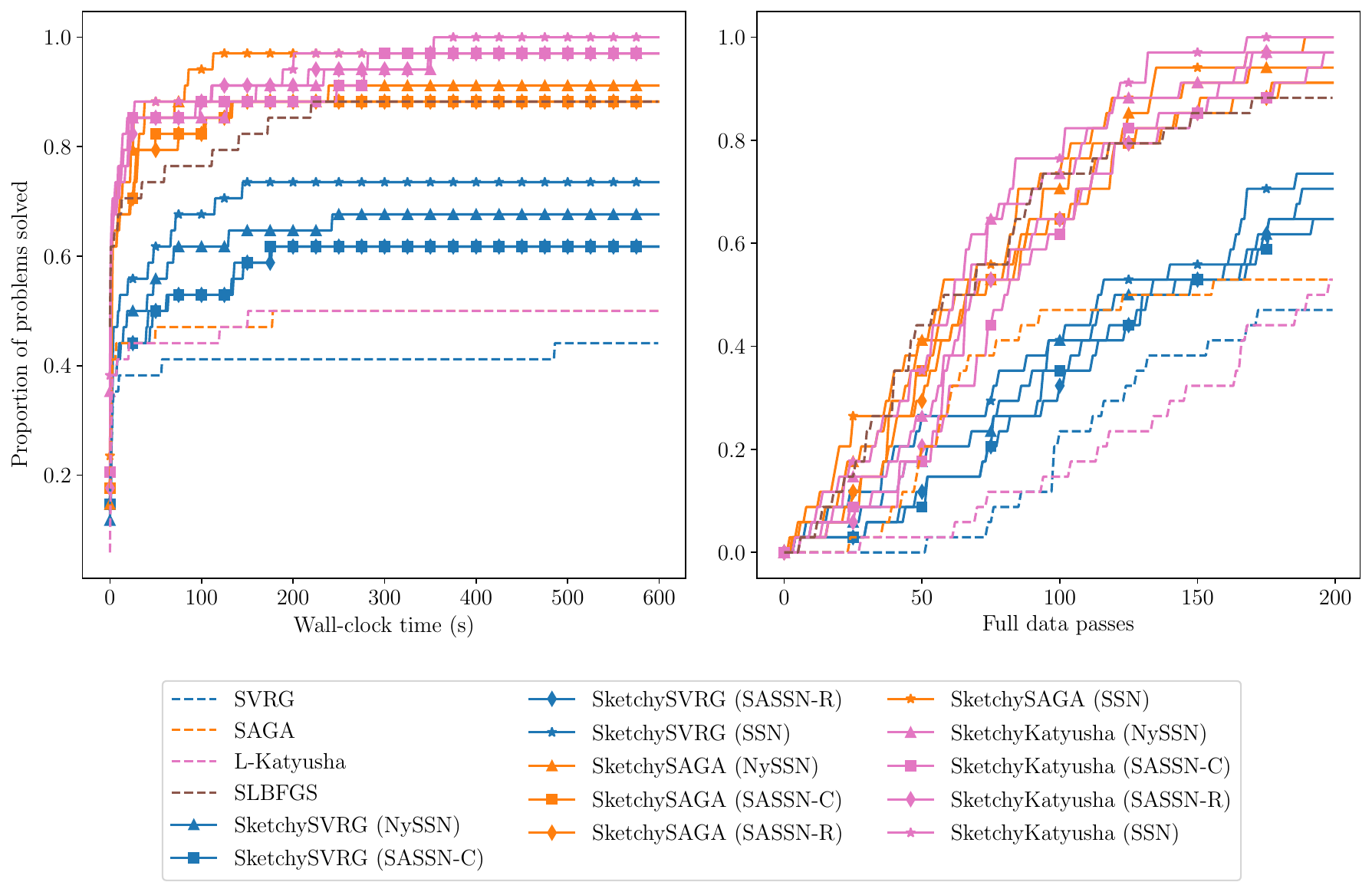}
    \caption{PROMISE methods (and SLBFGS) solve $l^2$-regularized logistic regression problems faster than competitors.}
    \label{fig:prop_solved_logistic}
\end{figure}

\ifpreprint
\begin{figure}[p]
    \centering
    \includegraphics[scale=0.4]{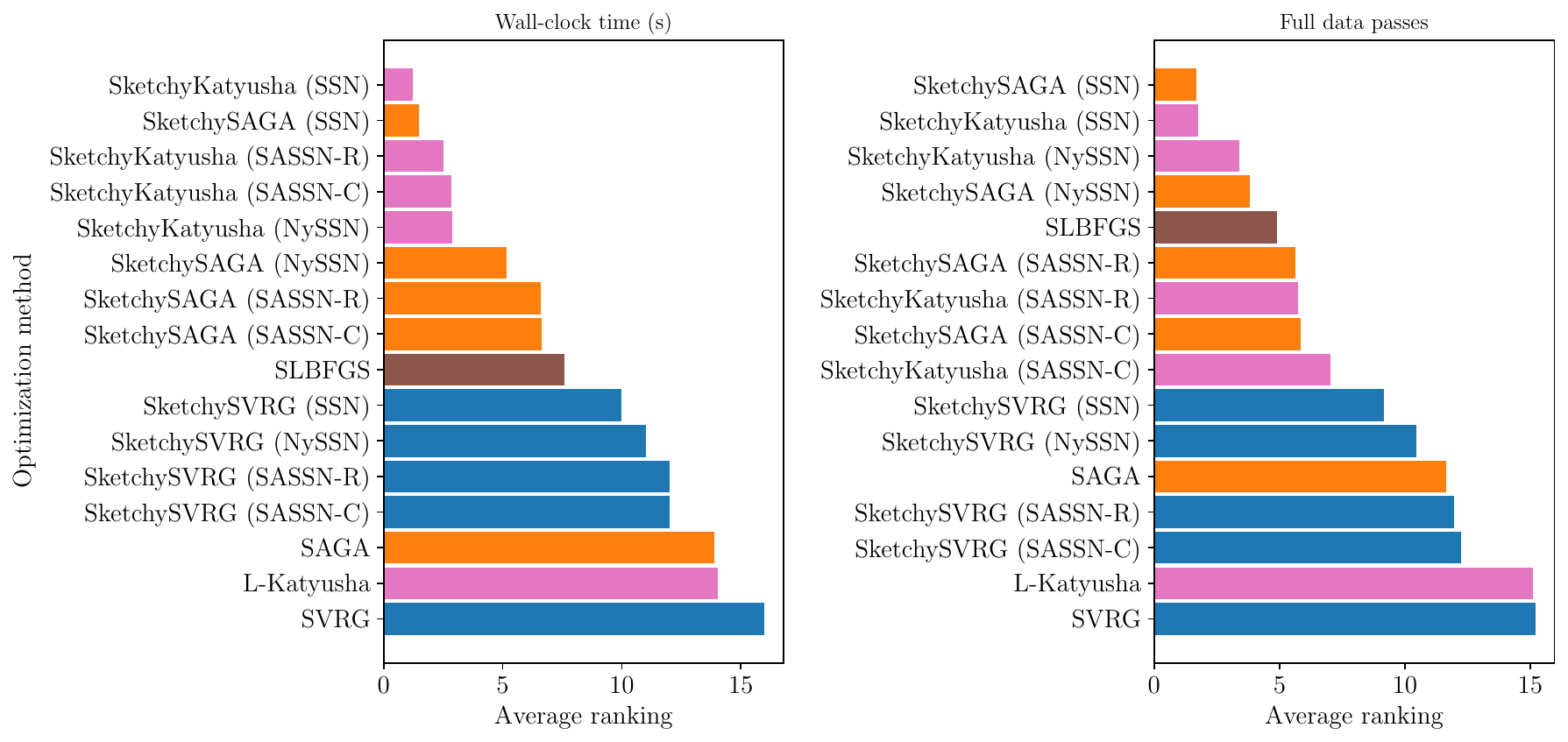}
    \caption{Average ranking of each method by number of $l^2$-regularized logistic regression problems solved with respect to wall-clock time (left) and full gradient computations (right).}
    \label{fig:ranking_logistic}
\end{figure}
\else
\fi

We solve $l^2$-regularized logistic regression problems of the form
\begin{align*}
    \textrm{minimize}_{w\in \mathbb{R}^p} ~ \frac{1}{n_{\mathrm{tr}}} \sum_{i=1}^{n_{\mathrm{tr}}} \log(1 + \exp(-b_i a_i^T w)) + \frac{\nu}{2} \|w\|_2^2,
\end{align*}
where $a_i \in \R^p$ is a datapoint, $b_i \in \{-1, 1\}$ a label, and $\nu > 0$ the regularization parameter.

These experiments use a testbed of $34$ datasets from LIBSVM. 
\ifpreprint
We apply random features to a few of the datasets; further details regarding preprocessing appear in \cref{subsubsection:performance_exp_data_appdx}.
\else
We apply random features to a few of the datasets; further details regarding preprocessing appear in the \href{https://arxiv.org/abs/2309.02014v2}{arxiv report}.
\fi

\ifpreprint
The results of these experiments appear in \cref{fig:prop_solved_logistic,fig:ranking_logistic}. \cref{fig:prop_solved_logistic} shows the proportion of problems solved by both our methods and the competitor methods as a function of wall-clock time and full data passes. 
\else
The results of these experiments appear in \cref{fig:prop_solved_logistic}, which shows the proportion of problems solved by both our methods and the competitor methods as a function of wall-clock time and full data passes. 
\fi
When combined with any one of the \ssn{}, \nyssn{}, \sassnc{}, and \sassnr{} preconditioners, SketchySVRG, SketchySAGA, and SketchyKatyusha uniformly outperform SVRG, SAGA and L-Katyusha. 
In addition, SketchySAGA and SketchyKatyusha outperform SLBFGS, which also employs preconditioning.
\ifpreprint
Just like in ridge regression, the \diagssn{} preconditioner leads to degraded performance (\cref{fig:prop_solved_logistic_appdx} in \cref{subsubsection:precond_comparison_appdx}), which further demonstrates the value of low-rank approximations to the subsampled Hessian.
\else
\fi
\ifpreprint
\cref{fig:ranking_logistic} shows the average ranking of all the optimization methods over the course of optimization. 
We calculate these rankings using the same approach as in ridge regression. 
The \ssn{} preconditioner outperforms all other preconditioners, while \nyssn{}, \sassnc{}, and \sassnr{} display similar performance to one another.
\else
\fi

\ifpreprint
We compare \ssn{} and \nyssn{} on both sparse and dense datasets in \cref{subsubsection:logistic_sparse_dense_appdx}; we find that \ssn{} tends to outperform \nyssn{} on sparse problems while \nyssn{} slightly outperforms \ssn{} on dense problems, which is in line with the preconditioner recommendations in \cref{table:precond_comp}.
\else
\fi


Overall, SketchySAGA and SketchyKatyusha perform much better than SketchySVRG 
here, supporting our recommendation in \cref{subsection:algo_recs} to use 
SketchyKatyusha (assuming we can compute full gradients) or SketchySAGA for logistic regression.



\subsection{Suboptimality experiments}
\label{subsection:subopt}
We examine the objective suboptimality (with respect to the lowest attained training loss for all methods) for SketchySVRG, SketchySAGA, and SketchyKatyusha, with their \textit{default} hyperparameters, and the competitor methods, with \textit{tuned} hyperparameters.
For simplicity, we only show PROMISE methods with the \nyssn{} and \ssn{} preconditioner.
Each optimizer is run for $200$ full data passes ($100$ epochs for SVRG, L-Katyusha, SLBFGS, SketchySVRG, and SketchyKatyusha, $200$ epochs for SAGA and SketchySAGA).
\ifpreprint
We provide additional details and a comparison with SketchySGD in \cref{subsection:subopt_exp_appdx}.
\else
\fi

\cref{fig:subopt_least_squares,fig:subopt_logistic} display objective suboptimality (with respect to the lowest attained training loss) for selected datasets on ridge and $l^2$-regularized logistic regression. 
The objective suboptimality for PROMISE methods decreases linearly for ridge and logistic regression, which matches the theoretical convergence guarantees in \cref{section:theory}. 
On ridge regression, PROMISE methods uniformly outperform the competition,
even reaching machine precision on the yolanda dataset!
On logistic regression, PROMISE methods generally outpeform SVRG, SAGA, and L-Katyusha.
Interestingly, SLBFGS outperforms PROMISE methods on ijcnn1.
However, SLBFGS can be unstable; for example, SLBFGS initially outperforms PROMISE methods on SUSY, but the training loss suddenly spikes and then diverges.

\begin{figure}[h]
    \centering
    \includegraphics[scale=0.5]{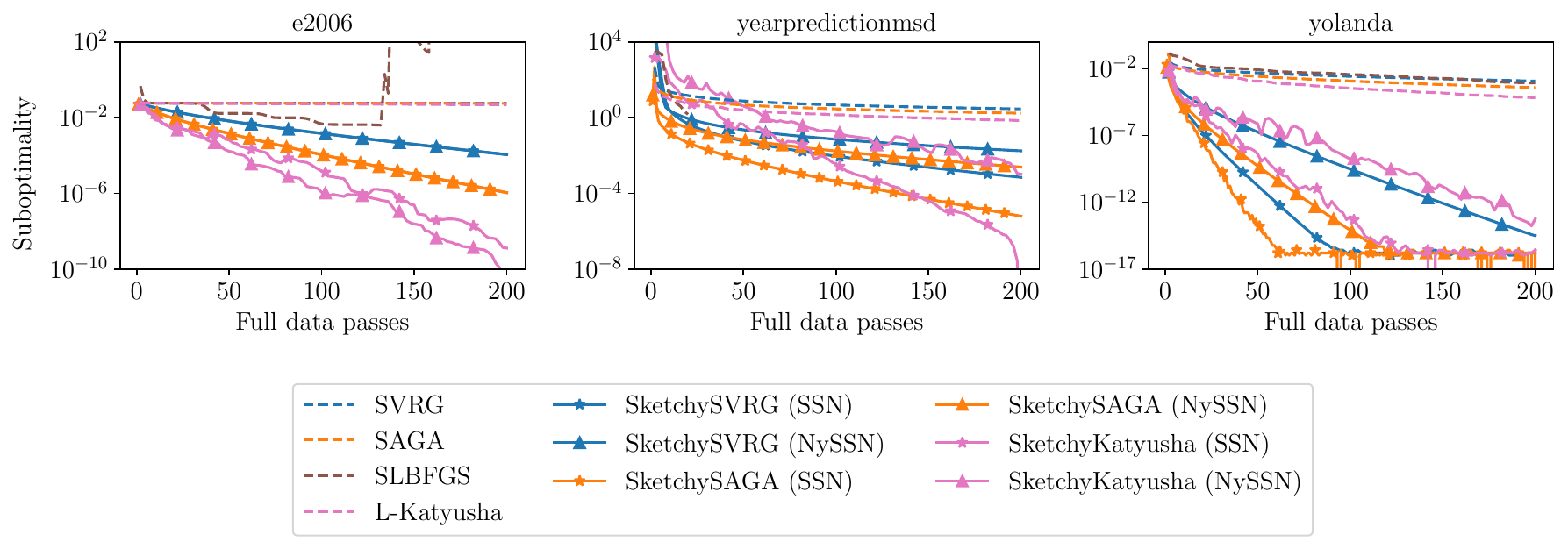}
    \caption{Suboptimality comparisons between our proposed methods and tuned competitor methods for selected datasets on ridge regression.}
    \label{fig:subopt_least_squares}
\end{figure}

\begin{figure}[h]
    \centering
    \includegraphics[scale=0.5]{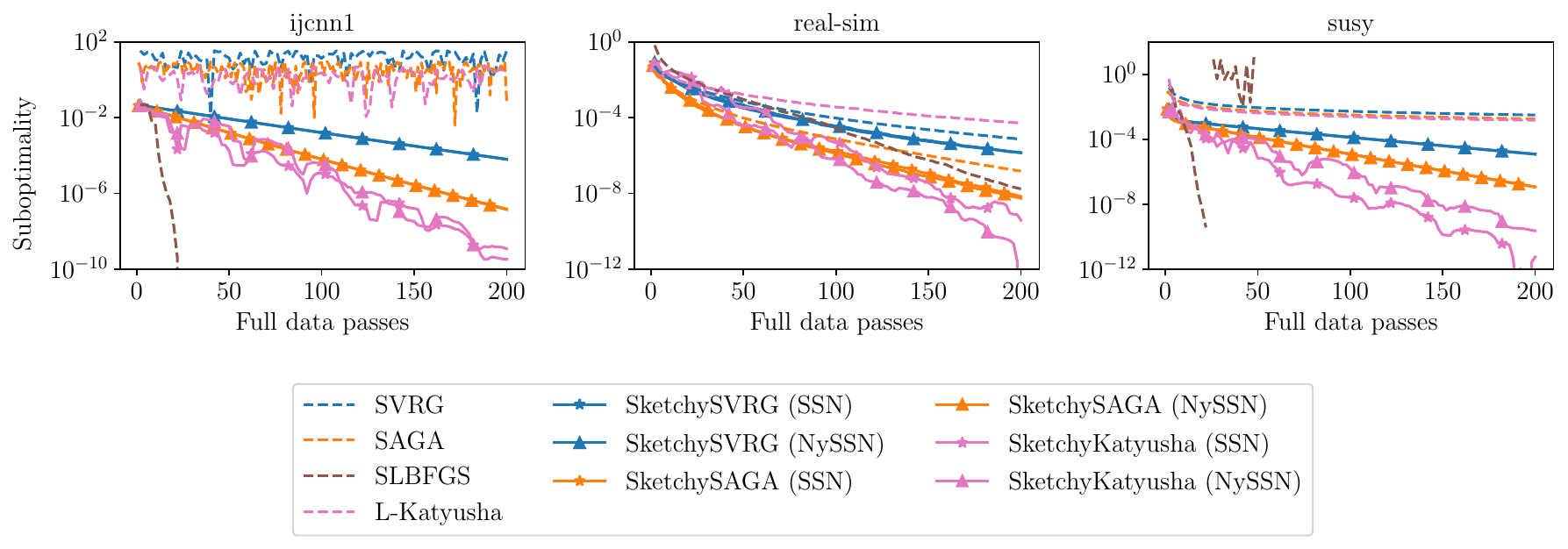}
    \caption{Suboptimality comparisons between our proposed methods and tuned competitor methods for selected datasets on $l^2$-regularized logistic regression.}
    \label{fig:subopt_logistic}
\end{figure}

\subsection{Showcase experiments}
\label{subsection:showcase}
Our second set of experiments compares the performance of SketchySVRG, SketchySAGA, and SketchyKatyusha, with their \textit{default} hyperparameters, to SVRG, SAGA, L-Katyusha, and SLBFGS with both default and tuned hyperparameters on the url, yelp, and acsincome datasets.
All three datasets originate in real-world applications: the url dataset is used to train a $l^2$-regularized logistic regression classifier that detects malicious websites using features derived from URLs, the yelp dataset is used to train a $l^2$-regularized logistic regression classifier that predicts sentiment from user reviews, and the acsincome dataset is used to train a ridge regression classifier that predicts income given demographic information such as age, employment, and education.
After preprocessing, all three of these datasets have $n_{\mathrm{tr}} > 10^6$ training examples, while url and yelp have $p > 10^6$ features, putting all three of these datasets in the big-data regime.
\ifpreprint
Additional information regarding these datasets appears in \cref{subsection:showcase_exp_appdx}.
\else
\fi
We provide two sets of comparisons: the first set compares our methods to SVRG, SAGA, and L-Katyusha with their \textit{default} hyperparameters, while the second set compares our methods to SVRG, SAGA, L-Katyusha, and SLBFGS with their \textit{tuned} hyperparameters.
We run each optimizer with a fixed time budget: 1 hour for url and yelp, and 2 hours for acsincome.

\ifpreprint
The first set of comparisons appears in \cref{fig:showcase_auto_train,fig:showcase_auto_test}.
\cref{fig:showcase_auto_train} shows the training loss for our methods and the competitor methods (with default hyperparameters) as a function of wall-clock time. 
When combined with either of the \ssn{} or \nyssn{} preconditioners, SketchySVRG, SketchySAGA, and SketchyKatyusha uniformly outperform the competitor methods on their default hyperparameters.
\cref{fig:showcase_auto_test} compares the same optimizers, but on test classification error (url, yelp) and test loss (acsincome). 
We observe our methods generalize better to test data than the competitor methods while running much faster.
\else
The first set of comparisons appears in \cref{fig:showcase_auto_test}, which compares our methods and the competitor methods (with default hyperparameters) on test classification error (url, yelp) and test loss (acsincome) as a function of wall-clock time.
When combined with either of the \ssn{} or \nyssn{} preconditioners, SketchySVRG, SketchySAGA, and SketchyKatyusha uniformly outperform the competitor methods on their default hyperparameters.
Our methods generalize better to test data than the competitor methods while running much faster.
\fi

\ifpreprint
\begin{figure}[h]
    \centering
    \includegraphics[scale=0.5]{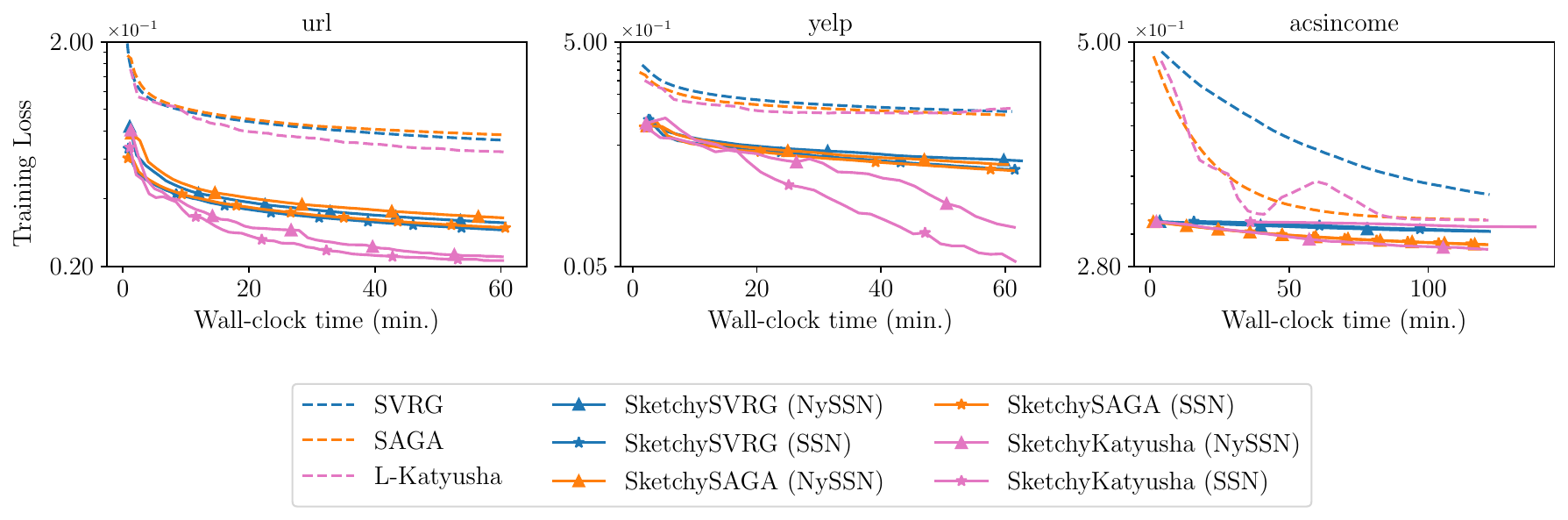}
    \caption{Comparisons to competitor methods on training loss with default learning rates (SVRG, SAGA) and smoothness parameters (L-Katyusha).}
    \label{fig:showcase_auto_train}
\end{figure}
\else
\fi

\begin{figure}[h]
    \centering
    \includegraphics[scale=0.5]{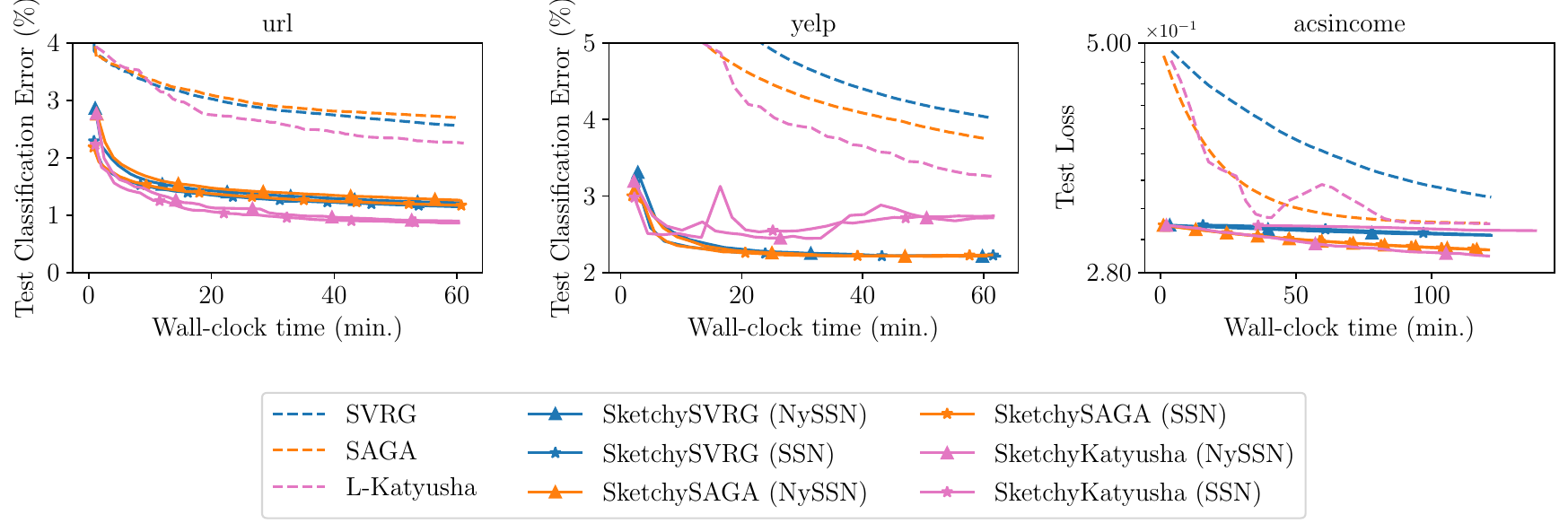}
    \caption{Comparisons to competitor methods on test metrics with default learning rates (SVRG, SAGA) and smoothness parameters (L-Katyusha).}
    \label{fig:showcase_auto_test}
\end{figure}

\ifpreprint
The second set of comparisons appears in \cref{fig:showcase_tuned_train,fig:showcase_tuned_test}.
\cref{fig:showcase_tuned_train} shows the training loss for our methods and the competitor methods (with tuned hyperparameters) a function of wall-clock time; \cref{fig:showcase_tuned_test} shows test metrics. 
PROMISE methods outperform the competition on url and acsincome and perform comparably on yelp. PROMISE methods are also substantially more stable along the optimization trajectory, whereas the competition often suffer sudden large increases in the loss.
Moreover, recall that the performance of the competitor methods is only possible after hyperparameter tuning, which is quite expensive for datasets of this size, whereas PROMISE methods still obtain good performance with default hyperparameters.
\else
The second set of comparisons appears in \cref{fig:showcase_tuned_test}, which compares our methods and the competitor methods (with tuned hyperparameters) on test metrics as a function of wall-clock time.
PROMISE methods outperform the competition on url and acsincome and perform comparably on yelp. 
Moreover, recall that the performance of the competitor methods is only possible after hyperparameter tuning, which is quite expensive for datasets of this size, whereas PROMISE methods still obtain good performance with default hyperparameters.
\fi


\ifpreprint
\begin{figure}[h]
    \centering
    \includegraphics[scale=0.5]{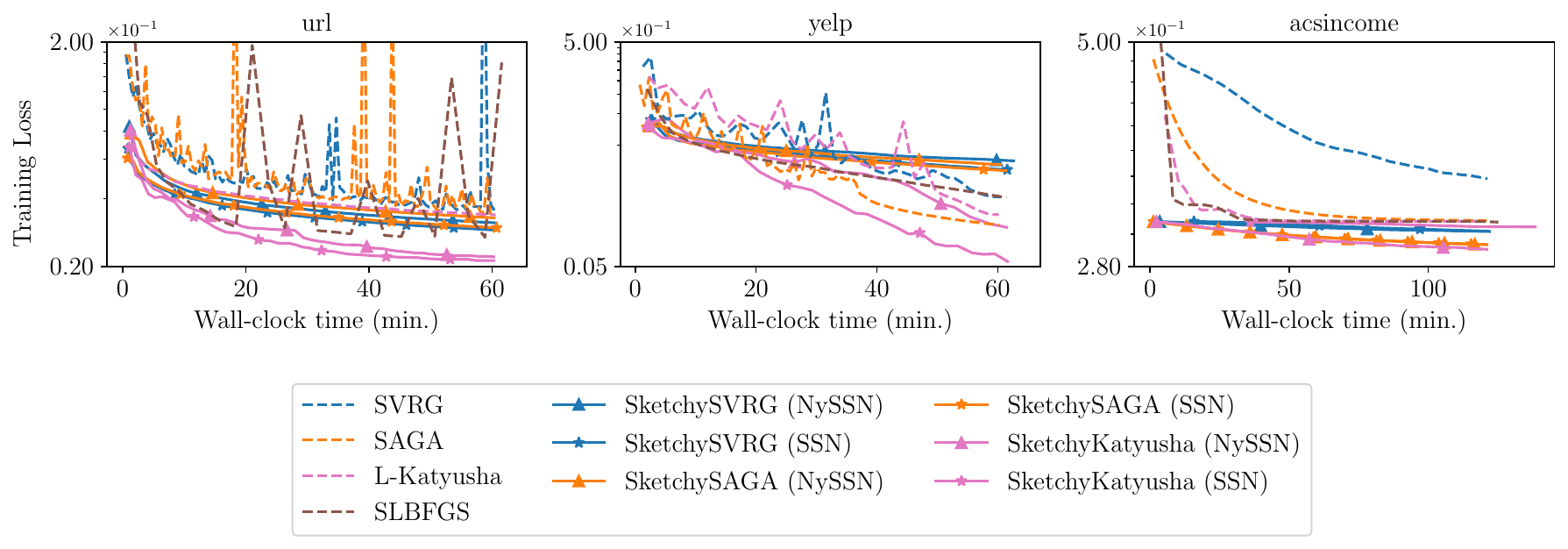}
    \caption{Comparisons to competitor methods on training loss with tuned learning rates (SGD, SVRG, SAGA, SLBFGS) and smoothness parameters (L-Katyusha).}
    \label{fig:showcase_tuned_train}
\end{figure}
\else
\fi

\begin{figure}[h]
    \centering
    \includegraphics[scale=0.5]{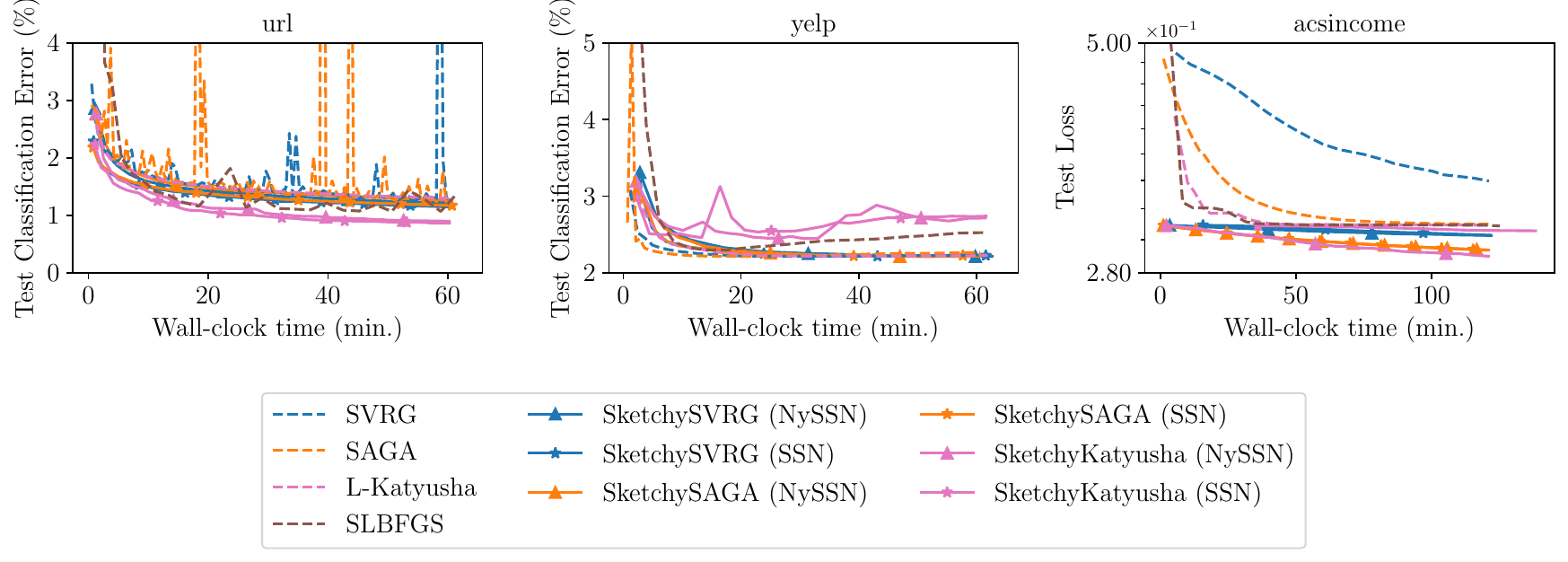}
    \caption{Comparisons to competitor methods on test metrics with tuned learning rates (SGD, SVRG, SAGA, SLBFGS) and smoothness parameters (L-Katyusha).}
    \label{fig:showcase_tuned_test}
\end{figure}

\subsection{Streaming experiments}
\label{subsection:large_scale}
\ifpreprint
We apply random features to the HIGGS and SUSY datasets to obtain transformed datasets with sizes $840$ and $720$ GB respectively (see \cref{subsubsection:streaming_exp_data_appdx} for more details). 
These transformed datasets are much larger than the hard drive and RAM of most computers.
We solve a $l^2$-regularized logistic regression problem on each transformed dataset.
To perform optimization, we load the original dataset in memory and at each iteration, form a minibatch of the transformed dataset by applying random features to a minibatch of the data.
In this setting, computing a full gradient of the objective is computationally prohibitive, so we exclude SVRG, L-Katyusha, SLBFGS, SketchySVRG, and SketchyKatyusha.
We compare our methods to SGD and SAGA with their \textit{tuned} hyperpameters (a comparison to SAGA with its \textit{default} hyperparameter is presented in \cref{subsubsection:streaming_exp_def_comp_appdx}). 
All optimization methods are run for $10$ epochs.
\else
We apply random features to the HIGGS dataset to obtain a transformed datasets with size $840$ GB (see the \href{https://arxiv.org/abs/2309.02014v2}{arxiv report} for more details). 
This transformed dataset is much larger than the hard drive and RAM of most computers.
We solve a $l^2$-regularized logistic regression problem on this transformed dataset.
To perform optimization, we load the original dataset in memory and at each iteration, form a minibatch of the transformed dataset by applying random features to a minibatch of the data.
In this setting, computing a full gradient of the objective is computationally prohibitive, so we exclude SVRG, L-Katyusha, SLBFGS, SketchySVRG, and SketchyKatyusha.
We compare our methods to SGD and SAGA with their \textit{tuned} hyperpameters.
All optimization methods are run for $10$ epochs.
\fi


\ifpreprint
The comparison to tuned versions of SGD and SAGA is presented in \cref{fig:higgs_tuned,fig:susy_tuned}.
On these problems, PROMISE methods (SketchySGD and SketchySAGA) perform well while the competitors (SGD and SAGA) struggle to make any progress.
The \nyssn{} preconditioner outperforms the \ssn{} preconditioner on these large, dense problems: it achieves similar test loss at each iteration but is faster on wall-clock time.
\else
The comparison to tuned versions of SGD and SAGA is presented in \cref{fig:higgs_tuned}.
On this problem, PROMISE methods (SketchySGD and SketchySAGA) perform well while the competitors (SGD and SAGA) struggle to make any progress.
The \nyssn{} preconditioner outperforms the \ssn{} preconditioner on this large, dense problem: it achieves similar test loss at each iteration but is faster on wall-clock time.
\fi
We only plot test loss, as computing the training loss suffers from the same computational issues as computing a full gradient.
The plots with respect to wall-clock time only show the time taken in optimization; they do not include the time taken in repeatedly applying the random features transformation.




\begin{figure}[h]
    \centering
    \includegraphics[scale=0.4]{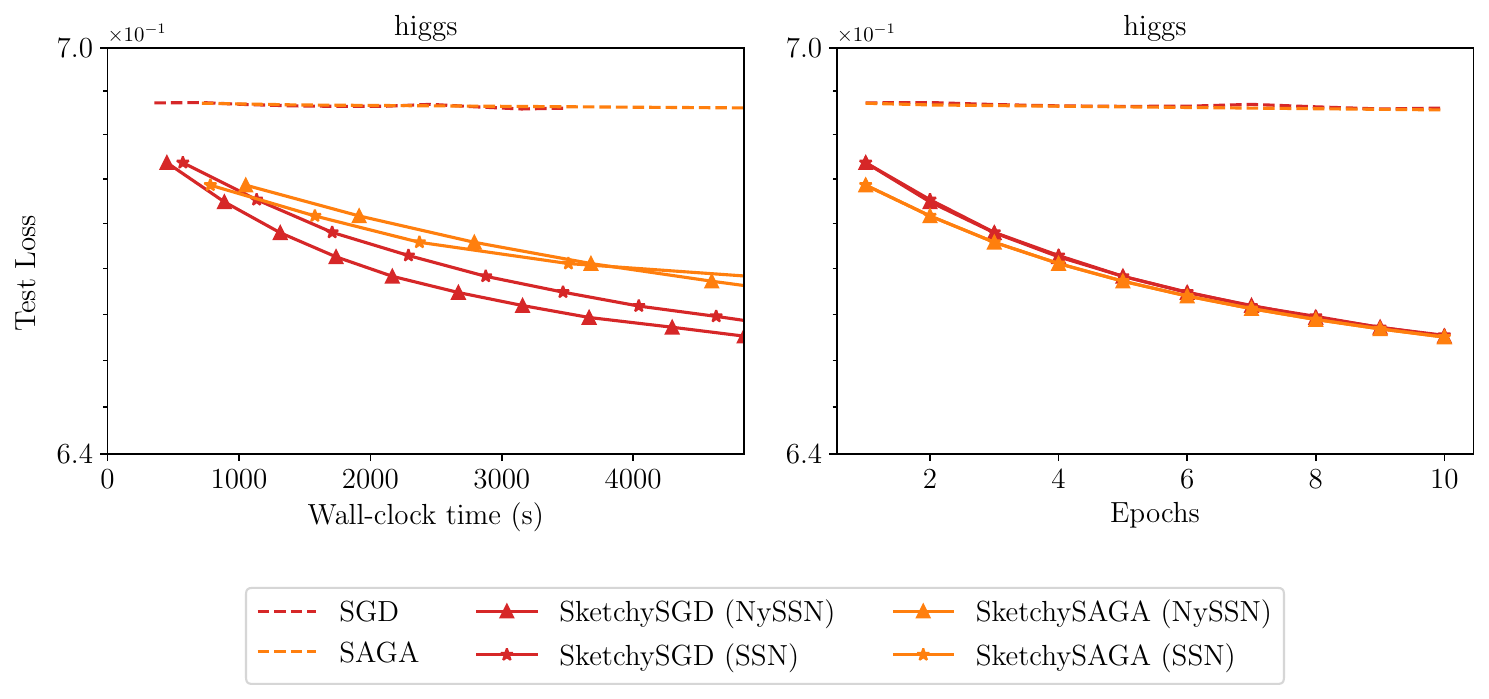}
    \caption{PROMISE methods outperform (tuned) SGD and SAGA on HIGGS.}
    \label{fig:higgs_tuned}
\end{figure}

\ifpreprint
\begin{figure}[h]
    \centering
    \includegraphics[scale=0.4]{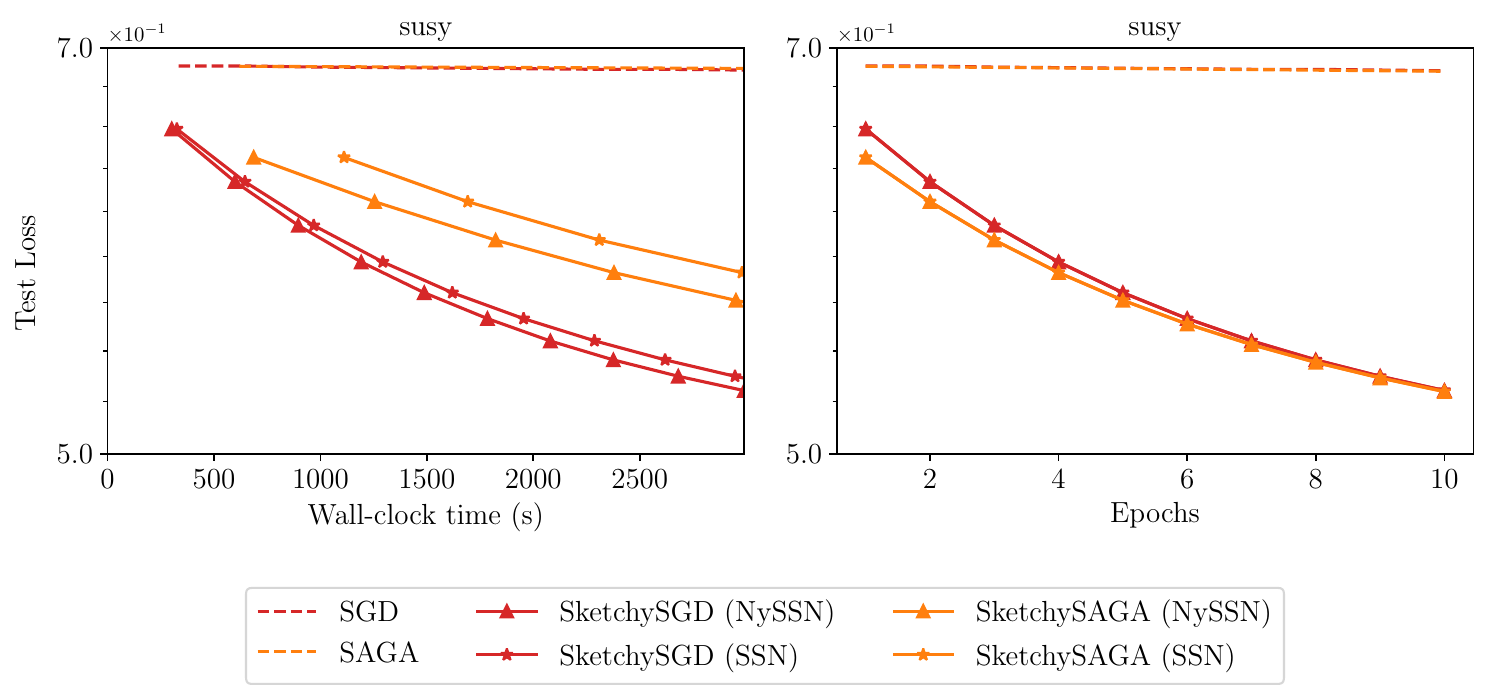}
    \caption{Comparison between SGD and SAGA with tuned learning rates, SketchySGD, and SketchySAGA on SUSY.}
    \label{fig:susy_tuned}
\end{figure}
\else
\fi

\ifpreprint
\subsection{Sensitivity study}
\label{subsection:sensitivity}
We investigate the sensitivity of SketchySAGA, with the \nyssn{} preconditioner, to the rank hyperparameter $r$ (\cref{subsubsection:sensitivity_r}) and update frequency hyperparameter $u$ (\cref{subsubsection:sensitivity_u}). 
In the first set of sensitivity experiments, we select ranks $r \in \{1, 2,5, 10, 20, 50\}$ while holding the update frequency fixed at $u = \left\lceil \frac{n_{\textrm{tr}}}{b_g} \right\rceil$ (1 epoch)\footnote{If we set $u = \infty$ in ridge regression, which fixes the preconditioner throughout the run of SketchySAGA, the potential gain from a larger rank $r$ may not be realized due to a poor initial Hessian approximation.}.
In the second set of sensitivity experiments, we select update frequencies $u \in \left\{0.5\left\lceil \frac{n_{\textrm{tr}}}{b_g} \right\rceil, 
\left\lceil \frac{n_{\textrm{tr}}}{b_g} \right\rceil, 
2\left\lceil \frac{n_{\textrm{tr}}}{b_g} \right\rceil, 
5\left\lceil \frac{n_{\textrm{tr}}}{b_g} \right\rceil,
\infty \right\}$, while holding the rank fixed at $r = 10$. 
We use the E2006-tfidf, YearPredictionMSD, yolanda, ijcnn1, real-sim, and SUSY datasets, with the same preprocessing as in \cref{subsection:performance_exp_appdx}. 
Each curve is the suboptimality (measured with respect to the best attained value) of the median training loss corresponding to a given $(r, u)$ pair across several random seeds ($10$ for E2006-tfidf, ijcnn1, and real-sim, $5$ for YearPredictionMSD and yolanda, and $3$ for SUSY).
We perform $80$ full data passes, which is equivalent to $80$ epochs of SketchySGD/SketchySAGA or $40$ epochs of SketchySVRG/SketchyKatyusha.


\subsubsection{Effects of changing the rank}
\label{subsubsection:sensitivity_r}
Looking at \cref{fig:sensitivity_r}, we see two distinct patterns: either increasing the rank has no noticeable impact on performance (E2006-tfidf, real-sim), or increasing the rank leads to faster convergence (YearPredictionMSD). 
We empirically observe that these patterns are related to the spectrum of each dataset (\cref{fig:spectrums}). 
For example, the spectrum of E2006-tfidf is highly concentrated in the first singular value and decays rapidly, and increasing the rank does not improve convergence.
On the other hand, the spectrum of YearPredictionMSD is not as concentrated in the first singular value, but still decays rapidly, so convergence steadily improves as we increase the rank from $r = 1$ to $r = 50$.
The spectrum of real-sim decays quite slowly in comparison to E2006-tfidf or YearPredictionMSD, 
so increasing the rank up to $50$ does not capture enough of the spectrum to improve convergence.
Rank sensitivity plots for yolanda, ijcnn1, and SUSY appear in \cref{subsubsection:sensitivity_exp_plots_appdx}.

\begin{figure}[h]
    \centering
    \includegraphics[scale=0.5]{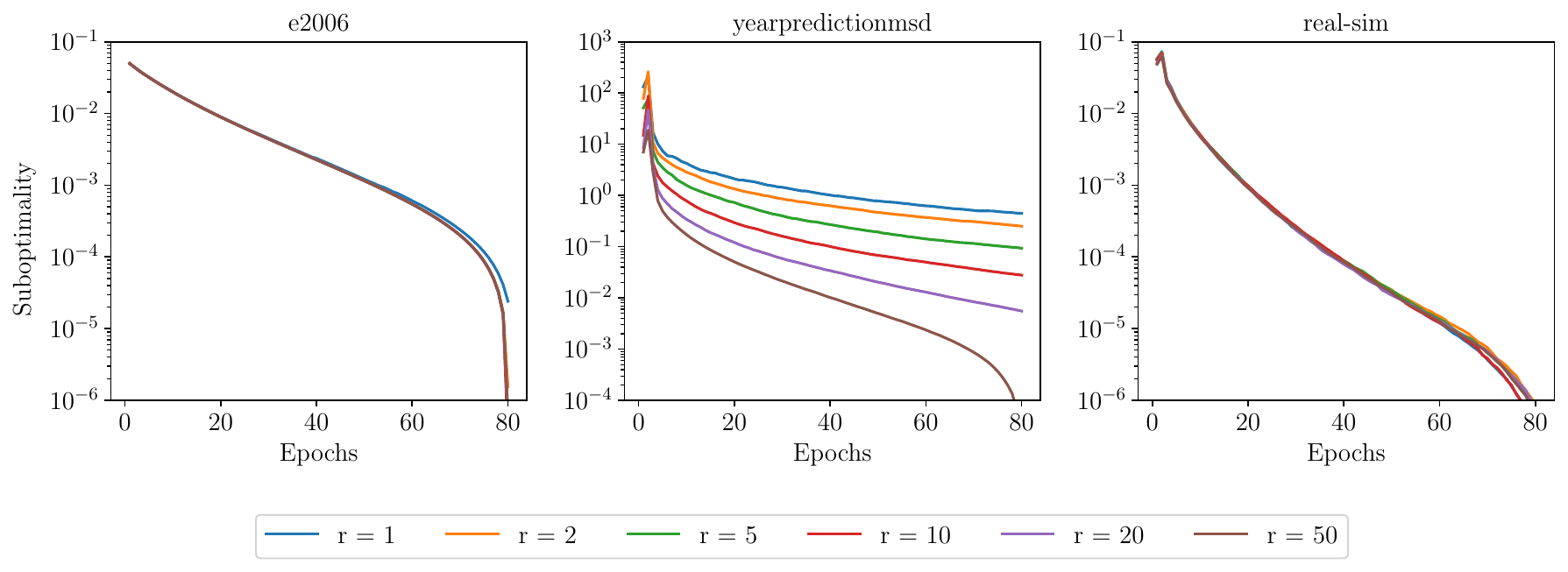}
    \caption{Sensitivity of SketchySAGA to rank $r$.}
    \label{fig:sensitivity_r}
\end{figure}

\begin{figure}[h]
    \centering
    \includegraphics[scale=0.5]{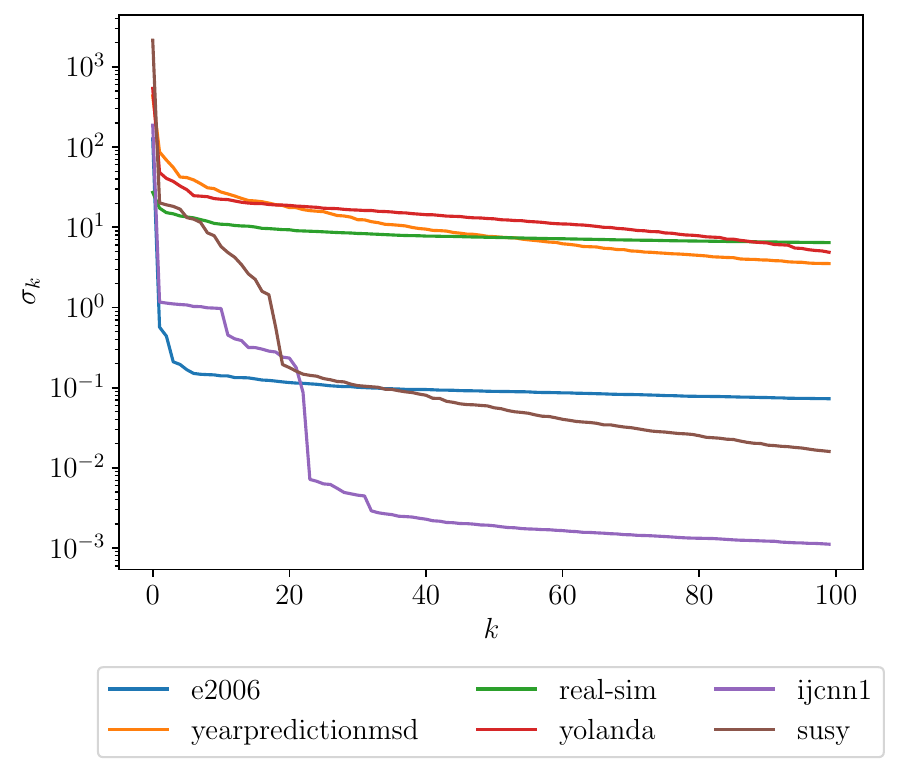}
    \caption{Top $100$ singular values of datasets after preprocessing.}
    \label{fig:spectrums}
\end{figure}

\subsubsection{Effects of changing the update frequency}
\label{subsubsection:sensitivity_u}
In this section, we display results only for logistic regression (\cref{fig:sensitivity_u}), since there is no benefit to updating the preconditioner for a quadratic problem such as ridge regression (\cref{subsubsection:sensitivity_exp_plots_appdx}): the Hessian in ridge regression is constant for all $w \in \R^p$. The impact of the update frequency depends on the spectrum of each dataset.
The spectra of ijcnn1 and susy are highly concentrated in the top $r = 10$ singular values and decay rapidly (\cref{fig:spectrums}).
Interestingly, it seems that for such problems the initial preconditioner approximates the curvature of the loss well throughout the optimization trajectory.
On the other hand, the spectrum of real-sim decays quite slowly, so the initial preconditioner does not capture most of the curvature information in the Hessian.
For real-sim, we find it is beneficial to update the preconditioner. These updates can be very infrequent: an update frequency of $5$ epochs yields nearly identical performance to updating every $0.5$ epochs.

\begin{figure}[h]
    \centering
    \includegraphics[scale=0.5]{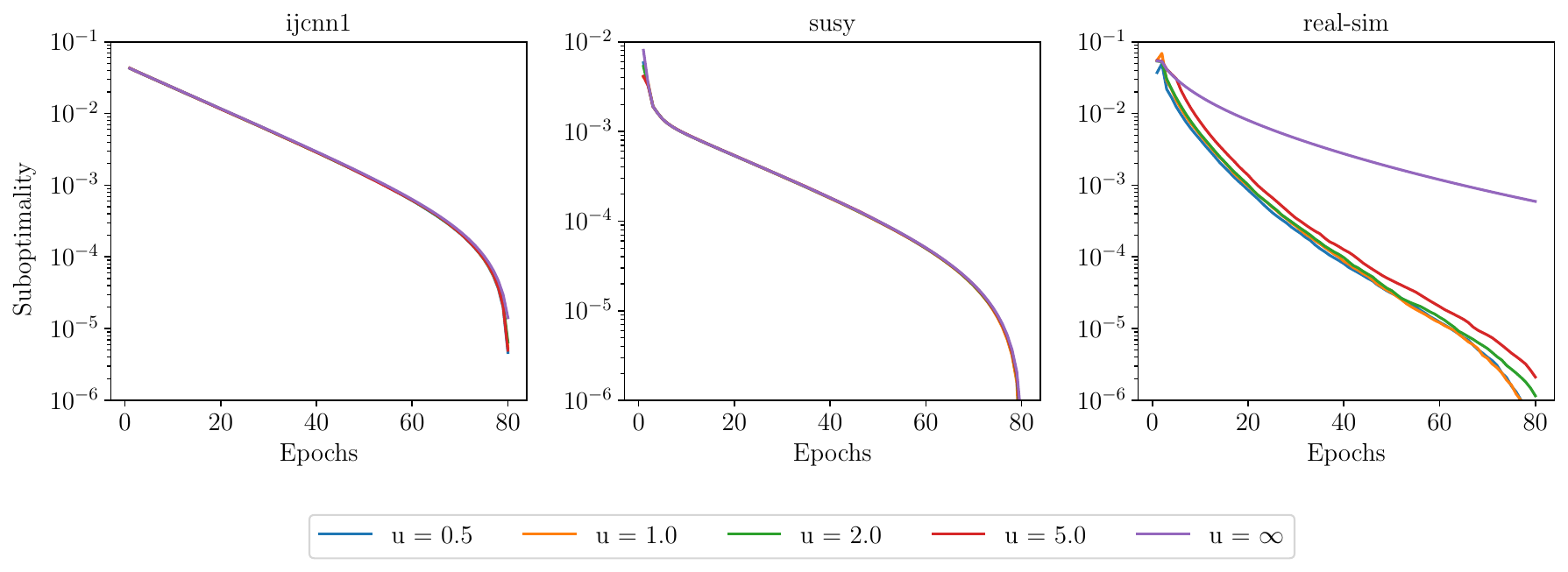}
    \caption{Sensitivity of SketchySAGA to update frequency $u$.}
    \label{fig:sensitivity_u}
\end{figure}
\else
\fi

\subsection{Regularity study: Why do PROMISE methods converge fast globally?}
\label{subsection:regularity}
\label{section:fast_global_convergence}
The theory in \cref{section:theory} shows that the PROMISE algorithms converge linearly, 
with the rate of convergence being controlled by the weighted quadratic regularity ratio $\qbar$.
It is encouraging that the convergence rate is no longer controlled by the condition number, 
yet the quantity $\qbar$ is not local, and so could be quite large.
Indeed, \cref{proposition:QuadRegConds} implies in the worst case it may be as large as $\kappa^2$, which is at odds with the empirical results in \cref{subsection:performance,subsection:subopt,subsection:showcase,subsection:large_scale}.
Here we provide an empirical argument that partly explains why PROMISE methods exhibit faster convergence than stochastic first-order methods.
The key observation in this regard is that the optimization trajectory does not arbitrarily traverse $\R^p$---it stays in localized regions.
Thus, the speed of convergence is determined by the \emph{local} values of $\qbar$, and not its global value over all of $\R^p$.
As the value of $\qbar$ over a localized region may be better behaved than over the whole space, PROMISE methods can take larger step-sizes, which leads to faster convergence.
Furthermore, as the iterates approach the optimum, the values of $\qbar$ over these localized regions approach $1$.
Hence we expect the local weighted quadratic regularity ratio to be small. 

In this section, we provide empirical evidence for the hypothesis of the preceding paragraph by studying a local version of the quadratic regularity ratio $\q$ along the optimization trajectory.
The reason for studying $\q$ instead of its weighted version $\qbar = \mathcal{L}_P / \gamma_\ell$ is because $\mathcal{L}_P$ requires knowledge of the Hessian similarity $\tau_\star^\nu$ and the individual values of $\gamma_{u_i}$, which are too expensive to compute.
Furthermore, $\q \approx \qbar$ for reasonably large gradient batchsizes, which makes $\q$ a reasonable proxy for $\qbar$.
To this end, we start by defining appropriate local versions of the quadratic regularity constants.
Close inspection of our analysis reveals that we only need quadratic regularity with $w_0$ equal to the iterate where we compute the preconditioner, $w_1$ set to be the current iterate, and $w_2$ set equal to the optimum.
Hence appropriate definitions for the local quadratic regularity constants $\gamma_{u,j}, \gamma_{\ell,j}$ and local quadratic regularity ratio $\q_j$ are given by
\begin{equation}
\label{eq:local_upper_qr}
    \gamma_{u,j} \coloneqq \max_{w \in S_j}\int^1_0 2(1 - t) \frac{\|w_\star-w\|_{\nabla^2F(w+t(w_\star-w))}^2}{\|w_\star-w\|^2_{\nabla^2 F(w_j)}}dt,
\end{equation} 
\begin{equation}
\label{eq:local_lower_qr}
  \gamma_{\ell,j} \coloneqq \min_{w \in S_j} \int^1_0 2(1 - t) \frac{\|w_\star-w\|_{\nabla^2F(w+t(w_\star-w))}^2}{\|w_\star-w\|^2_{\nabla^2 F(w_j)}}dt,
\end{equation}
\begin{equation}
    \q_j \coloneqq \frac{\gamma_{u,j}}{\gamma_{\ell,j}},
\end{equation}
where $w_j$ is the iterate where we compute the preconditioner, $S_j$ is the set of iterates associated with the preconditioner $P_j$ (i.e., iterates on the trajectory between $w_j$, inclusive and $w_{j+1}$, exclusive), and $w_\star$ is the optimum.


\cref{fig:qr_ratio} shows $\q_j$ for eight datasets over $50$ epochs of training using SketchySAGA with the \nyssn{} preconditioner. 
For all of these datasets, $\q_j \approx 1$ after $20$ epochs of training.
Furthermore, this phenomenon occurs before SketchySAGA converges close to the optimum; \cref{table:epochs_to_solve} demonstrates that $\q_j \approx 1$ well before the problem has been solved (within $10^{-4}$ of $F(w^\star)$ as in \cref{subsection:performance}). 
For example, ijcnn1 takes $94$ epochs to be solved by SketchySAGA with the \nyssn{} preconditioner, but $\q_j \approx 1$ in less than $5$ epochs.
\ifpreprint
See \cref{subsection:regularity_exp_appdx} for additional details.
\else
\fi

\begin{figure}[h]
    \centering
    \includegraphics[scale=0.4]{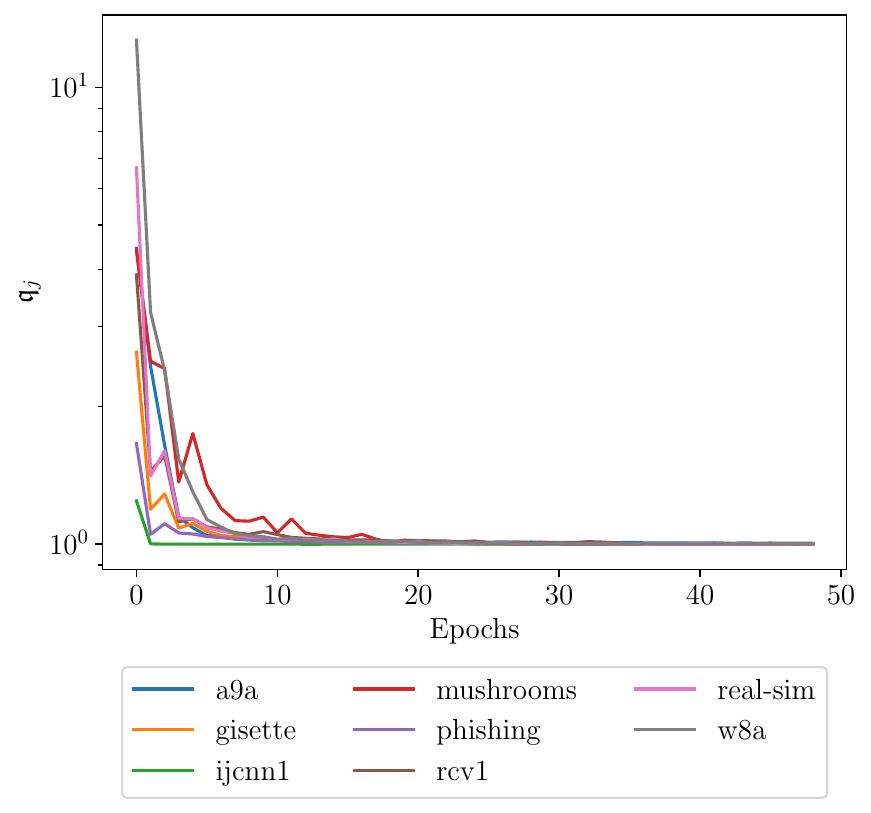}
    \caption{Plots of $\q_j$ over the optimization trajectory for selected datasets.}
    \label{fig:qr_ratio}
\end{figure}

\ifpreprint
\begin{table}[H]
\centering
    \begin{tabular}{C{2.5cm}C{2.5cm}}
        Dataset & \# of epochs \\ \hline
        a9a & 15 \\ \hline
        gisette & 86 \\ \hline
        ijcnn1 & 94 \\ \hline
        mushrooms & 20 \\ \hline
        phishing & 71 \\ \hline
        rcv1 & 49 \\ \hline
        real-sim & 39 \\ \hline
        w8a & 28 \\ \hline
    \end{tabular}
    \caption{
    Number of epochs to solve logistic regression problems on selected datsets.
    }
    \label{table:epochs_to_solve}
\end{table}
\else
\begin{table}[H]
    \centering
    \begin{tabular}{c|c|c|c|c|c|c|c|c}
        Dataset & a9a & gisette & ijcnn1 & mushrooms & phishing & rcv1 & real-sim & w8a \\ \hline
        \# of epochs & 15 & 86 & 94 & 20 & 71 & 49 & 39 & 28
    \end{tabular}
    \caption{Number of epochs to solve logistic regression problems on selected datsets.}
    \label{table:epochs_to_solve}
\end{table}
\fi

\section{Conclusion}
We introduce PROMISE, a framework for combining scalable preconditioning techniques with popular stochastic optimization methods.
In particular, we present a variety of preconditioning techniques (\ssn{}, \nyssn{}, \sassnc{}, \sassnr{}, \diagssn{}) and develop the preconditioned stochastic second-order methods SketchySVRG, SketchySAGA, and SketchyKatyusha.
Furthermore, we provide default hyperparameters for these preconditioners and algorithms, which enable them to work out-of-the-box, even on highly ill-conditioned data. 

To analyze the PROMISE methods, we introduce quadratic regularity and the quadratic regularity ratio, which generalize the notions of smoothness, strong convexity, and condition number to the Hessian norm.
We also introduce Hessian dissimilarity, which allows us to give practical requirements on the gradient batchsize, a first in the literature.
We show that PROMISE methods have global linear convergence, and that this convergence is condition-number free for ridge regression.
Moreover, we show that SketchySVRG converges linearly at a rate independent of the condition number, once the iterates are close enough to the optimum. 
Hence, SketchySVRG enjoys the fast local convergence one would expect of a Newton-type method. 
\ifpreprint
Last, we have discussed a theoretical hybrid algorithm that combines Katyusha with SketchySVRG, yielding a stochastic second-order method that possesses the best stochastic gradient query complexity, and which resolves an open question in the literature.  
Developing a practical stochastic second-order method that matches the performance of the hybrid algorithm provides an interesting direction for future work. 
\else
\fi

We empirically demonstrate the superiority of PROMISE methods over popular competitor methods for ridge and logistic regression.
PROMISE methods, with their default hyperparameters, consistently outperform the competition, even when they have been tuned to achieve their best performance.

\section{Acknowledgements}
We would like to thank Catherine Chen, John Duchi, Daniel LeJeune, Michael Mahoney, Mert Pilanci, Aaron Sidford, and Ali Teshnizi for helpful discussions regarding this work.
We would also like to thank GPT-4 for inspiring the title of this work.



\newpage

\appendix


\ifpreprint

\section{Additional preconditioner details}
\label{section:methodology_appdx}
We present additional information related to the \sassn{} and \diagssn{} preconditioners that was omitted in the main text.

\subsection{\sassn{} preconditioner}
\label{subsection:methodology_sassn_appdx}

An implementation of the \sassnc{}/\sassnr{} preconditioners is provided in the $\psassnc$/$\psassnr$ classes (\cref{tab:sassn_attr,alg:sassnc_gen,alg:sassnr_gen,alg:sassn_upd,alg:sassn_dir}). 
The attributes of the $\psassnc$/$\psassnr$ classes are given in \cref{tab:sassn_attr}, and pseudocode for the \texttt{update} and \texttt{direction} methods is provided in \cref{alg:sassn_upd,alg:sassn_dir}, respectively. 
$\psassnc$ and $\psassnr$ have their own \texttt{generate\_embedding} methods (\cref{alg:sassnc_gen,alg:sassnr_gen}) for generating a sparse embedding.

\begin{table}[H]
    \centering
    \begin{tabular}{c c}
        Attribute & Description \\ \hline
        $r$ & Rank for constructing preconditioner \\ \hline
        $\rho$ & Regularization for preconditioner \\ \hline
        $Y$ & Sketch of subsampled Hessian \\ \hline
        $L$ & Lower-triangular Cholesky factor for storing preconditioner \\ \hline
        $\lambda_\mathcal{P}$ & Estimate of preconditioned smoothness constant
    \end{tabular}
    \caption{Attributes of the $\psassnc$/$\psassnr$ classes.}
    \label{tab:sassn_attr}
\end{table}



        

The \texttt{generate\_embedding} method for $\psassnc$/$\psassnr$ takes a dimension $b$ as input and generates a column-sparse/row-sparse embedding that is used to construct the \sassnc{}/\sassnr{} preconditioner.

\begin{algorithm}[H]
    \centering
    \caption{Generate column-sparse embedding for $\psassnc$ preconditioner}
    \label{alg:sassnc_gen}
    \begin{algorithmic}
        \Require{$\psassnc$ object with attributes $r, \rho, Y, L, \lambda_\mathcal{P}$}
        \Function{$\psassnc$.\texttt{generate\_embedding}}{$b$}
        \State $r \gets \psassnc.r$ \Comment{Get attributes}
        \State $\zeta \gets \min(r, 8)$ \Comment{Sparsity parameter}
        \State $\textrm{cols} \gets [0, 1, \ldots, b - 1]^T \otimes 1_{\zeta}$ \Comment{Column indices of nonzero entries}
        \State $\textrm{rows} \gets \texttt{rand\_choice} (\{0, 1, \ldots, r - 1\}, \zeta b)$ \Comment{Row indices of nonzero entries}
        \State $\Phi \gets \texttt{sign}(\texttt{unif\_sample}(\zeta b) - 0.5)$ \Comment{Random sign matrix}
        \State $\Omega \gets \texttt{csr\_matrix} \left( \sqrt{\frac{1}{\zeta}}\Phi, \textrm{idx} = (\textrm{rows}, \textrm{cols}), \textrm{shape} = (r, b) \right)$
        \State \Return $\Omega$
        \EndFunction
    \end{algorithmic}
\end{algorithm}

\begin{algorithm}[H]
    \centering
    \caption{Generate row-sparse embedding for $\psassnr$ preconditioner}
    \label{alg:sassnr_gen}
    \begin{algorithmic}
        \Require{$\psassnr$ object with attributes $r, \rho, Y, L, \lambda_\mathcal{P}$}
        \Function{$\psassnr$.\texttt{generate\_embedding}}{$b$}
        \State $r \gets \psassnr.r$ \Comment{Get attributes}
        \State $\zeta \gets \min(b, 8)$ \Comment{Sparsity parameter}
        \State $\textrm{cols} \gets \texttt{rand\_choice} (\{0, 1, \ldots, b - 1\}, \zeta r)$ \Comment{Column indices of nonzero entries}
        \State $\textrm{rows} \gets [0, 1, \ldots, r - 1]^T \otimes 1_{\zeta}$ \Comment{Row indices of nonzero entries}
        \State $\Phi \gets \texttt{sign}(\texttt{unif\_sample}(\zeta r) - 0.5)$ \Comment{Random sign matrix}
        \State $\Omega \gets \texttt{csr\_matrix} \left( \sqrt{\frac{b}{\zeta r}}\Phi, \textrm{idx} = (\textrm{rows}, \textrm{cols}), \textrm{shape} = (r, b) \right)$ 
        \State \Return $\Omega$
        \EndFunction
    \end{algorithmic}
\end{algorithm}

The implementation of this method for $\psassnc$ and $\psassnr$ is based on \cite{martinsson2020randomized} and \cite{derezinski2021newtonless}, respectively. We limit the sparsity parameter $\zeta \leq 8$ based on recommendations from \cite{tropp2019streaming}.

The \texttt{update} method takes a GLM $\Model$, Hessian batches $\mathcal{B}_1, \mathcal{B}_2$, and vector $w \in \R^p$ as input. 

\begin{algorithm}[H]
    \centering
    \caption{Update $\psassnc$/$\psassnr$ preconditioner and preconditioned smoothness constant}
    \label{alg:sassn_upd}
    \begin{algorithmic}
        \Require{$\psassnc$/$\psassnr$ object with attributes $r, \rho, Y, L, \lambda_\mathcal{P}$}
        \Function{$\psassnc$/$\psassnr$.\texttt{update}}{$\Model, \mathcal{B}_1, \mathcal{B}_2, w$}
        \State $\rho \gets \psassnc/\psassnr.\rho$ \Comment{Get attributes}
        \State
        \State \# Phase 1: Update preconditioner
        \State $A_{\textrm{sub}} \gets \Model.\texttt{get\_data}(\mathcal{B}_1)$
        \State $d_{\textrm{sub}} \gets \Model.\texttt{get\_hessian\_diagonal}(\mathcal{B}_1, w)$
        \State $X \gets \diag(\sqrt{d_{\textrm{sub}}}) A_{\textrm{sub}}$ \Comment{Square root of subsampled Hessian}
        \State $\Omega \gets \psassnc/\psassnr.\texttt{generate\_embedding}(|\mathcal{B}_1|)$ \Comment{Sparse embedding}
        \State $Y \gets \Omega X$ \Comment{Sketched square root of subsampled Hessian}
        \State $L \gets \texttt{cholesky}(Y Y^T + \rho I)$

        \State
        \State \# Phase 2: Update estimated preconditioned smoothness constant
        \State $A_{\textrm{sub}} \gets \Model.\texttt{get\_data}(\mathcal{B}_2)$
        \State $d_{\textrm{sub}} \gets \Model.\texttt{get\_hessian\_diagonal}(\mathcal{B}_2, w)$
        \State $Z \gets A_{\textrm{sub}}^T \diag(d_{\textrm{sub}}) A_{\textrm{sub}} + \Model.\texttt{get\_reg}() I$ \Comment{Subsampled Hessian}
        \State $\lambda_\mathcal{P} \gets \texttt{eig}(Z (Y^T Y + \rho I)^{-1}, k = 1)$ \Comment{Compute largest eigenvalue}
        \State
        \State $\psassnc/\psassnr.Y \gets Y, \psassnc/\psassnr.L \gets L, \psassnc/\psassnr.\lambda_\mathcal{P} \gets \lambda_\mathcal{P}$ \Comment{Set attributes}
        \EndFunction
    \end{algorithmic}
\end{algorithm}

In the first phase, this method constructs the \sassnc{}/\sassnr{} preconditioner $P$ at $w$ by computing the square root of the subsampled Hessian, $X$, applying a sparse embedding to $X$, and computing a Cholesky factorization. 
The factors $Y$ and $L$ can be used later to apply the preconditioner to a vector.
In the second phase, this method estimates the preconditioned smoothness constant by taking a similar approach to the $\pssn$ class.

The \texttt{direction} method takes a vector $g \in \R^p$ (typically a stochastic gradient) as input. 

\begin{algorithm}[H]
    \centering
    \caption{Compute preconditioned $\psassnc$/$\psassnr$ direction}
    \label{alg:sassn_dir}
    \begin{algorithmic}
        \Require{$\psassnc$/$\psassnr$ object with attributes $r, \rho, Y, L, \lambda_\mathcal{P}$}
        \Function{$\psassnc$/$\psassnr$.\texttt{direction}}{$g$}
        \State $Y \gets \psassnc/\psassnr.Y, L \gets \psassnc/\psassnr.L, \rho \gets \psassnc/\psassnr.\rho$ \Comment{Get attributes}
        \State $v \gets Yg$
        \State $v \gets L^{-1} v$ \Comment{Triangular solve}
        \State $v \gets L^{-T} v$ \Comment{Triangular solve}
        \State $v \gets Y^T v$
        \State \Return $(g - v) / \rho$
        \EndFunction
    \end{algorithmic}
\end{algorithm}

The method then computes $P^{-1}g$ using the Woodbury formula with the preconditioner factors $Y$ and $L$; this computation has complexity $\mathcal{O}(rp)$ in general.


\subsection{\diagssn{} preconditioner}
\label{subsection:methodology_diagssn_appdx}

An implementation of the \diagssn{} preconditioner is provided in the $\pdiagssn$ class (\cref{tab:diagssn_attr,alg:diagssn_upd,alg:diagssn_dir}). 
The attributes of the $\pdiagssn$ class are given in \cref{tab:diagssn_attr}, and pseudocode for the \texttt{update} and \texttt{direction} methods is provided in \cref{alg:diagssn_upd,alg:diagssn_dir}, respectively.

\begin{table}[H]
    \centering
    \begin{tabular}{c c}
        Attribute & Description \\ \hline
        $\rho$ & Regularization for preconditioner \\ \hline
        $d$ & Vector for storing preconditioner \\ \hline
        $\lambda_\mathcal{P}$ & Estimate of preconditioned smoothness constant
    \end{tabular}
    \caption{Attributes of the $\pdiagssn$ class.}
    \label{tab:diagssn_attr}
\end{table}

        
        
        

The \texttt{update} method takes a GLM $\Model$, Hessian batches $\mathcal{B}_1, \mathcal{B}_2$, and vector $w \in \R^p$ as input. 

\begin{algorithm}[H]
    \centering
    \caption{Update $\pdiagssn$ preconditioner and preconditioned smoothness constant}
    \label{alg:diagssn_upd}
    \begin{algorithmic}
        \Require{$\pdiagssn$ object with attributes $\rho, d, \lambda_{\mathcal{P}}$}
        \Function{$\pdiagssn$.\texttt{update}}{$\Model, \mathcal{B}_1, \mathcal{B}_2, w$}
        \State $\rho \gets \pdiagssn.\rho$ \Comment{Get attributes}
        \State
        \State \# Phase 1: Update preconditioner
        \State $A_{\textrm{sub}} \gets \Model.\texttt{get\_data}(\mathcal{B}_1)$
        \State $d_{\textrm{sub}} \gets \Model.\texttt{get\_hessian\_diagonal}(\mathcal{B}_1, w)$
        \State $X \gets \diag(\sqrt{d_{\textrm{sub}}}) A_{\textrm{sub}}$ \Comment{Square root of subsampled Hessian}
        \State $d \gets \texttt{column\_norm}(X)^2$

        \State
        \State \# Phase 2: Update estimated preconditioned smoothness constant
        \State $A_{\textrm{sub}} \gets \Model.\texttt{get\_data}(\mathcal{B}_2)$
        \State $d_{\textrm{sub}} \gets \Model.\texttt{get\_hessian\_diagonal}(\mathcal{B}_2, w)$
        \State $Z \gets A_{\textrm{sub}}^T \diag(d_{\textrm{sub}}) A_{\textrm{sub}} + \Model.\texttt{get\_reg}() I$ \Comment{Subsampled Hessian}
        \State $\lambda_\mathcal{P} \gets \texttt{eig}(Z \diag(d + \rho)^{-1}, k = 1)$ \Comment{Compute largest eigenvalue}
        \State
        \State $\pdiagssn.d \gets d, \pdiagssn.\lambda_\mathcal{P} \gets \lambda_\mathcal{P}$ \Comment{Set attributes}
        \EndFunction
    \end{algorithmic}
\end{algorithm}

In the first phase, this method constructs the \diagssn{} preconditioner $P$ at $w$ by computing the square root of the subsampled Hessian, $X$, and then using $X$ to compute the diagonal of the subsampled Hessian, $d$, which is stored as a vector.
In the second phase, this method estimates the preconditioned smoothness constant by taking a similar approach to the $\pssn$ class.

The \texttt{direction} method takes a vector $g \in \R^p$ (typically a stochastic gradient) as input. 

\begin{algorithm}[H]
    \centering
    \caption{Compute preconditioned $\pdiagssn$ direction}
    \label{alg:diagssn_dir}
    \begin{algorithmic}
        \Require{$\pdiagssn$ object with attributes $\rho, d, \lambda_{\mathcal{P}}$}
        \Function{$\pdiagssn$.\texttt{direction}}{$g$}
        \State $d \gets \pdiagssn.d, \rho \gets \pdiagssn.\rho$ \Comment{Get attributes}
        \State \Return $g / (d + \rho) $ \Comment{Elementwise division}
        \EndFunction
    \end{algorithmic}
\end{algorithm}

The method then computes $P^{-1}g$ by elementwise division with $d$; this computation has complexity $\mathcal{O}(p)$.

\ifpreprint
\section{General lemmas}
Here we collect some general lemmas that will be used in our analysis.
The first three lemmas follow from elementary computation. 
\begin{lemma}[A-norm variance decomposition]
\label[lemma]{lem:AVarBnd}
Let $A$ be a symmetric positive definite matrix, and $X\in \R^p$ a random vector. Then
\[
\E\|X-\E X\|_A^2 = \E\|X\|_A^2 -\|\E X\|^2_A,
\]
and so 
\[
\E\|X-\E X\|_A^2\leq \E\|X\|_A^2.
\]
\end{lemma}

\begin{lemma}[$A$-norm parallelogram law inequality]
\label[lemma]{lem:A-norm-par}
    Let $A$ be a symmetric positive definite matrix, then for any vectors $a,b$ we have
    \[
    \|a+b\|_A^2 \leq 2\left(\|a\|_A^2+\|b\|_A^2\right).
    \]
\end{lemma}

\begin{lemma}[$A$-norm polarization identity]
\label[lemma]{lem:A-norm-pol}
Let $A$ be a symmetric positive definite matrix. Then for any vectors $a,b,c,d$, the following equality holds
\[
\langle a-b,c-d\rangle_A = \frac{1}{2}\left(\|a-d\|_A^2-\|a-c\|^2_A\right)+\frac{1}{2}\left(\|c-b\|_A^2-\|d-b\|^2_A\right).
\]
\end{lemma}

\begin{lemma}[Generalized Young Inequality]
\label[lemma]{lemma:GenYoungIneq}
    Let $A$ be a symmetric positive definite matrix. Then for any $a,b \in \R^p$ and $c_1,c_2>0$
    \[
    c_1 \langle a,b \rangle - \frac{c_2}{2}\|b\|_A^2 \leq \frac{c^2_1}{2c_2}\|a\|_{A^{-1}}^2.  
    \]
\end{lemma}

\begin{proof}
    Set $a' = \frac{c_1}{\sqrt{c_2}}A^{-1}a, b' = \sqrt{c_2}b$. Then
    \[ c_1 \langle a,b \rangle - \frac{c_2}{2}\|b\|_A^2 = \langle a',b'\rangle_A+\frac{1}{2}\|b'\|_A^2.\]
    Now, using $\|a'+b'\|_A^2 \leq 2(\|a'\|_{A}^2+\|b'\|_A^2$), expanding and rearranging, we reach
    \[
    \langle a',b'\rangle_A\leq \frac{1}{2}\left(\|a'\|_A^2+\|b'\|_A^2\right).
    \]
    Hence
    \[
    \langle a',b'\rangle_A-\frac{1}{2}\|b'\|_A^2\leq \frac{\|a'\|^2_A}{2}.
    \]
    Now, substituting in the values for $a'$, $b'$, and using $\|A^{-1}a\|_{A}^2 = \|a\|_{A^{-1}}^2$, completes the proof. 
\end{proof}
\

\begin{lemma}[Smoothness with respect to the $A$-norm]
\label[lemma]{lem:A-norm-smooth}
    Let $A\in \mathbb{S}_p^{++}(\R)$ , and $F(x)$ be a real-valued convex function satisfying:
    \[
    F(x)\leq F(y)+\langle \nabla F(y),x-y\rangle+\frac{L}{2}\|x-y\|_A^2 \quad \forall x,y\in \R^p. 
    \]
    Then
    \[
    F(x) \geq F(y)+\langle \nabla F(y),x-y\rangle+\frac{1}{2L}\|\nabla F(x)-\nabla F(y)\|_{A^{-1}}^2 \quad \forall x,y\in \R^p.
    \]
\end{lemma}
\begin{proof}
    It is a standard fact of convex optimization (see Theorem 5.8 in \cite{beck2017firstorder} or Theorem 2.1.5 in \cite{nesterov2018lectures}) that $F$ satisfies the inequality
    \[
    F(x)\leq F(y)+\langle \nabla F(y),x-y\rangle +\frac{L}{2}\|x-y\|^2 \quad \forall x,y \in \R^p,
    \]
    with respect to the norm $\|\cdot\|$, if and only if it satisfies the inequality
    \[
    F(x)\geq F(y)+\langle \nabla F(y),x-y\rangle+\frac{1}{2L}\|\nabla F(x)-\nabla F(y)\|_*^2 \quad \forall x,y \in \R^p,
    \]
    where $\|\cdot\|_*$ is the dual norm of $\|\cdot\|$.
    The claim follows by observing that the dual norm of $\|\cdot\|_A$ is $\|\cdot \|_{A^{-1}}$.
\end{proof}
\else
\fi

\section{Proofs of main results}
\label{sec:theory_pfs}

\ifpreprint
\subsection{Proof of \cref{lemma:EffDimBnd}}
\label{subsec:EffDimPf}
\begin{proof}
    The proof is by direct calculation. 
    Indeed, for the first part of the claim, observe that
    \begin{align*}
        &\deff(A^{T}\Phi''(Aw)A) = \sum_{j=1}^{n}\frac{\sigma_j^2(\Phi''(Aw)A)}{\sigma_j^2(\Phi''(Aw)A)+n\nu} \overset{(1)}{\leq} \sum^{n}_{j=1} \frac{C_1Bj^{-2\beta}}{C_1Bj^{-2\beta}+n\nu}\\
        & \overset{(2)}{=} \sum^{n}_{j=1} \frac{C}{C+(n\nu)j^{2\beta}} \leq \int_{0}^{\infty}\frac{C}{C+(n\nu)x^{2\beta}}dx \overset{(3)}{=} C(n\nu)^{-1/(2\beta)}\int_{0}^{\infty}\frac{1}{C+u^{2\beta}}du\\
        & \overset{(4)}{=} C(n\nu)^{-1/(2\beta)}\times C^{\frac{1}{2\beta}-1}\frac{\pi/(2\beta)}{\sin(\pi/(2\beta))} = \frac{\pi/(2\beta)}{\sin(\pi/(2\beta))}\left(\frac{C}{n\nu}\right)^{1/{2\beta}}. 
    \end{align*}
    Here $(1)$ uses $A^{T}\Phi''(Aw)A\preceq B$, our hypotheses on that $\sigma_j(A) = \bigO(j^{-2\beta})$, and that $\frac{x}{x+\nu}$ is increasing in $x$ for $x\geq 0$, $(2)$ sets $C = C_1B$ and multiplies the numerator and denominator by $j^{2\beta}$, $(3)$ uses the substitution $u = (n\nu)^{1/{2\beta}}$, and $(4)$ uses the fact that $\int_{0}^{\infty}\frac{1}{C+u^{2\beta}}du = C^{\frac{1}{2\beta}-1}\frac{\pi/(2\beta)}{\sin(\pi/(2\beta))}$ \citep{sutherland2017fixing}.

    The second claim now follows from the first by plugging in $\nu = \bigO(1/n)$.
\end{proof}
\else
\fi

\ifpreprint
\subsection{Proof of \cref{lem:deff_vs_kappa}}
\label{subsec:deff_vs_kappa_pf}
\begin{proof}
    By definition, $\chi^{\nu}(A^{T}\Phi''(Aw)A)\deff(A^{T}\Phi''(Aw)A)) = n\max_{i\in [n]}l_i^{\nu}(\Phi''(Aw)^{1/2}A)$.
    Now, for any $i$
    \begin{align*}
        &l_i^{\nu}(\Phi''(Aw)^{1/2}A) = \phi''(a_i^{T}w)\left[a_i^{T}(A^{T}\Phi''(Aw)A+n\nu I)^{-1}a_i^{T}\right] \leq \frac{\phi''(a_i^{T}w)\|a_i\|^2}{n\nu}\\
        & = \frac{\|\nabla^2 f_i(w)\|}{n\nu}\leq \frac{L_i}{n\nu}\leq \frac{\LMax}{n\nu} = \frac{\kappaMax}{n}.
    \end{align*}
    The preceding display immediately implies 
    \[
    \chi^{\nu}(A^{T}\Phi''(Aw)A)\deff(A^{T}\Phi''(Aw)A))\leq \kappaMax,
    \]
    which concludes the argument. 
\end{proof}

\subsection{Proofs of sampling and low-rank approximation bounds}
\label{subsec:samp_lr_pfs}
\subsubsection{Preliminaries}
The following result is an immediate consequence of Theorem 1 in \cite{cohen2016optimal} with $A = B = D^{1/2}_{\sigma}$.
\begin{proposition}
\label[proposition]{proposition:appx_mat_mu_sr}
     Let $G\in \R^{p\times p}$ with eigendecomposition $G = V\Lambda V^{T}$. Let $\sigma >0$ and define $D_\sigma = \Lambda(\Lambda+\sigma I)^{-1}$. Assume one of the following items holds:
     \begin{enumerate}
         \item $\frac{1}{\sqrt{r}}\Omega$, where $\Omega$ is a Gaussian random matrix with $r = \Upomega\left(\frac{d^{\sigma}_{\textup{eff}}(\widehat \nabla^2 f(w))+\log(\frac{1}{\delta})}{{\zeta_0^2}}\right)$ rows.
         \item $\Omega$ is an SRHT matrix with $ r =\Upomega\left(\frac{d^{\sigma}_{\textup{eff}}(\widehat \nabla^2 f(w))+\log(\frac{1}{\zeta_0 \delta})\log\left(\frac{d^{\sigma}_{\textup{eff}}(\widehat \nabla^2 f(w))}{\delta}\right)}{{\zeta_0^2}}\right)$ rows.
         \item $\Omega$ is an sparse sign embedding with sparsity $s = \Upomega\left(\frac{\log\left(\frac{d^{\sigma}_{\textup{eff}}(\widehat \nabla^2 f(w))}{\delta}\right)}{{\zeta_0}}\right)$ and $r =\Upomega\left(\frac{d^{\sigma}_{\textup{eff}}(\widehat \nabla^2 f(w))\log\left(\frac{d^{\sigma}_{\textup{eff}}(\widehat \nabla^2 f(w))}{\delta}\right)}{{\zeta_0^2}}\right)$ rows.
     \end{enumerate}
     Then
    \begin{equation}
        \|D_\sigma^{1/2}V^{T}\Omega^{T}\Omega VD^{1/2}_\sigma-D_\sigma\|\leq \zeta_0,
    \end{equation}
    with probability at least $1-\delta$.
\end{proposition}

We also recall the following lemma due to \cite{alaoui2015fast}.

\begin{lemma}[Lemma 1, \cite{alaoui2015fast}]
\label[lemma]{lem:reg_nys_err_bnd}
Let $H$ be a symmetric psd matrix, $\sigma>0$ and $\hat H = H\langle \Omega \rangle$ be a \nys{} approximation to $H$ with test matrix $\Omega$. 
Further, let the regularized \nys{} approximation $H \langle \Omega \rangle_\sigma = (H \Omega)(\Omega^T H \Omega + \sigma I)^{-1}(H \Omega)^T$.
Define $E = H-H\langle \Omega \rangle$ and $E_{\sigma} = H-H\langle \Omega \rangle_{\sigma}$. 
Suppose that $\|D^{1/2}_{\sigma}\Omega\Omega^{T}D^{1/2}_{\sigma}-D_{\sigma}\|\leq \widetilde \eta < 1$.
Then 
\begin{equation}
\|H-\hat H\| \leq \frac{\sigma}{1-\widetilde \eta}. 
\end{equation}
\end{lemma}

\subsubsection{Proof of \cref{proposition:SSNPrecond}}
\begin{proof}
    Suppose the subsampled Hessian $\widehat{\nabla}^2 f(w)$ is constructed with batch $\mathcal B \subseteq [n]$ such that $|\mathcal B| = b_H$. Since $f$ is a GLM, 
    \[\widehat{\nabla}^2 f(w) = \frac{1}{b_H} A_{\mathcal B}^T \Phi''(A_{\mathcal{B}} w) A_{\mathcal B},\]
    where $A$ is the associated data matrix.
    Now, as the Hessian batchsizes satisfies 
    \begin{align*}
    &b_H = \bigOmega\left(\frac{\chi^\rho(\nabla^2 f(w)) d^\rho_{\textup{eff}}(\nabla^2 f(w))\log\left(\frac{d^\rho_{\textup{eff}}(\nabla^2 f(w))}{\delta}\right)}{\zeta_0^2}\right) \\ 
    & = \bigOmega\left(\frac{\chi^\rho(\Phi''(Aw)^{1/2}A) d^\rho_{\textup{eff}}(\Phi''(Aw)^{1/2}A)\log\left(\frac{d^\rho_{\textup{eff}}(\Phi''(Aw)^{1/2}A)}{\delta}\right)}{\zeta_0^2}\right),
    \end{align*} 
    we may apply \cref{lemma:ridgesamplingappx} to conclude
    \[(1 - \zeta_0) \left( \frac{1}{b_H} A_{\mathcal B}^T \Phi''(A_{\mathcal{B}} w) A_{\mathcal B} + \rho I \right) \preceq \frac{1}{n} A^T \Phi''(Aw) A + \rho I \preceq (1 + \zeta_0) \left( \frac{1}{b_H} A_{\mathcal B}^T \Phi''(A_{\mathcal{B}} w) A_{\mathcal B} + \rho I \right) \]
    holds with probability at least $1 - \delta$.
    The previous display is equivalent to
    \[
    (1 - \zeta_0) (\widehat{\nabla}^2 f(w) + \rho I) \preceq \nabla^2 f(w) + \rho I \preceq (1 + \zeta_0) (\widehat{\nabla}^2 f(w) + \rho I).
    \]
    Since $\rho \geq \nu$, a routine calculation shows that
    \[
    \frac{\nu}{\rho} (1 - \zeta_0) (\widehat{\nabla}^2 f(w) + \rho I) \preceq \nabla^2 f(w) + \nu I \preceq (1 + \zeta_0) (\widehat{\nabla}^2 f(w) + \rho I).
    \]
    Thus, for $\zeta = \max \{1 - \frac{\nu}{\rho} (1 - \zeta_0), \zeta_0\}$,
    \[
    (1 - \zeta)(\widehat{\nabla}^2 f(w) + \rho I) \preceq \nabla^2 f(w) + \nu I \preceq (1 + \zeta) (\widehat{\nabla}^2 f(w) + \rho I).
    \]
    To conclude, recall $\zeta_0 < 1$, so the $\max$ defining $\zeta$ is always attained at $1 - \frac{\nu}{\rho} (1 - \zeta_0)$.
\end{proof}

\subsubsection{Proof of \cref{proposition:NysSSNPrecond}}
\begin{proof}
Let $E =  \widehat \nabla^2 f(w)-\hat{H}$, and recall that it satisfies $E\succeq 0$ \citep{tropp2017fixed}. 
To control the norm of $E$, we may use our hypotheses on $r$ to invoke \cref{proposition:appx_mat_mu_sr} (applied to $\Omega^{T}$)\footnote{If $\Omega$ is Gaussian, we apply the proposition to $\frac{1}{\sqrt{c}}\Omega^{T}$. We can do this, as the Nystr{\"o}m approximation is invariant to multiplication of $\Omega$ by scalars.} and \cref{lem:reg_nys_err_bnd} to guarantee that 
\[\|E\|\leq \zeta_0 \rho \]
with probability at least $1-\delta/2$. 
Now, by definition of $E$, the regularized subsampled Hessian satisfies
\begin{equation}
\label{eq:HessianDecomp}
    \widehat \nabla^2 f(w)+\rho I = \hat{H}+\rho I+E.
\end{equation}
Let $P = \hat{H}+\rho I$. Then combining \eqref{eq:HessianDecomp} with Weyl's inequalities yields
\begin{align*}
&\lambda_{1}\left (P^{-1/2} \left(\widehat \nabla^2 f(w)+\rho I \right) P^{-1/2}\right) \leq  \lambda_{1}\left(P^{-1/2}\left(\hat{H}+\rho I\right)P^{-1/2}\right)+\lambda_{1}\left(P^{-1/2}EP^{-1/2}\right) \\ =
& 1+\|P^{-1/2}E P^{-1/2}\| \leq 1+\|P^{-1}\|\|E\|\leq 1+\frac{\|E\|}{\rho}\leq 1+\zeta_0.
\end{align*}
To bound the smallest eigenvalue, observe that
\[
  \widehat \nabla^2 f(w)+\rho I \succeq \hat{H}+\rho I
\quad \implies \quad
\widehat \nabla^2 f(w)+\rho I \succeq P.\]
Thus, conjugating the preceding relation by $P^{-1/2}$ we obtain
\[P^{-1/2} \left(\widehat \nabla^2 f(w)+\rho I\right) P^{-1/2}\succeq I_{p},\]
where $I_{p}$ is the $p\times p$ identity matrix.
The preceding inequality immediately yields 
\[
\lambda_{p}(P^{-1/2} \left(\widehat \nabla^2 f(w)+\rho I\right) P^{-1/2})\geq 1.
\]
Hence,
\[1\leq\lambda_{p}(P^{-1/2} \left(\widehat \nabla^2 f(w)+\rho I\right) P^{-1/2})\leq \lambda_{1}(P^{-1/2} \left(\widehat \nabla^2 f(w)+\rho I\right) P^{-1/2})\leq 1+\zeta_0.\]
As an immediate consequence, we obtain the Loewner ordering relation
\[I_p\preceq P^{-1/2}\left(\widehat \nabla^2 f(w)+\rho I\right) P^{-1/2}\preceq (1+\zeta_0)I_p\]
which we conjugate by $P^{1/2}$ to show
\[\hat{H}+\rho I \preceq \widehat \nabla^2 f(w)+\rho I \preceq (1+\zeta_0)\left(\hat{H}+\rho I\right),\]
with probability at least $1-\delta/2$.

We now relate the quality of $P$ to the full Hessian $\nabla^2 f(w)+\nu I$. 
To accomplish this, we observe that by our hypothesis on $b_H$, \cref{proposition:SSNPrecond} implies the relation
\[
(1-\zeta_0)\left(\widehat \nabla^2 f(w)+\nu I\right) \preceq \nabla^2 f(w)+\nu I \preceq (1+\zeta_0)\left(\widehat \nabla^2 f(w)+\nu I\right),
\]
holds with probability at least $1-\delta/2$.
As $\rho \geq \nu$, the preceding display yields
\[
(1-\zeta_0)\frac{\nu}{\rho}\left(\widehat \nabla^2 f(w)+\rho I\right) \preceq \nabla^2 f(w)+\nu I \preceq (1+\zeta_0)\left(\widehat \nabla^2 f(w)+\rho I \right).
\]
Hence by union bound, with probability at least $1-\delta$, we have
\[
(1-\zeta_0)\frac{\nu}{\rho}\left(\hat{H}+\rho I\right) \preceq  \nabla^2 f(w)+\nu I \preceq (1+\zeta_0)^2\left(\hat{H}+\rho I\right).
\]
Scaling down $\zeta_0$ by a constant factor to $\zeta'$, so that $(1+\zeta')^2$ is smaller than $(1+\zeta_0)$, the previous display becomes
\[
(1-\zeta_0)\frac{\nu}{\rho}\left(\hat{H}+\rho I\right) \preceq  \nabla^2 f(w)+\nu I \preceq (1+\zeta_0)\left(\hat{H}+\rho I\right). 
\]
Setting $\zeta = 1-(1-\zeta_0)\frac{\nu}{\rho}$, we see that $\hat H+\rho I$ is a $\zeta$-spectral approximation, which concludes the argument. 
\end{proof}

\subsubsection{Proof of $\zeta$-spectral approximation for SASSN}
Sketch-and-Solve Subsampled Newton also yields a $\zeta$-spectral approximation under similar hypotheses as Nystr{\"o}m Subsampled Newton, as shown in the following proposition. 
\begin{proposition}
\label[proposition]{proposition:SASSNPrecond}
     Let $w\in \mathbb{R}^{p}$, $\zeta_0 \in (0, 1)$, suppose $F$ is a GLM, and that  we construct the subsampled Hessian with sample size $b_H =  \bigOmega\left(\frac{\chi^\nu(\nabla^2 f(w)) d^\nu_{\textup{eff}}(\nabla^2 f(w))\log\left(\frac{d^\nu_{\textup{eff}}(\nabla^2 f(w))}{\delta}\right)}{\zeta_0^2}\right)$. 
     Further assume one of the following holds:
     \begin{enumerate}
         \item $\frac{1}{\sqrt{r}}\Omega$, where $\Omega$ is a Gaussian random matrix with $r = \bigOmega\left(\frac{d^{\rho}_{\textup{eff}}(\widehat \nabla^2 f(w))+\log(\frac{1}{\delta})}{{\zeta_0^2}}\right)$ rows.
         \item $\Omega$ is an SRHT matrix with $ r =\bigOmega\left(\frac{d^{\rho}_{\textup{eff}}(\widehat \nabla^2 f(w))+\log(\frac{1}{\zeta \delta})\log\left(\frac{d^{\rho}_{\textup{eff}}(\widehat \nabla^2 f(w))}{\delta}\right)}{{\zeta_0^2}}\right)$ rows.
         \item $\Omega$ is an sparse sign embedding with sparsity $s = \bigOmega\left(\frac{\log\left(\frac{d^{\rho}_{\textup{eff}}(\widehat \nabla^2 f(w))}{\delta}\right)}{{\zeta_0}}\right)$ and $r =\bigOmega\left(\frac{d^{\rho}_{\textup{eff}}(\widehat \nabla^2 f(w))\log\left(\frac{d^{\rho}_{\textup{eff}}(\widehat \nabla^2 f(w))}{\delta}\right)}{{\zeta_0^2}}\right)$ rows.
         \item $\Omega$ is a LESS-uniform embedding with sparsity $s = \bigO \left(\chi^{\rho}(\widehat \nabla^2 f(w)) d^{\rho}_{\textup{eff}}(\widehat \nabla^2 f(w)) \right)$ and $r = \bigOmega\left(\frac{d^{\rho}_{\textup{eff}}(\widehat \nabla^2 f(w))\log\left(\frac{d^{\rho}_{\textup{eff}}(\widehat \nabla^2 f(w))}{\delta}\right)}{{\zeta_0^2}}\right)$ rows.
     \end{enumerate}
     Then with probability at least $1 - \delta$,
     \begin{equation}
        (1-\zeta) (\hat H+\rho I){\preceq} \nabla^2f(w)+\nu I\preceq (1+\zeta)(\hat H+\rho I),
    \end{equation}
    where $\zeta = (1-\zeta_0)\nu/\rho$.
\end{proposition}
\begin{proof}
    As $F$ is a GLM, the \emph{unregularized} portion of subsampled Hessian has the form
    \[
    A^{T}_{\mathcal B}\Phi(A_{\mathcal B}w)A_{\mathcal B} = R^{T}R,
    \]
    where $R = \Phi(A_{\mathcal B}w)^{1/2}A_{\mathcal B}$ and
    where $\mathcal B$ corresponds to the subsampled indices. 
    Hence
    \[
    \widehat \nabla^2 f(w)+\rho I = R^{T}R+\rho I.
    \]
    Now, consider the random variable 
    \[
    X \coloneqq \left(R^{T}R+\rho I\right)^{-1/2}\left(R^T\Omega^{T}\Omega R+\rho I\right)\left(R^{T}R+\rho I\right)^{-1/2}-I_p.
    \]
    Using the SVD, $R = U\Sigma V^{T}$, this may be rewritten as 
    \begin{align*}
        X &= V\left(\Sigma^2+\rho I\right)^{-1/2}V^{T}\left(V\Sigma U^{T}\Omega^{T}\Omega U\Sigma V^{T}+\rho I\right)V\left(\Sigma^2+\rho I\right)^{-1/2}V^{T}-I_p\\
        &= V\left(\Sigma^2+\rho I\right)^{-1/2}V^{T}\left(V\Sigma U^{T}\Omega^{T}\Omega U\Sigma V^{T}+V\left[-\Sigma^2+\Sigma^2+\rho I\right]V^{T}\right)V\left(\Sigma^2+\rho I\right)^{-1/2}V^{T}-I_p\\
        &= V\Sigma\left(\Sigma^2+\rho I\right)^{-1/2}\left(U^{T}\Omega^{T}\Omega U\right)\Sigma(\Sigma^2+\rho I)^{-1/2}V^T-V\Sigma^2\left(\Sigma^2+\rho I\right)^{-1}V^{T}\\
        &= VD_\rho^{1/2}U^{T}\Omega^{T}\Omega UD^{1/2}_{\rho}V^{T}-VD_\rho V^{T}.
    \end{align*}
    Hence we find
    \[
    \|X\| = \|D_\rho^{1/2}U^T\Omega^{T}\Omega UD^{1/2}_\rho-D_\rho\|.
    \]
    Now, using our hypotheses on $\Omega$  we may invoke \cref{proposition:appx_mat_mu_sr}\footnote{See Lemma 7 in \cite{derezinski2021newtonless} to understand why \cref{proposition:appx_mat_mu_sr} holds for item 4.} with $G = RR^{T} = U\Sigma^2U^{T}$, to reach
    \[
    \|X\| =\|D_\rho^{1/2}U^T\Omega^{T}\Omega UD^{1/2}_\rho-D_\rho\|\leq \zeta_0, 
    \]
    with probability at least $1-\delta/2$.
    It now follows immediately from the preceding display that
    \[
    (1-\zeta_0)\left(\widehat \nabla^2 f(w)+\rho I\right) \preceq R^T\Omega^{T}\Omega R +\rho I \preceq (1+\zeta_0)\left(\widehat \nabla^2 f(w)+\rho I\right),
    \]
    with probability at least $1-\delta/2$.
    Rearranging and scaling down $\zeta_0$, we reach
    \[
     (1-\zeta_0)\left(R^T\Omega^{T}\Omega R +\rho I\right) \preceq \left(\widehat \nabla^2 f(w)+\rho I\right) \preceq (1+\zeta_0)\left(R^T\Omega^{T}\Omega R +\rho I\right).
    \]
    Now applying the same argument as in the proof of \cref{proposition:NysSSNPrecond}, we conclude
    \[
     (1-\zeta_0)\frac{\nu}{\rho}\left(R^T\Omega^{T}\Omega R +\rho I\right) \preceq \left(\nabla^2 f(w)+\nu I\right) \preceq (1+\zeta_0)\left(R^T\Omega^{T}\Omega R +\rho I\right),
    \]
    with probability at least $1-\delta$. 
    Thus, setting $\zeta = 1 - \frac{\nu}{\rho} (1 - \zeta_0)$, we reach
    \[
    (1 - \zeta)(R^T \Omega^T \Omega R + \rho I) \preceq \nabla^2 f(w) + \nu I \preceq (1 + \zeta) (R^T \Omega^T \Omega R + \rho I),
    \]
    which concludes the argument. 
\end{proof}
\else
\fi

\ifpreprint
\subsection{Proofs for quadratic regularity}
\label{subsec:quad_reg_pfs}

\subsubsection{Proof of \cref{proposition:QuadRegConds}}
We first define the quadratic stability constant,
\begin{equation*}
    \Gamma(\mathcal C) \coloneqq \sup_{w_0 \in \mathcal C}\left(\sup_{w_1,w_2, w_3 \in \mathcal C,w_2\neq w_3}\frac{\|w_3-w_2\|^2_{\nabla^2 F(w_1)}}{\|w_3-w_2\|^2_{\nabla^2 F(w_0)}}\right) = \sup_{w_0 \in \mathcal C}\left(\sup_{w_1\in \mathcal C, d\in (\mathcal C-\mathcal C)\setminus\{0\}}\frac{\|d\|^2_{\nabla^2 F(w_1)}}{\|d\|^2_{\nabla^2 F(w_0)}}\right) .
\end{equation*}
\begin{lemma}
\label[lemma]{lem:quad_stab_bnd} 
    The upper and lower quadratic regularity constants satisfy the following bound:
    \[
    \frac{1}{\Gamma(\mathcal C)} \leq \gamma_\ell(\mathcal C)\leq \gamma_u(\mathcal C)\leq \Gamma(\mathcal C).
    \]
\end{lemma}
\begin{proof}
    Observe that,
    \[
    \frac{\|w_2-w_1\|^2_{\nabla^2 F(w_t)}}{\|w_2-w_1\|^2_{\nabla^2F(w_0)}} = \frac{\|w_t-w_1\|^2_{\nabla^2 F(w_t)}}{\|w_t-w_1\|^2_{\nabla^2F(w_0)}},
    \]
    where $w_t = w_1+t(w_2-w_1)\in \mathcal C$.
    Hence we have
    \begin{align*}
         \gamma_u(\mathcal C) &= \sup_{w_0\in \mathcal{C}}\left(\sup_{w_1,w_2\in \mathcal{C}, w_1\neq w_2} \int^1_0 2(1 - t) \frac{\|w_2-w_1\|_{\nabla^2F(w_1+t(w_2-w_1))}^2}{\|w_2-w_1\|^2_{\nabla^2 F(w_0)}}dt\right) \\
         & = \sup_{w_0\in \mathcal{C}}\left(\sup_{w_1,w_2\in \mathcal{C}, w_1\neq w_2} \int^1_0 2(1 - t) \frac{\|w_t-w_1\|^2_{\nabla^2 F(w_t)}}{\|w_t-w_1\|^2_{\nabla^2F(w_0)}}dt\right) \\
         & \leq \sup_{w_0\in \mathcal{C}}\left(\int^1_0 2(1-t) \sup_{w_1,w_2\in \mathcal{C}, w_1\neq w_2}\frac{\|w_t-w_1\|^2_{\nabla^2 F(w_t)}}{\|w_t-w_1\|^2_{\nabla^2F(w_0)}}dt\right) \\
         &\leq \int_{0}^{1}2(1-t)\sup_{w_0\in \mathcal{C}}\left(\sup_{w_1\in \mathcal C, d\in (\mathcal C-\mathcal C)\setminus\{0\}}\frac{\|d\|^2_{\nabla^2 F(w_1)}}{\|d\|^2_{\nabla^2 F(w_0)}}\right)dt\\
         &= \Gamma(\mathcal C).
    \end{align*}
    By observing that 
    \[
    \inf_{w_0 \in \mathcal C}\left(\inf_{w_1,w_2, w_3 \in \mathcal C,w_2\neq w_3}\frac{\|w_3-w_2\|^2_{\nabla^2 F(w_1)}}{\|w_3-w_2\|^2_{\nabla^2 F(w_0)}}\right)
    \geq \frac{1}{\Gamma(\mathcal C)},\]
    a nearly identical argument yields the lower bound on $\gamma_\ell(\mathcal C)$.
\end{proof}

Hence to prove \cref{proposition:QuadRegConds}, 
we just have to upper bound $\Gamma(\mathcal C)$, and this is the strategy we shall follow.
\begin{proof}
    \begin{enumerate}
        \item If $f$ is $L$-smooth and $\mu$-strongly convex, then for any $w_0,w_1\in \mathcal C$ and $d\in \left(\mathcal C-\mathcal C\right)\setminus\{0\}$, we have
        \[
        \frac{\mu}{L}\leq \frac{\|d\|^2_{\nabla^2F(w_1)}}{\|d\|^2_{\nabla^2F(w_0)}}\leq \frac{L}{\mu}.
        \]
        Hence $\Gamma(\mathcal C)\leq L/\mu$, and so 
        \[\frac{\mu}{L}\leq \gamma_\ell^\star\leq \gamma^\star_u \leq \frac{L}{\mu},\]
        by \cref{lem:quad_stab_bnd}.
        \item Observe that for any $w_0,w_1\in \mathcal C$ and $d\in \left(\mathcal C-\mathcal C\right)\setminus\{0\}$, we have
        \[
         \frac{\|d\|^2_{H(w_1)}}{\|d\|^2_{H(w_0)}}= 1+\frac{d^{T}\left[H(w_1)-H(w_0)\right]d}{\|d\|^2_{H(w_0)}} \leq 1+\frac{MD}{\mu}.
        \]
        Where the last inequality used that $f$ has an $M$-Lipschitz Hessian and that $\|d|\leq D.$
        Thus $\Gamma(\mathcal C) \leq 1+\frac{MD}{\mu}$, and the claim follows immediately from \cref{lem:quad_stab_bnd}.

        \item Recall if $F$ is generalized self-concordant, then by definition (see Definition 1 in \cite{marteau2019globally}) the function $\phi(t) = \|d\|^2_{\nabla^2F(w+td')}$ satisfies
        \[
        \phi'(t) \leq k\|d'\|\|d\|^2_{\nabla^2 F(w+td')} = k\|d'\|\phi(t) \quad \forall w \in \mathcal C, d,d'\in \left(\mathcal C-\mathcal C\right)\setminus\{0\}.
        \]
        Hence, 
        \begin{align*}
            \frac{\phi(1)}{\phi(0)} = \exp\left(\int_{0}^{1} \frac{\phi'(t)}{\phi(t)}dt\right) \leq \exp\left(\int_{0}^{1}k\|d'\|dt\right) \leq \exp(kD).  
        \end{align*}
        Setting  $w = w_0$ and $d' = w_1-w_0$, we reach 
        \[
        \frac{\|d\|^2_{\nabla^2 F(w_1)}}{\|d\|^2_{\nabla^2 F(w_0)}}\leq \exp(kD).
        \]
        It immediately follows that $\Gamma(\mathcal C)\leq \exp(kD)$.
       
       \item 
        As $F$ is a GLM, the Hessian has the form $\nabla^2 F(w) = \frac{1}{n}A^{T}\Upphi''\left(Aw\right)A+\nu I$.
        Hence with $\hat \Sigma = \frac{1}{n}A^{T}A$, we have by definition of $u$ and $\ell$ that
        \[\ell\hat\Sigma+\nu I\preceq \nabla^2 F(w)\preceq  u\hat \Sigma+\nu I.\]
        Thus for any $w_0,w_1\in \mathcal{C}$ and $d\in \left(\mathcal C-\mathcal C\right)\setminus\{0\} $, we have
        \begin{align*}
             &\frac{\|d\|^2_{\nabla^2F(w_1)}}{\|d\|^2_{\nabla^2F(w_0)}} = \frac{\frac{\|d\|^2_{\frac{1}{n}A^{T}\Upphi''\left(Aw_1\right)A}}{\|d\|^2}+\nu}{\frac{\|d\|^2_{\frac{1}{n}A^{T}\Upphi''\left(Aw_0\right)A}}{\|d\|^2}+\nu} \leq  \frac{\frac{u\|d\|^2_{\hat \Sigma}}{\|d\|^2}+\nu}{\frac{\ell \|d\|^2_{\hat \Sigma}}{\|d\|^2}+\nu} \leq \frac{u\lambda_1(\hat \Sigma)+\nu}{\ell\lambda_1(\hat \Sigma)+\nu},
        \end{align*}
        where the last inequality follows from the variational characterization of $\lambda_1(\hat \Sigma)$, and that the function
        \[h(x) = \frac{u x+\nu}{\ell x+\nu}\]
        is monotonically increasing for $x\geq 0.$
        Thus we conclude
        \[
        \Gamma(\mathcal C)\leq \frac{u\lambda_1(\hat \Sigma)+\nu}{\ell\lambda_1(\hat \Sigma)+\nu},
        \]
        which yields the claim via \cref{lem:quad_stab_bnd}.
    \end{enumerate}
\end{proof}
\else
\fi

\ifpreprint
\subsection{Proofs of Hessian dissimilarity bounds}
\label{subsec:hess_sim_pfs}
\subsubsection{Proof of \cref{lem:hess_sim}}
\begin{proof}
    Let $H_i(w) = \nabla^2 F(w)^{-1/2}\left(\nabla^2 f_i(w)+\nu I\right)\nabla^2 F(w)^{-1/2}$. 
    We start with the lower bound.
    Observe for any $w\in \mathcal C$ that
    \begin{align*}
        1 & = \lambda_1\left(\nabla^2F(w)^{-1/2}\nabla^2F(w)\nabla^2F(w)^{-1/2}\right) = \lambda_1\left(\nabla^2F(w)^{-1/2}\left[\frac{1}{n}\sum_{i=1}^{n}\nabla^2 F_i(w)\right]\nabla^2F(w)^{-1/2}\right) \\
        & = \lambda_{1}\left(\frac{1}{n}\sum_{i=1}^{n}H_i(w)\right)\leq \frac{1}{n}\sum_{i=1}^{n}\lambda_{1}(H_i(w)) \leq \max_{1\leq i\leq n}\lambda_1(H_i(w))\leq \tau_\star^{\nu}(\mathcal C),
    \end{align*}
    which proves the lower bound.
    To prove the upper bound, we note that for any $w\in \mathcal{C}$ and $i\in [n]$
    \begin{align*}
        \nabla^2 F_i(w)\preceq n \nabla^2 F(w), \quad \nabla^2 F_i(w)\preceq (L_{\textrm{max}}+\nu)I_p.
    \end{align*}
    So conjugation gives
    \begin{align*}
        H_i(w)\preceq n I_p, \quad H_i(w)\preceq \left(L_{\textrm{max}}+\nu\right)\nabla^2F(w)^{-1}\preceq 1+\frac{L_{\textrm{max}}}{\nu} I_p
    \end{align*}
    which immediately implies $\tau_\star^\nu(\mathcal C)\leq \min\{n,1+\frac{L_{\textrm{max}}}{\nu}\}$. 
\end{proof}

\subsubsection{Proof of \cref{prop:hess_sim_glm}}
\begin{proof}
    Let $H_i(w) = \nabla^2 F(w)^{-1/2}\left(\nabla^2 f_i(w)+\nu I\right)\nabla^2 F(w)^{-1/2}$ .
    As $F(w)$ is a GLM, we have that
    \begin{align*}
    H_i(w) &= \nabla^2 F(w)^{-1/2}\left(\nabla^2 f_i(w)+\nu I\right)\nabla^2 F(w)^{-1/2} \\
    &= \left(R^T(w)R(w)+\nu I\right)^{-1/2}(\phi''_i(a_i^T w)a_ia_i^{T}+\nu I)\left(R^T(w)R(w)+\nu I\right)^{-1/2},    
    \end{align*}
    where $R(w) = \frac{1}{\sqrt{n}}\Phi''(Aw)^{1/2}A.$
    Hence, a direction calculation yields
    \begin{align*}
        \lambda_{1}(H_i(w)) \overset{(1)}{\leq} {} & \lambda_1\left((R^T(w)R(w)+\nu I)^{-1/2}\phi''_i(a_i^T w)a_ia_i^{T}(R^T(w)R(w)+\nu I)^{-1/2}\right)\\
        &+\nu \lambda_1((R^T(w)R(w)+\nu I)^{-1}) \\
        \overset{(2)}{\leq} {} & 1+\phi''_i(a_i^T w)\lambda_1\left(a_ia_i^{T}(R^T(w)R(w)+\nu I)^{-1}\right) \\
        \overset{(3)}{\leq} {} & 1+\phi''_i(a_i^T w)\textrm{trace}\left(a_ia_i^{T}(R^T(w)R(w)+\nu I)^{-1}\right) \\
        \overset{(4)}{=} {} & 1+\phi''_i(a_i^T w)\textrm{trace}\left(a_i^{T}(R^T(w)R(w)+\nu I)^{-1}a_i\right) \\
        = {} & 1+\phi''_i(a_i^T w)a_i^{T}(R^T(w)R(w)+\nu I)^{-1}a_i \\
        = {} & 1+(\sqrt{\phi''_i(a_i^T w)}a_i)^{T}(R^T(w)R(w)+\nu I)^{-1}(\sqrt{\phi''_i(a_i^T w)}a_i) \\
        = {} & 1+n l_i^\nu(\Phi''(Aw)^{1/2}A),
    \end{align*}
    where $(1)$ uses Weyl's inequalities, $(2)$ uses matrix similarity, $(3)$ uses that the trace is the sum of matrix's eigenvalues and all the eigenvalues are non-negative, and $(4)$ uses the cyclic property of the trace.
    Hence we conclude
    \begin{align*}
        & \tau_\star^\nu =\sup_{w\in \R^p}\left(\max_{1\leq i\leq n}\lambda_1(H_i(w))\right) \leq \sup_{w\in \R^p} \left(1 + n \max_{1\leq i\leq n}l_i^\nu(\Phi''(Aw)^{1/2}A)\right) = \sup_{w\in \R^p}\left(1+n l_\infty^\nu(\Phi''(Aw)^{1/2}A)\right) \\ 
        &= 1 +\sup_{w\in \R^p}(\chi^{\nu}(\Phi''(Aw)^{1/2}A)d^{\nu}_{\textrm{eff}}(\Phi''(Aw)^{1/2}A)) \leq 1+\chi_\star^{\nu}\sup_{w\in \R^p} d^{\nu}_{\textrm{eff}}(\Phi''(Aw)^{1/2}A). 
    \end{align*}
    Now, to conclude the desired inequality, observe that
    \begin{align*}
        &d^{\nu}_{\textrm{eff}}(\Phi''(Aw)^{1/2}A) = \sum_{j=1}^{p}\frac{\lambda_j(A^{T}\Phi''(Aw)A)}{\lambda_j(A^{T}\Phi''(Aw)A)+n\nu} \leq \sum_{j=1}^{p}\frac{B\lambda_j(A^{T}A)}{B\lambda_j(A^{T}A)+n\nu} \\
  & =\sum_{j=1}^{p}\frac{\lambda_j(A^{T}A)}{\lambda_j(A^{T}A)+\frac{n\nu}{B}} = d^{\nu/B}_{\textrm{eff}}(A),
    \end{align*}
    where the first inequality follows as $A^{T}\Phi''(Aw)A \preceq B A^TA$ and $\frac{x}{x+\nu}$ is increasing in $x$ for $x\geq 0$.
    The claim regarding logistic regression and least squares immediately follows by observing that $B = 1$ in both cases. 
\end{proof}
\else
\fi

\ifpreprint
\subsection{Proof of \cref{prop:stable_precond}}
\label{subsec:stable_precond}
\begin{proof}
    We only prove the result for the $X_j$'s, as the argument for the $Y_j$'s is identical.
    To this end, observe that
    \begin{align*}
    \frac{\|w_j-w_\star\|^2_{P_{j+1}}}{\|w_j-w_\star\|^2_{P_{j}}} \leq \frac{1+\zeta}{\rho}\frac{\|w_j-w_\star\|^2_{\nabla^2 F(w_j)}}{\|w_j-w_\star\|^2} \leq (1+\zeta)\frac{L}{\rho}.
    \end{align*}
    Hence for any $j$, with probability 1
    \[
    X_j \leq \max\left\{1,(1+\zeta)\frac{L}{\rho}\right\}.
    \]
    The desired claim now follows immediately for some $\xi_j$ satisfying 
    \[
    \xi_j \leq \max\left\{1,(1+\zeta)\frac{L}{\rho}\right\}.
    \]
\end{proof}
\else
\fi

\subsection{Proof of \cref{prop:precond_grad_var}}
\label{subsec:precond_grad_var}
We begin by recalling the following fundamental result from \cite{gower2019sgd}. 
\begin{theorem}[Theorem 3.6 and Proposition 3.8, \cite{gower2019sgd}]
\label{thm:grad_var}
Suppose $F = \frac{1}{n}\sum_{i = 1}^{n}F_i(w)$, where $F_i:\R^p\mapsto \R$.
Let the following conditions hold:
\begin{enumerate}
    \item $F_i$ is convex, for every $i\in [n]$.
    \item For each $i\in [n]$, there exists a matrix $M_i \in \mathbb{S}_p^{++}(\R)$, such that for all $x,h \in \R^p$
    \[
    F_i(w+h)\leq F_i(w)+\langle \nabla F_i(w),h \rangle+\frac{1}{2}\|h\|_{M_i}^2.
    \]
    \item There exists a matrix $M \in \mathbb{S}_p^{++}(\R)$, such that for all $x,h \in \R^p$
    \[
    F(w+h)\leq F(w)+\langle \nabla F(w),h \rangle+\frac{1}{2}\|h\|_{M}^2.
    \]
\end{enumerate}
Then for any $w,w' \in \R^p$, it holds that
\[
\E\|\widehat \nabla F(w)-\widehat \nabla F(w')\|^2\leq 2\mathcal L\left(F(w)-F(w')-\langle\nabla F(w'),w-w'\rangle \right),
\]
where 
\[
\mathcal L = \frac{n(b_g-1)}{b_g(n-1)}\lambda_1\left(M\right)+\frac{n-b_g}{b_g(n-1)}\max_{i\in [n]}\lambda_1(M_i).
\]
\end{theorem}

With these preliminaries out of the way, we commence the proof of \cref{prop:precond_grad_var}.
\begin{proof}
    Observe that each $F_i$ satisfies:
    \[
    F_i(w+h)\leq F_i(w)+\langle \nabla F_i(w),h\rangle+\frac{1}{2}\|h\|_{M_i}^2,
    \]
    with $M_i = \gamma_{u_i}\nabla^2F_i(w_0)$, where $w_0$ is the point where the preconditioner $P$ is constructed. 
    Hence performing the change of variable $w = P^{-1/2}z$ and defining $F_{P_i}(z) = F_i(P^{-1/2}z), F_P(z) = F(P^{-1/2}z)$, we reach
    \[
    F_{P_i}(z+\tilde h)\leq F_{P_i}(z)+ \langle \nabla F_{P_i}(z),\tilde h \rangle +\frac{\gamma_{u_i}}{2}\|\tilde h\|_{\nabla^2 F_{P_i}(z_0)}^2, 
    \]
    \[
    F_{P}(z+\tilde h)\leq F_{P}(z)+ \langle \nabla F_{P}(z),\tilde h \rangle +\frac{\gamma_{u}}{2}\|\tilde h\|_{\nabla^2 F_{P}(z_0)}^2. 
    \]
    Hence the conditions of \cref{thm:grad_var} are satisfied with $M_i = \gamma_{u_i} \nabla^2 F_{P_i}(z_0), M = \gamma_u \nabla^2 F_{P}(z_0)$, and so we reach
    \[
    \E\|\widehat\nabla F_{P_i}(z)-  \widehat\nabla F_{P_i}(z')\|^2\leq 2\mathcal L\left(F_{P_i}(z)- F_{P_i}(z')-\langle \nabla F_{P_i}(z'),z-z'\rangle\right),
    \]
    with $\mathcal L$ as in \cref{thm:grad_var}.
    Thus, we obtain
    \[
     \E\|\widehat\nabla F(w)-  \widehat\nabla F(w')\|_{P^{-1}}^2\leq 2\mathcal L\left(F(w)- F(w')-\langle \nabla F(w'),w-w'\rangle\right).
    \]
    Now,
    \begin{align*}
    & \mathcal L  = \frac{n(b_g-1)}{b_g(n-1)}\lambda_1\left(M\right)+\frac{n-b_g}{b_g(n-1)}\max_{i\in [n]}\lambda_1(M_i) \\
    & = \frac{n(b_g-1)}{b_g(n-1)}\gamma_{u}\lambda_1\left(\frac{1}{n}\sum_{i=1}^{n}\nabla^2 F_{P_i}(z_0)\right)+\frac{n-b_g}{b_g(n-1)}\max_{i\in [n]}\lambda_1\left(\gamma_{u_i}\nabla^2 F_{P_i}(z_0)\right) \\ 
    &\overset{(1)}{\leq} \frac{n(b_g-1)}{b_g(n-1)}\gamma_u (1+\zeta)+\frac{n-b_g}{b_g(n-1)}\gammaMax \lambda_1\left(\nabla^2 F_{P_i}(z_0)\right)\\
    & \overset{(2)} \leq  \left(\frac{n(b_g-1)}{b_g(n-1)}\gamma_u+\tau_{\star}^{\nu}\frac{n-b_g}{b_g(n-1)}\gammaMax\right) (1+\zeta) = \mathcal L_P,
    \end{align*}
    where $(1)$, $(2)$ both use that $P$ is a $\zeta$-spectral approximation, and $(2)$ uses  $\nabla^2 F_i(w)\preceq \tau^\nu_{\star} \nabla^2F(w)$, which follows by definition of $\tau_\star^{\nu}$. 
    Hence for all $w,w' \in \R^p$
    \[
    \E\|\widehat\nabla F(w)-  \widehat\nabla F(w')\|_{P^{-1}}^2\leq 2\mathcal L_P\left(F(w)- F(w')-\langle \nabla F(w'),w-w'\rangle\right),
    \]
    as desired.
\end{proof}

\subsection{SketchySVRG: Fast local convergence}
\label{subsec:sksvrg_fast_local_convergence}
In this section, we prove \cref{thm:sksvrg_loc_con}, which shows local condition number-free convergence of SketchySVRG in the neighborhood 
\[
\Nstar = \left\{w\in \R^p: \|w-w_\star\|_{\nabla^2F(w_\star)}\leq\frac{\varepsilon_0 \nu^{3/2}}{2M}\right\}.
\]
The result proven here, substantially improves upon the local convergence result of \cite{derezinski2022stochastic}, which requires a gradient batch size of $\Otil(\kappa)$ to obtain fast local convergence.
In contrast, \cref{thm:sksvrg_loc_con} only requires the gradient batchsize to satisfy $b_g = \Otil(\tau_\star^{\nu}(\Nstar)$, which is often orders of magnitude smaller than $\kappa$ in the ill-conditioned setting (see \cref{cor:sksvrg_fast_glm}).

The overarching idea of the proof is similar to other local analyses of stochastic Newton methods \citep{li2020subsampled,derezinski2021newtonless,derezinski2022stochastic}.
Namely, we seek to show the iterates belong to progressively smaller neighborhoods of the optimum, with a contraction rate independent of the condition number. 
\ifpreprint
The key to achieving these objectives is that the Hessian always provides a good quadratic model in the sense of \cref{lem:lazy_hess_good_model} when the iterates are close to the optimum.
\else
\fi

We start with standard notation, which will be used throughout the proof. 
\subsubsection{Notation}
We define the following quantities:
    \[
    \Delta_k^{(s)} \coloneqq w_k^{(s)}-w_\star, \quad p_k^{(s)} \coloneqq \nabla^2 F(w_k^{(s)})^{-1}v^{(s)}_k, \quad \tilde{p}_k^{(s)} \coloneqq P^{-1}v^{(s)}_k. 
    \]
$\Delta^{(s)}_k$ is the distance of the current iterate to the optimum, $p_k^{(s)}$ is the exact Newton direction, and $\tilde{p}^{(s)}_k$ is the approximate Newton direction actually computed by the algorithm.  

\subsubsection{Preliminary lemmas}
We begin with the following technical lemma, which shows the following items hold in $\Nstar$: (1) the quadratic regularity constants are close to unity, (2) the Hessians are uniformly close in the Loewner ordering, (3) taking an exact Newton step moves the iterate closer to the optimum in the Hessian norm, (4) $\nabla F_i(w)$, $\nabla F(w)$ are $(1+\varepsilon_0)$ Lipschitz in $\Nstar$, and (5) $P^{-1}$ is uniformly good approximation to the inverse Hessian. 
\begin{lemma}
\label[lemma]{lem:sksvrg_local_basic_lemmas}
    Let $w,w'\in \Nstar$, and suppose $P$ is a $\varepsilon_0$-spectral approximation constructed at some $w_0\in \Nstar$, 
    then the following items hold.
    \begin{enumerate}
        \item
        \[
        \frac{1}{1+\varepsilon_0}\leq \gammaMin(\Nstar)\leq \gammaMax(\Nstar)\leq (1+\varepsilon_0).
        \]
        \item \[
            (1-\varepsilon_0)\nabla^2 F(w)\preceq \nabla^2 F(w') \preceq (1+\varepsilon_0) \nabla^2F(w). 
            \]
        \item \[
            \|w-w_\star-\nabla^2F(w)^{-1}\nabla F(w)\|_{\nabla^2 F(w)}\leq \varepsilon_0\|w-w_\star\|_{\nabla^2 F(w)}.
            \]
        \item \begin{align*}
        &\|\nabla F_i(w)-\nabla F_i(w_\star)\|_{\nabla^2 F_i(w')^{-1}}\leq (1+\varepsilon_0)\|w-w_\star\|_{\nabla^2 F_i(w')}, \quad \text{for all $i\in [n]$},\\
        & \|\nabla F(w)-\nabla F(w_\star)\|_{\nabla^2 F(w')^{-1}}\leq (1+\varepsilon_0)\|w-w_\star\|_{\nabla^2 F(w')}.
        \end{align*}
        \item 
        \[
        \left\|\nabla^2F(w)^{1/2}(\nabla^2F(w)^{-1}-P^{-1})\nabla^2F(w)^{1/2}\right\|\leq 3\varepsilon_0. 
        \]
    \end{enumerate}
\end{lemma}

\ifpreprint
\begin{proof}
    \begin{enumerate}
        \item Observing that $\|w-w_\star\|\leq \nu^{-1/2}\|w-w_\star\|_{\nabla^2F(w_\star)}\leq \varepsilon\nu/M$, and each $\nabla^2 F_i$ has $M$-Lipschitz Hessian, the first statement follows from item 2. of \cref{proposition:QuadRegConds}.
        \item For the second statement, observe that $F$ having $M$-Lipschitz Hessian implies
        \[
        \nabla^2 F(w)-M\|w'-w\|I_p \preceq \nabla^2 F(w') \preceq \nabla^2 F(w)+M\|w'-w\|I_p. 
        \]
        So, using $I_p \preceq 1/\nu \nabla^2F(w)$, we find
        \[
        \|w'-w\| \leq \nu^{-1/2}\left(\|w'-w_\star\|_{\nabla^2 F(w_\star)}+\|w-w_\star\|_{\nabla^2 F(w_\star)}\right)\leq \varepsilon_0\nu/M
        \] 
        Hence, we conclude
        \[
        (1-\varepsilon_0)\nabla^2 F(w)\preceq \nabla^2 F(w')\preceq (1+\varepsilon_0) \nabla^2 F(w). 
        \]
        \item This item follows by direct calculation:
        \begin{align*}
            & \|w-w_\star-\nabla^2F(w)^{-1}\nabla F(w)\|_{\nabla^2 F(w)} = \|\nabla F(w)-\nabla^2F(w)(w-w_\star)\|_{\nabla^2 F(w)^{-1}} \\
            & = \left\|\int_{0}^{1}[\nabla^2F(w_\star+t(w-w_\star))-\nabla^2F(w)](w-w_\star)dt\right\|_{\nabla^2 F(w)^{-1}} \\
            & = \left\|\int_{0}^{1}[\nabla^2F(w)^{-1/2}\nabla^2F(w_\star+t(w-w_\star))\nabla^2F(w)^{-1/2}-I_p]dt\nabla^2F(w)^{1/2}(w-w_\star)\right\|\\
            &\leq \varepsilon_0\|w-w_\star\|_{\nabla^2F(w)},
        \end{align*}
        where the last inequality uses item 2. and Cauchy-Schwarz.
        \item The last item follows directly from the definition of a $\varepsilon_0$-spectral approximation and item 2. Indeed,
        \[
        (1-\varepsilon_0)P\preceq \nabla^2F(w_0) \preceq (1+\varepsilon_0) P
        \]
        implies
        \[
        (1+\varepsilon_0)^{-1}\nabla^2F(w_0)\preceq P\preceq (1-\varepsilon_0)^{-1}\nabla^2F(w_0).
        \]
        Hence by the properties of the Loewner ordering
        \[
        (1-\varepsilon_0) \nabla^2 F(w_0)^{-1} \preceq P^{-1} \preceq (1+\varepsilon_0) \nabla^2F(w_0)^{-1}.
        \]
        Now applying item 2., we reach
        \[
       (1-\varepsilon_0)^2 \nabla^2F(w)^{-1}\preceq  P^{-1} \preceq (1+\varepsilon_0)^2 \nabla^2F(w)^{-1}.
        \]
        As $\varepsilon \in (0,1/6]$, the preceding display becomes
        \[
        (1-3\varepsilon_0) \nabla^2F(w)^{-1}\preceq  P^{-1} \preceq (1+3\varepsilon_0) \nabla^2F(w)^{-1},
        \]
        which immediately implies the desired statement. 
        
    \end{enumerate}
\end{proof}
\else
\fi

We will use the following version of Bernstein's inequality for vectors which slightly refines a result of \cite{minsker2017some}.

 \begin{lemma}[Bernstein's inequality for vectors]
 \label[lemma]{lem:bernstein_ineq}
     Let $\{X_{i}\}_{1\leq i\leq m}$, be a sequence of independent mean zero random vectors in $\R^p$ satisfying $\|X_i\|\leq R$  and $\E[\|X_i\|^2] \leq \varsigma^2$ for all $i\in [m]$,
     Then 
     \[
     \mathbb P\left(\left\|\frac{1}{m}\sum_{i=1}^{m}X_i\right\|\geq t\right) \leq 8\exp\left(-\min\left\{\frac{mt^2}{4\varsigma^2},\frac{-3mt}{4R}\right\}\right),
     \]
     for all $t\geq \sqrt{\frac{\varsigma^2}{m}}+\frac{R}{3m}$.
\end{lemma}
\begin{proof}
    The result follows immediately from applying Theorem 7.3.1. of \cite{tropp2015introduction} to the scaled sequence $\{X_i/m\}_{1\leq i\leq m}$.
\end{proof}

Our next lemma controls the deviation of the stochastic gradient: $\widehat \nabla F(w)-\widehat \nabla F(w_\star)-\nabla F(w)$, in the norm $\|\cdot\|_{\nabla^2F(w')^{-1}}$.
This lemma is the key to improving over the local convergence analysis of \cite{derezinski2022stochastic}, which requires a gradient batchsize of $\Otil(\kappa)$.
The improvement is possible thanks to quadratic regularity and Hessian dissimilarity. 
Quadratic regularity enables us to directly reason in the norms  $\left(\|\cdot\|_{\nabla^2F(w')},\|\cdot\|_{\nabla^2F(w')^{-1}}\right)$, while Hessian dissimilarity allows for the tightest control possible over the gradient batchsize. 
\begin{lemma}
\label[lemma]{lem:sksvrg_local_bernstein_bound}
    Let $\beta_g\in (0,1)$. Suppose $w,w' \in \Nstar$, and $b_g \geq \frac{32\tau_{\star}^{\nu}(\Nstar)\log(\frac{8}{\delta})}{\beta^2_g}$. Then with probability at least $1-\delta$,
    \[
    \|\widehat \nabla F(w)-\widehat \nabla F(w_\star)-\nabla F(w)\|_{\nabla^2 F(w')^{-1}}\leq \beta_g\|w-w_\star\|_{\nabla^2 F(w')}.
    \]
   
\end{lemma}
\begin{proof}
    We begin by observing that 
    \[
    \left\|\widehat \nabla F(w)-\widehat \nabla F(w_\star)-\nabla F(w)\right\|_{\nabla^2 F(w')^{-1}}^2 = \left\|\frac{1}{b_g}\sum_{i\in \mathcal B} \tilde v_i\right\|^{2},
    \]
    where 
    \[\tilde v_i = \nabla^2 F(w')^{-1/2}(\nabla F_i(w)-\nabla F_i(w_\star)-\nabla F(w)).\]
    Now,
    \begin{align*}
        \|\tilde v_i\|^2 &\leq 2\|\nabla F_i(w)-\nabla F_i(w_\star)\|_{\nabla^2 F(w')^{-1}}^2+2\|\nabla F(w)\|_{\nabla^2 F(w')^{-1}}^2 \\
        &\leq 2\tau_{\star}^{\nu}(\Nstar)\|\nabla F_i(w)-\nabla F_i(w_\star)\|_{\nabla^2 F_i(w')^{-1}}^2+2\|\nabla F(w)\|_{\nabla^2 F(w')^{-1}}^2 \\
        &= 2\tau_{\star}^{\nu}(\Nstar)\|\nabla F_i(w)-\nabla F_i(w_\star)\|^2_{\nabla^2 F_i(w')^{-1}}+2\|\nabla F(w)-\nabla F(w_\star)\|_{\nabla^2 F(w')^{-1}}^2.
     \end{align*}
     Here, the second inequality is due to the definition of Hessian dissimilarity.
     Now, invoking item 4 of \cref{lem:sksvrg_local_basic_lemmas} we reach
     \begin{align*}
         \|\tilde v_i\|^2 &\leq 2\tau_{\star}^{\nu}(\Nstar)(1+\varepsilon_0)^2\|w-w_\star\|^2_{\nabla^2 F_i(w')}+2(1+\varepsilon_0)^2\|w-w_\star\|^2_{\nabla^2 F(w')}\\
         & \leq 4(1+\varepsilon_0)^2\tau_{\star}^{\nu}(\Nstar)^2\|w-w_\star\|^2_{\nabla^2 F(w')},
     \end{align*}
     So, for each $i\in \mathcal B$, it holds with probability 1 that 
     \[
     \|\tilde v_i\|\leq 2 (1+\varepsilon_0)\tau_{\star}^{\nu}(\Nstar)\|w-w_\star\|_{\nabla^2 F(w')}.
     \]
     Thus, we may set $R = 2(1+\varepsilon_0)\tau_{\star}^{\nu}(\Nstar)\|w-w_\star\|_{\nabla^2 F(w')}$.
     
     Next, observe that
     \begin{align*}
            \E[\|\tilde v_i\|]^2 &\leq \E\|\nabla F_i(w)-\nabla F_i(w_\star)\|_{\nabla^2 F(w')^{-1}}^2\\
            &\leq \tau_{\star}^{\nu}(\Nstar)\E\|\nabla F_i(w)-\nabla F_i(w_\star)\|_{\nabla^2 F_i(w')^{-1}}^2\\
            &\overset{(1)}{\leq} \tau_{\star}^{\nu}(\Nstar) \E[2(1+\varepsilon_0)(F_i(w)-F_i(w_\star)-\langle \nabla F_i(w_\star),w-w_\star\rangle)]\\ 
            &\leq 2(1+\varepsilon_0)\tau_{\star}^{\nu}(\Nstar)\left(F(w)-F(w_\star)\right) \\
            &\overset{(2)}{\leq} 2(1+\varepsilon_0)^2\tau_{\star}^{\nu}(\Nstar)\|w-w_\star\|^2_{\nabla^2 F(w')}. 
     \end{align*}
     Here $(1)$ uses $F_i(w)\geq F_i(w_\star)+\langle \nabla F_i(w_\star),w-w_\star \rangle+\frac{1}{2(1+\varepsilon_0)}\|w-w_\star\|^2_{\nabla^2F_i(w_\star)}$, which follows from lower quadratic regularity of $F$ and item 1 of \cref{lem:sksvrg_local_basic_lemmas}.
     Last, $(2)$ applies  upper quadratic regularity and item 1 of \cref{lem:sksvrg_local_basic_lemmas}.
     Hence, the variance is bounded by
     \[
     \varsigma^2 = 2(1+\varepsilon_0)^2\tau_{\star}^{\nu}(\Nstar)\|w-w_\star\|^2_{\nabla^2 F(w')}.
     \]
     
     Now, we invoke \cref{lem:bernstein_ineq} to find 
     \[
        \mathbb{P}\left(\left\|\frac{1}{b_g}\sum_{i\in \mathcal B} \tilde v_i\right\|\geq \sqrt{\frac{4\varsigma^2\log\left(\frac{8}{\delta}\right)}{b_g}}+\frac{4R\log\left(\frac{8}{\delta}\right)}{3b_g}\right)\leq \delta.
    \]
     The last display immediately implies with probability at least $1-\delta$, that 
     \begin{align*}
        &\left\|\widehat \nabla F(w)-\widehat \nabla F(w_\star)-\nabla F(w)\right\|_{\nabla^2 F(w')^{-1}}^2\leq \frac{8\varsigma^2\log(8/\delta)}{b_g}+\frac{32R^2\log^2(8/\delta)}{9b_g^2}\\
        &= \left[\frac{16(1+\varepsilon_0)^2\tau_{\star}^{\nu}(\Nstar)\log(8/\delta)}{b_g}+2\left(\frac{8(1+\varepsilon_0)\tau_{\star}^{\nu}(\Nstar)\log(8/\delta)}{3b_g}\right)^2\right] \|w-w_\star\|^2_{\nabla^2 F(w')} \\
        &\leq \beta^2_g \|w-w_\star\|^2_{\nabla^2 F(w')}
        \end{align*}
     where the last inequality follows from $b_g \geq \frac{32\tau_{\star}^{\nu}(\Nstar)\log(\frac{8}{\delta})}{\beta^2_g}$ and $\varepsilon\in(0,1/6)$.
     The desired claim now follows by taking square roots.   
\end{proof}

The next lemma shows that for sufficiently $b_g$ (with high probability) the distance to the optimum in the $\nabla^2F(w_\star)$-norm decreases when an \emph{exact} Newton step based on the current iterate is taken.   
\begin{lemma}
\label[lemma]{lem:sksvrg_local_dist_norm}
    Let $w_k^{(s)}\in \Nstar$, and $\beta_g\in (0,1)$. Suppose the gradient batchsize satisfies $b_g = \bigO\left(\frac{\tau_\star^{\nu}(\Nstar)\log\left(\frac{k+1}{\delta}\right)}{\beta^2_g}\right)$. 
    Then with probability at least $1-\frac{\delta}{(k+1)^2}$,
    \[
    \|\Delta_k^{(s)}-p^{(s)}_k\|_{\nabla^2 F(w_\star)}\leq (1+\varepsilon_0)\left[(\varepsilon_0+\beta_g)\|\Delta_k^{(s)}\|_{\nabla^2 F(w_\star)}+\beta_g\|\Delta_0^{(s)}\|_{\nabla^2 F(w_\star)}\right].
    \]
\end{lemma}
\begin{proof}
    We begin by applying the triangle inequality to reach
    \begin{align*}
        &\|\Delta_k^{(s)}-p^{(s)}_k\|_{\nabla^2 F(w_k^{(s)})} = \|\Delta_k^{(s)}-\nabla^2F(w_k^{(s)})^{-1}v^{(s)}_k\|_{\nabla^2 F(w_k^{(s)})} \\
        & = \|\nabla^2 F(w_k^{(s)})\Delta_k^{(s)}-(\widehat \nabla F(w_k^{(s)})-\widehat \nabla F(\hat w^{(s)})+\nabla F(\hat w^{(s)}))\|_{\nabla^2 F(w_k^{(s)})^{-1}} \\
        &\leq \|\nabla^2 F(w_k^{(s)})\Delta_k^{(s)}-\nabla F(w_k^{(s)})\|_{\nabla^2 F(w_k^{(s)})^{-1}}\\
        &{}+\|\nabla F(w_k^{(s)})-\widehat \nabla F(w_k^{(s)})+\widehat \nabla F(\hat w^{(s)})-\nabla F(\hat w^{(s)})\|_{\nabla^2 F(w_k^{(s)})^{-1}}.
    \end{align*}
   To bound the first term, we apply item 3. of \cref{lem:sksvrg_local_basic_lemmas}, which yields
    \[
      \|\Delta_k^{(s)}-\nabla^2F(w_k^{(s)})^{-1}\nabla F(w_k^{(s)})\|_{\nabla^2 F(w_k^{(s)})} \leq \varepsilon_0\|\Delta_k^{(s)}\|_{\nabla^2 F(w_k^{(s)})}.
    \]
    For the second term, the triangle inequality yields
    \begin{align*}
       &\|\nabla F(w_k^{(s)})-\widehat \nabla F(w_k^{(s)})+\widehat \nabla F(\hat w^{(s)})-\nabla F(\hat w^{(s)})\|_{\nabla^2 F(w_k^{(s)})^{-1}}\\ 
       &\leq \|\widehat \nabla F(w_k^{(s)})-\widehat \nabla F(w_\star)-\nabla F(w_k^{(s)})\|_{\nabla^2 F(w_k^{(s)})^{-1}}\\
       &{}+\|\widehat \nabla F(\hat w^{(s)})-\widehat \nabla F(w_\star)-\nabla F(\hat w^{(s)})\|_{\nabla^2 F(w_k^{(s)})^{-1}}. 
    \end{align*}
    Now, we can apply \cref{lem:sksvrg_local_bernstein_bound}, to find that
    \begin{align*}
        & \|\widehat \nabla F(w_k^{(s)})-\widehat \nabla F(w_\star)-\nabla F(w_k^{(s)})\|_{\nabla^2 F(w_k^{(s)})^{-1}}\leq \beta_g \|\Delta_k^{(s)}\|_{\nabla^2 F(w_k^{(s)})}, \\
        & \|\widehat \nabla F(\hat w^{(s)})-\widehat \nabla F(w_\star)-\nabla F(\hat w^{(s)})\|_{\nabla^2 F(w_k^{(s)})^{-1}}\leq \beta_g \|\Delta_0^{(s)}\|_{\nabla^2 F(w_k^{(s)})}, 
    \end{align*}
    with probability at least $1-\frac{\delta}{(k+1)^2}$.
    So,
    \begin{align*}
     &\|\nabla F(w_k^{(s)})-\widehat \nabla F(w_k^{(s)})+\widehat \nabla F(\hat w^{(s)})-\nabla F(\hat w^{(s)})\|_{\nabla^2 F(w_k^{(s)})^{-1}}\\  
     &\leq \beta_g\left(\|\Delta_k^{(s)}\|_{\nabla^2 F(w_k^{(s)})}+ \|\Delta_0^{(s)}\|_{\nabla^2 F(w_k^{(s)})}\right), 
    \end{align*}
    Combining the upper bounds on terms $1$ and $2$, we find  
    \begin{align*}
        \|\Delta_k^{(s)}-p^{(s)}_k\|_{\nabla^2 F(w_k^{(s)})} \leq (\varepsilon_0+\beta_g)\|\Delta_k^{(s)}\|_{\nabla^2 F(w_k^{(s)})}+\beta_g\|\Delta_0^{(s)}\|_{\nabla^2 F(w^{(s)}_k)}. 
    \end{align*}
    Hence applying item 2. of \cref{lem:sksvrg_local_basic_lemmas} twice, we conclude
    \begin{align*}
        \|\Delta_k^{(s)}-p^{(s)}_k\|_{\nabla^2 F(w_\star)} \leq (1+\varepsilon_0)\left[(\varepsilon_0+\beta_g)\|\Delta_k^{(s)}\|_{\nabla^2 F(w_\star)}+\beta_g\|\Delta_0^{(s)}\|_{\nabla^2 F(w_\star)}
        \right],
    \end{align*}
     with probability at least $1-\frac{\delta}{(k+1)^2}$.
\end{proof}

Next we have the following result, which shows (with high probability) the distance to the optimum of the iterate actually computed by \cref{alg:sksvrg} is decreasing in the $\nabla^2 F(w_\star)$-norm. 
In particular, this implies $w_{k+1}^{(s)}$ remains in $\Nstar$.  
\begin{lemma}
\label[lemma]{lem:sksvrg_local_next_iterate_bound}
    Instate the hypotheses of \cref{lem:sksvrg_local_dist_norm}. Then the following items hold with probability at least $1-\frac{\delta}{(k+1)^2}$.
    \begin{enumerate}
        \item $\|\Delta_{k+1}^{(s)}\|_{\nabla^2F(w_\star)}\leq \frac{7}{12}\|\Delta^{(s)}_k\|_{\nabla^2F(w_\star)}+\frac{1}{4} \|\Delta^{(s)}_0\|_{\nabla^2F(w_\star)}$
        \item $w^{(s)}_{k+1} \in \Nstar$.
    \end{enumerate}
\end{lemma}
\begin{proof}
    To start off, we apply the triangle inequality to reach 
    \begin{align*}
        & \|\Delta_{k+1}^{(s)}\|_{\nabla^2F(w_\star)} \leq \|\Delta_{k}^{(s)}-p^{(s)}_{k}\|_{\nabla^2F(w_\star)}+\|p^{(s)}_{k}-\tilde p^{(s)}_{k} \|_{\nabla^2F(w_\star)}. \\
    \end{align*}
    The first term may be bounded by invoking \cref{lem:sksvrg_local_dist_norm}, so for now we focus on bounding the second term, which represents the error from computing an approximate Newton step. 
    To this end, observe that 
    \begin{align*}
        \|p^{(s)}_{k}-\tilde p^{(s)}_{k} \|_{\nabla^2F(w_\star)} & = \left\|\nabla^2 F(w^{(s)}_k)^{1/2}(p^{(s)}_{k}-\tilde p^{(s)}_{k})\right\|_{\nabla^2 F(w^{(s)}_k)^{-1/2}\nabla^2F(w_\star)\nabla^2 F(w^{(s)}_k)^{-1/2}} \\
        &\overset{(1)}{\leq} (1+\varepsilon_0)\left\|\nabla^2 F(w^{(s)}_k)^{1/2}(p^{(s)}_{k}-\tilde p^{(s)}_{k})\right\|\\
        &= (1+\varepsilon_0)\left\|\nabla^2 F(w^{(s)}_k)^{1/2}(\nabla^2 F(w^{(s)}_k)^{-1}-P^{-1})\nabla^2 F(w^{(s)}_k)^{1/2}(\nabla^2 F(w^{(s)}_k)^{1/2}p^{(s)}_{k})\right\|\\
        &\overset{(2)}{\leq} \frac{21}{6}\varepsilon_0\|p_k^{(s)}\|_{\nabla^2 F(w_k^{(s)})}\leq 4\varepsilon_0\|p_k^{(s)}\|_{\nabla^2 F(w_\star)} \\
        &\leq 4\varepsilon_0\left(\|\Delta_{k}^{(s)}\|_{\nabla^2F(w_\star)}+\|\Delta_{k}^{(s)}-p^{(s)}_k\|_{\nabla^2F(w_\star)}\right),
    \end{align*}
    where $(1)$ uses item 2 of \cref{lem:sksvrg_local_basic_lemmas}, and $(2)$ uses item 4 of \cref{lem:sksvrg_local_basic_lemmas}, along with $\varepsilon_0\leq 1/6$.
    Combining the preceding upper bound with our initial bound, we reach
    \begin{align*}
        \|\Delta_{k+1}^{(s)}\|_{\nabla^2F(w_\star)} &\leq (1+4\varepsilon_0)\|\Delta_{k}^{(s)}-p^{(s)}_{k}\|_{\nabla^2F(w_\star)}+4\varepsilon_0\|\Delta_{k}^{(s)}\|_{\nabla^2F(w_\star)}.
    \end{align*}
    Now, invoking \cref{lem:sksvrg_local_dist_norm} to bound $\|\Delta_{k}^{(s)}-p^{(s)}_{k}\|_{\nabla^2F(w_\star)}$, we find with probability at least $1-\delta/(k+1)^2$, that
    \begin{align*}
        \|\Delta_{k+1}^{(s)}\|_{\nabla^2F(w_\star)} &\leq \left[(1+\varepsilon_0)(1+4\varepsilon_0)(\varepsilon_0+\beta_g)+4\varepsilon_0\right]\|\Delta_k^{(s)}\|_{\nabla^2 F(w_\star)}+(1+\varepsilon_0)(1+4\varepsilon_0)\beta_g\|\Delta_0^{(s)}\|_{\nabla^2 F(w_\star)}.
    \end{align*}
    Using $\varepsilon_0 \leq \frac{1}{6}$, the preceding display becomes
    \begin{align*}
        \|\Delta_{k+1}^{(s)}\|_{\nabla^2F(w_\star)} &\leq \left(\frac{1}{3}+2\beta_g\right)\|\Delta_k^{(s)}\|_{\nabla^2 F(w_\star)}+2\beta_g\|\Delta_0^{(s)}\|_{\nabla^2 F(w_\star)}.
    \end{align*}
    Now, let us set $\beta_g = \frac{1}{8}$, then the preceding display simplifies to
    \[
     \|\Delta_{k+1}^{(s)}\|_{\nabla^2F(w_\star)} \leq \frac{7}{12}\|\Delta_k^{(s)}\|_{\nabla^2 F(w_\star)}+\frac{1}{4}\|\Delta_0^{(s)}\|_{\nabla^2 F(w_\star)},
    \]
    which proves the first claim. 
    To see the second claim, note that 
    \[
    \max\{\|\Delta_k^{(s)}\|_{\nabla^2 F(w_\star)},\|\Delta_0^{(s)}\|_{\nabla^2 F(w_\star)}\} \leq \varepsilon_0 \nu^{3/2}/(2M),
    \]
    as $w_k^{(s)},\hat w^{(s)}\in \Nstar$.
    Hence, the second claim follows immediately from the first. 
\end{proof}
    \begin{lemma}[One-stage analysis]
    \label[lemma]{lem:sksvrg_local_one_stage_bnd}
        Let $\hat w^{(s)}\in \Nstar$. Run \cref{alg:sksvrg} with $m = 6$ inner iterations and gradient batchsize satisfies $b_{g} =\bigO\left(\tau_\star^{\nu}(\Nstar)\log\left(\frac{m+1}{\delta}\right)\right)$. 
        Then with probability at least $1-\delta$:
        \begin{enumerate}
            \item $\hat w^{(s+1)} \in \mathcal N_{\frac{2}{3}\varepsilon_0}(w_\star)$.
            \item  $F(\hat w^{(s+1)})-F(w_\star)\leq \frac{2}{3}(F(\hat w^{(s)})-F(w_\star))$.
        \end{enumerate}
    \end{lemma}
    \begin{proof}
    As $\hat{w}^{(s)}\in \Nstar$, it follows by union bound that the conclusions of \cref{lem:sksvrg_local_next_iterate_bound} hold for all $w_{k}^{(s)}$, where  $k\in \{0,\dots, m-1\}$, with probability at least 
    \[
    1-\sum^{m-1}_{k=0}\frac{\delta}{(m+1)^2} = 1-\frac{m}{(m+1)^2}\delta\geq 1-\delta. 
    \]
    Consequently, 
    \[
    \|\Delta_{m}^{(s)}\|_{\nabla^2F(w_\star)}\leq \frac{7}{12}\|\Delta^{(s)}_{m-1}\|_{\nabla^2F(w_\star)}+\frac{1}{4}\|\Delta^{(s)}_{0}\|_{\nabla^2F(w_\star)}.
    \] 
    We now recurse on the previous display, and use $m = 6> \frac{\log(1/15)}{\log(7/12)}$, to reach 
    \begin{align*}
        &\|\Delta_{m}^{(s)}\|_{\nabla^2F(w_\star)} \leq \left(\frac{7}{12}\right)^{m}\|\Delta^{(s)}_0\|_{\nabla^2F(w_\star)}+\left(\sum_{k=0}^{m-1}\left(\frac{7}{12}\right)^{k}\right)\frac{1}{4}\|\Delta^{(s)}_0\|_{\nabla^2F(w_\star)}\\
        &\leq \frac{1}{15}\|\Delta^{(s)}_0\|_{\nabla^2F(w_\star)}+\frac{1}{4(1-\frac{7}{12})}\|\Delta^{(s)}_0\|_{\nabla^2F(w_\star)}\\ 
        &= \left(\frac{1}{15}+\frac{3}{5}\right)\|\Delta^{(s)}_0\|_{\nabla^2F(w_\star)} \leq 
       \frac{2}{3}\|\Delta^{(s)}_0\|_{\nabla^2F(w_\star)}.
    \end{align*}
    Hence $\hat w^{(s+1)} = w_m^{(s)}\in \mathcal N_{\frac{2}{3}\varepsilon_0}(w_\star)$. 
    Using this last inclusion, and applying upper quadratic regularity, followed by lower quadratic regularity, we find 
    \begin{align*}
        &F(\hat w^{(s+1)})-F(w_\star)\leq \frac{1+\varepsilon_0}{2}\|\Delta_{m}^{(s)}\|^2_{\nabla^2F(w_\star)} \leq \frac{1+\varepsilon_0}{2} \frac{4}{9} \|\Delta^{(s)}_0\|^2_{\nabla^2F(w_\star)} \\
        &\leq (1+\varepsilon_0)^2\frac{4}{9}\left(F(\hat w^{(s)})-F(w_\star)\right) \leq \frac{2}{3}\left(F(\hat w^{(s)})-F(w_\star)\right),
    \end{align*}
    as desired. 
\end{proof}
\subsubsection{Proof of \cref{thm:sksvrg_loc_con}}
We now come to the proof of \cref{thm:sksvrg_loc_con}, which is reduced to union bounding over the conclusion of \cref{lem:sksvrg_local_one_stage_bnd}. 
\begin{proof}
    By hypothesis, we may invoke \cref{lem:sksvrg_local_one_stage_bnd} to conclude the output of the first outer iteration satisfies
    \[
    F(\hat w^{(1)})-F(w_\star)\leq \frac{2}{3}(F(w_0)-F(w_\star)), \quad \text{and}~ \hat w^{(1)}\in \Nstar
    \]
    with probability at least $1-\delta$.
    Hence we can apply \cref{lem:sksvrg_local_one_stage_bnd} again to $\hat w^{(s)}$, the output of the second outer iteration. 
    Repeating this logic for all the remaining outer iterations, we find by union bound, that with probability at least $1-s\delta$,
    \[
    F(\hat w^{(s)})-F(w_\star) \leq \left(\frac{2}{3}\right)^{s}(F(w_0)-F(w_\star))\leq \epsilon.
    \]
    The theorem now follows by scaling $\delta$ down by $ 3\log\left((F(w_0)-F(w_\star))/\epsilon\right)$.
\end{proof}

\ifpreprint
\subsection{SketchySAGA}
\label{subsec:sksaga_convergence}
In this subsection, we prove convergence of SketchySAGA (\Cref{alg:sksaga}).
The convergence analysis is based on a Lyapunov function argument, with the Lyapunov function defined in the main text. 
We note, that unlike the proofs of the previous theorems, the intuition behind the convergence argument is less obvious, which is often the case with Lyapunov function based arguments.  
\subsubsection{Notation}
Throughout this section, we shall need the following quantities:
\[
\gammatlm = (1 - \zeta)\gammaMin, \quad \kappa_P = n + \frac{\mathcal L_P}{\gammatlm}.
\]

\subsubsection{Preliminary lemmas}
We start by observing that all the $F_i's$ are quadratically regular.
\begin{lemma}[$F_i$'s are quadratically regular]
\label[lemma]{lem:sksaga_Fi_qr}
    Instate the hypotheses of \cref{assm:SmoothSC}. Then $F_i$ is quadratically regular for each $i \in [n]$.
\end{lemma}
\begin{proof}
    By definition $F_i(w) = f_i(w)+\frac{\nu}{2}\|w\|^2$, hence it smooth and strongly convex. Thus, $F_i(w)$ is quadratically regular by \cref{proposition:QuadRegConds}.
\end{proof}

For the remaining lemmas, we will always assume that \cref{assm:SmoothSC} holds, so the $F_i$'s will immediately be quadratically regular by \cref{lem:sksaga_Fi_qr}.

\begin{lemma}\label[lemma]{lemma:sketchysaga_uqr_ub}
    Instate the hypotheses of \cref{assm:SmoothSC,assm:GoodPrecond}. Then for any iteration $k$ and for all $x, y \in \R^p$,
    \begin{align*}
        F_i(y) \leq F_i(x) + \langle \nabla F_i(x), y - x \rangle + \frac{\Lmc_P}{2} \| y - x \|_{P_k}^2.
    \end{align*}
\end{lemma}

\begin{proof}
    As each $F_i$ is $\gammaMax$-upper quadratically regular, and $\mathcal L_P = (1+\zeta)\tau_\star^{\nu}\gammaMax$, the claim immediately follows from \cref{lem:approx_hess_good_model}.
\end{proof}

\begin{lemma}\label[lemma]{lemma:sketchysaga_lqr_ub}
    Instate the hypotheses of \cref{assm:SmoothSC,assm:GoodPrecond}. Then for any iteration $k$ and for all $x, y \in \R^p$,
    \begin{align*}
        F(y) \leq F(x) + \langle \nabla F(x), y - x \rangle + \frac{1}{2 (1 - \zeta) \gamma_\ell} \| \nabla F(y) - \nabla F(x) \|_{P_k^{-1}}^2. 
    \end{align*}
\end{lemma}

\begin{proof}
    Define $g_x(z) = F(z) - F(x) - \langle \nabla F(x), z - x \rangle$. We have $\nabla g_x(x) = \nabla F(x) - \nabla F(x) = 0$. Since $g$ is convex, $0 = g_x(x) = \min_v g_x(v)$. Furthermore, since $F$ is $\gamma_\ell$ lower-quadratically regular, $g_x$ is $\gamma_\ell$ lower-quadratically regular.

    Therefore,
    \begin{align*}
        0 = \min_v g_x(v) &\geq \min_v [g_x(y) + \langle \nabla g_x(y), v - y \rangle + \frac{\gamma_\ell}{2} \| v - y \|_{\nabla^2 F(w_k)}^2] \\
        &\geq \min_v [g_x(y) + \langle \nabla g_x(y), v - y \rangle + \frac{(1 - \zeta) \gamma_\ell}{2} \| v - y \|_{P_k}^2] \\
        &= g_x(y) - \frac{1}{2 (1 - \zeta) \gamma_\ell} \| \nabla g_x(y) \|_{P_k^{-1}}^2 \\
        &= F(y) - F(x) - \langle \nabla F(x), y - x \rangle - \frac{1}{2 (1 - \zeta) \gamma_\ell} \| \nabla F(y) - \nabla F(x) \|_{P_k^{-1}}^2,
    \end{align*}
    where the first inequality follows from $g_x$ being $\gamma_\ell$ lower-quadratically regular, and the second inequality follows from \cref{assm:GoodPrecond}.
    Rearranging the final display yields the result.
\end{proof}

\begin{lemma}\label[lemma]{lemma:sketchysaga_inner_product_ub}
    Instate the hypotheses of \cref{assm:SmoothSC,assm:GoodPrecond}. 
    Then for any iteration $k$ and for all $w \in \R^p$,
    \begin{align*}
         \langle \nabla F(w), w_\star - w \rangle \leq {} & \frac{\Lmc_P - \gammatlm}{\Lmc_P} [F(w_\star) - F(w)] - \frac{\gammatlm}{2} \| w_\star - w \|_{P_k}^2 \\
         & - \frac{1}{2 \Lmc_P n} \sum_{i = 1}^n \| \nabla F_i (w_\star) - \nabla F_i(w) \|_{P_k^{-1}}^2.
    \end{align*}
\end{lemma}

\begin{proof}
    Define $g_i(w) = F_i(w) - \frac{(1 - \zeta) \gammaMin}{2} \| w \|_{P_k}^2 = F_i(w) - \frac{\gammatlm}{2} \| w \|_{P_k}^2$. Since $F_i$ is $\gamma_{\ell_i}$ lower-quadratically regular and \cref{assm:GoodPrecond} holds, $g_i$ is convex. Furthermore, since $F_i$ is $\gammaMax$ upper-quadratically regular, we use \cref{lemma:sketchysaga_uqr_ub} to conclude $g_i$ is $(\Lmc_P - \gammatlm)$-smooth in the $P_k$-norm, i.e.,
    \begin{align*}
        g_i(x) \leq g_i(y) + \langle \nabla g_i(y), x - y \rangle + \frac{\Lmc_P - \gammatlm}{2} \| x - y \|_{P_k}^2
    \end{align*}
    for all $x, y \in \R^p$.
    Hence by \cref{lem:A-norm-smooth}, we have 
    \begin{align}\label{eq:sketchysaga_inner_product_ub-lb}
        g_i(x) \geq g_i(y) + \langle \nabla g_i(y), x - y \rangle + \frac{1}{2 (\Lmc_P - \gammatlm)} \| \nabla g_i(x) - \nabla g_i(y) \|_{P_k^{-1}}^2.
    \end{align}
    Substituting the definition of $g_i$ into \cref{eq:sketchysaga_inner_product_ub-lb}, we find
    \begin{align*}
        F_i(x) \geq {} & \frac{\gammatlm}{2} \|x\|_{P_k}^2 + F_i(y) - \frac{\gammatlm}{2} \|y\|_{P_k}^2 + \langle \nabla F_i(y) - \gammatlm P_k y, x - y \rangle \\
        & + \frac{1}{2 (\Lmc_P - \gammatlm)} \| \nabla F_i(x) - \nabla F_i(y) - \gammatlm P_k (x - y) \|_{P_k^{-1}}^2 \\
        = {} & F_i(y) + \langle \nabla F_i(y), x - y \rangle + \frac{\gammatlm}{2}\left[  \|x\|_{P_k}^2 -  \|y\|_{P_k}^2 - 2 \langle P_k y, x - y \rangle \right] \\
        & + \frac{1}{2 (\Lmc_P - \gammatlm)} \| \nabla F_i(x) - \nabla F_i(y) \|_{P_k^{-1}}^2 + \frac{\gammatlm}{\Lmc_P - \gammatlm} \langle \nabla F_i(x) - \nabla F_i(y), y - x \rangle \\
        & + \frac{(\gammatlm)^2}{2 (\Lmc_P - \gammatlm)} \|x - y\|_{P_k}^2 \\
        = {} & F_i(y) + \langle \nabla F_i(y), x - y \rangle + \frac{1}{2 (\Lmc_P - \gammatlm)} \| \nabla F_i(x) - \nabla F_i(y) \|_{P_k^{-1}}^2 \\
        & + \frac{\gammatlm \Lmc_P}{2 (\Lmc_P - \gammatlm)} \|y - x\|_{P_k}^2 + \frac{\gammatlm}{\Lmc_P - \gammatlm} \langle \nabla F_i(x) - \nabla F_i(y), y - x \rangle.
    \end{align*}
    Setting $x = w_\star$, $y = w$, averaging over the $F_i$'s, and using $\nabla F(w_\star) = 0$, we obtain the claimed result.
\end{proof}

\begin{lemma}\label[lemma]{lemma:sketchysaga_variance_reduction}
    Instate the hypotheses of \cref{assm:SmoothSC,assm:GoodPrecond}.
    Suppose each $F_i$ is $\gamma_{u_i}$ upper-quadratically regular. Then for any iteration $k$, and for all $\psi^i \in \R^p$, 
    \begin{align*}
        \frac{1}{n} \sum_{i = 1}^n \| \nabla F_i(\psi^i) - \nabla F_i(w_\star) \|_{P_k^{-1}}^2 \leq 2 \Lmc_P \left[ \frac{1}{n} \sum_{i = 1}^n F_i(\psi^i) - F(w_\star) - \frac{1}{n} \sum_{i = 1}^n \langle \nabla F_i(w_\star), \psi^i - w_\star \rangle \right].
    \end{align*}
\end{lemma}

\begin{proof}
    Using \cref{lemma:sketchysaga_uqr_ub} and \cref{lem:A-norm-smooth}, we have
    \begin{align*}
        F_i(y) \geq F_i(x) + \langle \nabla F_i(x), y - x \rangle + \frac{1}{2 \Lmc_P} \| \nabla F_i(y) - \nabla F_i(x) \|_{P_k^{-1}}^2
    \end{align*}
    for all $x, y \in \R^p$.
    Rearranging this inequality and using $y = \psi^i$ and $x = w_\star$ gives
    \begin{align*}
        \| \nabla F_i(\psi^i) - \nabla F_i(w_\star) \|_{P_k^{-1}}^2 \leq 2 \Lmc_P [F_i(\psi^i) - F_i(w_\star) - \langle \nabla F_i(w_\star), \psi^i - w_\star \rangle].
    \end{align*}
    Summing over $i = 1, 2, \ldots, n$ and dividing by $n$ yields
    \begin{align*}
        \frac{1}{n} \sum_{i = 1}^n \| \nabla F_i(\psi^i) - \nabla F_i(w_\star) \|_{P_k^{-1}}^2 \leq 2 \Lmc_P \left[ \frac{1}{n} \sum_{i = 1}^n F_i(\psi^i) - F(w_\star) - \frac{1}{n} \sum_{i = 1}^n \langle \nabla F_i(w_\star), \psi^i - w_\star \rangle \right].
    \end{align*}
\end{proof}

\begin{lemma}\label[lemma]{lemma:sketchysaga_iteration_distance}
    For any $\psi_k^i, w_\star, w_k$ and $\alpha > 0$, with $w_{k+1}$ defined as in \cref{alg:sksaga}, we have
    \begin{align*}
        \E \|w_{k+1} - w_k\|_{P_k}^2 \leq {} & \eta^2 (1 + \alpha^{-1}) \E \| \nabla F_j(\psi_k^j) - \nabla F_j(w_\star) \|_{P_k^{-1}}^2  \\
        & + \eta^2 (1 + \alpha) \E \| \nabla F_j(w_k) - \nabla F_j(w_\star) \|_{P_k^{-1}}^2 - \eta^2 \alpha \|\nabla F(w_k) \|_{P_k^{-1}}^2.
    \end{align*}
\end{lemma}

\begin{proof}
    By the definition of the SketchySAGA update,
    \begin{align*}
        & \E \|w_{k+1} - w_k\|_{P_k}^2 \\
        = {} & \E \left\| P_k^{-1} \left( -\frac{\eta}{n} \sum_{i = 1}^n \nabla F_i(\psi_k^i) + \eta [\nabla F_j(\psi_k^j) - \nabla F_j(w_k)] \right) \right\|_{P_k}^2 \\
        = {} & \E \left\| -\frac{\eta}{n} \sum_{i = 1}^n \nabla F_i(\psi_k^i) + \eta [\nabla F_j(\psi_k^j) - \nabla F_j(w_k)] \right\|_{P_k^{-1}}^2 \\
        = {} & \eta^2 \E \left\| \left( \nabla F_j (\psi_k^j) - \nabla F_j(w_\star) - \frac{1}{n} \sum_{i = 1}^n \nabla F_i(\psi_k^i) \right) - (\nabla F_j(w_k) - \nabla F_j(w_\star) - \nabla F(w_k))\right\|_{P_k^{-1}}^2 \\
        & + \eta^2 \| \nabla F(w_k) \|_{P_k^{-1}}, 
    \end{align*}
    where the final equality follows from using $\E[\nabla F_j (\psi_k^j) - \frac{1}{n} \nabla F_j (\psi_k^j) - \nabla F_j(w_k)] = -\nabla F(w_k)$ and the identity $\E \| X \|_A^2 = \E \| X - \E X \|_A^2 + \| \E X\|_A^2$. Using the facts (1) $\|x+y\|_A^2 \leq (1 + \alpha^{-1}) \|x\|_A^2 + (1 + \alpha) \|y\|_A^2$, (2) $\E \| X - \E X \|_A^2 \leq \E \| X \|_A^2$, and (3) $\E \| X \|_A^2 = \E \| X - \E X \|_A^2 + \| \E X\|_A^2$, we find 
    \begin{align*}
        &\eta^2 \E \left\| \left( \nabla F_j (\psi_k^j) - \nabla F_j(w_\star) - \frac{1}{n} \sum_{i = 1}^n \nabla F_i(\psi_k^i) \right) - (\nabla F_j(w_k) - \nabla F_j(w_\star) - \nabla F(w_k))\right\|_{P_k^{-1}}^2 \\
        & + \eta^2 \| \nabla F(w_k) \|_{P_k^{-1}} \\
        \stackrel{\textrm{(1)}}{\leq} {} & \eta^2 (1 + \alpha^{-1}) \E \left\| \nabla F_j (\psi_k^j) - \nabla F_j(w_\star) - \frac{1}{n} \sum_{i = 1}^n \nabla F_i(\psi_k^i) \right\|_{P_k^{-1}}^2 \\
        & + \eta^2 (1 + \alpha) \E \left\| \nabla F_j(w_k) - \nabla F_j(w_\star) - \nabla F(w_k) \right\|_{P_k^{-1}}^2 + \eta^2 \| \nabla F(w_k) \|_{P_k^{-1}} \\ 
        \stackrel{\textrm{(2)}}{\leq} {} & \eta^2 (1 + \alpha^{-1}) \E \left\| \nabla F_j (\psi_k^j) - \nabla F_j(w_\star) \right\|_{P_k^{-1}}^2 \\
        & + \eta^2 (1 + \alpha) \E \left\| \nabla F_j(w_k) - \nabla F_j(w_\star) - \nabla F(w_k) \right\|_{P_k^{-1}}^2 + \eta^2 \| \nabla F(w_k) \|_{P_k^{-1}} \\
        \stackrel{\textrm{(3)}}{=} {} & \eta^2 (1 + \alpha^{-1}) \E \| \nabla F_j(\psi_k^j) - \nabla F_j(w_\star) \|_{P_k^{-1}}^2 \\
        & + \eta^2 (1 + \alpha) \E \| \nabla F_j(w_k) - \nabla F_j(w_\star) \|_{P_k^{-1}}^2 - \eta^2 \alpha \|\nabla F(w_k) \|_{P_k^{-1}}^2.
    \end{align*}
\end{proof}

\begin{lemma}[Contraction lemma]
\label[lemma]{lem:sksaga_contraction}
    Instate the hypotheses of \cref{assm:SmoothSC,assm:GoodPrecond}. Then
    \begin{align*}
        \E_k [T_{k+1}] \leq \left( 1 - \frac{1}{2 (n + \kappa_P)} \right) \beta_k T_k,
    \end{align*}
    where $\kappa_P = \frac{\Lmc_P}{\gammatlm} = \frac{\Lmc_P}{(1 - \zeta) \gammaMin}$ and $\beta_k$ is defined as in \cref{prop:stable_precond}.
\end{lemma}

\begin{proof}
    By the definition of $T_j$ for general $j$, we can write $T_{k+1}$ as 
    \begin{align*}
        T_{k+1} = B_{k+1} \left( \frac{1}{n} \sum_{i = 1}^n F_i(\psi_{k+1}^i) - F(w_\star) - \frac{1}{n} \sum_{i = 1}^n \langle \nabla F_i(w_\star), \psi_{k+1}^i - w_\star \rangle + c \| w_{k+1} - w_\star \|_{P_k}^2 \right).
    \end{align*}
    The expectations of the first and third terms of $T_{k+1}$ simplify as follows:
    \begin{align*}
        \E \left[ \frac{1}{n} \sum_{i = 1}^n F_i(\psi_{k+1}^i) \right] &= \frac{1}{n} F(w_k) + \left( 1 - \frac{1}{n} \right) \frac{1}{n} \sum_{i = 1}^n F_i(\psi_k^i) \\
        \E \left[ \frac{1}{n} \sum_{i = 1}^n \langle \nabla F_i(w_\star), \psi_{k+1}^i - w_\star \rangle \right] &= \left( 1 - \frac{1}{n} \right) \frac{1}{n} \sum_{i = 1}^n \langle \nabla F_i(w_\star), \psi_k^i - w_\star \rangle.
    \end{align*}
    We now bound the expectation of the final term in $T_{k+1}$,
    \begin{align*}
        & c \E \| w_{k+1} - w_\star \|_{P_k}^2 \\
        = {} & c \E \|w_k - w_\star + w_{k+1} - w_k\|_{P_k}^2 \\
        = {} & c \| w_k - w_\star \|_{P_k}^2 + 2c \E \langle w_{k+1} - w_k, w_k - w_\star \rangle_{P_k} + c \E \| w_{k+1} - w_k\|_{P_k}^2 \\
        = {} & c \| w_k - w_\star \|_{P_k}^2 + 2c \langle -\eta P_k^{-1} \nabla F(w_k), w_k - w_\star \rangle_{P_k} + c \E \| w_{k+1} - w_k\|_{P_k}^2 \\
        = {} & c \| w_k - w_\star \|_{P_k}^2 - 2c \eta \langle \nabla F(w_k), w_k - w_\star \rangle + c \E \| w_{k+1} - w_k\|_{P_k}^2 \\
        \leq {} & c \| w_k - w_\star \|_{P_k}^2 - 2c \eta \langle \nabla F(w_k), w_k - w_\star \rangle - c \eta^2 \alpha \| \nabla F(w_k) \|_{P_k^{-1}}^2 \\
        & + c \eta^2 (1 + \alpha^{-1}) \E \| \nabla F_j(\psi_k^j) - \nabla F_j(w_\star) \|_{P_k^{-1}}^2 + c \eta^2 (1 + \alpha) \E \| \nabla F_j(w_k) - \nabla F_j(w_\star) \|_{P_k^{-1}}^2,
    \end{align*}
    where the third equality follows from the definition of the SketchySAGA update, and the inequality follows from applying \cref{lemma:sketchysaga_iteration_distance}. We will set the value of $\alpha$ later in the proof to reach the claimed result.
    Using \cref{lemma:sketchysaga_inner_product_ub} to bound $- 2c \eta \langle \nabla F(w_k), w_k - w_\star \rangle$ and \cref{lemma:sketchysaga_variance_reduction} to bound $\E \| \nabla F_j(\psi_k^j) - \nabla F_j(w_\star) \|_{P_k^{-1}}^2$, we obtain
    \begin{align*}
        c \E \| w_{k+1} - w_\star \|_{P_k}^2 \leq {} & c(1 - \eta \gammatlm) \| w_k - w_\star \|_{P_k}^2 + \left( (1 + \alpha) c \eta^2 - \frac{c \eta}{\Lmc_P} \right) \E \| \nabla F_j(w_k) - \nabla F_j(w_\star) \|_{P_k^{-1}}^2 \\
        & - \frac{2 c \eta (\Lmc_P - \gammatlm)}{\Lmc_P} [F(w_k) - F(w_\star)] - c \eta^2 \alpha \| \nabla F(w_k) \|_{P_k^{-1}}^2 \\
        & + 2 c \eta^2 (1 + \alpha^{-1}) \Lmc_P \mathsf{Breg}_k
    \end{align*}
    where we have defined $\mathsf{Breg}_k \coloneqq \frac{1}{n} \sum_{i = 1}^n F_i(\psi_k^i) - F(w_\star) - \frac{1}{n} \sum_{i = 1}^n \langle \nabla F_i(w_\star), \psi_k^i - w_\star \rangle$ for notational convenience.
    We now apply backwards stable evolution (\cref{prop:stable_precond}) and \cref{lemma:sketchysaga_lqr_ub} with $x = w_\star,y = w_k, \gamma = \gamma_{\ell_{\min}}$, from which we find
\begin{align*}
      c \E \| w_{k+1} - w_\star \|_{P_k}^2 \leq {} & \frac{c (1 - \eta \gammatlm)}{\beta_k} \| w_k - w_\star \|_{P_{k-1}}^2 \\
        & + \left( (1 + \alpha) c \eta^2 - \frac{c \eta}{\Lmc_P} \right) \E \| \nabla F_j(w_k) - \nabla F_j(w_\star) \|_{P_k^{-1}}^2 \\
        & + \left(-\frac{2 c \eta (\Lmc_P - \gammatlm)}{\Lmc_P} - 2 c \eta^2 \gammatlm \alpha \right)[F(w_k) - F(w_\star)] \\
        & + 2 c \eta^2 (1 + \alpha^{-1}) \Lmc_P \mathsf{Breg}_k.
\end{align*}
Now we combine the bounds we have computed for the terms in $\E T_{k+1}$:
\begin{align*}
        \E T_{k+1} \leq {} & B_{k+1} \left( \frac{1}{n} F(w_k) + \left( 1 - \frac{1}{n} \right) \frac{1}{n} \sum_{i = 1}^n F_i(\psi_k^i) - F(w_\star) - \left( 1 - \frac{1}{n} \right) \frac{1}{n} \sum_{i = 1}^n \langle \nabla F_i(w_\star), \psi_k^i - w_\star \rangle \right) \\
        & + B_{k+1} \frac{c (1 - \eta \gammatlm)}{\beta_k} \| w_k - w_\star \|_{P_{k-1}}^2 + B_{k+1} \left( (1 + \alpha) c \eta^2 - \frac{c \eta}{\Lmc_P} \right) \E \| \nabla F_j(w_k) - \nabla F_j(w_\star) \|_{P_k^{-1}}^2 \\
        & + B_{k+1} \left(-\frac{2 c \eta (\Lmc_P - \gammatlm)}{\Lmc_P} - 2 c \eta^2 \gammatlm \alpha \right)[F(w_k) - F(w_\star)] \\
        & + 2 B_{k+1} c \eta^2 (1 + \alpha^{-1}) \Lmc_P \mathsf{Breg}_k.
\end{align*}
Applying the definitions of $T_k$ and $B_k$ in the previous display, we obtain
\begin{align*}
        \E T_{k+1} \leq {} & \frac{B_{k+1}}{n} F(w_k) - \frac{B_{k+1}}{n} F(w_\star) + \frac{B_{k+1}}{B_k} T_k \\
        & - \frac{B_{k+1}}{n} \left( \frac{1}{n} \sum_{i = 1}^n F_i(\psi_k^i) - F(w_\star) - \frac{1}{n} \sum_{i = 1}^n \langle \nabla F_i(w_\star), \psi_k^i - w_\star \rangle \right) \\
        & + c \left( (1 - \eta \gammatlm) \frac{B_{k+1}}{\beta_k} - B_{k+1} \right) \| w_k - w_\star \|_{P_{k-1}}^2 \\
        & + B_{k+1} \left( (1 + \alpha) c \eta^2 - \frac{c \eta}{\Lmc_P} \right) \E \| \nabla F_j(w_k) - \nabla F_j(w_\star) \|_{P_k^{-1}}^2 \\
        & + B_{k+1} \left(-\frac{2 c \eta (\Lmc_P - \gammatlm)}{\Lmc_P} - 2 c \eta^2 \gammatlm \alpha \right)[F(w_k) - F(w_\star)] \\
        & + 2 B_{k+1} c \eta^2 (1 + \alpha^{-1}) \Lmc_P \mathsf{Breg}_k \\
        = {} & \frac{B_{k+1}}{B_k} T_k + c \left( (1 - \eta \gammatlm) \frac{B_{k+1}}{\beta_k} - B_{k+1} \right) \| w_k - w_\star \|_{P_{k-1}}^2 \\
        & + B_{k+1} \left( (1 + \alpha) c \eta^2 - \frac{c \eta}{\Lmc_P} \right) \E \| \nabla F_j(w_k) - \nabla F_j(w_\star) \|_{P_k^{-1}}^2 \\
        & + B_{k+1} \left(\frac{1}{n} -\frac{2 c \eta (\Lmc_P - \gammatlm)}{\Lmc_P} - 2 c \eta^2 \gammatlm \alpha \right)[F(w_k) - F(w_\star)] \\
        & + B_{k+1} \left(2c \eta^2 (1 + \alpha^{-1}) \Lmc_P  - \frac{1}{n} \right) \mathsf{Breg}_k \\
        = {} & \beta_k T_k + c (1 - \eta \gammatlm - \beta_k) B_k  \| w_k - w_\star \|_{P_{k-1}}^2 \\
        & + B_{k+1} \left( (1 + \alpha) c \eta^2 - \frac{c \eta}{\Lmc_P} \right) \E \| \nabla F_j(w_k) - \nabla F_j(w_\star) \|_{P_k^{-1}}^2 \\
        & + B_{k+1} \left(\frac{1}{n} -\frac{2 c \eta (\Lmc_P - \gammatlm)}{\Lmc_P} - 2 c \eta^2 \gammatlm \alpha \right)[F(w_k) - F(w_\star)] \\
        & + B_{k+1} \left(2c \eta^2 (1 + \alpha^{-1}) \Lmc_P  - \frac{1}{n} \right) \mathsf{Breg}_k \\
        = {} & \left( 1 - \frac{1}{\kappa} \right)\beta_k T_k + c (1 - \eta \gammatlm - \beta_k + \frac{1}{\kappa} \beta_k) B_k  \| w_k - w_\star \|_{P_{k-1}}^2 \\
        & + B_{k+1} \left( (1 + \alpha) c \eta^2 - \frac{c \eta}{\Lmc_P} \right) \E \| \nabla F_j(w_k) - \nabla F_j(w_\star) \|_{P_k^{-1}}^2 \\
        & + B_{k+1} \left(\frac{1}{n} -\frac{2 c \eta (\Lmc_P - \gammatlm)}{\Lmc_P} - 2 c \eta^2 \gammatlm \alpha \right) [F(w_k) - F(w_\star)] \\
        & + B_{k+1} \left(2c \eta^2 (1 + \alpha^{-1}) \Lmc_P  - \frac{1}{n} + \frac{1}{\kappa} \right) \mathsf{Breg}_k
\end{align*}
Notice that the four quantities $\| w_k - w_\star \|_{P_{k-1}}^2, \E \| \nabla F_j(w_k) - \nabla F_j(w_\star) \|_{P_k^{-1}}^2, F(w_k) - F(w_\star)$, and $\mathsf{Breg}_k$ in the preceding set of inequalities are all non-negative. By setting $\eta = \frac{1}{2(\gammatlm n + \Lmc_P)}$, $c = \frac{1}{2 \eta (1 - \eta \gammatlm) n}, \kappa = \frac{1}{\eta \gammatlm}$, and $\alpha = \frac{2 \gammatlm n + \Lmc_P}{\Lmc_P}$, we guarantee that the coefficients corresponding to these four quantities are all $\leq 0$. Therefore, 
\begin{align*}
        \E_k T_{k+1} \leq \left( 1 - \frac{1}{\kappa} \right) \beta_k T_k = \left( 1 - \frac{1}{2 (n + \kappa_P)} \right) \beta_k T_k,
\end{align*}
where $\kappa_P = \frac{\Lmc_P}{\gammatlm} = \frac{\Lmc_P}{(1 - \zeta) \gammaMin}$.
\end{proof}

\subsubsection{SketchySAGA convergence: Proof of \Cref{thm:conv_sksaga}}
\begin{proof}
    By the contraction lemma, we have 
    \[
    \E_k [T_{k}] \leq \left( 1 - \frac{1}{2(n + \kappa_P)}\right)\beta_{k-1}T_{k-1}.  
    \]
    Hence taking the total expectation over all iterations and recursing, we reach
    \[
    \E [T_{k}]\leq \left( 1 - \frac{1}{2(n + \kappa_P)}\right)^{k}\left(\prod^{k-1}_{j=0} \beta_{j}\right)T_0\leq  \left( 1 - \frac{1}{2(n + \kappa_P)}\right)^{k}B_PT_0.
    \]
    
\end{proof}
Now, by definition $T_k\geq c\|w_k-w_\star\|_{P_k}^2$ for all $k$, so we obtain
\[
\E \|w_k-w_\star\|_{P_k}^2 \leq \left( 1 - \frac{1}{2(n + \kappa_P)}\right)^{k}\frac{B_P}{c}T_0.
\]
Now using $c = \frac{1}{2\eta(1-\eta \gammatlm)n}$ and $\eta = \frac{1}{2(n\gammatlm+\mathcal L_P)}$, we find
\[
\frac{1}{c} = 2\eta(1-\eta \gammatlm)n = \frac{n}{n\gammatlm+\mathcal L_P}\left(1-\frac{\gammatlm}{2(n \gammatlm+\mathcal L_P)}\right) = \frac{(2n-1)\gammatlm+\mathcal L_P}{2(n\gammatlm+\mathcal L_P)^2}n \leq \frac{n}{n\gammatlm+\mathcal L_P}
\]
Therefore
\begin{align*}
    \E \|w_k-w_\star\|_{P_k}^2  \leq \left( 1 - \frac{1}{2(n + \kappa_P)}\right)^{k}B_P\left(\| w_0 - w_\star \|^2-\frac{n}{n\gammatlm+\mathcal L_P}\left[F(w_0) - F(w_\star)\right]\right).
\end{align*}
Hence it immediately follows that
\[
\E \|w_k-w_\star\|_{P_k}^2  \leq \epsilon,
\]
after $k = \bigO((n + \kappa_P) \log \left( \frac{1}{\epsilon} \right))$ iterations.

\subsection{SketchyKatyusha}
\label{subsec:skkat_convergence}
In this subsection, we prove \cref{thm:conv_skkat}, which shows linear convergence of SketchyKatyusha (\cref{alg:skkat}).
The argument is based on the Lyapunov function introduced in the main text. 
Similar to the previous theorems, we break down the proof of \cref{thm:conv_skkat} into a series of lemmas, which together lead to an easy proof of the theorem.
Unfortunately, similar to the analysis of classical accelerated gradient descent, the intuition behind the overall proof is somewhat opaque. 

\subsubsection{Notation}
Throughout the analysis, we shall make use of the following quantities:
\[
 \tilde \gamma_\ell \coloneqq (1-\zeta)\gamma_\ell,
\]
\[
\kappa_P \coloneqq \frac{\mathcal L_P}{\tilde \gamma_\ell} = \frac{\qbar}{1-\zeta}, \quad \sigma \coloneqq \frac{1}{\kappa_P}.
\]

\subsubsection{Preliminary lemmas}
\begin{lemma}[Variance bound]
\label[lemma]{lem:skkat_var_bnd}
Let $v_k = \widehat{\nabla}F(x_k) - \widehat{\nabla} F(y_k) + \nabla F(y_k)$ be the variance-reduced stochastic gradient at iteration $k$. Then
\begin{align*}
\E \|v_k-\nabla F(x_k)\|^2_{P^{-1}_k} &\leq 2 \mathcal L_P[F(y_k) - F(x_k) -\langle \nabla F(x_k),y_k-x_k\rangle].
\end{align*}
\end{lemma}
\begin{proof}
    The proof is via direct calculation:
    \begin{align*}
        & \E\|v_k-\nabla F(x_k)\|_{P^{-1}_k}^2 \overset{(1)}{=} \E\left\|\widehat \nabla F(y_k)-\widehat \nabla F(x_k)-\E\left[\widehat \nabla F(y_k)-\widehat\nabla F(x_k)\right]\right\|^2_{P^{-1}_k}\overset{(2)}{\leq} \\ 
        & \E\|\widehat \nabla F(y_k)-\widehat \nabla F(x_k)\|_{P^{-1}_k}^2 \overset{(3)}{\leq} 2 \mathcal L_P[F(y_k) - F(x_k) -\langle \nabla F(x_k),y_k-x_k\rangle],
    \end{align*}
    where (1) simply substitutes in the value of $v_k$, (2) uses $\E\|X-\E[X]\|^2_A\leq \E\|X\|_A^2$, and (3) invokes \cref{prop:precond_grad_var}.
\end{proof}

The next two lemmas are somewhat mysterious in their content.
However, they serve a concrete purpose, which is to couple various quantities involved with the Lyapunov function together in a way that will be useful for establishing contraction in \cref{lem:skkat_contract}. 
Given that this is the role they play, we adopt the language of \cite{allenzhu2018katyusha}(for whom two similar lemmas appear), and refer to \cref{lem:skkat_coup_1} and \cref{lem:skkat_coup_2} as coupling lemmas.

\begin{lemma}[Coupling Lemma 1]
\label[lemma]{lem:skkat_coup_1}
    Let $\sigma = \frac{\tilde\gamma_\ell}{\mathcal L_P}$. Then for any iteration $k$, the following inequality holds
    \begin{equation*}
        \langle v_k, w_\star-x_k \rangle +\frac{\tilde\gamma_\ell}{2}\|x_k-w_\star\|_{P_k}^2\geq \frac{\mathcal L_P}{2\eta}\|z_k-z_{k+1}\|_{P_k}^2+
        \frac{1}{\xi_k}\mathcal{Z}_{k+1}-\frac{1}{(1+\eta \sigma)}\mathcal Z_{k}.
    \end{equation*}
\end{lemma}

\begin{proof}
Observe the following equality
     \begin{align*}
         &\langle v_k, x_k-w_\star \rangle = \langle P_k^{-1} v_k, x_k-w_\star \rangle_{P_k} = \mathcal L_P \sigma \langle x_k-z_{k+1},x_k-w_\star\rangle_{P_k} +\frac{\mathcal L_P}{\eta}\langle z_k-z_{k+1},z_{k+1}-w_\star \rangle_{P_k}\\
         &=  \tilde\gamma_\ell \langle x_k-z_{k+1},x_k-w_\star\rangle_{P_k} +\frac{\mathcal L_P}{\eta}\langle z_k-z_{k+1},z_{k+1}-w_\star \rangle_{P_k},\\
     \end{align*}
     where the third equality uses 
     \[
     \frac{\eta}{\mathcal L_P}P_k^{-1}v_k = \eta\sigma\left(x_k-z_{k+1}\right)+z_k-z_{k+1}.
     \] 
     Invoking \cref{lem:A-norm-pol} twice with $A = P_k$, we reach 
     \begin{align*}
         &\langle v_k, x_k-w_\star \rangle = \frac{\tilde \gamma_\ell}{2}\left(\|x_k-w_\star\|_{P_k}^2-\|x_k-z_{k+1}\|_{P_k}^2-\|z_{k+1}-w_\star\|_{P_k}^2\right)\\
         &+\frac{\mathcal L_P}{2\eta}\left(\|z_k-w_\star\|_{P_k}^2-\|z_k-z_{k+1}\|_{P_k}^2-\|z_{k+1}-w_\star\|_{P_k}^2\right)\\
         &\leq \frac{\tilde \gamma_\ell}{2}\|x_k-w_\star\|_{P_k}^2 +\frac{\mathcal L_P}{2\eta}\left(\|z_k-w_\star\|_{P_k}^2-(1+\eta\sigma)\|z_{k+1}-w_\star\|_{P_k}^2\right)-\frac{\mathcal L_P}{2\eta}\|z_{k}-z_{k+1}\|_{P_k}^2.
     \end{align*}
     Using $\mathcal Z_k = \frac{\mathcal L_P(1+\eta \sigma)}{2\eta}\|z_k-w_\star\|_{P_k}^2$,  $\|z_{k+1}-w_\star\|_{P_{k+1}}^2\leq \xi_k \|z_{k+1}-w_\star\|^2_{P_k} $ rearranging, we find
     \begin{align*}
          \langle v_k, w_\star-x_k \rangle +\frac{\tilde \gamma_\ell}{2}\|x_k-w_\star\|_{P_k}^2&\geq \frac{\mathcal L_P}{2\eta}\|z_{k}-z_{k+1}\|_{P_k}^2 +\frac{1}{\xi_k}\mathcal Z_{k+1} - \frac{1}{1+\eta \sigma}\mathcal Z_k.
     \end{align*}
\end{proof}

\begin{lemma}[Coupling Lemma 2]
\label[lemma]{lem:skkat_coup_2}
    For each iteration $k$, we have
    \begin{align*}
        &\frac{1}{\theta_1}\left(F(w_{k+1})-F(x_k)\right)-\frac{\theta_2}{2\theta_1\mathcal L_P}\|v_k-\nabla F(x_k)\|^2_{P^{-1}_k} \leq \frac{\mathcal L_P}{2\eta}\|z_{k+1}-z_k\|^2_{P_k} +\langle v_k,z_{k+1}-z_k\rangle.
    \end{align*}
\end{lemma}
\begin{proof}
By construction of the algorithm, the following equality holds
\[
\frac{\mathcal{L}_P}{2\eta}\|z_{k+1}-z_k\|^2_{P_k}+\langle v_k,z_{k+1}-z_k\rangle = \frac{1}{\theta_1}\left(\frac{\mathcal L_P}{2\eta \theta_1}\|w_{k+1}-x_k\|_{P_k}^2+\langle v_k,w_{k+1}-x_k\rangle\right).
\]
Now, adding and subtracting zero twice, we reach 
\begin{align*}
    &\frac{\mathcal{L}_P}{2\eta}\|z_{k+1}-z_k\|^2_{P_k}+\langle v_k,z_{k+1}-z_k\rangle = 
    \\ & \frac{1}{\theta_1}\left(\frac{\mathcal L_P}{2\eta\theta_1}\|w_{k+1}-x_k\|^2_{P_k}+\langle \nabla F(x_k), w_{k+1}-x_k\rangle+\langle v_k-\nabla F(x_k),w_{k+1}-x_k\rangle\right)\\
    &= \frac{1}{\theta_1}\left(\frac{\mathcal L_P}{2}\|w_{k+1}-x_k\|^2_{P_k}+\langle \nabla F(x_k), w_{k+1}-x_k\rangle\right)\\ 
    &+\frac{\mathcal L_P}{2\theta_1}\left(\frac{1}{\eta \theta_1}-1\right)\|w_{k+1}-x_k\|_{P_k}^2+\frac{1}{\theta_1}\langle v_k-\nabla F(x_k),w_{k+1}-x_k\rangle.\\
\end{align*}
Now combining that $P$ is $\zeta$-spectral approximation with upper quadratic regularity, we have
\[
F(w_{k+1})-F(x_k)\leq \langle \nabla F(x_k), w_{k+1}-x_k\rangle+\frac{\mathcal L_P}{2}\|w_{k+1}-x_k\|^2_{P_k}.
\]
Using the preceding relation, we find
\begin{align*}
    & \frac{\mathcal{L}_P}{2\eta}\|z_{k+1}-z_k\|^2_{P_k}+\langle v_k,z_{k+1}-z_k\rangle\geq \\
    & \frac{1}{\theta_1}(F(w_{k+1})-F(x_k))+\frac{\mathcal L_P}{2\theta_1}\left(\frac{1}{\eta \theta_1}-1\right)\|w_{k+1}-x_k\|_{P_k}^2+\frac{1}{\theta_1}\langle v_k-\nabla F(x_k),w_{k+1}-x_k\rangle\\
    &\overset{(1)}{\geq} \frac{1}{\theta_1}(F(w_{k+1})-F(x_k))-\frac{\eta\theta_1}{2 \mathcal L_P(1-\eta\theta_1)}\|v_k-\nabla F(x_k)\|^2_{P^{-1}_k}\\
    &\overset{(2)}{=} \frac{1}{\theta_1}\left(F(w_{k+1})-F(x_k)\right)-\frac{\theta_2}{2\theta_1\mathcal L_P}\|v_k-\nabla F(x_k)\|^2_{P^{-1}_k}.
\end{align*}
Where $(1)$ uses Lemma \ref{lemma:GenYoungIneq} with $c_1 = \frac{1}{\theta_1}, c_2 = \frac{\mathcal L_P(1-\eta \theta_1)}{\eta \theta^2_1}$, and $(2)$ uses $\theta_2 = \frac{\eta \theta_1}{1-\eta \theta_1}$.
\end{proof}

With the coupling lemmas in hand, we now establish that conditioned on the $k$th iterate, the Lyapunov function contracts in expectation. 
\begin{lemma}[Contraction Lemma]
\label[lemma]{lem:skkat_contract}
    \begin{align*}
        \E_k \left[\mathcal W_{k+1}+\mathcal Y_{k+1}+\frac{1}{\xi_k}\mathcal Z_{k+1}\right] \leq \left(1-\frac{\pi\theta_1}{1+\theta_1}\right)\mathcal W_k +\left(1-\theta_1(1-\theta_2)\right)\mathcal Y_k +\frac{1}{1+\eta \sigma }\mathcal Z_k
    \end{align*} 
\end{lemma}
\begin{proof}
By lower quadratic regularity of $F$ and $P_k$ being a $\zeta$-spectral approximation, we have 
\begin{align*}
    &F(w_\star) \geq F(x_k)+\langle \nabla F(x_k),w_\star-y_k\rangle +\frac{\tilde\gamma_\ell}{2}\|x_k-w_\star\|_{P_k}^2.
\end{align*}
Rearranging, we reach
\begin{align*}
    F(w_\star) &\overset{(1)}{\geq} F(x_k)+\frac{\tilde \gamma_\ell}{2}\|x_k-w_\star\|_{P_k}^2+\langle \nabla F(x_k),w_\star-z_k+z_k-x_k\rangle \\
    &\overset{(2)}{=} F(x_k)+\frac{\tilde \gamma_\ell}{2}\|x_k-w_\star\|_{P_k}^2+\langle \nabla F(x_k),w_\star -z_k\rangle +\frac{\theta_2}{\theta_1}\langle \nabla F(x_k),x_k-y_k\rangle \\
    &+\frac{1-\theta_1-\theta_2}{\theta_1}\langle \nabla F(x_k),x_k-w_k\rangle \\
    &\overset{(3)}{\geq} F(x_k)+\frac{\tilde \gamma_\ell}{2}\|x_k-w_\star\|_{P_k}^2 +\langle \nabla F(x_k),w_\star -z_k\rangle\\
    &+\frac{\theta_2}{\theta_1}\langle \nabla F(x_k),x_k-y_k\rangle +\frac{1-\theta_1-\theta_2}{\theta_1}(F(x_k)-F(w_k)).\\
\end{align*}
Where $(1)$ adds and subtracts zero, $(2)$ uses 
\[
z_k-x_k = -\frac{(1-\theta_1)}{\theta_1}x_k-\frac{\theta_2}{\theta_1}y_k-\frac{1-\theta_1-\theta_2}{\theta_2}w_k = -\frac{\theta_2}{\theta_1}(x_k-y_k)-\frac{(1-\theta_1-\theta_2)}{\theta_1}(x_k-w_k),
\]
and $(3)$ invokes convexity of $F$.
Now, observe that
\begin{align*}
    \frac{\tilde \gamma_\ell}{2}\|x_k-w_\star\|_{P_k}^2+\langle \nabla F(x_k),w_\star -z_k\rangle &= \E_k\ \left[\frac{\tilde \gamma_\ell}{2}\|x_k-w_\star\|_{P_k}^2+\langle v_k,w_\star-z_{k+1}\rangle +\langle v_k, z_{k+1}-z_k\rangle\right].
\end{align*}
Hence, we can use the preceding relation with \cref{lem:skkat_coup_1} and \cref{lem:skkat_coup_2}, to reach 
\begin{align*}
        &F(w_\star)\geq f(x_k)+\frac{\theta_2}{\theta_1}\langle \nabla F(x_k),x_k-w_k\rangle +\frac{1-\theta_1-\theta_2}{\theta_1}(F(x_k)-F(w_k))+ \E_k\left[\frac{1}{\xi_k}\mathcal Z_{k+1}-\frac{1}{1+\eta \sigma} \mathcal Z_k\right]\\
        &+\E_k\left[\langle v_k,z_{k+1}-z_k\rangle +\frac{\mathcal L_P}{2\eta}\|z_k-z_{k+1}\|_{P_k}^2\right] \\
        &\geq F(x_k)+\frac{\theta_2}{\theta_1}\langle \nabla F(x_k),x_k-w_k\rangle +\frac{1-\theta_1-\theta_2}{\theta_1}(F(x_k)-F(w_k))+ \E_k\left[\frac{1}{\xi_k}\mathcal Z_{k+1}-\frac{1}{1+\eta \sigma} \mathcal Z_k\right]\\
        &+ \E_k\left[\frac{1}{\theta_1}\left(F(w_{k+1})-F(x_k)\right)-\frac{\theta_2}{2\theta_1\mathcal L_P}\|v_k-\nabla F(x_k)\|^2_{P^{-1}_k}\right]. \\
\end{align*}
Now, invoking the variance bound in \cref{lem:skkat_var_bnd}, we find 
\begin{align*}
        F(w_\star) &\geq F(x_k)+\frac{\theta_2}{\theta_1}\langle \nabla F(x_k),x_k-y_k\rangle +\frac{1-\theta_1-\theta_2}{\theta_1}(F(x_k)-F(w_k))+ \E_k\left[\frac{1}{\xi_k}\mathcal Z_{k+1}-\frac{1}{1+\eta \sigma} \mathcal Z_k\right]\\
        &+\frac{1}{\theta_1}\E_k\left[F(w_{k+1})-F(x_k)\right]-\frac{\theta_2}{\theta_1}\left(F(y_k)-F(x_k)\right)-\frac{\theta_2}{\theta_1}\langle\nabla F(x_k),x_k-y_k\rangle.\\
\end{align*}
Hence,
\begin{align*}
    F(w_\star)&\geq F(x_k)+\frac{1-\theta_1-\theta_2}{\theta_1}(F(x_k)-F(w_k))-\frac{1}{1+\eta \sigma} \mathcal Z_k-\frac{\theta_2}{\theta_1}\left(F(y_k)-F(x_k)\right) \\
    &+\E_k\left[\frac{1}{\theta_1}(F(w_{k+1})-F(x_k))+\frac{1}{\xi_k}\mathcal Z_{k+1}\right].
\end{align*}
Rearranging, and using that $w_\star$ is the optimum, we reach
\begin{align*}
    &\frac{1-\theta_1-\theta_2}{\theta_1}(F(w_k)-F(y_k))+\frac{\theta_2}{\theta_1}\left(F(y_k)-F(x_k)\right)+\frac{1}{1+\eta \sigma} \mathcal Z_k +\left(\frac{1}{\theta_1}-1\right)(F(x_k)-F(w_\star))\\
    &\geq \E_k\left[\frac{1}{\theta_1}(F(w_{k+1})-F(w_\star))+\frac{1}{\xi_k}\mathcal Z_{k+1}\right].
\end{align*}

Adding and subtracting $F(w_\star)$ in the first two terms in parentheses on the left handside, combining the resulting terms, and recalling the definitions of $\mathcal W_k, \mathcal Y_k, \mathcal Z_k$, we obtain
    \begin{align*}
        \E_k\left[\mathcal W_{k+1}+\frac{1}{\xi_k}\mathcal Z_{k+1}\right] \leq (1-\theta_1-\theta_2)\mathcal W_k+\frac{\pi}{1+\theta_1}\mathcal Y_k +\frac{1}{1+\eta \sigma}\mathcal Z_k.
    \end{align*}
Now, using the relation $\E_k[\mathcal Y_{k+1}] = (1-\pi)\mathcal Y_k+\theta_2(1+\theta_1)\mathcal W_k$, which holds by construction of the algorithm, we conclude
\begin{align*}
        \E_k\left[\mathcal W_{k+1}+ \mathcal Y_{k+1}+\frac{1}{\xi_{k}}\mathcal Z_{k+1}\right]\leq \left(1-\theta_1(1-\theta_2)\right)\mathcal W_k+\left(1-\frac{\pi\theta_1}{1+\theta_1}\right)\mathcal Y_k +\frac{1}{1+\eta \sigma }\mathcal Z_k. 
    \end{align*}
\end{proof}

\subsubsection{SketchyKatyusha convergence: Proof of \cref{thm:conv_skkat}}

\begin{proof}
    \cref{lem:skkat_contract} shows that
    \begin{align*}
        \E_k\left[\mathcal W_{k+1}+\mathcal Y_{k+1}+\mathcal Z_{k+1}\right] \leq \xi_k\alpha \left(\mathcal W_k+\mathcal Y_k +\mathcal Z_k\right),
    \end{align*}
    where $\alpha = \max\left\{\left(1-\theta_1(1-\theta_2)\right),\left(1-\frac{\pi\theta_1}{1+\theta_1}\right),\frac{1}{1+\eta\sigma}\right\}$.
    Hence, taking the total expectation over all iterations and recursing, we find 
    \[
    \E\left[\Psi_{k}\right]\leq \epsilon\Psi_0,
    \]
    after $k\geq \max\{\frac{1}{\theta_1(1-\theta_2)},\pi^{-1}(1+\theta_1^{-1}),1+(\eta\sigma)^{-1}\}\log\left(\frac{\mathcal E_P}{\epsilon}\right)$
    Now, recall $\sigma = 1/\kappa_P$ and $\pi = b_g/n$, so that $\theta_1 = \min\{\sqrt{\frac{n}{b_g\kappa_P}}\theta_2,1/2\}$, and $\pi = b_g/n$.
    We see the the analysis breaks down to two cases: $\kappa_P \leq \frac{n}{b_g}$ and $\kappa_P>\frac{n}{b_g}$.
    \paragraph{Case 1: $\left(\kappa_P \leq \frac{n}{b_g}\right)$}
    Let us set $\theta_2 = 1/2$. 
    Then for Case 1, the minimum defining $\theta_1$ is achieved by $1/2$. 
    Using this in conjunction with $\eta = \theta_1/[(1+\theta_2)\theta_1]$, we find
     \[
    \max\left\{\frac{1}{\theta_1(1-\theta_2)},\pi^{-1}(1+\theta_1^{-1}),1+(\eta\sigma)^{-1}\right\} = \max\left\{4,3\frac{n}{b_g},1+3/2\kappa_P\right\} = 3\frac{n}{b_g},
    \]
    Hence, as $1\leq \kappa_P\leq n/b_g$, we conclude
    \[
     \E \left[\Psi_k\right] \leq \epsilon\Psi_0
    \]
    after $k\geq \frac{3n}{b_g}\log\left(\frac{\mathcal E_P}{\epsilon}\right)$ iterations. 
    \paragraph{Case 2: $\left(\kappa_P \geq \frac{n}{b_g}\right)$}
    Once again set $\theta_2 = 1/2$.
    Given the hypothesis on $\kappa_P$, we have $\theta_1 = \sqrt{n/(b_g\kappa_P)}\theta_2$.
    Using the value of $\theta_1$, and performing some straightforward computations, we find 
    \begin{align*}
    &\max\left\{\frac{1}{\theta_1(1-\theta_2)},\pi^{-1}(1+\theta_1^{-1}),1+(\eta\sigma)^{-1}\right\} \\
    & = \max\left\{1/2\sqrt{\kappa_P b_g/n},n/b_g+2\sqrt{n\kappa_P/b_g},1+3/4\sqrt{n\kappa_P/b_g}\right\} \\
    & = n/b_g+2\sqrt{n\kappa_P/b_g}.
    \end{align*}
    
    Combining both cases, we see that with $\pi = b_g/n$, we have
    \[
        \E[\Psi_{k}] \leq \epsilon\Psi_0,
    \]
    after $k = \max\left\{3n/b_g,n/b_g+2\sqrt{n\kappa_P/b_g}\right\}\log(\frac{\mathcal E_P}{\epsilon})$ iterations. 
    Plugging in $\kappa_P = \qbar/(1-\zeta)$, we obtain the desired claim. 
\end{proof}
\else
\fi
\section{Experimental details}
\label{appndx:experiment_details}
Here we provide additional experimental details that were omitted from the main text.

\subsection{Computational resources}
The experiments in this paper are run on servers with the computational resources listed in \cref{table:computational_resources}.
Although these servers have 64-core CPUs, calls to BLAS are limited to 16 threads, i.e., we only use 16 cores for any given experiment.

\begin{table}[H]
\centering
    \begin{tabular}{C{4cm}C{3cm}C{3cm}}
        Experiment section & CPU & Memory (RAM) \\ \hline
        Performance experiments \newline (\cref{subsection:performance}) & 64-core Xeon E5-2698 v3 @ 2.30GHz & 1007.8 GB \\ \hline
        Showcase experiments \newline (\cref{subsection:showcase}) & 64-core Xeon E5-2698 v3 @ 2.30GHz & 755.8 GB \\ \hline
        Streaming experiments \newline (\cref{subsection:large_scale}) & 64-core Xeon E5-2698 v3 @ 2.30GHz & 755.8 GB \\ \hline
        Sensitivity study \newline (\cref{subsection:sensitivity}) & 64-core Xeon E5-2698 v3 @ 2.30GHz & 1007.8 GB \\ \hline
        Regularity study \newline (\cref{subsection:regularity}) & 64-core Xeon E5-2698 v3 @ 2.30GHz & 755.8 GB \\ \hline
    \end{tabular}
    \caption{Computational resources used for each set of experiments.}
    \label{table:computational_resources}
\end{table}

\subsection{Optimizer details}
\label{subsection:optimizer_appdx}
\paragraph{Initialization} All optimizers are initialized at $w_0 = 0$.

\paragraph{Default hyperparameters for SVRG/SAGA/L-Katyusha} The theoretical analyses of SVRG, SAGA, and L-Katyusha all yield recommended learning rates that lead to linear convergence.
In practical implementations such as scikit-learn \citep{pedregosa2011scikit}, these recommendations are used to compute the default learning rate. 
For SAGA, the theoretical learning rate is given by
\[
\eta = \max\left\{\frac{1}{3L_{\max}}, \frac{1}{2\left(L_{\max}+n_{\textrm{tr}}\mu\right)}\right\},
\]
where $L_{\max} \coloneqq \max_i L_i$ is the maximum smoothness constant among the $f_i$ s and $\mu$ is the strong convexity constant.

Now, standard computations show that the smoothness constants for least-squares and logistic regression satisfy 
\[
L_{\textrm{least-squares}} \leq \frac{1}{n_{\textrm{tr}}}\sum_{i=1}^{n_{\textrm{tr}}}\|a_i\|^2 \eqqcolon \hat{L}_{\textrm{avg}} \leq \max_i \|a_i\|^2 \eqqcolon \hat{L}_{\max},  
\]
\[
L_{\textrm{logistic}} \leq \frac{1}{4n_{\textrm{tr}}}\sum_{i=1}^{n_{\textrm{tr}}}\|a_i\|^2 \eqqcolon \hat{L}_{\textrm{avg}} \leq \frac{1}{4} \max_i \|a_i\|^2 \eqqcolon \hat{L}_{\max}.
\]

The scikit-learn software package uses the preceding upper-bound $\hat{L}_{\max}$ in place of $L_{\max}$ to set $\eta$ in their implementation of SAGA. 
However, $\hat{L}_{\max}$ can lead to a very conservative learning rate (especially if the data is poorly scaled), so we use $\hat{L}_{\textrm{avg}}$ to calculate the default learning rate instead.

The theoretical analysis of SVRG suggests a step-size of $\eta = \frac{1}{10\mathcal L}$, where $\mathcal L$ is the expected-smoothness constant. 
We have found this setting to pessimistic relative to the SAGA default, so we use the same default for SVRG as we do for SAGA. 
For L-Katyusha the hyperparameters $\theta_1$ and $\theta_2$ are controlled by how we specify $L^{-1}$, the reciprocal of the smoothness constant; we use $\hat{L}_{\textrm{avg}}$ in place of $L$ as the default. 

\paragraph{Grid search parameters (\cref{subsection:performance,subsection:subopt,subsection:showcase})} 
For ridge regression, we set $[10^{-3}, 10^{2}]$ as the search range for the learning rate in SGD, SVRG and SAGA, and $[10^{-2}, 10^{0}]$ as the search range for the smoothness parameter $L$ in L-Katyusha. 
For SLBFGS, we set the search range to be $[10^{-5}, 10^{0}]$ in order to have the same log-width as the search range for SGD, SVRG, and SAGA. 
In logistic regression, the search ranges for SGD/SVRG/SAGA, L-Katyusha, and SLBFGS become $[4 \cdot 10^{-3}, 4 \cdot 10^{2}], [2.5 \cdot 10^{-3}, 2.5 \cdot 10^{-1}]$, and $[4 \cdot 10^{-5}, 4 \cdot 10^{0}]$, respectively. 
The grid corresponding to each range samples $10$ equally spaced values in log space.

\paragraph{Grid search parameters (\cref{subsection:large_scale})} Instead of using a search range of $[4 \cdot 10^{-3}, 4 \cdot 10^{2}]$ for SGD/SAGA, we narrow the range to $[4 \cdot 10^{-2}, 4 \cdot 10^{1}]$ and sample $4$ equally spaced values in log space. 
The reason for reducing the search range and grid size is to reduce the total computational cost of running the experiments on the HIGGS and SUSY datasets. 
Furthermore, we find that $4 \cdot 10^{0}$ is the best learning rate for HIGGS and SUSY, while $4 \cdot 10^{1}$ leads to non-convergent behavior, meaning these search ranges are appropriate.

\paragraph{Additional hyperparameters} For SVRG and SLFBGS we perform a full gradient computation at every epoch. 
For SLFBGS we update the inverse Hessian approximation every epoch and set the Hessian batchsize to $\sqrt{n_{\mathrm{tr}}}$, which matches the Hessian batchsize hyperparameter in our proposed methods. 
In addition, we follow \cite{moritz2016linearly} and set the memory size of SLFBGS to $10$. 
For L-Katyusha, we initialize the update probability $p = b_g / n_{\mathrm{tr}}$ to ensure the average number of iterations between full gradient computations is equal to one epoch. 
We follow \cite{kovalev2020lkatyusha} and set $\mu$ equal to the $l^2$-regularization parameter, $\sigma = \frac{\mu}{L}, \theta_1 = \min\{\sqrt{2\sigma n_{\mathrm{tr}}/3}, \frac{1}{2} \}$, and $\theta_2 = \frac{1}{2}$. 
In \cref{subsection:performance,subsection:subopt,subsection:sensitivity,subsection:regularity} all algorithms use a batchsize of $256$ for computing stochastic gradients; in \cref{subsection:showcase,subsection:large_scale} all algorithms use a batchsize of 4,096 for computing stochastic gradients.

\subsection{Performance experiments (additional)}
\label{subsection:performance_exp_appdx}

\subsubsection{Datasets and preprocessing}
\label{subsubsection:performance_exp_data_appdx}
The ridge regression experiments in \Cref{subsubsection:ridge} use datasets from both LIBSVM and OpenML. 
The details regarding dimensions and random features are provided in \Cref{table:performance_least_squares_data_libsvm,table:performance_least_squares_data_openml}.
All OpenML datasets that are used do not have a designated training and test set, so we use a random 80/20 split.
We preprocess each dataset from LIBSVM to have unit row norm.
We standardize each dataset from OpenML, i.e., every feature has mean $0$ and standard deviation $1$ after preprocessing; we also standardize the labels if they are non-binary.
A bandwidth of $1$ is used in the cases where Gaussian random features are applied (performed after normalization/standardization).
For simplicity, we remove non-numeric features from the datasets Santander and Airlines\_DepDelay\_1M before running experiments.

\begin{table}[H]
\centering

    \begin{tabular}{C{4cm}C{1.5cm}C{1.5cm}C{1.5cm}C{4cm}}
        Dataset & $n_{\textrm{tr}}$ & $n_{\textrm{tst}}$ & $p$ & RF type/dimension \\ \hline
        E2006-tfidf &16,087 &3,308 &150,360 & \\
        YearPredictionMSD &463,715 &51,630 &90 &ReLU/4,367 \\
    \end{tabular}
    \caption{LIBSVM datasets used in ridge regression experiments from \Cref{subsubsection:ridge}. $n_{\textrm{tr}}$ denotes the number of training samples, $n_{\textrm{tst}}$ denotes the number of test samples, and $p$ denotes the number of features.}
    \label{table:performance_least_squares_data_libsvm}
\end{table}

\begin{table}[H]
\centering
\scriptsize
    \begin{tabular}{C{3cm}C{1.5cm}C{2cm}C{1.5cm}C{1cm}C{3cm}}
        Dataset & ID & Binary labels? & $n$ & $p$ &RF type/dimension \\ \hline
        Santander &42395 &\cmark &200,000 &201 &Gaussian/1,000 \\
        Jannis &44079 &\cmark &57,580 &54 &Gaussian/460 \\
        Yolanda &42705 &\xmark &400,000 &100 &Gaussian/1,000 \\
        MiniBoone &41150 &\cmark &130,064 &50 &Gaussian/1,000 \\
        Guillermo &41159 &\cmark &20,000 &4,296 & \\
        creditcard &1597 &\cmark &284,807 &30 &Gaussian/1,000 \\
        ACSIncome &43141 &\xmark &1,664,500 &11 &Gaussian/1,000 \\
        Medical-Appointment &43617 &\cmark &61,214 &18 &Gaussian/489 \\
        Airlines\_DepDelay\_1M &42721 &\xmark &1,000,000 &9 &Gaussian/1,000 \\
        Click\_prediction\_small &1218 &\cmark &1,997,410 &11 &Gaussian/1,000 \\
        mtp &405 &\xmark &4,450 &202 & \\
        elevators &216 &\xmark &16,599 &18 &Gaussian/132 \\
        ailerons &296 &\xmark &13,750 &40 &Gaussian/110 \\
        superconduct &44006 &\xmark &21,263 &79 &Gaussian/170 \\
        sarcos &44976 &\xmark &48,933 &21 &Gaussian/391 \\
    \end{tabular}
    \caption{OpenML datasets used in ridge regression experiments from \Cref{subsubsection:ridge}. ID is a unique identifier for the dataset on OpenML, $n$ is the number of samples before splitting into training/test sets, and $p$ is the number of features.}
    \label{table:performance_least_squares_data_openml}
\end{table}

All datasets used in the logistic regression experiments in \Cref{subsubsection:logistic} are from LIBSVM.
The details regarding dimensions, splits, and random features are provided in \Cref{table:performance_logistic_data}. 
The covtype.binary, german.numer, mushrooms, news20.binary, phishing, real-sim, sonar, and websample-unigram datasets do not have a designated training and test set, so we use a random 80/20 split.
We preprocess each dataset to have unit row norm.
A bandwidth of $1$ is used in the cases where Gaussian random features are applied (performed after normalization). 
We note that websample-unigram is listed on LIBSVM as having 16,609,143 features, but has only $254$ non-zero columns.

\begin{table}[H]
\centering
\scriptsize
    \begin{tabular}{C{2.7cm}C{1.5cm}C{1.5cm}C{1.5cm}C{2.7cm}}
        Dataset & $n_{\textrm{tr}}$ & $n_{\textrm{tst}}$ & $p$ & RF type/dimension \\ \hline
        a1a &1,605 &30,956 &123 & \\
        a2a &2,265 &30,296 &123 & \\
        a3a &3,185 &29,376 &123 & \\
        a4a &4,781 &27,780 &123 & \\
        a5a &6,414 &26,147 &123 & \\
        a6a &11,220 &21,341 &123 & \\
        a7a &16,100 &16,461 &123 & \\
        a8a &22,696 &9,865 &123 & \\
        a9a &32,561 &16,281 &123 & \\
        covtype.binary &581,012 & &54 &Gaussian/100 \\
        epsilon &400,000 &100,000 &2,000 & \\
        german.numer &1,000 & &24 &Gaussian/100 \\
        gisette &6,000 &1,000 &5,000 & \\
        HIGGS &10,500,000 &500,000 &28 &Gaussian/500 \\
        ijcnn1 &49,990 &91,701 &22 &Gaussian/2,500 \\
        madelon &2,000 &600 &500 & \\
        mushrooms &8,124 & &112 & \\
        news20.binary &19,996 & &1,355,191 & \\
        phishing &11,055 & &68 &Gaussian/100 \\
        rcv1.binary &20,242 &677,399 &47,236 & \\
        real-sim &72,309 & &20,958 & \\
        splice &1,000 &2,175 &60 &Gaussian/100 \\
        sonar &208 & &60 &Gaussian/100 \\
        SUSY &4,500,000 &500,000 &18 &Gaussian/1,000 \\
        svmguide3 &1,243 &41 &21 &Gaussian/100 \\
        w1a &2,477 &47,272 &300 & \\
        w2a &3,470 &46,279 &300 & \\
        w3a &4,912 &44,837 &300 & \\
        w4a &7,366 &42,383 &300 & \\
        w5a &9,888 &39,861 &300 & \\
        w6a &17,188 &32,561 &300 & \\
        w7a &24,692 &25,057 &300 & \\
        w8a &49,749 &14,951 &300 & \\
        websample-unigram &350,000 & &16,609,143 & \\
    \end{tabular}
    \caption{LIBSVM datasets used in logistic regression experiments from \Cref{subsubsection:logistic}. $n_{\textrm{tr}}$ denotes the number of training samples, $n_{\textrm{tst}}$ denotes the number of test samples, and $p$ denotes the number of features.}
    \label{table:performance_logistic_data}
\end{table}

\subsubsection{Performance plots for $\nu = 10^{-1}/n_{\mathrm{tr}}$}
\label{subsubsection:performance_mu_1e-1_appdx}
\cref{fig:prop_solved_least_squares_nu_1e-1,fig:prop_solved_logistic_nu_1e-1} compare PROMISE methods to competitor methods using the same setting as the performance experiments in \cref{subsection:performance}; the only difference is that $\nu = 10^{-1}/n_{\mathrm{tr}}$ instead of $10^{-2}/n_{\mathrm{tr}}$.
PROMISE methods still outperform the competition on both ridge and $l^2$-regularized logistic regression.

\begin{figure}[p]
    \centering
    \includegraphics[scale=0.45]{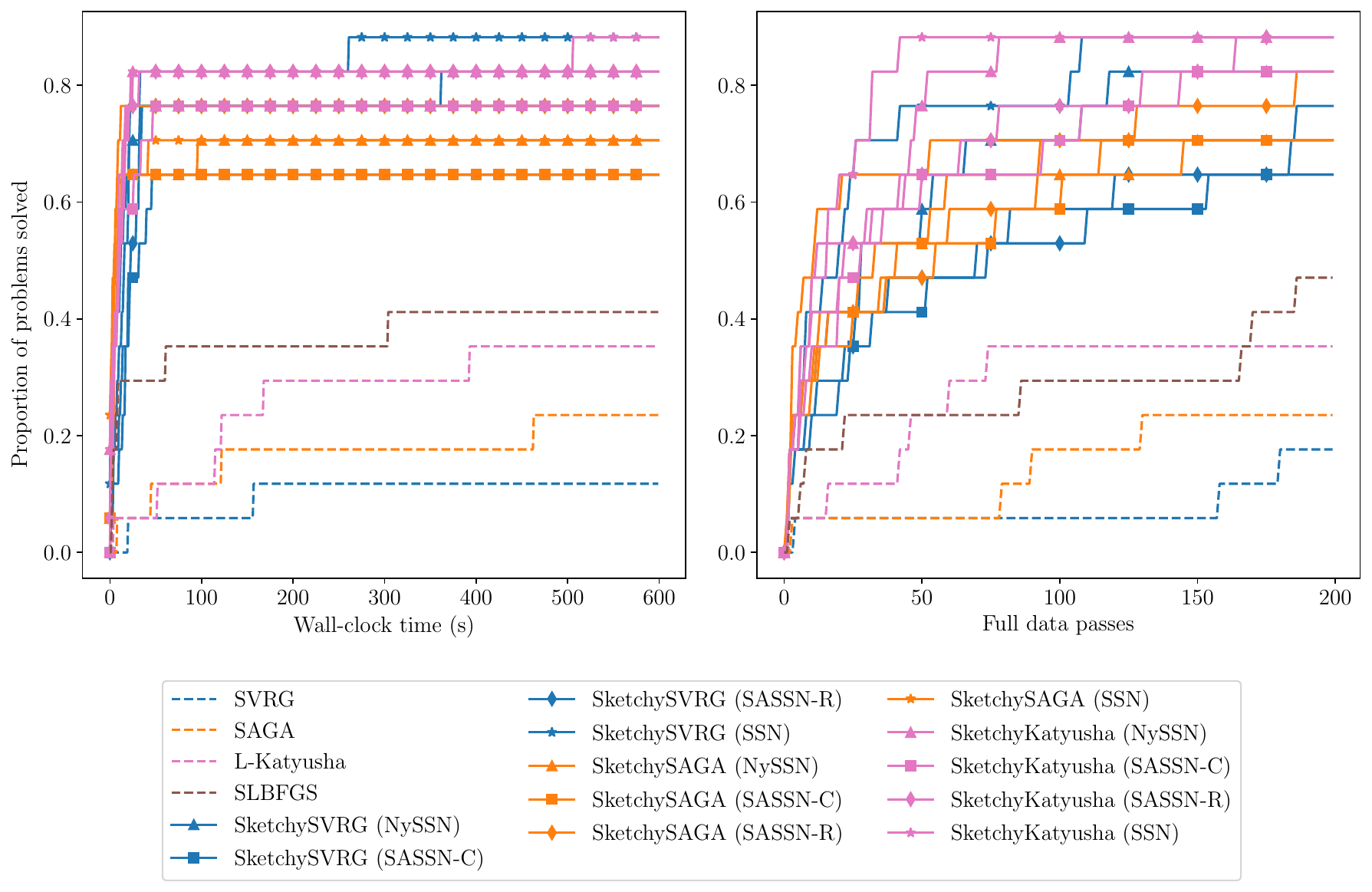}
    \caption{Proportion of ridge regression problems solved by our proposed methods and competitor methods when $\nu = 10^{-1}/n_{\mathrm{tr}}$.}
    \label{fig:prop_solved_least_squares_nu_1e-1}
\end{figure}
\begin{figure}[p]
    \centering
    \includegraphics[scale=0.45]{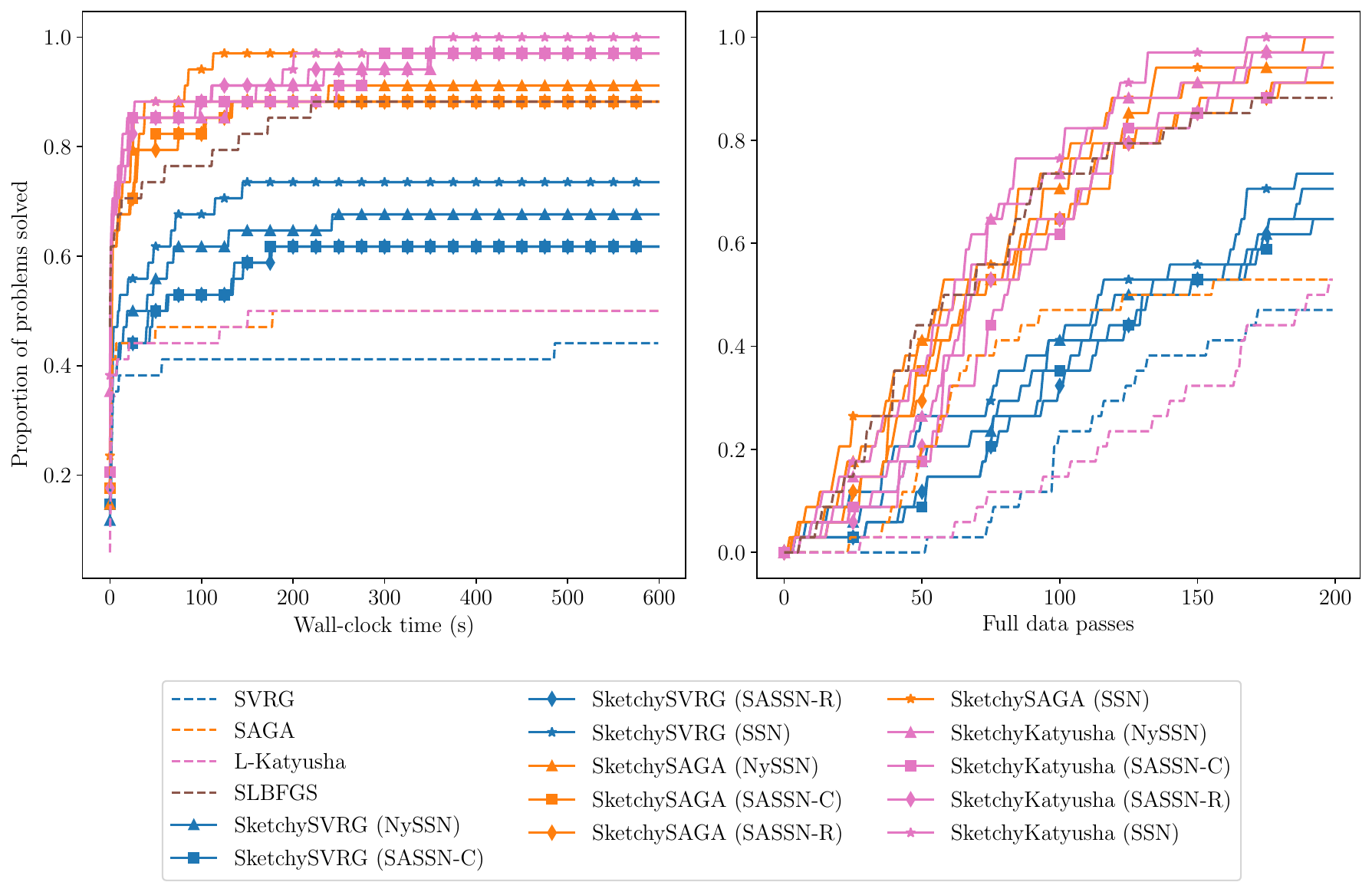}
    \caption{Proportion of $l^2$-regularized logistic regression problems solved by our proposed methods and competitor methods when $\nu = 10^{-1}/n_{\mathrm{tr}}$.}
    \label{fig:prop_solved_logistic_nu_1e-1}
\end{figure}

\subsubsection{Comparison of \ssn{} and \nyssn{} on sparse/dense logistic regression}
\label{subsubsection:logistic_sparse_dense_appdx}

Here we compare the performance of the \ssn{} and \nyssn{} preconditioners on both sparse and dense logistic regression problems.
We consider a problem to be sparse if $< 20 \%$ of the entries in the (preprocessed) data matrix are nonzero. According to this metric, the datasets a1a, a2a, a3a, a4a, a5a, a6a, a7a, a8a, a9a, mushrooms, news20, rcv1, real-sim, w1a, w2a, w3a, w4a, w5a, w6a, w7a, and w8a are sparse, while the datasets covtype, epsilon, german.numer, gisette, HIGGS, ijcnn1, madelon, phishing, splice, sonar, SUSY, svmguide3, and webspam are dense.

On sparse problems (\cref{fig:prop_solved_logistic_sparse,fig:ranking_logistic_sparse}), the \ssn{} preconditioner outperforms the \nyssn{} preconditioner with respect to both wall-clock time and full data passes.
On the other hand, the \nyssn{} preconditioner slightly outperforms the \ssn{} preconditioner on dense problems (\cref{fig:prop_solved_logistic_dense,fig:ranking_logistic_dense}).
These findings are in line with our preconditioner recommendations in \cref{table:precond_when_to_use}.

\begin{figure}[p]
    \centering
    \includegraphics[scale=0.5]{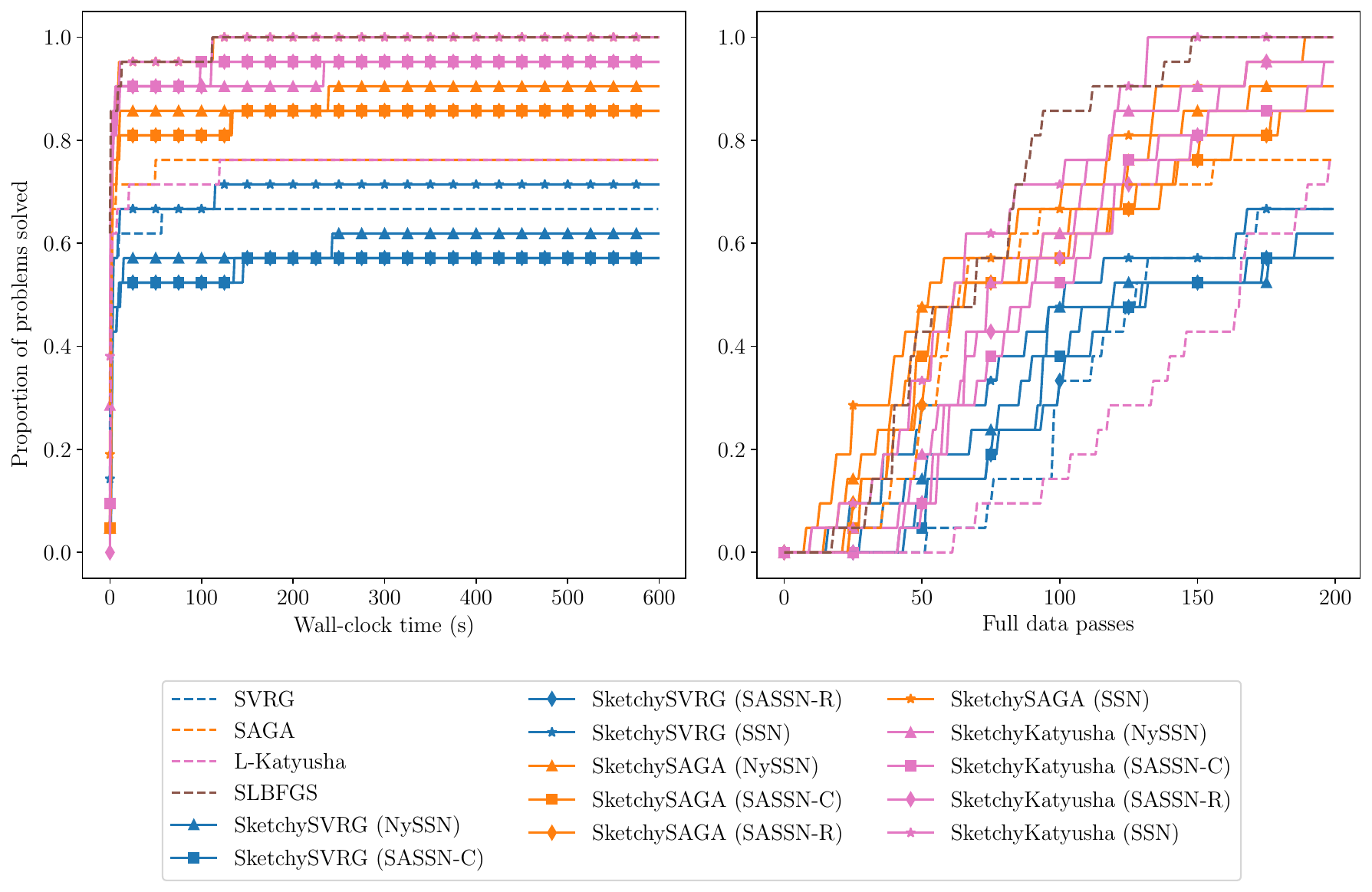}
    \caption{Proportion of $l^2$-regularized logistic regression problems, with sparse data, solved by our proposed methods and competitor methods.}
    \label{fig:prop_solved_logistic_sparse}
\end{figure}

\begin{figure}[p]
    \centering
    \includegraphics[scale=0.5]{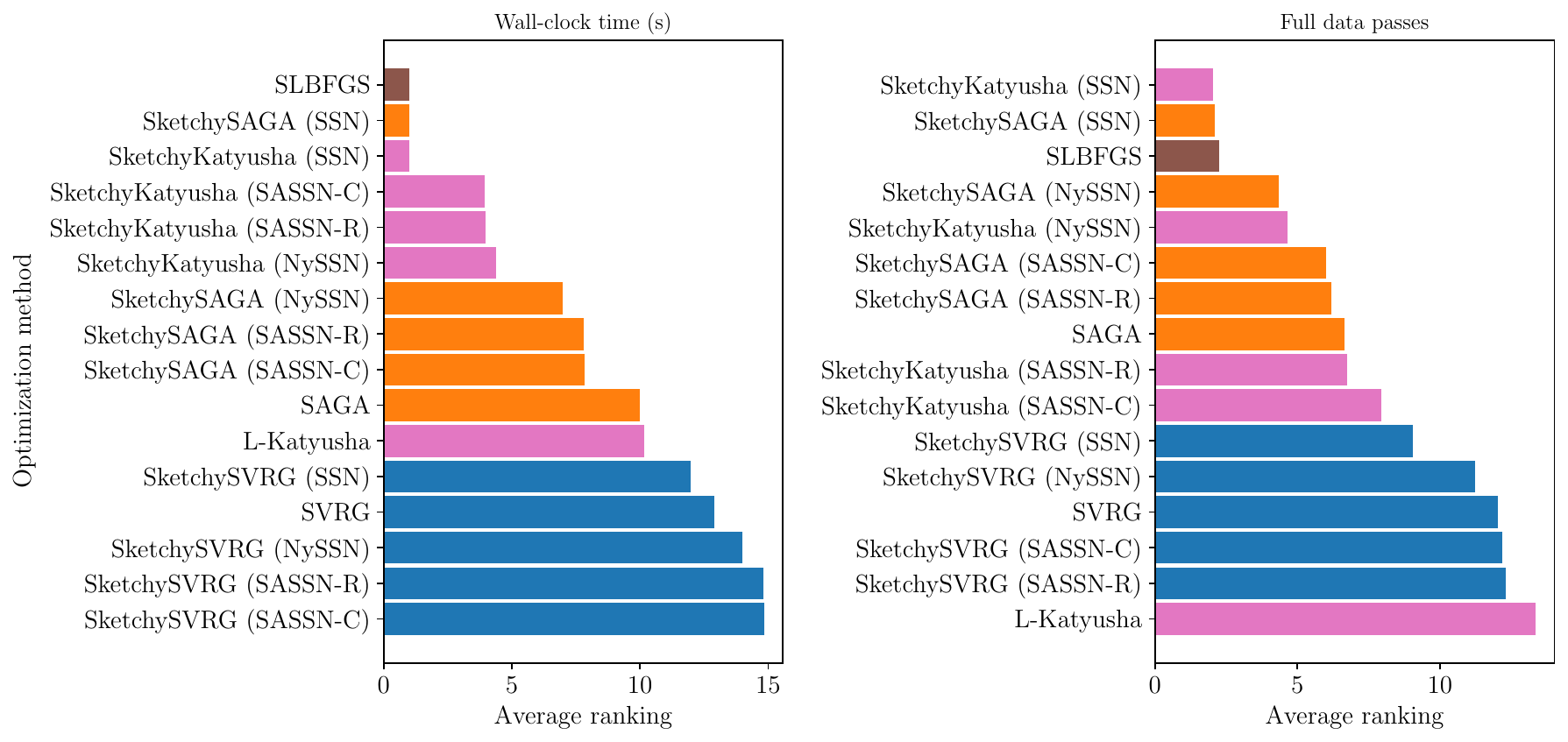}
    \caption{Average ranking of each method by number of $l^2$-regularized logistic regression problems, with sparse data, solved with respect to wall-clock time (left) and full gradient computations (right).}
    \label{fig:ranking_logistic_sparse}
\end{figure}

\begin{figure}[p]
    \centering
    \includegraphics[scale=0.5]{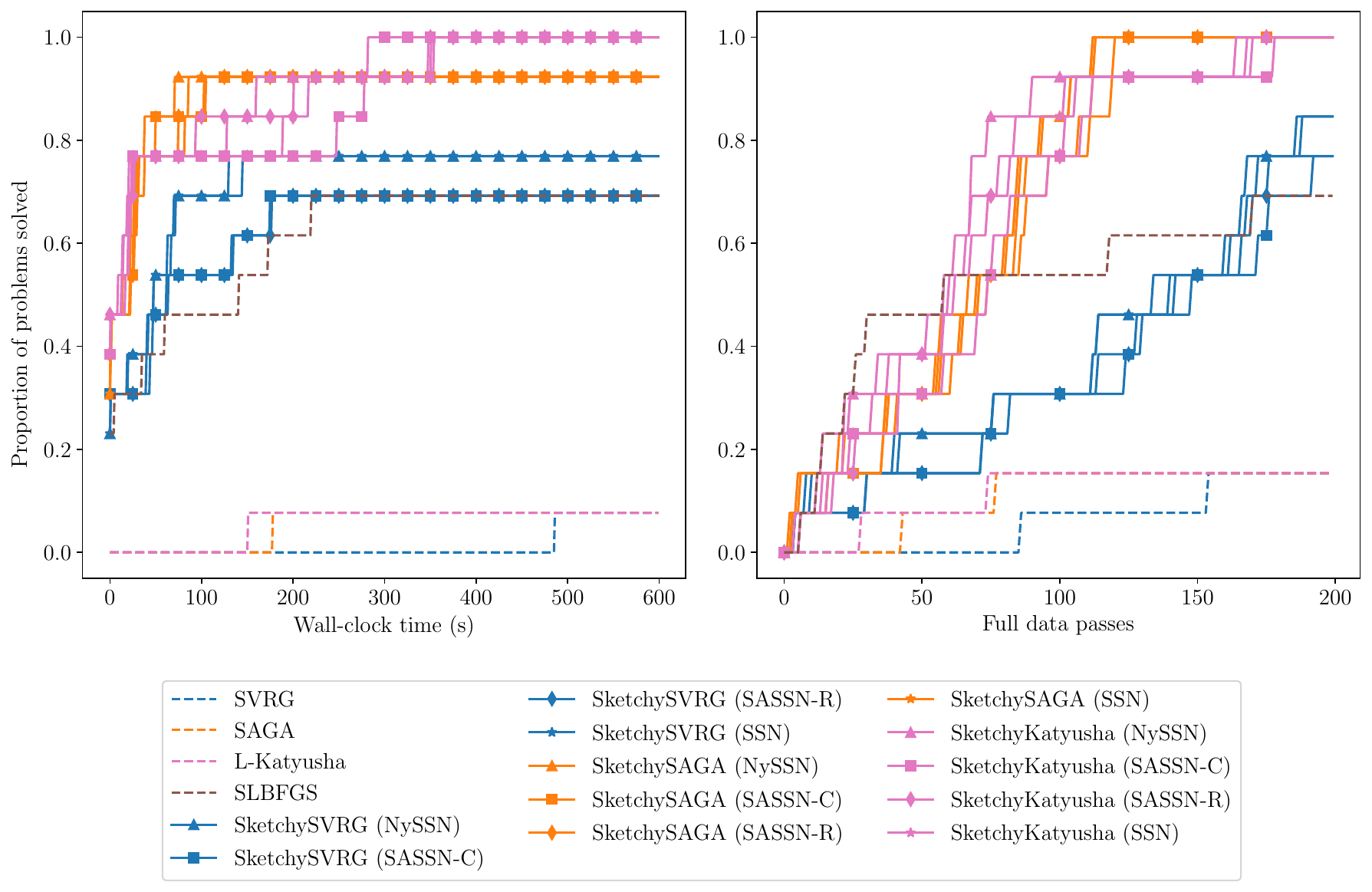}
    \caption{Proportion of $l^2$-regularized logistic regression problems, with dense data, solved by our proposed methods and competitor methods.}
    \label{fig:prop_solved_logistic_dense}
\end{figure}

\begin{figure}[p]
    \centering
    \includegraphics[scale=0.5]{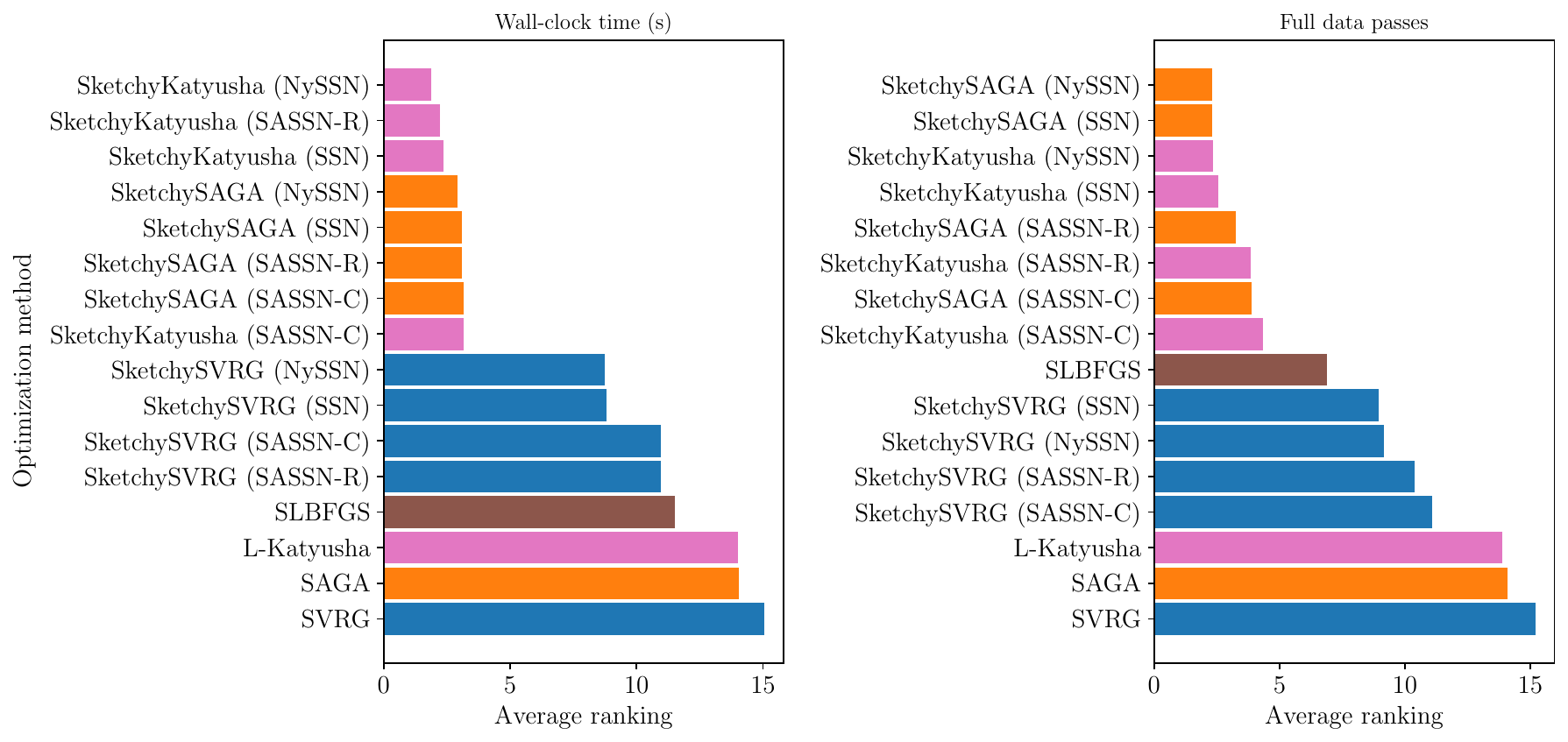}
    \caption{Average ranking of each method by number of $l^2$-regularized logistic regression problems, with dense data, solved with respect to wall-clock time (left) and full gradient computations (right).}
    \label{fig:ranking_logistic_dense}
\end{figure}

\subsubsection{\diagssn{} vs. \ssn{}/\nyssn{}/\sassnc{}/\sassnr{}}
\label{subsubsection:precond_comparison_appdx}

\cref{fig:prop_solved_least_squares_appdx,fig:prop_solved_logistic_appdx} compare the performance of the \diagssn{} preconditioner to the other preconditioners proposed in the main text. 
We observe that using the \diagssn{} preconditioner significantly degrades the performance of our proposed methods.

\begin{figure}[p]
    \centering
    \includegraphics[scale=0.45]{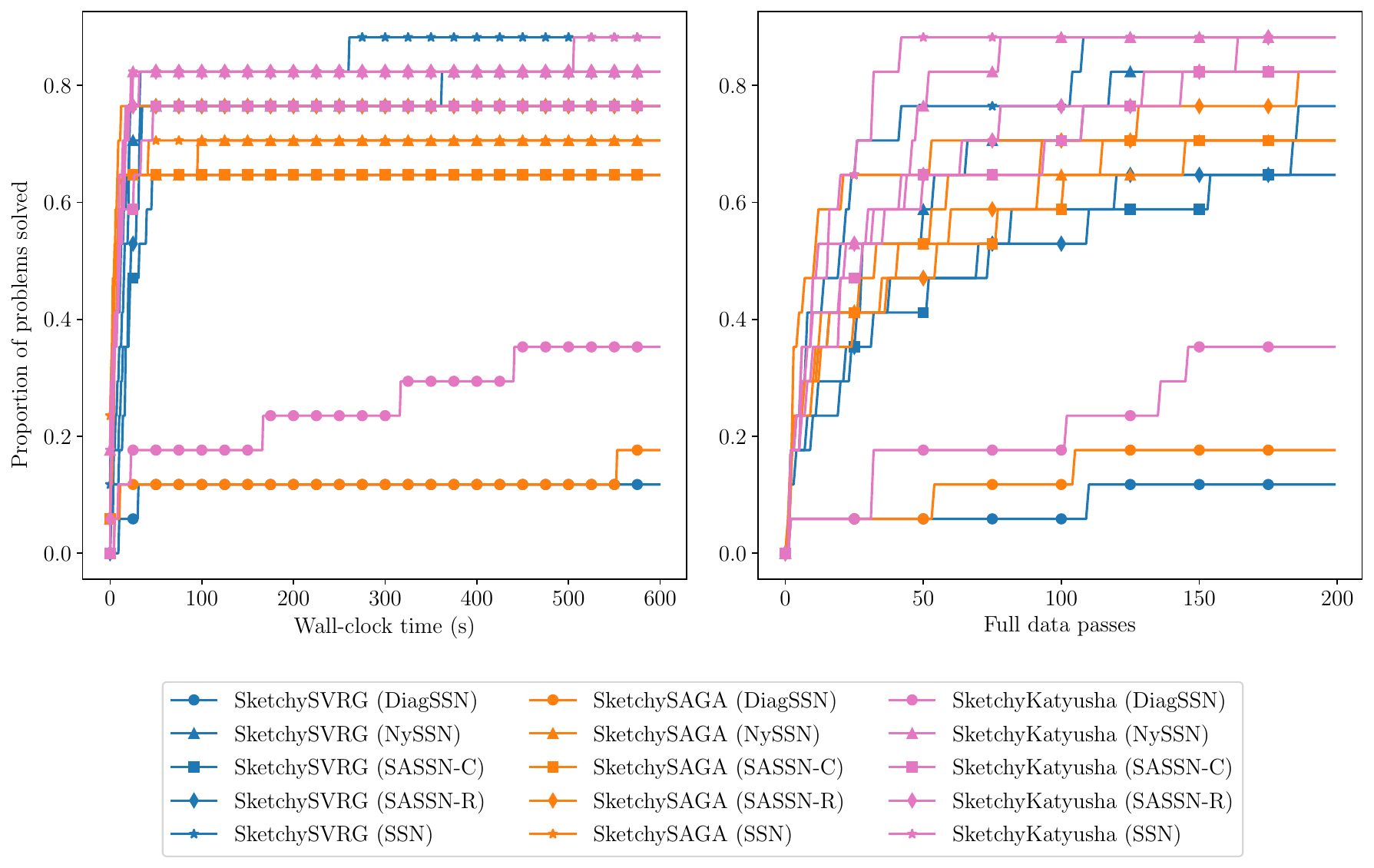}
    \caption{Proportion of ridge regression problems solved by \diagssn{} versus the other preconditioning methods.}
    \label{fig:prop_solved_least_squares_appdx}
\end{figure}

\begin{figure}[p]
    \centering
    \includegraphics[scale=0.45]{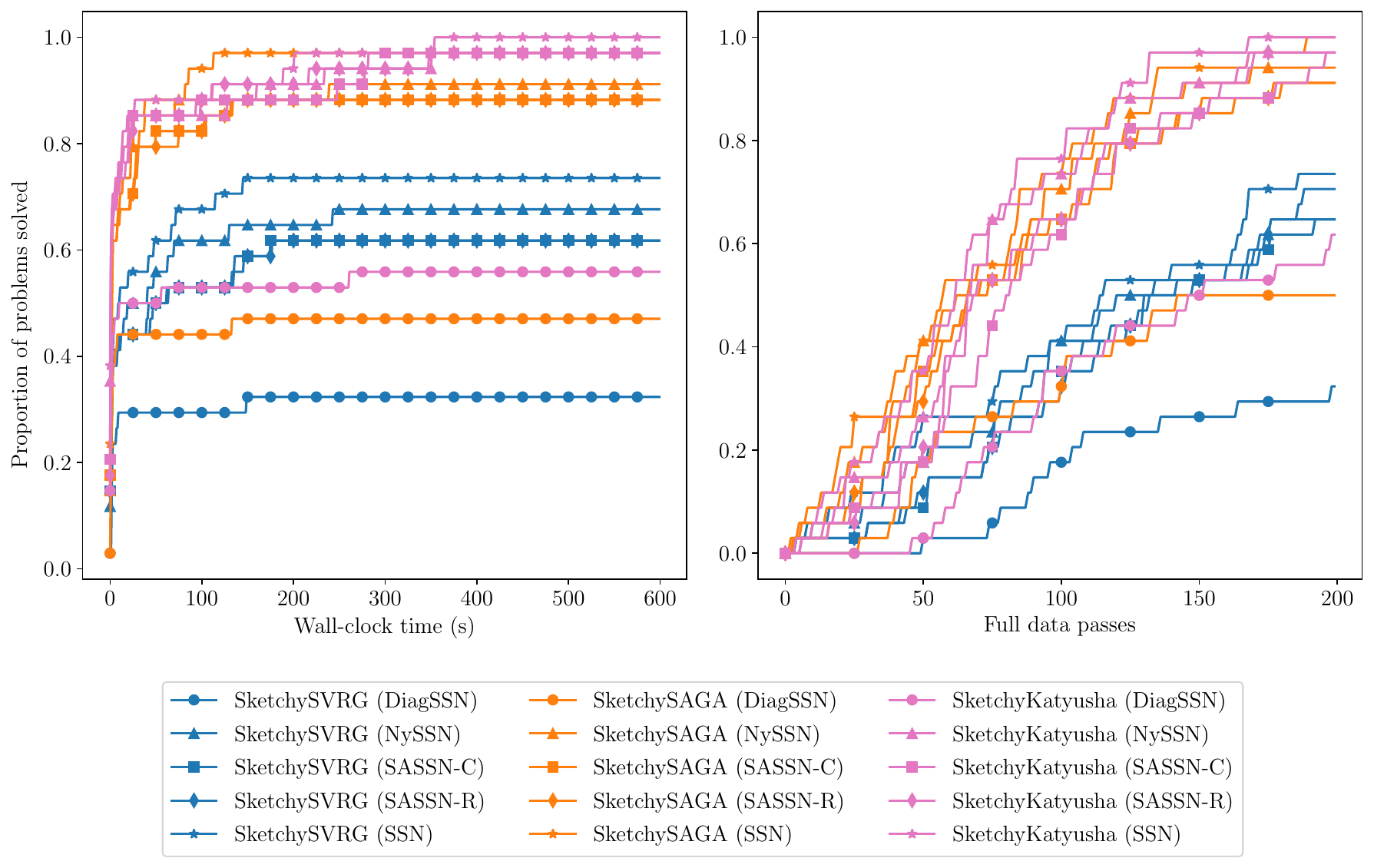}
    \caption{Proportion of $l^2$-regularized logistic regression problems solved by \diagssn{} versus the other preconditioning methods.}
    \label{fig:prop_solved_logistic_appdx}
\end{figure}

\subsection{Suboptimality experiments (additional)}
\label{subsection:subopt_exp_appdx}
The suboptimality experiments in \cref{subsection:subopt} use the E2006-tfidf, YearPredictionMSD, yolanda, ijcnn1, real-sim, and SUSY datasets.
We preprocess these datasets in the same way as \cref{subsection:performance}; see \cref{subsubsection:performance_exp_data_appdx} for details.

\subsubsection{SketchySGD vs. SketchySVRG/SketchySAGA/SketchyKatyusha}
\label{subsubsection:sketchy_opt_comparison_appdx}
We investigate how SketchySVRG, SketchySAGA, and SketchyKatyusha, which have linear convergence guarantees, provide improved convergence in comparison to SketchySGD.
We use the same datasets and run experiments using the same approach as in \cref{subsection:subopt}.

The results are shown in \cref{fig:opt_least_squares} and \cref{fig:opt_logistic}. 
SketchySGD consistently converges to a ball of noise, while SketchySVRG, SketchySAGA, and SketchyKatyusha attain lower suboptimality due to variance reduction. 
In the case of YearPredictionMSD, yolanda, real-sim, and SUSY, SketchySGD does not reach a suboptimality near the other three methods.

\begin{figure}[h]
    \centering
    \includegraphics[scale=0.5]{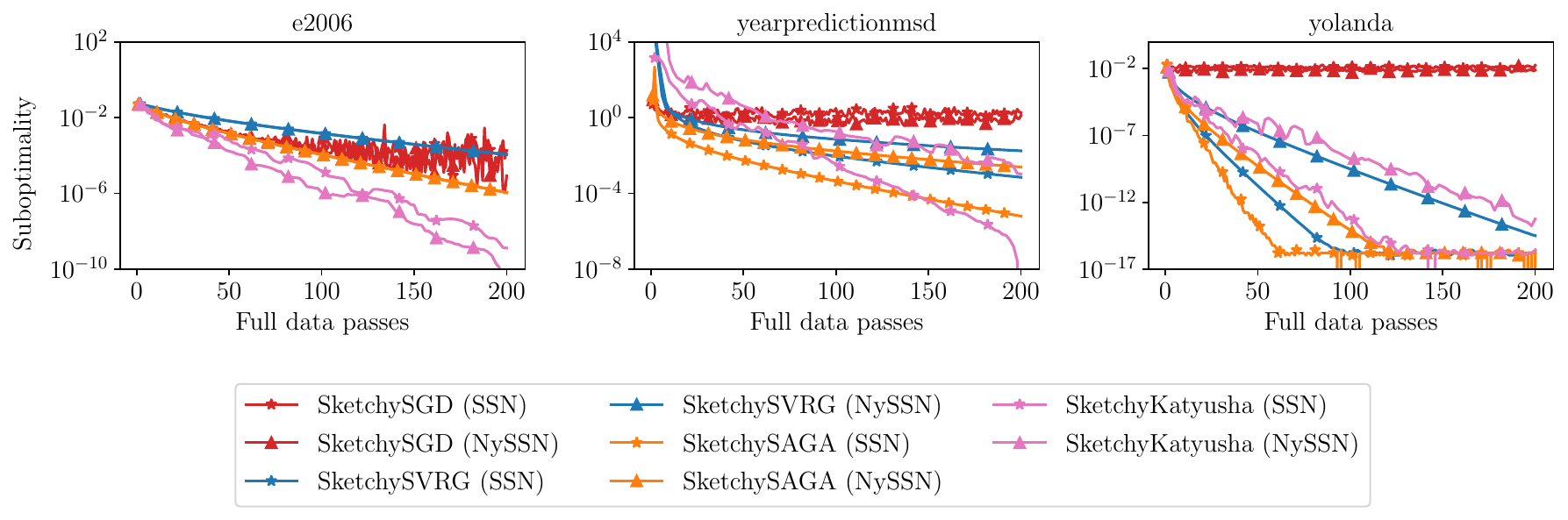}
    \caption{Comparison of SketchySGD, SketchySVRG, SketchySAGA, and SketchyKatyusha with default hyperparameters on ridge regression.}
    \label{fig:opt_least_squares}
\end{figure}

\begin{figure}[h]
    \centering
    \includegraphics[scale=0.5]{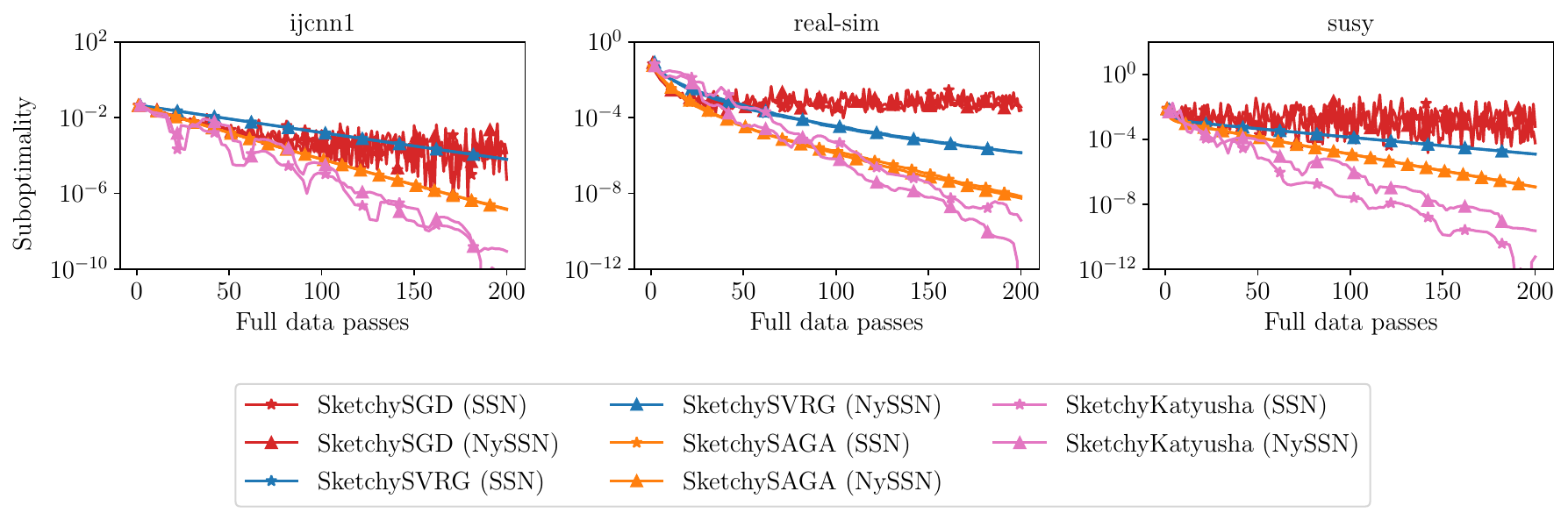}
    \caption{Comparison of SketchySGD, SketchySVRG, SketchySAGA, and SketchyKatyusha with default hyperparameters on $l^2$-regularized logistic regression.}
    \label{fig:opt_logistic}
\end{figure}

\subsection{Showcase experiments (additional)}
\label{subsection:showcase_exp_appdx}
The showcase experiments in \cref{subsection:showcase} use the url, yelp, and acsincome datasets.
The details regarding dimensions and random features are provided in \cref{table:showcase_data}.
None of the datasets have a designated training and test set, so we use a random 80/20 split.
The url dataset is preprocessed to have unit row norm.
The yelp dataset is created by taking the file \texttt{yelp\_academic\_dataset\_review.json} (accessed on May 20, 2023), preprocessing each review into unigrams and bigrams, assigning positive (4 or 5 star) reviews a label of $+1$, assigning negative (1 or 2 star) reviews a label of $-1$, and removing neutral (3 star) reviews.
The acsincome dataset is preprocessed by first removing outliers, i.e., samples whose labels are more than $3$ standard deviations from the mean. 
We then standardize the features and labels to have mean $0$ and standard deviation $1$.  
After standardization, we apply Gaussian random features with a bandwidth of $1$.

\begin{table}[H]
\centering
\scriptsize
    \begin{tabular}{C{2.7cm}C{3cm}C{1.5cm}C{1.5cm}C{2.7cm}}
        Dataset & Source & $n$ & $p$ & RF type/dimension \\ \hline
        url & LIBSVM & 2,396,130 & 3,231,961 &  \\
        yelp & \citep{yelp2023} & 5,038,676 & 2,084,724 & \\
        ACSIncome & OpenML (ID: 43141) & 1,632,867 & 11 & Gaussian/7,500 \\
    \end{tabular}
    \caption{Datasets used in showcase experiments from \cref{subsection:showcase}. $n$ is the number of samples before splitting into training/test sets (after appropriate preprocessing) and $p$ is the number of features.} 
    \label{table:showcase_data}
\end{table}

\subsection{Streaming experiments (additional)}
\label{subsection:streaming_exp_appdx}

\subsubsection{Datasets and preprocessing}
\label{subsubsection:streaming_exp_data_appdx}
We use the HIGGS and SUSY datasets from LIBSVM; details regarding dimensions, splits, and random features are provided in \cref{table:streaming_logistic_data}. 
We preprocess each dataset to have unit row norm, and then apply Gaussian random features with bandwidth $1$ to both datasets.

\begin{table}[H]
\centering
\scriptsize
    \begin{tabular}{C{2.7cm}C{1.5cm}C{1.5cm}C{1.5cm}C{2.7cm}}
        Dataset & $n_{\textrm{tr}}$ & $n_{\textrm{tst}}$ & $p$ & RF type/dimension \\ \hline
        HIGGS &10,500,000 &500,000 &28 &Gaussian/10,000 \\
        SUSY &4,500,000 &500,000 &18 &Gaussian/20,000
    \end{tabular}
    \caption{Details for HIGGS and SUSY datasets in \cref{subsection:large_scale}. $n_{\textrm{tr}}$ denotes the number of training samples, $n_{\textrm{tst}}$ denotes the number of test samples, and $p$ denotes the number of features.}
    \label{table:streaming_logistic_data}
\end{table}

\subsubsection{Comparison to SAGA with default hyperparameters}
\label{subsubsection:streaming_exp_def_comp_appdx}
The comparison to SAGA with its default hyperparameter is presented in \cref{fig:higgs_auto,fig:susy_auto}.
The plots are constructed using the same approach as in \cref{subsection:large_scale}.
We find that SAGA makes little to no progress in decreasing the test loss, while SketchySGD and SketchySAGA, when combined with either the \ssn{} or \nyssn{} preconditioner, decrease the test loss significantly.

\begin{figure}[h]
    \centering
    \includegraphics[scale=0.5]{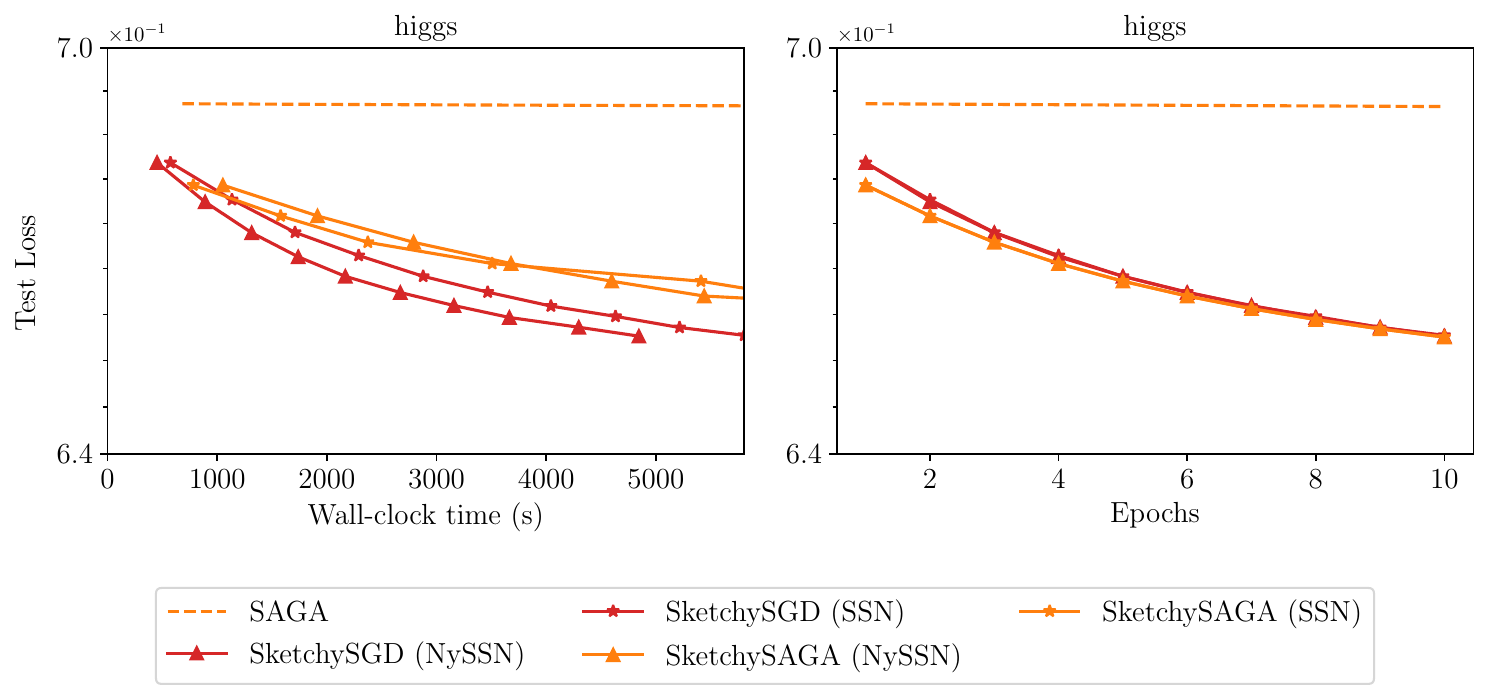}
    \caption{Comparison between SAGA with default learning rate, SketchySGD, and SketchySAGA on HIGGS.}
    \label{fig:higgs_auto}
\end{figure}

\begin{figure}[h]
    \centering
    \includegraphics[scale=0.5]{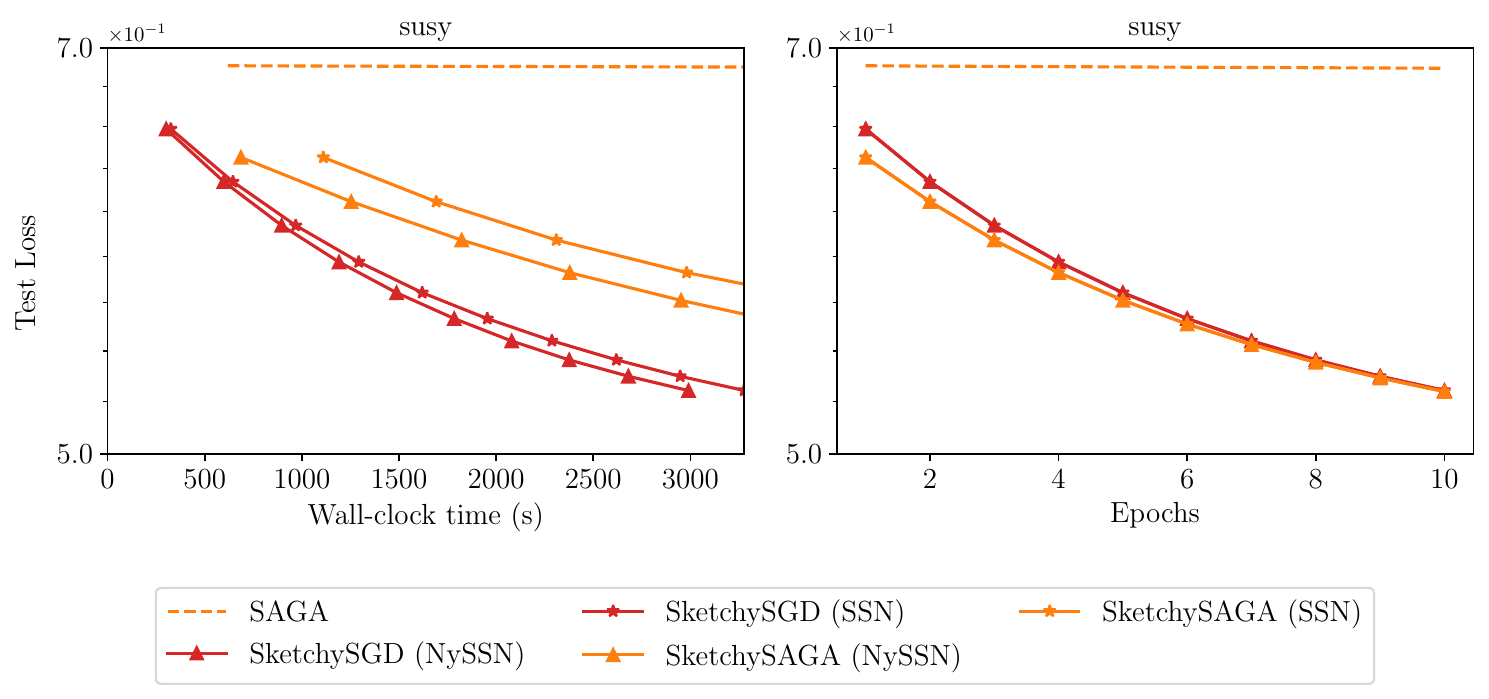}
    \caption{Comparison between SAGA with default learning rate, SketchySGD, and SketchySAGA on SUSY.}
    \label{fig:susy_auto}
\end{figure}

\subsection{Gradient batchsize relative to dataset size}
In this subsection show how the gradient batchsize used for each experiment class compares to $n$, and the baseline $\sqrt{n}$ suggested by \cref{cor:sksvrg_fast_glm}. 
Specifically, we report the median value of $b_g/n$ (percentage) and $b_g/\sqrt{n}$ for all the datasets used in each class of experiment. 
The results are reported in \cref{table:grad_batchsize}.
\label{subsection:grad_batchsize}
\begin{table}[htbp]
\centering
\scriptsize
    \begin{tabular}{C{3cm}C{3cm}C{3cm}}
        Experiment & $b_g/n~(\%)$ & $b_g/\sqrt{n}$ \\ \hline 
        Performance (Ridge)  & 0.418 & 1.03 \\ \hline 
        Performance ($l^2$-Logistic) & 2.31 & 2.43 \\ \hline
        Showcase & 0.171 & 2.65 \\ \hline
        Streaming & 0.065 & 1.60 \\ \hline
    \end{tabular}
    \caption{Comparison of $b_g$ relative to $n$ and $\sqrt{n}$. PROMISE does not require large gradient to achieve the excellent results observed \cref{section:experiments}. Moreover, the batchsizes used by PROMISE are modest multiples of $\sqrt{n}$, which agrees with \cref{cor:sksvrg_fast_glm}, and supports setting $b_g$ to be $\lfloor C\sqrt{n}\rfloor $, where $C\in\{1,2,3\}$.}
    \label{table:grad_batchsize}
\end{table}

\subsection{Sensitivity study (additional)}
\label{subsection:sensitivity_exp_appdx}

\subsubsection{Additional sensitivity plots}
\label{subsubsection:sensitivity_exp_plots_appdx}
We present additional plots that were omitted from the sensitivity study (\cref{subsection:sensitivity}) in the main text. \Cref{fig:sensitivity_r_appdx} shows how changing the rank $r$ impacts the performance of SketchySAGA on the yolanda, ijcnn1, and SUSY datasets. We again observe that the effect of increasing the rank depends on the spectrum (\cref{fig:spectrums}) of each dataset. Both yolanda and YearPredictionMSD have similar patterns of spectral decay, and increasing the rank for yolanda leads to faster convergence. On the other hand, the spectra of ijcnn1 and SUSY are highly concentrated in the first singular value, just like E2006-tfidf, and increasing the rank does not improve convergence. \cref{fig:sensitivity_u_appdx} demonstrates that adjusting the update frequency $u$ does not impact convergence on ridge regression problems.

\begin{figure}[t]
    \centering
    \includegraphics[scale=0.5]{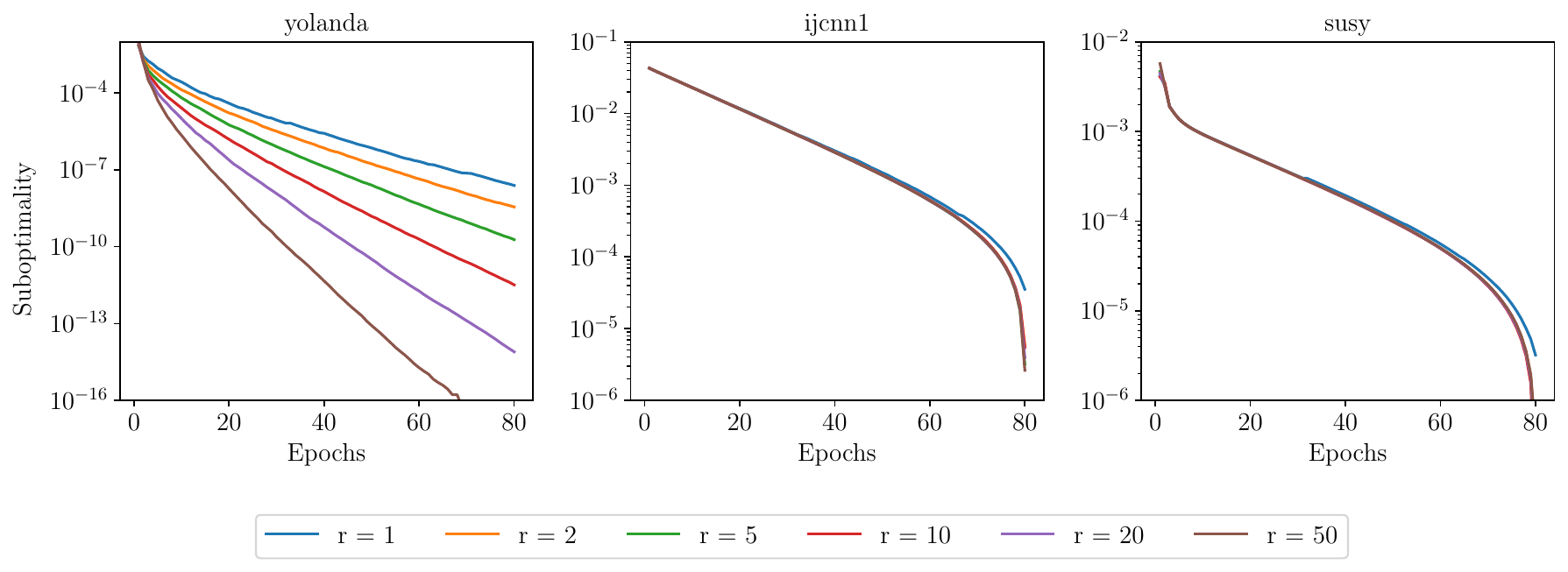}
    \caption{Sensitivity of SketchySAGA to rank $r$.}
    \label{fig:sensitivity_r_appdx}
\end{figure}

\begin{figure}[h]
    \centering
    \includegraphics[scale=0.5]{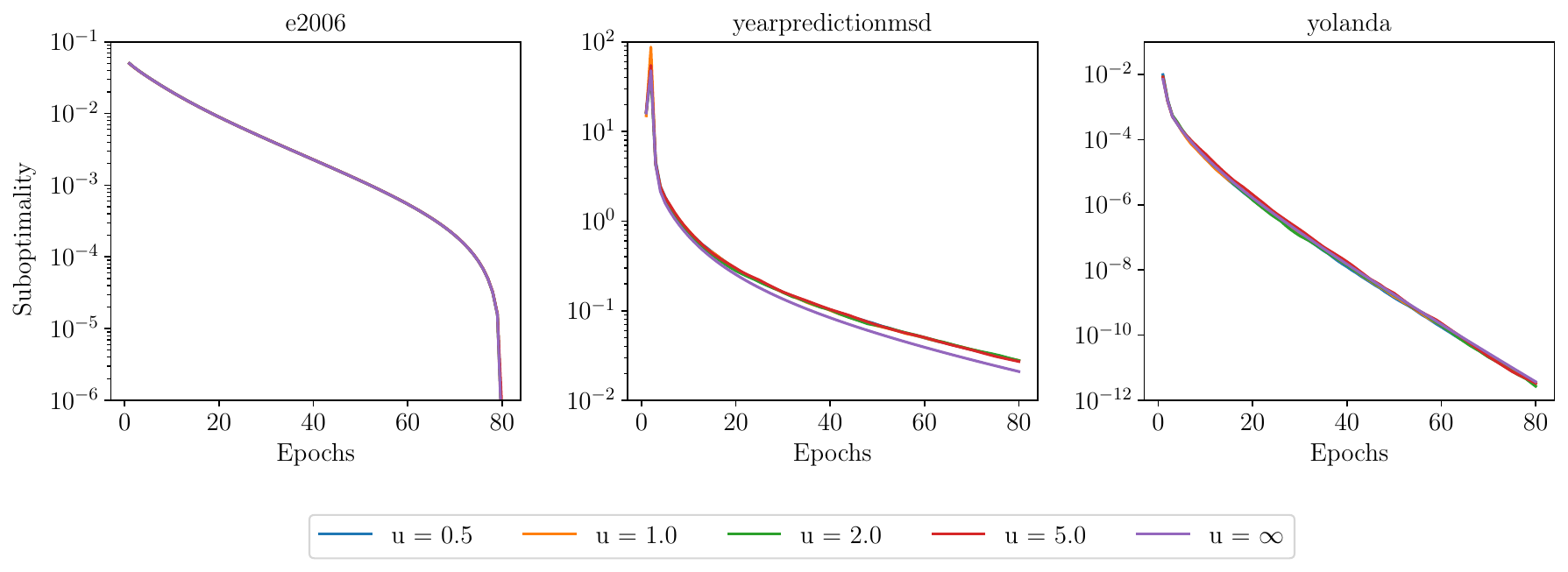}
    \caption{Sensitivity of SketchySAGA to update frequency $u$.}
    \label{fig:sensitivity_u_appdx}
\end{figure}



\subsection{Regularity study (additional)}
\label{subsection:regularity_exp_appdx}
The regularity study in \cref{subsection:regularity} uses the a9a, gisette, ijcnn1, mushrooms, phishing, rcv1, real-sim, and w8a datasets.
We preprocess these datasets in the same way as \cref{subsection:performance}; see \cref{subsubsection:performance_exp_data_appdx} for details.

\clearpage

\else

\fi

\bibliography{references}

\end{document}